\documentclass[12pt]{amsart}

\usepackage{amsmath,amstext,amsfonts,amsthm,amssymb,epigraph}
\usepackage{eucal}
\usepackage{framed}
\usepackage{xcolor}
\usepackage{bbm}
\usepackage[all]{xy}

\numberwithin{equation}{section}

\newcommand{\nc}{\newcommand}
\nc{\one}{\mbox{\bf 1}}
\nc{\invtensor}{\underset{\leftarrow}{\otimes}}
\nc{\const}{\operatorname{const}}

\nc{\ad}{\operatorname{ad}}
\nc{\cl}{\operatorname{cl}}
\nc{\ev}{\operatorname{ev}}
\nc{\tr}{\operatorname{tr}}
\nc{\Gr}{\mathscr{K}}
\nc{\rGr}{\operatorname{rGr}}
\nc{\atyp}{\operatorname{atyp}}
\nc{\tp}{\operatorname{top}}
\nc{\rank}{\operatorname{rank}}
\nc{\corank}{\operatorname{corank}}
\nc{\codim}{\operatorname{codim}}
\nc{\sdim}{\operatorname{sdim}}
\nc{\nult}{\operatorname{Mult}}
\nc{\ds}{\operatorname{ds}}
\nc{\tail}{\operatorname{tail}}
\nc{\Howl}{\operatorname{Howl}}
\nc{\triv}{\operatorname{triv}}
\nc{\spn}{\operatorname{span}}
\nc{\Sym}{\operatorname{Sym}}
\nc{\Core}{\operatorname{Core}}
\nc{\id}{\operatorname{id}}
\nc{\Id}{\operatorname{Id}}
\nc{\adm}{\operatorname{adm}}
\nc{\Ree}{\operatorname{Re}}
\nc{\ree}{\operatorname{re}}
\nc{\Ima}{\operatorname{Im}}
\nc{\ima}{\operatorname{im}}
\nc{\htt}{\operatorname{ht}}
\nc{\at}{\operatorname{at}}
\nc{\an}{\operatorname{an}}
\nc{\Sp}{\operatorname{Sp}}
\nc{\Sk}{\operatorname{Sk}}

\nc{\str}{\operatorname{str}}
\nc{\Ker}{\operatorname{Ker}}
\nc{\rker}{\operatorname{rKer}}
\nc{\osp}{\mathfrak{osp}}
\nc{\sgn}{\operatorname{sgn}}
\nc{\F}{\operatorname{F}}
\nc{\Mod}{\operatorname{Mod}}
\nc{\DS}{\operatorname{DS}}
\nc{\Soc}{\operatorname{Soc}}
\nc{\Inj}{\operatorname{Inj}}
\nc{\Hom}{\operatorname{Hom}}
\nc{\End}{\operatorname{End}}
\nc{\supp}{\operatorname{supp}}
\nc{\smult}{\operatorname{smult}}
\nc{\Dyn}{\operatorname{Dyn}}
\nc{\Card}{\operatorname{Card}}
\nc{\Ann}{\operatorname{Ann}}
\nc{\Arc}{\operatorname{Arc}}
\nc{\arc}{\operatorname{arc}}
\nc{\Ind}{\operatorname{Ind}}
\nc{\Coind}{\operatorname{Coind}}
\nc{\Int}{\operatorname{Int}}
\nc{\hwt}{\operatorname{hwt}}
\nc{\rk}{\operatorname{rank}}
\nc{\ch}{\operatorname{ch}}
\nc{\sch}{\operatorname{sch}}
\nc{\mdim}{\operatorname{mdim}}
\nc{\Stab}{\operatorname{Stab}}
\nc{\Snow}{\operatorname{Snow}}

\nc{\Irr}{\operatorname{Irr}}
\nc{\Spec}{\operatorname{Spec}}
\nc{\Res}{\operatorname{Res}}
\nc{\res}{\operatorname{res}}
\nc{\Aut}{\operatorname{Aut}}
\nc{\Ext}{\operatorname{Ext}}
\nc{\Prec}{\operatorname{Prec}}
\nc{\Fract}{\operatorname{Fract}}
\nc{\gr}{\operatorname{gr}}
\nc{\vol}{\operatorname{vol}}
\nc{\diag}{\operatorname{diag}}
\nc{\deff}{\operatorname{def}}
\nc{\core}{\operatorname{core}}
\nc{\HC}{\operatorname{HC}}
\nc{\Ch}{\operatorname{Ch}}
\nc{\Sch}{\operatorname{Sch}}
\nc{\dpth}{\operatorname{dpth}}
\nc{\aff}{\operatorname{aff}}

\nc{\pari}{\operatorname{par}}
\nc{\pos}{\operatorname{pos}}
\nc{\Ad}{\operatorname{Ad}}

\nc{\wdchi}{\widetilde{\chi}}
\nc{\wdH}{\widetilde{H}}
\nc{\wdN}{\widetilde{N}}
\nc{\wdM}{\widetilde{M}}
\nc{\wdO}{\widetilde{O}}
\nc{\wdR}{\widetilde{R}}

\nc{\wdV}{\widetilde{V}}

\nc{\wdC}{\widetilde{C}}

\nc{\fin}{\operatorname{fin}}
\nc{\nonzero}{\operatorname{nonzero}}

\nc{\Nonzero}{\operatorname{Nonzero}}

\nc{\diam}{\operatorname{diam}}

\nc{\Obj}{\operatorname{Obj}}
\nc{\Dglie}{\operatorname{{\mathcal D}glie}}
\nc{\Fin}{\operatorname{{\mathcal F}in}}
\nc{\pr}{\operatorname{pr}}
\nc{\is}{\operatorname{iso}}
\nc{\Adm}{\operatorname{\mathcal{A}dm}}

\nc{\fg}{\mathfrak g}

\nc{\Sg}{{\cS(\fg)}}
\nc{\Shg}{{\cS(\fhg)}}
\nc{\Ug}{{\cU(\fg)}}
\nc{\Uhg}{{\cU(\fhg)}}
\nc{\Sh}{{\cS(\fh)}}
\nc{\Uh}{{\cU(\fh)}}
\nc{\Uhh}{{\cU(\fhh)}}
\nc{\Zg}{{{\mathcal{Z}}(\fg)}}

\nc{\Vir}{{\mathcal{V}ir}}
\nc{\NS}{{\mathcal{N}S}}

\nc{\tZg}{{\widetilde{\mathcal Z}({\mathfrak g})}}
\nc{\Zk}{{\mathcal Z}({\mathfrak k})}

\nc{\Up}{{\mathcal U}({\mathfrak p})}
\nc{\Ah}{{\mathcal A}({\mathfrak h})}
\nc{\Ag}{{\mathcal A}({\mathfrak g})}
\nc{\Ap}{{\mathcal A}({\mathfrak p})}
\nc{\Zp}{{\mathcal Z}({\mathfrak p})}
\nc{\cR}{\mathcal R}
\nc{\cS}{\mathcal S}
\nc{\cP}{\mathcal P}
\nc{\cT}{\mathcal{T}}
\nc{\cC}{\mathcal C}
\nc{\cA}{\mathcal A}
\nc{\cV}{\mathcal V}
\nc{\cU}{\mathcal U}
\nc{\cZ}{\mathcal Z}
\nc{\cM}{\mathcal M}
\nc{\cL}{\mathcal L}
\nc{\cF}{\mathcal F}

\nc{\cB}{\mathcal{B}}

\nc{\fo}{\mathfrak o}

\nc{\fa}{\mathfrak a}

\nc{\CO}{\mathcal O}
\nc{\CR}{\mathcal R}

\nc{\cK}{\mathcal{K}}
\nc{\cW}{\mathcal{W}}
\nc{\bM}{\mathbf{M}}
\nc{\bL}{\mathbf{L}}
\nc{\bN}{\mathbf{N}}

\nc{\zq}{\mathpzc q}

\nc{\fl}{\mathfrak l}
\nc{\fn}{\mathfrak n}
\nc{\fm}{\mathfrak m}
\nc{\fp}{\mathfrak p}
\nc{\fh}{\mathfrak h}
\nc{\ft}{\mathfrak t}
\nc{\fk}{\mathfrak k}
\nc{\fb}{\mathfrak b}
\nc{\fc}{\mathfrak c}
\nc{\fs}{\mathfrak s}
\nc{\fB}{\mathfrak B}

\nc{\vareps}{\varepsilon}
\nc{\varesp}{\varepsilon}
\nc{\veps}{\varepsilon}

\nc{\fsl}{\mathfrak{sl}}
\nc{\fgl}{\mathfrak{gl}}
\nc{\fso}{\mathfrak{so}}
\nc{\fosp}{\mathfrak{osp}}
\nc{\fsp}{\mathfrak{sp}}
\nc{\fq}{\mathfrak q}
\nc{\fsq}{\mathfrak{sq}}
\nc{\fpsq}{\mathfrak{psq}}
\nc{\fpq}{\mathfrak{pq}}
\nc{\fz}{\mathfrak{z}}
\nc{\fpsl}{\mathfrak{psl}}

\nc{\fhg}{\hat{\fg}}
\nc{\fhn}{\hat{\fn}}
\nc{\fhh}{\hat{\fh}}
\nc{\fhb}{\hat{\fb}}
\nc{\hrho}{\hat{\rho}}

\nc{\hsl}{\hat{\fsl}}
\nc{\fpo}{\mathfrak{po}}
\nc{\dirlim}{\underset{\rightarrow}{\lim}\,}
\nc{\nen}{\newenvironment}
\nc{\ol}{\overline}
\nc{\ul}{\underline}
\nc{\ra}{\rightarrow}
\nc{\lra}{\longrightarrow}
\nc{\Lra}{\Longrightarrow}
\nc{\bo}{\bar{1}}
\nc{\Lla}{\Longleftarrow}

\nc{\Llra}{\Longleftrightarrow}

\nc{\thla}{\twoheadleftarrow}

\nc{\lang}{(}
\nc{\rang}{)}

\nc{\hra}{\hookrightarrow}

\nc{\iso}{\overset{\sim}{\lra}}

\nc{\ssubset}{\underset{\not=}{\subset}}

\nc{\vac}{|0\rangle}

\nc{\simka}{{\ \scriptscriptstyle _{{\sim}}^\text{\tiny{k}}\ }}

\nc{\Thm}[1]{Theorem~\ref{#1}}
\nc{\Prop}[1]{Proposition~\ref{#1}}
\nc{\Lem}[1]{Lemma~\ref{#1}}
\nc{\Cor}[1]{Corollary~\ref{#1}}
\nc{\Conj}[1]{Conjecture~\ref{#1}}
\nc{\Claim}[1]{Claim~\ref{#1}}
\nc{\Defn}[1]{Definition~\ref{#1}}
\nc{\Exa}[1]{Example~\ref{#1}}
\nc{\Rem}[1]{Remark~\ref{#1}}
\nc{\Note}[1]{Note~\ref{#1}}
\nc{\Quest}[1]{Question~\ref{#1}}
\nc{\Hyp}[1]{Hypoth\`ese~\ref{#1}}
\nen{thm}[1]{\label{#1}{\bf Theorem.\ } \em}{}
\nen{prop}[1]{\label{#1}{\bf Proposition.\ } \em}{}
\nen{lem}[1]{\label{#1}{\bf Lemma.\ } \em}{}
\nen{cor}[1]{\label{#1}{\bf Corollary.\ } \em}{}
\nen{conj}[1]{\label{#1}{\bf Conjecture.\ } \em}{}

\nen{claim}[1]{\label{#1}{\bf Claim.\ } \em}{}

\nen{defn}[1]{\label{#1}{\bf Definition.\ } }{}
\nen{exa}[1]{\label{#1}{\bf Example.\ } }{}
\nen{rem}[1]{\label{#1}{\em Remark.\ } }{}
\nen{note}[1]{\label{#1}{\em Note.\ } }{}
\nen{exer}[1]{\label{#1}{\em Exercise.\ } }{}
\nen{sket}[1]{\label{#1}{\em Sketch of proof.\ } }{}
\nen{quest}[1]{\label{#1}{\bf Question.\ } \em}{}

\nen{hyp}[1]{\label{#1}{\bf Hypoth\`ese.\ } \em}{}
\begin{document}

\setcounter{section}{0}
\setcounter{tocdepth}{1}
\title{(Quasi-)admissible modules over symmetrizable Kac-Moody superalgebras}
\author{Maria Gorelik}
\address{M.G.: Department of Mathematics, Weizmann Institute of Science,
Rehovot 761001, Israel; maria.gorelik@weizmann.ac.il}
\author{Victor G.~Kac}
\address{V. K.:  Department of Mathematics, MIT, 77 Mass. Ave, Cambridge, MA 02139;
kac@mit.edu}
%\thanks{}

\begin{abstract} 
The theory of admissible modules over symmetrizable anisotropic
Kac-Moody superalgebras, introduced by Kac and Wakimoto in late 80's, is
a well-developed subject with many applications, including
representation theory of vertex algebras.
Recently this theory was developed in a more general setup
by Gorelik and Serganova. In the present paper we develop
in this more general setup the theory of admissible modules over
arbitrary symmetrizable Kac-Moody superalgebras.
\end{abstract}

%\subjclass[2010]{17B10, 17B20, 17B55, 18D10.}

\medskip

\keywords{Kac-Moody superalgebra, quasi-admissible and admissible weight, vertex algebra, Enright functor}

\maketitle

\section{Introduction}\label{sect:intro}
Let $\fg(A)$ be a Kac-Moody 
algebra over $\mathbb{C}$, with a symmetrizable generalized Cartan matrix $A$,
and let $\fh$ be its Cartan subalgebra. Let $\Delta^{\ree}$ (resp., $\Delta^{\ree}_+$) denote the set of all (resp., positive)
real roots, and $W=\langle s_{\alpha}|\ \alpha\in \Delta^{\ree}\rangle$ be the Weyl group. Let $(\cdot , \cdot)$ be  
a non-degenerate  symmetric invariant bilinear form on $\fg(A)$, defined by the symmetrized $A$, see~\cite{Kbook}, Chapter 2.
Let 
$X:=\{\lambda\in\fh^*|\ (\lambda,\alpha)\geq 0\ \text{ for all, but finite number of }\alpha\in \Delta^{\ree}_+ \text{ and $(\lambda,\alpha)\not=0$ for any  isotropic root $\alpha$}\}$. 
Hereafter, for $a\in\mathbb{C}$ we write $a\geq 0$ if either $\Ree a>0$ or $\Ree a=0$ and $\Ima\ a \geq 0$.
See~\cite{Kbook}, Chapters 1--5 for basic definitions and results on Kac-Moody algebras.

For $\lambda\in\fh^*$ let ${R}_{\lambda}=\{\alpha\in \Delta^{\ree}|\ 
\langle\lambda,\alpha^{\vee}\rangle\in\mathbb{Z}\}$,
let $R_{\lambda}^+=R_{\lambda}\cap \Delta^{\ree}_+$, 
and $W_{\lambda}=\langle s_{\alpha}|\ \alpha\in R_{\lambda}\rangle\subset W$.
Let $D=\prod_{\alpha\in\Delta_+}(1-e^{-\alpha})^{\dim \fg_{\alpha}}$ be the Weyl denominator. 
Let $L(\lambda)$ denote the irreducible
$\fg(A)$-module with highest weight $\lambda$, see~\cite{Kbook}, Chapter 9.
In~\cite{KW88}, Section 2, the following theorem was proved.

\subsection{}
\begin{thm}{thm:0.1}
Let $\lambda\in\fh^*$ be such that  
\begin{equation}\label{eq:0.1}
\lambda+\rho\in X,\end{equation}
\begin{equation}\label{eq:0.2}
\langle \lambda+\rho,\alpha^{\vee}\rangle>0\ \ \text{ for all } \alpha\in R^+_{\lambda}.
\end{equation}
Then
\begin{equation}\label{eq:0.3}
e^{\rho}D\ch L(\lambda)=\sum\limits_{w\in W_{\lambda}} (\det w) e^{w(\lambda+\rho)}.
\end{equation}
\end{thm}
 Note that, by this theorem, if $\lambda\in\fh^*$ satisfies~(\ref{eq:0.1}), (\ref{eq:0.2}), then 
 \begin{equation}\label{eq:0.4}
e^{\rho}D\ch L(\lambda)\ \text{ is $W_{\lambda}$-anti-invariant.}
\end{equation}

We call $\lambda$ satisfying~(\ref{eq:0.1}), (\ref{eq:0.2}) a {\em quasi-admissible weight}, and,
as in~\cite{KW88}, we call $\lambda$ an {\em admissible weight} if, in addition, the $\mathbb{Q}$-span of $R_{\lambda}$ coincides with the $\mathbb{Q}$-span of $\Delta^{\ree}$.

\subsection{}
If $\fg=\fg(A)$ is an affine Lie algebra (see~\cite{Kbook}, Chapter 6--8), then  condition~(\ref{eq:0.1})
means 
\begin{equation}\label{eq:0.5}
k+h^{\vee}>0,
\end{equation}
where $k=\lambda(K)$, the level of $\lambda$, and $h^{\vee}$ is the dual Coxeter number; condition~(\ref{eq:0.2}) means 
\begin{equation}\label{eq:0.6}
\langle\lambda+\rho,\alpha^{\vee}\rangle\not\in\mathbb{Z}_{\leq 0}\ \ \text{ for all }\alpha\in \Delta^{\ree}_+,
\end{equation}
or, equivalently, that $\lambda+\rho$ is the maximal element in its $W$-orbit.  It follows from 
the character formula~(\ref{eq:0.3}) that $e^{a\delta} \ch L(\lambda)$ is a ratio of Jacobi forms for a suitable constant
$a$, having non-positive weight if $\lambda$ is quasi-admissible,
and zero weight if   $\lambda$ is admissible (see~\cite{KW89}). Thus, the (normalised)
character of the $\fg$-module $L(\lambda)$ with admissible highest weight $\lambda$ is a modular  function.

\subsection{}
A result, similar to~\Thm{thm:0.1}, was also proved in~\cite{KW88} for anisotropic Kac-Moody superalgebras $\fg=\fg(A,\tau)$,
introduced in~\cite{K78}.  Here $A$ is a symmetrizable generalized Cartan matrix, satisfying, in addition to the
usual requirements, that all entries in the $i$-th row
are even if $i\in\tau$, and $\tau$ is the set of indices, for which the Chevalley generators $e_i$ and $f_i$ are odd. 
For these superalgebras, ~\Thm{thm:0.1} still holds, with the following change in the definition of $R_{\lambda}$ 
(see~\cite{KW88}, Section 6): let $\Delta^{\ree}_{ev}$ and $\Delta^{\ree}_{odd}$ denote the sets
of even and odd real roots, and let $R_{\lambda}=R_{\lambda;ev}\cup R_{\lambda;odd}$,
where
$$R_{\lambda;ev}=\{\alpha\in \Delta^{\ree}_{ev}|\ \langle \lambda+\rho,\alpha^{\vee}\rangle\in \mathbb{Z}, \alpha/2\not\in \Delta^{\ree}\}, 
R_{\lambda;odd}=\{\alpha\in \Delta^{\ree}_{odd}|\ \langle \lambda+\rho,\alpha^{\vee}\rangle\in \mathbb{Z}_{odd}\}.$$

\subsection{}
Let $\CO_{\adm}$ denote the full subcategory of the category $\CO$, whose objects are modules over a symmetrizable
Kac-Moody algebra, for which all irreducible subquotients are quasi-admissible. 
Using~\cite{DGK}, it is proved in~\cite{KW89}
that the category $\CO_{\adm}$ is semisimple.  The same proof works for all anisotropic Kac-Moody superalgebras.

\subsection{}
The theory of admissible modules over a non-twisted affine Lie algebras $\fg$ is related to representation theory
of vertex algebras in the following way. Recall that the $\fg$-module $L(k\Lambda_0)$ carries a structure of a simple vertex 
algebra $V_k(\fg)$. A number $k\in\mathbb{C}$ is called an {\em admissible level}
for $\fg$ and $V_k(\fg)$ if $k\Lambda_0$ is an admissible weight.

For a simple Lie algebra $\dot{\fg}$ an admissible level $k$ of the corresponding non-twisted affine Lie algebra $\fg$ and the simple vertex algebra $V_k(\fg)$ is of the form
$k=-h^{\vee} +\frac{p}{q}$, where $p$ and $q$ are coprime positive integers, and $k$ is either
{\em principal}, when $q$ is coprime with the lacety $r^{\vee}$ of $\dot{\fg}$ and 
$p\geq h^{\vee}$, 
or {\em subprincipal}, when $q$ is divisible by $r^{\vee}>1$ and 
$p\geq h$,
where $h^{\vee}$ and $h$ are the dual Coxeter number and the Coxeter number of 
$\fg$ respectively~\cite{KW08}.
An admissible weight $\lambda$ is called {\em principal} (resp., {\em subprincipal}) admissible
if  its level $k$ is principal (resp,. subprincipal) admissible and
$R_{\lambda}$ is isometric to $R_{k\Lambda_0}$.
(Subsets $S$ and $S_1$ of vector spaces with symmetric bilinear forms are called {\em isometric}
if there exists a linear map $\psi$, preserving the bilinear forms,
of the span of $S$ to the span of $S_1$, which restricts to a bijection between $S$ and $S_1$
\footnote{Notice that, in our definition, the relation "$S$ is isometric to $S_1$" is  transitive, but not symmetric, since the map
$\psi$ may have a non-zero kernel.}.)

Adamovi\'c and Milas in~\cite{AM}  proved for ${\fg}=\mathfrak{sl}_2^{(1)}$, that for $k$, such that
$k\Lambda_0$ is an admissible weight for ${\fg}$, a ${\fg}$-module $L(\lambda)$ is actually a $V_k(\fg)$-module if and only if
$\lambda$ is an admissible weight of level $k$, and they 
stated a conjecture  for arbitrary affine Lie algebras.
This conjecture was proved by Arakawa in~\cite{A}, using~\cite{AM}. This is the following theorem.

{\em Arakawa Theorem.} Let $k$ be such that $k\Lambda_0$ is an admissible weight of a non-twisted affine Lie 
algebra ${\fg}$. Then an irreducible highest weight ${\fg}$-module $L(\lambda)$ is a $V_k(\fg)$-module if and only if 
$\lambda$ is an admissible weight of level $k$, such that $R_{\lambda}$ is isometric to  $R_{k\Lambda_0}$.

This theorem was extended to the only anisotropic simple Lie superalgebras 
$\dot{\fg}$, which are not  Lie algebras,
$\mathfrak{osp}(1|2n)$ in~\cite{GSsnow}.

\subsection{}\label{1.6}
In the present paper we study quasi-admissible 
(called snowflake in~\cite{GSsnow}) and admissible $\fg$-modules
over arbitrary  symmetrizable Kac-Moody superalgebras.

Recall that for any $\ell\times\ell$ matrix $A=(a_{ij})$
over $\mathbb{C}$ one associates a Lie algebra $\fg(A)$
with Cartan subalgebra $\fh$ and Chevalley generators $e_i,f_i$, $i=1\ldots,\ell$, described in~\cite{Kbook}, Chapter 1,
and called a {\em contragredient Lie algebra}. Given a subset $\tau\subset\{1,\ldots,\ell\}$,
one associates to the pair $(A,\tau)$ a contragredient Lie superalgebra
by letting the $e_i$ and $f_i$ with $i\in\tau$  be odd elements~\cite{Ksuper},\cite{K78}.
The contragredient Lie superalgebra structure  is called {\em integrable} if all its Chevalley generators
are ad-locally nilpotent. An integrable contragredient Lie algebra is called a Kac-Moody algebra
if $a_{ii}\not=0$ for all $i$.

If $a_{ii}=0$ for some $i$, then the Lie superalgebra $\fg(A,\tau)$ can be given another structure of
a contragredient Lie superalgebra by applying the odd reflexion,
 corresponding to the index $i$.
The superalgebra $\fg(A,\tau)$ is called a {\em Kac-Moody superalgebra} if all
contragredient Lie superalgebra structures, obtained from it by a sequence of odd reflexions, 
are integrable and $a_{ii}\not=0$ for $i\not\in\tau$ for all of them~\cite{S}.

\subsection{}
If for the matrix $A$ there exists a non-singular diagonal matrix $D=diag(d_1,\ldots,d_{\ell})$
such that the matrix $DA$ is symmetric, $A$ is called {\em symmetrizable}.
In this case $\fg(A,\tau)$  is called  symmetrizable, and this Lie superalgebra
carries a (essentially unique) non-degenerate supersymmetric invariant bilinear form $(\cdot,\cdot)$, whose restriction to $\fh$ is non-degenerate (cf.~\cite{Kbook}, Chapter 2).

\subsection{}
The class of symmetrizable indecomposable Kac-Moody superalgebras consists of three classes~\cite{S}:
\begin{itemize}
\item[(a)]  finite-dimensional, classified in~\cite{Ksuper};
\item[(b)] anisotropic, i.e. those integrable $\fg(A,\tau)$
for which all the diagonal entries 
of $A$ are non-zero, studied in~\cite{K78};
\item[(c)] symmetrizable indecomposable infinite-dimensional contragredient Lie superalgebras
of finite growth (=Gelfand-Kirillov dimension), 
classified in~\cite{vdLeur} in the super case and in~\cite{K68}
in the non-super case; they are called {\em symmetrizable affine Lie superalgebras} 
since all of them can be obtained by the same construction
as the affine Lie algebras~\cite{K68},\cite{vdLeur}.
\end{itemize}

 \subsection{}
  Admissible modules over affine Lie algebras, i.e. those $L(\lambda)$, 
  such that $\lambda$ satisfies~(\ref{eq:0.5}), (\ref{eq:0.6})
 and, in addition, $\mathbb{Q}R_{\lambda}=\mathbb{Q}\Delta^{\ree}$, have been
classified in~\cite{KW89}. It was proved there that the span of the (normalised) characters
of admissible modules  for given level is a finite-dimensional 
$SL_2(\mathbb{Z})$-invariant space (this is called modular
invariance property).
Admissible modules over some affine Lie superalgebras and their characters
 were studied a series of papers
starting with~\cite{KW01} through~\cite{KW17}.

It was shown in these papers that, though modular
invariance fails, for a modification of the 
 characters of these modules it is restored.

\subsection{} It follows from~(\ref{eq:0.3}) that for an anisotropic Kac-Moody 
superalgebra $\fg$ and a $\fg$-module $L(\lambda)$ with 
a quasi-admissible highest weight $\lambda$ we have 
\begin{equation}\label{eq:0.7}
D e^{\rho} \ch L(\lambda) \text{ is $W_{\lambda}$-anti-invariant.}\end{equation}
In~\cite{GSsnow} Gorelik and Serganova found another approach to quasi-admissible modules
$L(\lambda)$ over anisotropic symmetrizable Kac-Moody superalgebras, defining them by an analogue of the property~(\ref{eq:0.7}),
and the property
\begin{equation}\label{eq:0.8}
\lambda \text{ is non-critical}.
\end{equation}
(We find this term more natural than "snowflake", used in~\cite{GSsnow}.)
The analogue of the property~(\ref{eq:0.7}) is anti-invariance
of $D_{\pi} e^{\rho_{\pi}}\ch L(\lambda)$ with respect to the group $W_{\lambda}\cap W[\pi]$
where $\pi$ stands for the set of simple roots of the "largest" component of $\Delta_{\ol{0}}$.

Condition~(\ref{eq:0.8}) means that $2(\lambda+\rho,\alpha)\not=(\alpha,\alpha)$
for all positive imaginary roots $\alpha$. For a (symmetrizable) affine Kac-Moody 
superalgebra, this simply means that $k$ is a non-critical level: $k+h^{\vee}\not=0$,
which is weaker than~(\ref{eq:0.5}). It is proved in ~\cite{GSsnow}
that for an anisotropic symmetrizable Kac-Moody superalgebra $\fg$
and a non-critical weight $\lambda$ conditions~(\ref{eq:0.2}) and~(\ref{eq:0.7})
are actually equivalent.

Note also that admissible modules over Kac-Moody superalgebras were introduced in~\cite{KRW03} and classified for $\fsl(2|1)^{(1)}$ there.
For $\fg=\fsl(m|n)^{(1)}$ this definition coincides with that of the present paper.

\subsection{}
Our main results are the following:
\begin{itemize}
\item[$\bullet$]  \Thm{thm:Deltalambda}, 
stating that for any non-critical $\lambda\in\fh^*$ of a
Kac-Moody superalgebra $\fg$
each indecomposable component of the "integral root system"
$\Delta_{\lambda+\rho}$ is the set of real roots of a symmetrizable
Kac-Moody superalgebra $\fg'$, 
unless this component is of type $A(1|1)^{(2)}$ (which may happen only
if $\fg=A(2m-1|2n-1)^{(2)}$).
\item[$\bullet$] \Thm{thm:criterion}, stating
that quasi-admissibility of a $\fg$-module $L(\lambda)$ is equivalent to partial 
integrability (i.e., integrability with respect to 
the "largest" component of $\fg'_{\ol{0}}$) of the $\fg'$-module 
$L'(\lambda')$, where $\fg'$ is a Kac-Moody superalgebra,
if $\Delta_{\lambda+\rho}$ does not have indecomposable components  of type $A(1|1)^{(2)}$.
\item[$\bullet$] Theorem~\ref{thm:arak}, stating conditions on a $V^k(\fg)$-module
of admissible level $k$ to descend to the  $V_k(\fg)$-module, reducing the problem for the affine Kac-Moody superalgebra $\fg$ to its affine subalgebra
$\fg^{\#}$.
\item[$\bullet$] 
Arakawa-type Theorem~\ref{thm:cor:arak} for affine Kac-Moody superalgebras.
\item[$\bullet$] \Prop{prop:VkVk}, describing the maximal ideal
of $V^k(\fg)$ of admissible level $k$, as demonstrated in~\Prop{prop:sl21} in the case $\fg=\fsl(2|1)^{(1)}$ (cf.~\cite{Qing}),
which can be used in the computation of the Zhu algebra of $V_k(\fg)$.
\item[$\bullet$] \Prop{prop:boundary} on $V_k(\fg)$-modules for the boundary admissible level $k$ (cf.~\cite{Qing} for $\fg=\fsl(2|1)^{(1)}$).
\end{itemize}

We classify in $\S$\ref{classadmneBmn} and Section~\ref{classadmBmn} the admissible levels
of non-twisted affine Kac-Moody superalgebras 
${\fg}$, which are not Lie algebras
and are not of type $D(2|1,a)^{(1)}$:
\begin{itemize}
\item[$\bullet$]
For ${\fg}=A(m|n)^{(1)}$, 
$D(m|n)^{(1)}$ ($(m,n)\not=(2,1)$), $G(3)^{(1)}$, $F(4)^{(1)}$
any admissible level is principal admissible and is of
the form $k=-h^{\vee}+\frac{p}{q}$,
where $p,q$ are coprime positive integers, $p\geq h^{\vee}$,
and  $q$ is coprime with the lacety of $\pi$ (i.e. 
$q$ is odd for $F(4)^{(1)}$, $D(m|n)^{(1)}$ with $n\geq m$, and $q$
is not divisible by $3$ for $G(3)^{(1)}$).

\item[$\bullet$] Let ${\fg}=B(m|n)^{(1)}$ with $m\leq n$ (so that $h^{\vee}=n-m+\frac{1}{2}$). The level 
$k$ is principal admissible if and only if $k=-h^{\vee}+\frac{p}{2q}$,
where  $p,q$ are coprime odd positive integers and $p\geq 2h^{\vee}$.
If $m\not=0$ and $(m,n)\not=(1,1)$, then any admissible level is principal admissible.
For $B(0|n)^{(1)}$,  $k$ is subprincipal admissible if and only if $k=-h^{\vee}+\frac{p}{2q}$, where  $p,q$ are coprime  positive integers, $pq$ is even and $p\geq n$.
For $B(1|1)^{(1)}$,  $k$ is subprincipal admissible if and only if $k=-h^{\vee}+\frac{p}{2q}$, where  $p,q$ are coprime  positive integers and  $p$ is even.

\item[$\bullet$] Let ${\fg}=B(m|n)^{(1)}$ with $m>n>0$ (so that $h^{\vee}=2(m-n)-1$). The level 
$k$ is principal admissible if and only if $k=-h^{\vee}+\frac{p}{q}$,
where  $p,q$ are coprime  positive integers, $q$ is odd,  and $p\geq h^{\vee}$.
The level 
$k$ is subprincipal admissible if and only if $k=-h^{\vee}+\frac{p}{q}$,
where  $p,q$ are coprime  positive integers, $q$ is even,  and $p\geq 2(m-n)$.
\end{itemize}

\subsubsection{Comment}\label{commentintro}
The conditions on $p$ can be unified by the formula: $p\geq u$, where $u$
is the dual Coxeter number of $\fg'$ for $k\Lambda_0+\rho$.
If $k$ is principal  admissible, then  $\Delta_{k\Lambda_0+\rho}\cong {\Delta}^{\ree}$
so $u=h^{\vee}$. 

It turns out that if $k$ is subprincipal admissible, then  $\Delta_{k\Lambda_0+\rho}$ can be described as follows:
we recall that  the set of simple roots for $\fg$ takes the form
${\Sigma}=\dot{\Sigma}\cup\{\delta-\theta\}$
where $\dot{\Sigma}$ is the set of simple roots for 
$\dot{\fg}$ and $\theta$ is the maximal root for $\dot{\fg}$; 
the set of simple roots for $\fg'$ is of the form
$\dot{\Sigma}\cup\{\delta-\theta'\}$ where $\theta'$
is a short root. We have the following table
(where $\fg^L$ stands for the Langlands dual of $\fg$ - the 
Dynkin diagram of $\fg^L$ is obtained from the Dynkin diagram of
$\fg$ by reversing all arrows):

$$\begin{array}{|c||l|l|l|l|c|c|}
 \hline
\fg & B^{(1)}_m& C_m^{(1)}  &   G_2^{(1)}& F_4^{(1)} & B(0|n)^{(1)} & B(m|n)^{(1)}, m>0 \\
 \hline
\fg' & D_{m+1}^{(2)}& A_{2m-1}^{(2)} & D_4^{(3)} & E_6^{(2)} & A(0|2n-1)^{(2)} & D(m+1|n)^{(2)}\\
 \hline
 \fg^L &  A_{2m-1}^{(2)} & D_{m+1}^{(2)}& D_4^{(3)} & E_6^{(2)} & - & - \\ 
 \hline
 u & 2m & 2m & 6 & 12 & n-\frac{1}{2} & \frac{n-m}{2} \text{ if }n\geq m,\  m-n  \text{ if } m>n\\
 \hline
 \end{array}$$
For $n>0$ the Dynkin diagram 
obtained from the Dynkin diagram of
$B(m|n)^{(1)}$  by reversing all arrows 
is not the Dynkin diagram of a Kac-Moody superalgebra,
so $\fg^L$ is not defined in this case.
Notice that $A_{2m-1}^{(2)}$, $ D_{m+1}^{(2)}$ have the same dual Coxeter number (=$2n$),
so  $u=h^{\vee,L}$ is the dual Coxeter number of $\fg^L$ in all cases except for
 $B(m|n)^{(1)}$ with $n>0$.

\subsection{Index of definitions and notation} \label{sec:app-index}
Throughout the paper the ground field is $\mathbb{C}$. We will frequently use the following notation.

\begin{center}
\begin{tabular}{lcl} 
$N_{\nu}$, $\Omega(N)$ ,
$\CO^{\inf}$, base for a triangular decomposition, $\CO_{\Sigma}$, $\cR(\Sigma)$, $\CO$,
$\CO^{\fin}$ & & \ref{triangular}\\
Weyl denominator $D$ & & \ref{Weyldenominator} \\
reflexion:=$r_{\alpha}\Sigma$ & & \ref{reflexion}\\
spine:=$\Sp$, principal roots:=$\Sigma_{\pr}$, Weyl group:=$W$ & & \ref{Spine}\\ 
roots real:=${\Delta}^{\ree}$,  $\ol{\Delta}^{\ree}$,
isotropic:=$\Delta^{\is}$, non-isotropic:=$\Delta^{\an}$;  imaginary:=${\Delta}^{\ima}$
& & \ref{Spine}\\ 
the Weyl vector $\rho$ & & \ref{Weylvector}\\
 $W[E]$, natural action of the Weyl group & & \ref{Weylgroup}\\
non-critical weight, non-critical module, short root & & \ref{shortroot}\\
$\ol{R}_{\lambda},\  R_{\lambda},\ \Delta_{\lambda}$ & & \ref{Deltalambda}\\
$\Delta_N$ & & \ref{defn:DeltaL}\\
$\psi$ & & \ref{KMsubsystem}\\
$L',\fg'$ & & \ref{glambda}\\
$U(\lambda)$ & & \ref{Ulambda}\\
$\pi$-integrable & & \ref{defn:pi-integrable}\\
$\fg_{\pi}$  & & \ref{newnot}\\
$\rho_{\pi}$, $D_{\pi}$ & & \ref{Wnatural}\\
$\cl$, $\pi^{\#}$,  $\Delta^{\#}$ & & \ref{notataff}\\
$V^k$ & & \ref{vacmod}\\
(principal) admissible level & & \ref{admlevel}\\
$V^k(\fg)$, $V_k(\fg)$, $\CO(\fg)^k$, $\CO^{\inf}(\fg)^k$, $\fg^{\#}$ & & \ref{notvertex}
\end{tabular}
\end{center}

{\em Acknowledgment.}
We are grateful to  Vladimir Hinich and Vera Serganova for 
important suggestions and observations. V.K. was partially 
supported by Simons collaboration grant and by ISF Grant 1957/21.
M. G. was partially supported by  ISF Grant 1957/21 
and by the Minerva foundation with funding from 
the Federal German Ministry for Education and Research.

\section{Contragredient Lie superalgebras $\fg(A,\tau)$}
\subsection{Definition of $\fg(A,\tau)$}
Let $I$ be an index set, which we assume to be finite, unless otherwise stated, and let $A=(a_{ij})_{i,j\in I}$
be a matrix with entries in $\mathbb{C}$.
A triple $(\fh,\Sigma, \Sigma^{\vee})$, such that $\fh$ is a vector space, $\Sigma=\{\alpha_i\}_{i\in I}$ 
(resp., $\Sigma^{\vee}=\{h_i\}_{i\in I}$)
 is a linearly independent subset of $\fh^*$ (resp., $\fh$), such that
 $$\alpha_j(h_i)=a_{ij}, \ \ i,j\in I,$$
is called a {\em realization} of the matrix $A$. 

A realization of $A$ is called {\em reduced} if $\fh$ cannot be replaced by a strictly smaller subspace; if $I$ is finite, then 
for the reduced realization of $A$ one has:
$$\dim \fh=|I|+corank A.$$

Let $\tau$ be a subset of $I$.  Define the Lie superalgebra
$\tilde{\fg}(A,\tau)$ generated by  elements $e_i,f_i$, $i\in I$, and $\fh$, where $\fh$ consists of even elements and $e_i,f_i$ are
even if and only if $i\not\in\tau$, subject to the following relations  ($i,j\in I$):
\begin{equation}\label{eq:Chevalleyrel}
\begin{array}{lll}
(a) & & [h,h']=0\ \text{ for }h,h'\in\fh;\ \ \  [h,e_i]=\alpha_i(h)e_i,\ \ 
 [h,f_i]=-\alpha_i(h)f_i \text{ for }h\in\fh; \\
(b) & & [e_i,f_j]=\delta_{ij}h_i.
\end{array}\end{equation}

Let $J(A,\tau)$ be the sum of all ideals of $\tilde{\fg}(A,\tau)$,
intersecting $\fh$ trivially, and let
$$\fg(A,\tau):=\tilde{\fg}(A,\tau)/J(A,\tau).$$
This is called the {\em 
contragredient Lie superalgebra} with  the Cartan datum $(A,\tau)$. Elements $e_i,f_i$,
 $i\in I$ are called the {\em Chevalley generators}. The Lie algebras
 $\fg(A,\emptyset)$ were introduced in~\cite{K68}, and the general $\fg(A,\tau)$ in~\cite{Ksuper}. 
 
 Note that replacing each $h_i$ by $b_ih_i$, where $b_i$ are non-zero numbers, produces an isomorphic contragredient Lie superalgebra,
 whose Cartan matrix is $A_1=diag(b_i)_{i\in I} A$.
 The Cartan matrices $A$ and $A_1$ are called equivalent.
 
 We will follow~\cite{Kbook},
 Chapters 1 and 2 on basic properties of $\fg(A)$, they remain
 valid for $\fg(A,\tau)$, except that Proposition 1.7 (b) is valid if $A\not=(0)$.

 \subsection{Triangular and root space decompositions}\label{base}

We have the triangular decomposition
$$\fg(A,\tau)=\fn^-\oplus \fh\oplus \fn^+,$$
where $\fn^-$ (resp., $\fn^+$) is generated by $f_i$ (resp., $e_i$),
$ i\in I$. Furthermore, $\fn^+$ and $\fn^-$
are normalized by the Cartan subalgebra $\fh$,
so that we have the root space decomposition:
$$\fn^{\pm}=\bigoplus_{\alpha\in \pm \Delta^+(\Sigma)} \fg_{\alpha},
\ \text{ where }  
\fg_{\alpha}=\{a\in \fg|\ [h,a]=\alpha(h) a\}.
$$
 The set $\Delta^+(\Sigma)=\{\alpha\in\fh^*|\ 0\not=\fg_{\alpha}\subset \fn^+\}$ is called the set of
 {\em positive roots}, corresponding to the base $\Sigma$,
 and the set $\Delta=-\Delta^+(\Sigma)\coprod\Delta^+(\Sigma)$
is the set of {\em all roots}. Note that
$\Delta^+(\Sigma)=\Delta\cap \mathbb{Z}_{\geq 0}\Sigma$,
and that $\dim \fg_{\alpha}<\infty$ for all $\alpha\in\Delta$.
The set $\Sigma$ is  called a {\em base} of $\Delta$.

We define a linear map
$p:\mathbb{Z}\Delta\to \mathbb{Z}_2$ given by
$p(\alpha_i)=1$ if $i\in\tau$ and $p(\alpha_i)=0$
if $i\not\in\tau$. Then all vectors in $\fg_{\alpha}$
are even (resp., odd) if $p(\alpha)=0$  (resp., $p(\alpha)=1$).
Thus $\Delta=\Delta_{\ol{0}}\coprod\Delta_{\ol{1}}$.

% We introduce the Weyl denominator by the usual formulas:
%$$D_0:=\prod_{\alpha\in\Delta^+_{\ol{0}}} 
%(1-e^{-\alpha})^{\dim \fg_{\alpha}},\ \ \ 
%D_1:=\prod_{\alpha\in\Delta^+_{\ol{1}}} (1+e^{-\alpha})^{\dim %\fg_{\alpha}},
%\ \ \ 
%D:=D_0 /D_1.$$

\subsection{Odd reflexions}
We start from the following lemma.

\subsubsection{}
\begin{lem}{lem:ass}
One has in $\fg(A,\tau)$
$$\begin{array}{lcl}
(a) & & [e_s,e_s]=0=[f_s,f_s]\ \ \text{ if } s\in \tau \text{ and } a_{ss}=0, \text{ hence }\  (\ad e_s)^2=(\ad f_s)^2=0.\\
(b) & & [e_s,e_s]\not=0\ \text{ and }[f_s,f_s]\not=0\ \ \text{ if } s\in \tau \text{ and } a_{ss}\not=0.\\
(c) & & [e_i,e_j]=0=[f_i,f_j]\ \ \text{ if } a_{ij}=0=a_{ji}  \text{ and } i\not=j.\\
(d) & & [e_i,e_j]\not=0\ \text{ and } [f_i,f_j]\not=0\ \ \text{ if } a_{ij}\not=0 \text{ and } i\not=j.
\end{array}$$
\end{lem}
\begin{proof}
For all $s\in\tau$ and $j\in J$ we have 
$[f_j,[e_s,e_s]]=-2\delta_{sj}a_{ss} e_s$ in
$\tilde{\fg}(A,\tau)$, so
 $[e_s,e_s]\in J(A,\tau)$ if and only if $a_{ss}=0$. 
This proves (a) and (b).  The proof of (c) and (d) is similar.
\end{proof}

\subsubsection{}
\begin{prop}{prop:ass}
Suppose that $s\in\tau$ is such that $a_{ss}=0$
and $a_{sj}=0$ implies $a_{js}=0$.  
Then the Lie superalgebra $\fg(A,\tau)$
carries another structure of a contragredient Lie superalgebra
$\fg(A',\tau')$
with the same $\fh$, and new Chevalley generators
$e'_i,f'_i$ ($i\in I$) given by $e'_i:=e_i$, $f'_i:=f_i$
 if $a_{sj}=0$
and $j\not=s$, $e'_s:=f_s$, $f'_s:=e_s$,
and $e'_i:=[e_s,e_i]$, $f'_i:=[f_s,f_i]$
if  $a_{sj}\not=0$. One has 
$$\begin{array}{ll}
\tau'=\{i\in \tau|\ a_{si}\not=0\}\cup \{i\in I\setminus \tau| \ a_{si}\not=0\},\\
\Sigma'=\{-\alpha_s\}\cup\{\alpha_i|\ i\in I, \ i\not=s,\ a_{si}=0\}
\cup\{\alpha_i+\alpha_s|\ i\in I, a_{si}\not=0\},\\
(\Sigma')^{\vee}=\{h'_i\}_{i\in I}, \text{ where } h'_i=h_i \text{ if }  a_{si}=0,
\ \  h'_i=(-1)^{p(\alpha_i)}(a_{is}h_s+a_{si}h_i) \text{ if } \  a_{si}\not=0.\\
\end{array}$$
\end{prop}
\begin{proof}
One readily sees that the sets $\Sigma'$ and $(\Sigma')^{\vee}$
are linearly independent. 

Let us verify that $\fg(A,\tau)$ is generated by the new Chevalley generators and $\fh$.
Let $\fp$ be the subalgebra generated by $\fh$ and
$e'_i, f'_i$ for $i\in I$.
For all $i\not=s\in I$ we have
$[e_s,[f_s,f_i]]=-a_{si} f_i$, so
$[f'_s,f'_i]=-a_{si} f_i$
if $a_{si}\not=0$. Therefore $f_i\in \fp$
for all $i\in I$. Similarly, $e_i\in\fp$
for all $i\in I$, so $\fp=\fg(A,\tau)$.

The relations~(\ref{eq:Chevalleyrel}) (a) 
for $e'_i$, $f'_i$ and the relations $[e'_i,f'_i]=h'_i$
are straightforward. For
$i\not=j\not=s$ one has $\alpha'_i-\alpha'_j\not\in\mathbb{Z}_{\geq 0}\Sigma$
which gives $[e'_i,f'_j]=0$.
It remains to verify that
$[e'_s,f'_i]=[e'_i,f'_s]=0$ for $i\not=s$.

If $a_{sj}\not=0$, then
$[e'_s,f'_j]=[f_s,[f_s,f_j]]=0$ by~\Lem{lem:ass} (a).
If $a_{sj}=0$,  then, by the assumption,
$a_{js}=0$, so $[e'_s,f'_j]=[f_s,f_j]=0$
by~\Lem{lem:ass} (c). Hence $[e'_s,f'_j]=0$
for all $j\in I$. Similarly, $[e'_j,f'_s]=0$
for all $j\in I$. 
\end{proof}

\subsubsection{Remark}\label{reflexion}
For $(A,\tau)$ and $s\in I$ as in the proposition we say that $(A',\tau')$ are obtained from $(A,\tau)$ by 
an {\em  odd reflexion}  
\footnote{In~\cite{GHS} these are called "isotropic reflexions".}
$r_{\alpha_s}$.
We shall often write $\Sigma'=r_{\alpha_s}\Sigma$. Note that
$r_{-\alpha_s}r_{\alpha_s}\Sigma=\Sigma$.

\subsubsection{}
The Cartan matrix $A$ is called {\em weakly symmetrizable}  if $a_{ij}=0$ implies $a_{ji}=0$ for all $i,j\in I$. By~\Prop{prop:ass},
if $A$ is weakly symmetrizable, then the odd reflexion
$r_{\alpha_s}$ is well-defined for any $s\in \tau$
with $a_{ss}=0$.

\subsubsection{}
For an odd reflexion $r_{\alpha_s}$ we have
\begin{equation}\label{eq:Delta+oddrefl}
\Delta^+(r_{\alpha_s}\Sigma)=\Delta^+(\Sigma)\setminus\{\alpha_s\}\cup\{-\alpha_s\}.\end{equation}

Indeed, if $\gamma\in -\Delta^+(r_{\alpha_s}\Sigma)\cap \Delta^+(\Sigma)$, then
$\gamma=\sum_{i\in I} m_i\alpha_i=-\sum_{i\in I} n_i\alpha'_i$
where $m_i,n_i\in\mathbb{Z}_{\geq 0}$. One has
$\alpha'_s=-\alpha_s$.
For $i\not=s$ we have 
$\alpha'_i=\alpha_i+k_i\alpha_s$ for $k_i\in\{0,1\}$, so
$m_i=n_i$. Thus $m_i=0$ for all $i\not=s$.
Hence $\gamma\in \mathbb{Z}_{\geq 0}\alpha_s$. By~\Lem{lem:ass} (a), $[e_s,s_s]=0$, so
 $\Delta\cap\mathbb{Z}\alpha_s=\{\pm \alpha_s\}$. Thus $\gamma=\alpha_s$
as required.

\subsection{Integrable $\fg(A,\tau)$}
The contragredient Lie superalgebra
$\fg(A,\tau)$ is called {\em integrable}
if all its Chevalley generators and
all Chevalley generators obtained by 
any sequence of odd reflexions are $\ad$-locally nilpotent.

An integrable contragredient Lie superalgebra
$\fg(A,\tau)$ is called a {\em Kac-Moody superalgebra}
if for all $(A',\tau')$ obtained by any 
sequence of odd reflexions,
 the Cartan matrix $A'$ is weakly symmetrizable and
 $a'_{ii}\not=0$ for all $i\not\in\tau'$.

\subsubsection{Example}
Suppose that all diagonal entries of the Cartan matrix
are non-zero, so that it is equivalent to a Cartan matrix $A=(a_{ij})$ with all diagonal entries equal to $2$. In this case the contragredient Lie superalgebra is integrable if and only if
$A$ is weakly symmetrizable, and for $i\not=j$,
$a_{ij}\in\mathbb{Z}_{\leq 0}$
and $a_{ij}$ is even if $i\in\tau$.

These contragredient Lie superalgebras are called {\em anisotropic Kac-Moody superalgebras}.

\subsection{Indecomposable $\fg(A,\tau)$}
If the set $I$ decomposes in a disjoint union of non-empty subsets
$I_1$ and $I_2$, such that $a_{ij}=a_{ji}=0$
for $i\in I_1$, $j\in I_2$, then we have direct product decomposition of Lie superalgebras
$$\fg(A,\tau)\cong \fg(A_1,\tau_1)\times \fg(A_2,\tau_2) \times \ol{\fh},$$
where the matrices $A_1$ and $A_2$ are the submatrices of $A$,
corresponding to the sets of indices $I_1$
and $I_2$, $\tau_i\subset I_i$, $\fg(A_i,\tau_i)$
correspond to the reduced realizations of the $A_i$,
and $\ol{\fh}\subset\fh$. In this case the Cartan matrix is called {\em decomposable}.

The contragredient Lie superalgebra $\fg(A,\tau)$
is called {\em indecomposable}
if $A$ is {\em indecomposable} and its realization is reduced.

\subsection{Invariant bilinear form on symmetrizable $\fg(A,\tau)$}
A  Lie superalgebra $\fg(A,\tau)$ (and its Cartan matrix $A$)
is called {\em symmetrizable} if an equivalent Cartan matrix
$A_1=diag(b_i)_{i\in I} A$ is symmetric.
Note that a symmetrizable Cartan matrix is weakly symmetrizable.

A symmetrizable contragredient Lie superalgebra $\fg(A,\tau)$
admits a non-degenerate supersymmetric invariant bilinear form $(-,-)$, such that $(-,-)_{\fh\times\fh}$
is non-degenerate, $(\fg_{\alpha},\fg_{\beta})=0$
unless $\alpha=-\beta$, and the root spaces
$\fg_{\alpha}$ and $\fg_{-\alpha}$
are non-degenerately paired, such that
$$[e_{\alpha}, e_{-\alpha}]=(e_{\alpha},e_{-\alpha})h_{\alpha},$$
where $h_{\alpha}$ is defined by 
$\alpha(h)=(h,h_{\alpha})$ for all $h\in\fh$
(cf.~\cite{Kbook}, Chapter 2).

\subsection{}\label{Sigmaproperties}
Let $\fg(A,\tau)$ be a Kac-Moody superalgebra.
Let $\alpha\in\Sigma$. Then  $\dim\fg_{\pm\alpha}=1$
and $[\fg_{\alpha},\fg_{-\alpha}]$ is a one-dimensional subspace
of $\fh$; we denote by $\alpha^{\vee}$ a non-zero element 
in this subspace which satisfies the condition
 $\langle \alpha,\alpha^{\vee}\rangle\in \{0,2\}$ ($\alpha^{\vee}$ 
is unique if $\langle \alpha,\alpha^{\vee}\rangle=2$).
 One has 
$\mathbb{Z}\alpha\cap\Delta=\{\pm \alpha\}$ except for the case when
$p(\alpha)=\ol{1}$  and $\langle \alpha,\alpha^{\vee}\rangle\not=0$;
in the latter case $\mathbb{Z}\alpha\cap\Delta=\{\pm \alpha;\pm 2\alpha\}$.
Moreover the subalgebra generated 
by $\fg_{\alpha}, \fg_{-\alpha}$ coincides with $\mathbb{C}\alpha^{\vee}+\sum\limits_{j=1}^{\infty} \fg_{\pm j\alpha}$ and 
is isomorphic to one of the following Lie superalgebras (cf.~\cite{Ksuper}):
\begin{itemize}
\item
 $\fsl_2$ spanned by $\fg_{\pm\alpha}$ and $\alpha^{\vee}$;
in this case $p(\alpha)=\ol{0}$ and $\langle \alpha,\alpha^{\vee}\rangle=2$;
\item  $\mathfrak{sl}(1|1)$ spanned by $\fg_{\pm\alpha}$ and $\alpha^{\vee}$;
in this case $p(\alpha)=\ol{1}$ and
$\langle \alpha,\alpha^{\vee}\rangle=0$;
\item  $\mathfrak{osp}(1|2)$ spanned by $\fg_{\pm\alpha}$,  $\fg_{\pm2\alpha}$ and 
$\alpha^{\vee}$;
in this case $p(\alpha)=\ol{1}$ and $\langle \alpha,\alpha^{\vee}\rangle=2$.
\end{itemize}

\subsection{Spine, principal roots, Weyl group, real and imaginary roots}\label{Spine}
Let $\fg=\fg(A,\tau)$ be a symmetrizable Kac-Moody superalgebra.
The {\em spine} of $\fg$ is the set $\Sp$
of all bases of $\Delta$, obtained 
from $\Sigma$ by a chain of odd reflexions. 

For each non-isotropic vector  $\alpha\in\fh^*$
we define the reflection 
$s_{\alpha}\in \End(\fh^*)$ by $s_{\alpha}(\lambda):=\lambda-\frac{2(\lambda,\alpha)}{(\alpha,\alpha)}\alpha$.

Following notation of~\cite{S}, \cite{GSsnow} 
we call
 an even  root $\alpha$ {\em principal} if $\alpha$ or $\alpha/2$
lies in $\Sigma_1\in\Sp$. 
Since all even roots in $\Sigma_1$ are non-isotropic 
and, by~\Lem{lem:ass} (a), $2\beta\not\in\Delta$ 
for any isotropic root in $\Sigma_1$,
all principal roots are non-isotropic.
We denote the set of principal roots by $\Sigma_{\pr}$.
Note that $(\alpha,\alpha)\not=0$ if $\alpha$ is a principal root.
The subgroup $W$ of the group $GL(\fh^*)$
generated by  the reflections with respect to
principal roots, is called the {\em Weyl group} of $\fg$. We have: $W(\Delta)=\Delta$ (cf.~\cite{Kbook}, Section 3)
and $W(\Delta_{\ol{i}})=\Delta_{\ol{i}}$ for $i=0,1$.

A root $\gamma\in\Delta$ is called {\em real} if $\gamma$
or $\gamma/2$  lies in $w\Sigma_1$ for some $w\in W$ and $\Sigma_1\in\Sp$,
and is called {\em imaginary} otherwise. 
The sets of real and imaginary roots are denoted by $\Delta^{\ree}$
and $\Delta^{\ima}$ respectively. 

Using~$\S$~\ref{Sigmaproperties} it is easy to see that
$\Delta^{\ree}=\ol{\Delta}^{\ree}\coprod \{2\alpha|\ \alpha\not\in \ol{\Delta}^{\ree}_{\ol{1}}\}$,  where 
\begin{equation}\label{eq:real}
\ol{\Delta}^{\ree}=:=\{\gamma\in\Delta^{\ree}|\ \frac{1}{2}\gamma\in\Delta\}=
\{w\alpha|\ w\in W,\ \alpha\in\Sigma_1\in \Sp\}.
\end{equation}

By~$\S$~\ref{Sigmaproperties},
$\dim\fg_{\alpha}=1$ if $\alpha\in\Delta^{\ree}$, and
$[\fg_{\alpha},\fg_{-\alpha}]=\mathbb{C}\alpha^{\vee}$,
where $\langle \alpha,\alpha^{\vee}\rangle\in \{0,2\}$,
and $\alpha\in\Delta^{\ree}$ is non-isotropic (reps., isotropic)
if $\langle \alpha,\alpha^{\vee}\rangle=2$ (resp.,  
$=0$). We denote by $\Delta^{\an}$
(resp., $\Delta^{\is}$) the set of all non-isotropic (reps., isotropic) real roots. We let $\ol{\Delta}^{\an}=\Delta^{\an}\cap \ol{\Delta}^{\ree}$.
Notice that all even real roots are non-isotropic.
The Weyl group contains the reflections $s_{\gamma}$
for all $\gamma\in\Delta^{\an}$.

\subsubsection{}
If $\alpha\in\ol{\Delta}^{\an}$, then $w\alpha\in\Sigma_1\in\Sp$, so
for all $\beta\in\Sigma_1$,  the value
$\langle \beta, (w\alpha)^{\vee}\rangle$ is integral and is even
if $w\alpha$ is odd. 
Therefore for any $\beta\in\Delta$ we have
 \begin{equation}\label{eq:alphaveebeta} 
 \langle \beta,\alpha^{\vee}\rangle\in\mathbb{Z}\ \text{ if } \ 
 \alpha\in\ol{\Delta}^{\an}, \ \ \ \ \ 
\langle \beta,\alpha^{\vee}\rangle\in 2\mathbb{Z}\ \text{ if } \ 
\alpha\in\ol{\Delta}^{\an}_{\ol{1}}.\end{equation}

By~(\ref{eq:Delta+oddrefl}) the sets  $\Delta^{\ima +}=\Delta^{\ima}\cap \Delta^+$, 
$\Delta^{\an+}=\Delta^{\an}\cap\Delta^+$ and $\Delta^+_{\ol{0}}$
do not depend on the choice of a base in the spine.

The set of principal roots coincides with the set of indecomposable elements in 
the set of positive even real roots
$(\Delta_{\ol{0}}^{\ree})^+$, i.e. 
 these are the elements which can not be decomposed as a sum of several other elements 
in $(\Delta_{\ol{0}}^{\ree})^+$ (see~\cite{Shay}, Proposition 3.5).

Recall that $W$ is generated by reflections $s_{\alpha}$
with respect to the principal roots. For a principal root
$\alpha$ such that $\alpha$ or $\alpha/2$ lies in $\Sigma_1\in \Sp$ and any 
$\beta\in \Delta^+(\Sigma_1)$ one has $s_{\alpha}\beta\in \Delta^+(\Sigma_1)$
if $\beta\not\in\mathbb{Z}\alpha$. This implies
 $W(\Delta^{\ima +})=\Delta^{\ima +}$.

\subsubsection{Remark}\label{rem:prince}
Each subset ${\Sigma}'\subset\Sigma$ defines a Kac-Moody subalgebra ${\fg}'\subset \fg$
with the same Cartan subalgebra $\fh$ and the base $\Sigma'$. There is a natural embedding of
the spine of $\Sigma'$  to the spine of $\Sigma$; this gives $\Sigma'_{\pr}\subset \Sigma_{\pr}$ (cf.~\cite{Kbook}, Chapter 3).

\subsubsection{}\label{Bpr}
Since even real roots are non-isotropic we have $\langle \alpha,\alpha^{\vee}\rangle =2$
for each  $\alpha\in\Sigma_{\pr}$. By~\cite{S}, for all $\alpha,\beta\in\Sigma_{\pr}$ we have 
$\langle \beta,\alpha^{\vee}\rangle\in\mathbb{Z}_{\leq 0}$, so
 the matrix $B_{\pr}:=\bigl(\langle \alpha^{\vee},\beta\rangle\bigr)_{\alpha,\beta\in \Sigma_{\pr}}$
is the   Cartan matrix of a Kac-Moody algebra (from Hoyt-Serganova classification~\cite{Hoyt},\cite{S},
it follows that $\Sigma_{\pr}$ is finite).

Let a triple $(\fh',\pi,\pi^{\vee})$ be a realization of $B_{\pr}$ in the sense of~\cite{Kbook}, Chapter I. 
Consider linear epimorphisms $\phi: \mathbb{Z}\pi\to \mathbb{Z}\Sigma_{\pr}$
and $\phi^*:\mathbb{Z}\pi^{\vee}\to \mathbb{Z}\Sigma^{\vee}_{\pr}$ 
which map $\pi$ to $\Sigma_{\pr}$ and $\pi^{\vee}$ to $\Sigma_{\pr}^{\vee}$
respectively. The restriction of $\phi$ gives a bijection between
the set
of real roots of $\fg(B_{\pr})$ and $\Delta^{\ree}_{\ol{0}}$. Moreover, $\phi$
induces a group isomorphism between the Weyl group of $\fg(B_{\pr})$
and the Weyl group $W$. In particular,
$W$ is the Coxeter group generated by $s_{\alpha}$ with $\alpha\in\Sigma_{\pr}$
and $\Delta^{\ree}_{\ol{0}}=W(\Sigma_{\pr})$.

\subsubsection{}
\begin{lem}{lem:standard}
Any $\gamma\in \Delta^{\ree+}_{\ol{0}}$
can be written as $\gamma=s_{\alpha_j}\ldots s_{\alpha_2}\alpha_1$ 
(with $\gamma=\alpha_1$ for $j=1$)
where $\alpha_1,\ldots,\alpha_j\in\Sigma_{\pr}$ are such that
$$s_{\alpha_{i+1}}(s_{\alpha_i}\ldots s_{\alpha_2}\alpha_1)\in \bigl(
s_{\alpha_i}\ldots s_{\alpha_2}\alpha_1+ \mathbb{Z}_{>0}\alpha_{i+1}\bigr)
$$
for each $i=1,\ldots, j-1$.
\end{lem}
\begin{proof}
Using the induction on  the partial order 
given by $\nu\geq \mu$ if $\nu-\mu\in \mathbb{Z}_{\geq 0}\Sigma_{\pr}$
we reduce the statement to the following: for 
 each $\gamma\in (\Delta^{\ree+}_{\ol{0}}\setminus \Sigma_{\pr})$
there exists $\alpha\in\Sigma_{\pr}$ such that $\langle\gamma,\alpha^{\vee}\rangle>0$.
For Kac-Moody algebras this assertion is Proposition 5.1 (e) in~\cite{Kbook}.
Using $\phi: \mathbb{Z}\pi\to \mathbb{Z}\Sigma_{\pr}$ described above,
we deduce the required assertion
from the corresponding assertion for the Kac-Moody algebra
$\fg(B_{\pr})$.
\end{proof}

\subsection{Weyl vector  $\rho=\rho_{\Sigma}$}\label{Weylvector}
Define $\rho\in\fh^*$ satisfying the condition
$$2\langle\rho,\alpha^{\vee}\rangle=\langle\alpha,\alpha^{\vee}\rangle\ \ \text{ 
for all  }\alpha\in\Sigma.$$
For each $\Sigma_1\in\Sp$, the element 
\begin{equation}\label{eq:rho}
\rho_{\Sigma_1}:=\rho-\sum\limits_{\gamma\in \Delta^+(\Sigma)\setminus
\Delta^+(\Sigma_1)} (-1)^{p(\gamma)}\gamma\end{equation}
is a Weyl vector for $\Sigma_1$. This follows from the observation 
that $\rho_{r_{\alpha}\Sigma}=\rho+\alpha$
if $\alpha\in \Sigma\cap\Delta^{\is}$.

Combining~(\ref{eq:alphaveebeta})  and (\ref{eq:rho}) we get
\begin{equation}\label{eq:rhoalpha}
\langle\rho,\alpha^{\vee}\rangle\in \mathbb{Z}\ \text{ for }\alpha\in\ol{\Delta}^{\an}_{\ol{0}},\ \ \ \ \ \langle\rho,\alpha^{\vee}\rangle\in 2\mathbb{Z}+1\ \text{ for }\alpha\in \Delta^{\an}_{\ol{1}}.\end{equation}

\subsection{Categories $\CO^{\fin}\subset \CO_{\Sigma}\subset \CO^{\inf}$}
\label{triangular}
Let $\fg$ be an arbitrary Lie superalgebra and let $\fh\subset\fg_{\ol{0}}$
be an abelian, self-normalizing subalgebra which acts $\ad$-semisimply  on $\fg$, so that
$$\fg=\fh\oplus(\oplus_{\alpha\in \Delta} \fg_{\alpha}),
\ \ \fg_{\alpha}:=\{x\in\fg|\ ]h,x]=\alpha(h)x\ \text{ for all } h\in\fh\},$$
and $\Delta:=\{\alpha\in\fh^*|\ \alpha\not=0,\ \fg_{\alpha}\not=0\}$.

\subsubsection{Triangular decomposition}\label{Sherman}
Let $\omega:\mathbb{Z}\Delta\to \mathbb{R}$  
be a group homomorphism satisfying
$\omega(\alpha)\not=0$ for all $\alpha\in \Delta$ (such a homomorphism exists because
$\fh^*\cong \mathbb{R}$ as a $\mathbb{Q}$-vector space~\cite{GSS}). We set
$\Delta^{\pm}:=\{\alpha\in \Delta|\ \pm \omega(\alpha)>0\}$.
We introduce  the corresponding triangular decomposition
$\fg=\fg^-\oplus \fh\oplus \fg^+$
 by setting $\fg^{\pm}:=\sum\limits_{\alpha\in\Delta^{\pm}(\fg)} \fg_{\alpha}$.
Note that a triangular decomposition of $\fg$ induces a triangular decomposition of 
any $\ad\fh$-invariant subalgebra of $\fg$.

\subsubsection{}
Given $\lambda\in\fh^*$, we define
a Verma module  $M(\lambda)$  in the usual way,
 as the $\fg$-module, induced from the $1$-dimensional $\fh+\fg^+$-module
 $\mathbb{C}v_{\lambda}$, such that $\fg^+v_{\lambda}=0$ and
 $hv_{\lambda}=\lambda(h)v_{\lambda}$ for $h\in\fh$.
 This module has a unique simple quotient, denoted by 
$L(\lambda)$.

We denote by ${\CO}^{\inf}(\fg)$ the full subcategory of the category of $\fg$-modules with the objects 
which are $\fh$-semisimple and $\fg^+$-locally finite (i.e., 
every $v\in N$ generates a finite-dimensional $\fg^+$-submodule).
Note that for any $N\in {\CO}^{\inf}(\fg)$ each irreducible subquotient of
$N$ is isomorphic to $L(\lambda)$ for some $\lambda\in\fh^*$.

We will use the following standard notation:
for each semisimple $\fh$-module $N$ and $\nu\in\fh^*$,
$N_{\nu}$ denotes the weight space with weight $\nu$,
 $\Omega(N):=\{\nu|\ N_{\nu}\not=0\}$, and 
$\ch N=\sum\limits_{\mu\in\fh^*} (\dim N_{\mu})\, e^{\mu}$
if  $\dim N_{\mu}<\infty$ for all $\mu$.
For an  $\fh$-invariant subalgebra $\fl \subseteq \fg$ we set $\Delta(\fl):=\Omega(\fl)$.

\subsubsection{}\label{categoryO}
Let $\fg$ be a  Lie superalgebra with a fixed triangular decomposition.
Assume that 
\begin{itemize}
\item[(a)]
all root spaces $\fg_{\alpha}$ are finite-dimensional;
\item[(b)]
 there exists a linearly independent set $\Sigma\subset\fh^*$
such that $\Delta^{\pm}(\fg)\subset \pm\mathbb{Z}_{\geq 0}\Sigma$
(note that we do not assume that $\Sigma\subset\Delta(\fg)$).
\end{itemize}

 If $\fg'\subset \fg$ is any subalgebra 
containing $\fh$, then $\fg'$   inherits the triangular decomposition 
satisfying
the above properties (for the same set $\Sigma$). The restriction functor
$\Res^{\fg}_{\fg'}$ carries $\CO^{\inf}(\fg)$ to $\CO^{\inf}(\fg')$.

We denote by $\CO_{\Sigma}(\fg)$ the full subcategory of 
$\CO^{\inf}(\fg)$
of  modules $N$ satisfying the following condition: 
$\fh$ has finite-dimensional weight spaces,
 and
 $\Omega(N)$ lies 
in a finite union of the sets of the form $\lambda-\mathbb{Z}_{\geq 0}\Sigma$
where $\lambda\subset\fh^*$.  This notion is a slight variation of 
the category $\CO$ introduced in~\cite{KK}.

By (a), $M(\lambda)$ and $L(\lambda)$ are objects of $\CO_{\Sigma}(\fg)$,
and any irreducible module in $\CO_{\Sigma}(\fg)$ is isomorphic to $L(\lambda)$
for some $\lambda\in\fh^*$. Any  $N\in\CO_{\Sigma}(\fg)$ has the following properties:
\begin{enumerate}
\item 
$N\in \CO^{\inf}(\fg)$;
\item
all   subquotients of $N$ lie in $\CO_{\Sigma}(\fg)$;
\item
$\Res^{\fg}_{\fg'} N\in \CO_{\Sigma}(\fg')$  if $\fg'\subset \fg$ is a subalgebra 
containing $\fh$ (note that the induced triangular decomposition of $\fg'$
satisfies (a) and (b));
\item the multiplicity
$[N:L(\mu)]$ is well defined 
and $\ch N=\sum_{\nu\in\fh^*} [ N:L(\nu)]\ch L(\nu)$.
\end{enumerate}

Property (i) is easy to check; (ii), (iii) follow from definition. For (iv),
 using the arguments of~\cite{DGK} (Proposition 3.2),
one shows that for each $\mu\in\fh^*$ any module $N\in\CO_{\Sigma}(\fg)$ 
admits a {\em weak composition series at $\mu$} which is a chain of submodules
$$0=N_0\subset N_1\subset \ldots N_r=N$$
such that for each $i=1,\ldots, r$ either
 $N_i/L_{i-1}\cong L_(\nu)$ for some
 $\nu\in (\mu+\mathbb{Z}_{\geq 0}\Sigma)$
or $(N_i/N_{i-1})_{\nu}=0$ for all $\nu\in (\mu+\mathbb{Z}_{\geq 0}\Sigma)$.

Assume that the triangular decomposition satisfies (a) and (b). 
We introduce the algebra $\cR(\Sigma)$ over $\mathbb{R}$ as in~\cite{GKadm}:
 the elements of $\cR(\Sigma)$ are finite linear combinations
of the elements of the form $B:=\sum\limits_{\mu\in\mathbb{Z}_{\geq 0}\Sigma}
b_{\mu} e^{\lambda-\mu}$ where $b_{\nu}\in \mathbb{R}$ and $\lambda\in\fh^*$. These elements can be multiplied in the obvious way.
Any element in $\cR(\Sigma)$ can be uniquely written as 
$x=\sum\limits_{\mu\in \fh^*}
x_{\nu} e^{\nu}$ and we set $\supp(x):=\{\nu|\ x_{\nu}\not=0\}$.
The ring $\cR(\Sigma)$ contains $\ch N$ for any $N\in {\CO}_{\Sigma}(\fg)$.

\subsubsection{}
We call  $\Sigma$  satisfying  (b) a {\em base} (or a set of simple roots) 
if $\Sigma\subset\Delta^+(\fg)$.
If the triangular decomposition admits a base, such base is unique.

\subsubsection{}\label{COfin}
We denote by $\CO^{\fin}(\fg)$ the BGG-category, i.e.
the full subcategory of $\CO^{\inf}(\fg)$
with finitely generated modules $N$.

It is easy to see (\cite{GSsnow}, Lemma 3.1.3)
that $N$ lies in $\CO^{\fin}(\fg)$
if and only if $N$ admits a finite filtration $0=N_0\subset N_1\subset N_r=N$
such that for each $i=1,\ldots, r$ the quotient 
$N_i/N_{i-1}$ is isomorphic to a quotient of a Verma module.
As a result, $\CO^{\fin}(\fg)$ is a subcategory of $\CO_{\Sigma}(\fg)$
for any choice of $\Sigma$ compatible with the 
triangular decomposition of $\fg$.

\subsubsection{Example}\label{triangKM}
 Let $\fg$ be a Kac-Moody superalgebra with a Cartan subalgebra $\fh$.
% Any base is a base in ``Kac-Moody sense'', see~$\S$~\ref{base}. Moreover,
 If $\fg$ is finite-dimensional or affine, then
any triangular decomposition admits a base 
(and $\fg$ coincides with the Kac-Moody superalgebra
constructed for this base), see~\cite{Shay}, Theorem 0.4.3.
A triangular decomposition of $\fg$ induces a triangular decomposition of 
$\fg_{\ol{0}}$.
The latter decomposition 
admits a base if $\fg$ is finite-dimensional and 
does not admit a base for most of affine Kac-Moody superalgebras.
For any $N\in \CO_{\Sigma}(\fg)$ we have $\Res^{\fg}_{\fg_{\ol{0}}} N\in \CO_{\Sigma}(\fg_{\ol{0}})$
so the multiplicity $[N:L_{\fg_{\ol{0}}}(\nu)]$
is well defined.

\subsubsection{Remark}
The reason why we consider several versions of the category $\CO$ is the following.

In this paper we mostly consider the following Lie superalgebras:
an affine  Kac-Moody superalgebra $\fg$ with the Cartan subalgebra
$\fh$ and a  base $\Sigma\subset\fh^*$, a Kac-Moody superalgebra $\fg_{\pi}\subset \fg$
containing $\fh$, and a Kac-Moody superalgebra $\fg^{\#}\subset \fg_{\pi}$ such that $\fg_{\pi}=\fg^{\#}+\fh$.
The Cartan subalgebra $\fh$ is always fixed,
but triangular decompositions (and $\Sigma$) of $\fg$ vary.

The functor $\Res^{\fg}_{\fg_{\pi}}$  carries $\CO^{\inf}(\fg)$
to $\CO^{\inf}(\fg_{\pi})$ and $\CO_{\Sigma}(\fg)$
to $\CO_{\Sigma}(\fg_{\pi})$ 
(notice that $\Sigma$, usually, does not lie in $\Delta_{\pi}$).
The above construction allows to  define the multiplicity 
$[N:L_{\fg_{\pi}}(\nu)]$
for any $N\in\CO_{\Sigma}(\fg)$.

The functor $\Res^{\fg}_{\fg^{\#}}$  carries $\CO^{\inf}(\fg)$
to $\CO^{\inf}(\fg^{\#})$ (but the image of $N\in\CO_{\Sigma}(\fg)$
might be not in 
$\CO_{\Sigma}(\fg^{\#})$ since the weight spaces on $N$ with 
respect to the Cartan subalgebra $\fh\cap\fg^{\#}$ might be infinite-dimensional).

The category $\CO^{\fin}(\fg)$ does not behave well under the restriction functor,
since the modules $\Res^{\fg}_{\fg_{\pi}} M(\lambda)$ might be  not finitely generated. However
$\CO^{\fin}(\fg)$ has the following advantages
comparing to $\CO_{\Sigma}(\fg)$: it  does not depend on the choice
of $\Sigma_1\in \Sp$, see~$\S$~\ref{hwtSpine}, and is stable under the action of the Enright functors, see~$\S$~\ref{Enright}.

\subsection{Ring $\cR(\Sigma)$ for Kac-Moody $\fg(A,\tau)$}\label{Weylgroup}
Now let $\fg=\fg(A,\tau)$ be a symmetrizable Kac-Moody superalgebra.
The triangular decomposition satisfies the conditions (a), (b)
in~$\S$~\ref{categoryO} and we let $\cR(\Sigma)$ be the ring introduced 
there.

For any $E\subset \Delta^{\an}$ we denote by $W[E]$ the subgroup of $W$ generated by $s_{\alpha}$ for 
$\alpha\in E\cap \Delta^{\an}$.

If $x:=\sum x_{\nu} e^{\nu}\in \cR(\Sigma)$ and $w\in W$ are such that 
$w\bigl(\sum x_{\nu} e^{\nu}\bigr):=\sum x_{\nu} e^{w\nu}\in \cR(\Sigma)$, 
we say that $w$ {\em acts naturally } on  $x$.
We say that $x\in \cR(\Sigma)$ is {\em naturally } $w$-invariant and, respectively,  $w$-anti-invariant if $x$ is $wx=x$
and, respectively,  $wx=\sgn(w) x$  for the natural action of $w$ (the map
$\sgn: W\to \mathbb{Z}_2$ is given by $\sgn (s_{\alpha})=-1$ for all $\alpha\in\Delta^{\an}$).
In other words, $\sum x_{\nu} e^{\nu}$ is naturally $w$-invariant if $x_{w\nu}=x_{\nu}$ and
naturally $w$-anti-invariant if $x_{w\nu}=\sgn w\cdot x_{\nu}$.

\subsubsection{The Weyl denominator $D$}\label{Weyldenominator} It is
 the following element in $\cR(\Sigma)$
$$D=D_{\ol{0}}/D_{\ol{1}},\ \ \text{
where }D_{\ol{0}}=\prod_{\alpha\in\Delta^+(\Sigma)_{\ol{0}}} (1-e^{-\alpha})^{\dim \fg_{\alpha}},\ 
D_{\ol{1}}=\prod_{\alpha\in\Delta^+(\Sigma)_{\ol{1}}} (1+e^{-\alpha})^{\dim \fg_{\alpha}},$$
and by $(1+e^{-\alpha})^{-1}$ we mean $\sum\limits_{i=0}^{\infty} (-e^{-\alpha})^i\in \cR(\Sigma).$

\subsection{Category $\CO^{\fin}(\fg)$ for  Kac-Moody superalgebra $\fg(A,\tau)$}\label{hwtSpine}
An important advantage of BGG-category $\CO^{\fin}(\fg)$ is that this category does not depend on the
choice of $\Sigma_1\in \Sp$: if
$N\in \CO^{\fin}(\fg)$ for the triangular decomposition with a base $\Sigma$, then 
$N\in \CO^{\fin}(\fg)$ for the triangular decomposition with any base $\Sigma_1\in \Sp$
(see~\cite{GSsnow}, 1.11.2).

For a base $\Sigma_1\in \Sp$ we denote by $L(\lambda,\Sigma_1)$ the 
irreducible module of the highest weight $\lambda$ with respect to
 the corresponding triangular decomposition.
 For $L=L(\lambda,\Sigma)$ we set
 $\ \hwt_{\Sigma} L:=\lambda+\rho_{\Sigma}$. Then for $\alpha\in (\Sigma\cap \Delta^{\is})$ we have
\begin{equation}\label{eq:hwt}
\hwt_{r_{\alpha}\Sigma} L=\left\{\begin{array}{lcl}
\hwt_{\Sigma} L& & \text{ if   }\ \ \ 
\langle \hwt_{\Sigma} L,\alpha^{\vee}\rangle\not=0\\
\hwt_{\Sigma} L+\alpha& & \text{ if   }\ \ \ 
\langle \hwt_{\Sigma} L,\alpha^{\vee}\rangle=0.
\end{array}\right.\end{equation}
(The last formula means that $L(\lambda,r_{\alpha}\Sigma)=L(\lambda,\Sigma)$
if $\langle \lambda,\alpha^{\vee}\rangle=0$ and $\alpha$ is isotropic.)

The $\fg$-module $L(\lambda)$ is called {\em typical} if $\langle \lambda+\rho,\alpha^{\vee}\rangle\not=0$ for all isotropic roots $\alpha\in \Delta^{\is}$. In this case $\hwt_{\Sigma_1} L=\hwt_{\Sigma} L$
for all $\Sigma_1\in\Sp$.

\section{Symmetrizable Kac-Moody superalgebras}
We will use terminology  of~\cite{GHS} and some results of~\cite{S} and~\cite{GHS}.
If $\fg$ is an indecomposable  symmetrizable Kac-Moody superalgebra, we always
assume that the invariant bilinear form $(-,-)$ is normalised 
in such a way that $(\alpha,\alpha)\in\mathbb{Q}^*$
for some  $\alpha\in\Delta(\fg)$, unless $\fg=\fgl(1|1)$ 
(if $\fg=\fgl(1|1)$, then $\Delta(\fg)=\{\pm\alpha\}$ and
$(\alpha,\alpha)=0$).

\subsection{Classification}
Any symmetrizable indecomposable Kac-Moody superalgebra 
 lies in  one of the following classes~\cite{Hoyt},\cite{S}:

\begin{itemize}
\item[(An)] anisotropic: integrable $\fg(A,\tau)$ with
$a_{ii}\not=0$ for all $i\in I$
(these superalgebras are studied in~\cite{K78});
\item[(Fin)]   finite-dimensional: this class consists of  the simple Lie algebras and
the following simple  Lie superalgebras: 
\begin{equation}\label{eq:findimlist}
\mathfrak{sl}(m|n) \text{ with } m\not=n, m,n\geq 1, 
\ \osp(m|2n), \text{ with } m,n>0, \ D(2|1,a), F(4), G(3) 
\end{equation}
and 
the non-simple one $\fgl(n|n), n\geq 1$, see~\cite{Ksuper}. It is often convenient to use
Cartan type notation~\cite{Ksuper}: $A(m|n)=\fsl(m+1|n+1)$ if $m\not=n$,
$A(n|n)=\mathfrak{gl}(n+1|n+1)$,
$B(m|n)=\mathfrak{osp}(2m+1|2n)$, $C(n)=\mathfrak{osp}(2|2n)$,
$D(m|n)=\mathfrak{osp}(2m|2n)$ (with $m\geq 1$)-- these superalgebras are described in~\cite{Ksuper};

\item[(Aff)]  affine: affine Kac-Moody algebras~\cite{Kbook}, and
symmetrizable affine Kac-Moody superalgebras, which are not Lie algebras, described in~\cite{vdLeur}; their construction is explained in~$\S$~\ref{symmaffsuper}.
\end{itemize}

\subsubsection{}
The classes (Fin) and (Aff) do not intersect 
(i.e., all affine superalgebras are infinite-dimensional),
(An)$\cap$ (Fin) consists of finite-dimensional simple Lie algebras and
  $\mathfrak{osp}(1|2\ell)=B(0|\ell)$,  and the intersection (An)$\cap$ (Aff) 
 consists of affine Lie algebras and  the series  $B(0|n)^{(1)}$, $A(0|2n)^{(4)}$, $A(0|2n+1)^{(2)}$, $C(n+1)^{(2)}=D(1|n)^{(2)}$ with $n\geq 1$~\cite{K78}.

 \subsubsection{}\begin{rem}{rem:delta}
The above classification  can be also described in terms of $\Delta^{\ima}$: 
 $\Delta^{\ima}=\emptyset$ if $\fg$ is  of type (Fin), 
 and $\Delta^{\ima}=\mathbb{Z}\delta\setminus\{0\}$
 for some $\delta\in\Delta^+$ if $\fg$ is of type (Aff).
 If $\fg$ is not of types (Fin) and (Aff), then $\Delta^{\ima}$ contains at least two non-proportional
 imaginary roots, see~\cite{GHS},\cite{Shay}. If $\fg$ is of type (Aff), then 
 $\mathbb{Z}\delta=\{\nu\in\mathbb{Z}\Delta|\ (\nu,\Delta)=0\}$.

It can be also described  via dimensions:
 all finite-dimensional algebras lie in (Fin), and all infinite-dimensional of finite growth  lie in (Aff). C.~Hoyt and V.~Serganova classified all indecomposable Kac-Moody superalgebras;
 from this classification it follows that all indecomposable symmetrizable Kac-Moody superalgebras
 of infinite growth are anisotropic.
 \end{rem}

 \subsection{Definitions}\label{shortroot}
Let $\fg$ be a symmetrizable Kac-Moody superalgebra. 
\subsubsection{Definition}
\label{def:non-critical} We say that $\lambda\in\fh^*$ is {\em non-critical} if  $2(\lambda+\rho,\alpha)\not\in\mathbb{Z}_{>0}(\alpha,\alpha)$  for all 
$\alpha\in\Delta^{\ima+}$  and  call a $\fg$-module
 $N$  {\em non-critical} if all
 irreducible subquotients of $N$ are highest weight modules with non-critical
 highest weights.

\subsubsection{}\begin{rem}{rem:non-critical0}
By~\cite{Kbook}, Proposition 5.1 (see also~\cite{Shay},
(1) for Lie superalgebras), $n\alpha\in \Delta^{\ima+}$
if $\alpha\in\Delta^{\ima+}$ and 
 $n\in \mathbb{Z}_{>0}$. Therefore
 $\lambda\in\fh^*$ is non-critical if and only if  $2(\lambda+\rho,\alpha)\not=(\alpha,\alpha)$. 
\end{rem}
\subsubsection{Remark}
Let $\fg$ be a Kac-Moody algebra.
 By~\cite{Kbook}, Proposition 3.12, one has 
$$\{\lambda\in\fh^*|\ (\lambda,\alpha)\geq 0\ \text{ for all, but finite number of }\alpha\in \Delta^{\ree}_+ \}=\cup_{w\in W} w(C),$$ 
where $C:=\{\nu\in \fh^*|\ (\nu,\alpha)\geq 0\ \text{ for all }\alpha\in \Delta^+\}$. 
Let $X$
 be as in Section~\ref{sect:intro} that is 
$X=\{\lambda\in\cup_{w\in W} w(C) |\  (\lambda,\alpha)\not=0 \text{ for any isotropic  root $\alpha$  }\}$.
Then~\cite{DGK}, Proposition 5.2 (i), the set $X-\rho$ coincides with the set of non-critical weights
in $-\rho+\cup_{w\in W} w(C)$. This follows 
the fact that  $\alpha\in\Delta$ is imaginary if and only if 
$(\alpha,\alpha)\leq 0$, see~\cite{Kbook}, Proposition 5.2.
This immediately implies that a non-critical weight
in $-\rho+\cup_{w\in W} w(C)$ lies in $X-\rho$. For the inverse inclusion
let $\nu\in -\rho+\cup_{w\in W} w(C)$ be a critical weight that is 
 $2(\nu+\rho,\alpha)=j(\alpha,\alpha)$
for some $\alpha\in\Delta^{\ima+}$ and $j\in\mathbb{Z}_{>0}$. Writing $\nu+\rho$ as $w\mu$ for $w\in W$
and $\mu\in C$ we obtain
$2(\mu,w^{-1}\alpha)=j(\alpha,\alpha)$. Since $\mu\in C$ and 
$w^{-1}\alpha\in \Delta^{\ima+}$ this implies $(\alpha,\alpha)\geq 0$.
Since $\alpha$ is an imaginary root, $(\alpha,\alpha)=0$, so
$\nu+\rho\not\in X$.

\subsubsection{}\begin{rem}{rem:non-critical}
By~\Rem{rem:delta}, all weights are non-critical in type (Fin), and
in type (Aff) the weight 
$\lambda$ is non-critical if and only if $(\lambda+\rho,\delta)\not=0$.
\end{rem}

 \subsubsection{Definition}\label{shortroot1}
We call   $\alpha_s\in\Delta^{\an}$ a {\em short  root} if  
 $(\beta,\beta)\in\mathbb{Z}(\alpha_s,\alpha_s)$ for all $\beta\in\Delta$.
 For example, for $\fsl(m|n)$, all non-isotropic roots are short.

\subsubsection{}\label{alphas+delta}
 We will use the following fact: if 
 $\alpha_s$ is a short root of an affine Kac-Moody superalgebra
 with the minimal imaginary root $\delta$,
 then $\alpha_s+j\delta\in \Delta$ is equivalent to $j\in\mathbb{Z}$
 (this follows from the description of affine root systems in~\cite{Kbook}, Proposition 6.3; \cite{vdLeur}, Table V; this can be also proved in a unified way).

 \subsubsection{Remark}
Assume that  $\fg$ is of type (Fin) or (Aff) and $\Delta\not=D(2|1,a)$, $D(2|1,a)^{(1)}$.
 From the classifications in~\cite{Ksuper},\cite{vdLeur} it is easy to see that
 $\Delta$ contains a short root; if, in addition, $\fg$ is not in (An), then 
  $\Delta$ contains a short even root.

\subsection{Type (Aff)}\label{symmaffsuper}
It follows from Hoyt-Serganova classification that 
any symmetrizable indecomposable infinite-dimensional 
Kac-Moody superalgebra, which is not anisotropic, is of finite growth.
The indecomposable
symmetrizable contragredient Lie superalgebras of finite growth
were classified by van de Leur in~\cite{vdLeur}.
All these superalgebras are Kac-Moody, and the class (Aff) coincides with 
the class of the indecomposable
symmetrizable contragredient Lie superalgebras of finite growth.
Such Lie algebras were classified previously in~\cite{K68}.

These Lie superalgebras can be constructed by the following procedure (cf.~\cite{Kbook}, Chapter 8).
Let $\sigma$ be an automorphism of order $r$ of one of the Lie 
superalgebras~(\ref{eq:findimlist}), which we denote by $\fg$. Let
$\fg=\oplus_{j=0}^{r-1} \fg_{j\,\text{mod}\,r}$ be the eigenspace decomposition for $\sigma$, 
where $\fg_j=\{a\in\fg|\ \sigma(a)=e^{\frac{2\pi\sqrt{-1}}{r}j} a\}$. We shall assume that 
$\sigma$ preserves the invariant bilinear form $(. , .)$, so that $\fg_j$ and $\fg_{-j}$
are non-degenerally paired. If $\fg\not=\fgl(n|n)$ we let
$${\fg}[t,t^{-1}]^{(r,\sigma)}=\oplus_{j\in\mathbb{Z}} (\fg_{j\,\text{mod}\,r} t^j)\subset \fg[t,t^{-1}].$$

This Lie superalgebra is called a twisted loop superalgebra. The loop superalgebra ${\fg}[t,t^{-1}]$ has a well-known $2$-cocycle
$$\psi(at^m,bt^n)=m\delta_{m,-n}(a,b),\ \ \text{ where } m,n\in\mathbb{Z},\ a,b\in\fg.$$
We denote by $\hat{\fg}=\fg[t,t^{-1}]\oplus \mathbb{C}K$ the central extension of ${\fg}[t,t^{-1}]$
with this cocycle, where $K$ is the central element. Then the non-twisted Kac-Moody  superalgebra is, provided that $\fg\not=\fgl(n|n)$:

$$\fg^{(1)}=\hat{\fg}\oplus \mathbb{C}d,\ \ \text{ where } d=t\frac{d}{dt}$$
and $\fg^{(1)}$ is its subquotient 
for $\fg=\fgl(n|n)$, obtained by substituting the loop superalgebra 
$\fgl(n|n)[t,t^{-1}]$
by its subquotient, which is the subalgebra
$\fgl(n|n) +\sum_{j\in\mathbb{Z}}\fsl(n|n)t^j$  quotient by 
central elements $I_{2n}t^s$, $s\not=0$.

The subalgebra  $\fg^{(r,\sigma)}={\fg}[t,t^{-1}]^{(r,\sigma)}\oplus \mathbb{C}d\subset \fg^{(1)}$ has a structure
of symmetrizable contragredient Lie superalgebra, provided that
$\fg_0$-module $\fg_1$ is irreducible, constructed in the same way as in~\cite{Kbook}, Chapter 8,
using that $\fg_0$ is a contragredient Lie superalgebra and that
the $\fg_0$-modules $\fg_1$ and $\fg_{-1}$ are contragredient.

\subsubsection{}
It follows from~\cite{Kbook}, Proposition 8.5,
that the Kac-Moody superalgebras 
 $\fg^{(r,\sigma)}$ and $\fg^{(r',\sigma')}$ 
 are isomorphic if $\sigma$ and $\sigma'$
 lie in the same connected component of the group $\Aut\fg$.
 Since  $\fg^{(r,\sigma)}$ can be a contragredient 
 Lie superalgebra only for $\sigma$ preserving the invariant
 bilinear form, in order to construct all
 affine Kac-Moody superalgebras up to isomorphism,
 we should pick a representative $\mu$ in each connected component
 of the group of automorphisms 
 of $\fg$, preserving the invariant
 bilinear form, and verify that the corresponding
 Lie superalgebra
 $\fg^{(r,\sigma)}$ is contragredient. But this follows from
 Table 4 of van de Leur's paper~\cite{vdLeur}
 and the description of $\Aut\fg$
 in~\cite{Saut}. We thus obtain that up to isomorphism 
 a complete list of indecomposable 
 symmetrizable affine Kac-Moody superalgebras,
which are not Lie algebras,
 beyond the non-twisted ones, is
\begin{equation}\label{eq:affinelist}
A(m-1|2n-1)^{(2)}, \ A(2m-2|2n)^{(4)},\ D(m|n)^{(2)},
\end{equation}
where $m,n\in\mathbb{Z}_{\geq 1}$  except for
$m=2$, $n=1$ in the first case 
(observe that  $G(3)^{(2)}$ in~\cite{vdLeur} is actually
isomorphic to $G(3)^{(1)}$, see~\cite{Shay}, Section 8.5).

\subsubsection{Remark}
Using the description of automorphisms of  finite-dimensional Lie superalgebras in~\cite{Saut}, 
a result, similar to that in~\cite{Kbook}, Chapter 8 holds:

{\em Up to isomorphism, the affine Lie superalgebras $\fg^{(r,\sigma)}$ correspond bijectively to connected  components of $\Aut \fg$, preserving the  bilinear form $(-,-)$, containing $\sigma$ or $\sigma^{-1}$.}

\subsection{}
\begin{prop}{prop:rhodelta}
View $\fg^{(r,\sigma)}$ as a subalgebra of $\fg^{(1)}$ with  the non-degenerate
bilinear form $(-,-)$ satisfying
$(at^i,bt^j)=\delta_{ij} (a,b)$ where $a,b\in\fg$. Then  
$$r(\rho,\delta)=c,$$ 
where
$2c$ is the eigenvalue of the  Casimir operator of $\fg$ on the adjoint representation.
\end{prop}

We start from the following lemma.

\subsubsection{}
\begin{lem}{lem:Cxy}
Let $\fg$ be  an indecomposable in a direct sum finite-dimensional 
Lie superalgebra endowed with
an even non-degenerate supersymmetric invariant bilinear form $(- ,- )$,
and let $\{a^j\}_{j=1}^n$, $\{a_j\}_{j=1}^n$ be dual bases of $\fg$ (i.e.,
$(a^i,a_j)=\delta_{ij}$). Then for all $x,y\in\fg$ we have
$$\sum\limits_{i=1}^n [[x,a^i],[a_i,y]] =c [x,y],$$
where $2c$ is the eigenvalue of the  Casimir operator 
$C:=\sum\limits_{i=1}^n a^i a_i$ on the adjoint representation
of $\fg$.
\end{lem}
\begin{proof}
Let $\Delta:\cU(\fg)\to\cU(\fg)\otimes\cU(\fg)$ be the coproduct of the universal enveloping superalgebra $\cU(\fg)$.
We have the  commutative diagram
$$  
\xymatrix{&\Ad\otimes\Ad\ar^{\ \ [-,-]}[r]\ar^{\Delta(C)}[d]&\Ad\ar^{C}[d]\\
                 &\Ad\otimes\Ad\ar^{\ \ [-,-]}[r]&\Ad\\
}
$$
Since $p(a_i)=p(a^i)$ we have
$\Delta(C)=C\otimes 1 +1\otimes C+\sum\limits_{i=1}^n a^i\otimes a_i+\sum\limits_{i=1}^n (-1)^{p(a_i)} a_i\otimes a^i$, so
$$\begin{array}{rl}
(\ad C)([x,y])=&[(\ad C)(x),y]+[x,(\ad C)(y)]\\
&+
\sum\limits_{i=1}^n (-1)^{p(a_i)p(x)} [[a^i,x],[a_i,y]]+
\sum\limits_{i=1}^n (-1)^{p(a^i)p(x)+p(a_i)}  [[a_i,x],[a^i,y]].\end{array}$$

Note that $\{(-1)^{p(a_i)} a_i\}_{i=1}^n$ and $\{a^j\}_{j=1}^n$ are dual bases. Since
the sum $\sum_{i=1}^n a^i\otimes a_j$ does not depend on the choice of dual bases,
we obtain

$$\begin{array}{l}\sum\limits_{i=1}^n [[a^i,x],[a_i,y]]=\sum\limits_{i=1}^n (-1)^{p(a_i)} [[a_i,x],[a^i,y]].\end{array}$$

Since $(\ad C)(z)=2cz$ for any $z\in \fg$ we get
$$\begin{array}{l}
2c[x,y]=4c[x,y]+2\sum\limits_{i=1}^n (-1)^{p(a_i)p(x)} [[a^i,x],[a_i,y]]=
4c[x,y]-2\sum\limits_{i=1}^n  [[x,a^i],[a_i,y]]
\end{array}$$ 
which implies the required formula.
\end{proof}

\subsubsection{Proof of~\Prop{prop:rhodelta}}
We set $\hat{\fg}:=\fg^{(r,\sigma)}$ and let $\hat{\fh}$ be a Cartan subalgebra
of $\hat{\fg}$.

We view $\hat{\fg}$ as a $\mathbb{Z}$-graded superalgebra
with 
$\hat{\fg}_i=\fg_{i\mod r} t^i$ for $i\not=0$ and 
$$\hat{\fg}_0=(\fg_0\times \mathbb{C}K)\rtimes\mathbb{C}d$$ except for the case $\hat{\fg}=\mathfrak{psl}(n|n)^{(1)}$
(with $n>1$)
where 
$$\hat{\fg}_0=(\fgl(n|n)\times \mathbb{C}K)\rtimes\mathbb{C}d=(\fg_0\times \mathbb{C}K_1\times \mathbb{C}K)\rtimes (\mathbb{C}d_1\times \mathbb{C}d).$$
We fix a triangular decomposition $\hat{\fg}=\hat{\fn}^-\oplus\hat{\fh}\oplus\hat{\fn}^+$  and consider the induced triangular decomposition  
$\fg_0=(\fn')^-\oplus\fh'\oplus \fn'$. Let $\Sigma$ 
and $\Sigma'$ be the corresponding bases
(then $\Sigma=\Sigma'\cup\{\delta-\theta\}$ where 
$\theta$ is the highest weight of $\fg_0$-module
$\fg_1$). 
We choose a triangular decomposition of $\hat{\fg}$  in such a way that $\Sigma'$
contains an even  root $\beta$.  
Fix $f\in\fg_{-\beta}$, $e\in\fg_{\beta}$ such that
$h:=[e,f]=\nu(\beta)$ where $\nu:\hat{\fh}^*\to \hat{\fh}$ 
is the isomorphism induced by the 
bilinear form $(-,-)$.

Let $v$ be a highest weight vector of
the Verma $\hat{\fg}$-module $M(r\delta+\beta)$.
Let $h_1,\ldots,h_{\ell}$ and $h^1,\ldots,h^{\ell}$ be dual bases of $\hat{\fh}$.
 The Casimir operator of $\hat{\fg}$ is given by the following formula (see~\cite{Kbook}, $\S$~2.5)
 $$\hat{C}:=2\nu^{-1}(\rho)+\sum_{i=1}^{\ell} h_i h^i+2\sum_{\alpha\in\Delta^+} \sum_i e_{-\alpha}^{(i)}e_{\alpha}^{(i)}$$
 where 
$\{e_{\alpha}^{(i)}\}$ is a basis of $\fg_{\alpha}$ and $\{e_{-\alpha}^{(i)}\}$ is the dual basis of $\fg_{-\alpha}$ (i.e.,  $(e_{-\alpha}^{(i)},e_{\alpha}^{(j)})=\delta_{ij}$).
The Casimir operator acts on $M(r\delta+\beta)$ by multiplication on 
$$(\beta+r\delta+2\rho,\beta+r\delta)=2(\beta,\beta)+2r(\rho,\delta).$$
One has $(et^r)(ft^{-r})v=(\beta,\beta)v$, so 
$$(et^r) \hat{C} (ft^{-r})v=2((\beta,\beta)^2+r(\rho,\delta)  (\beta,\beta))v.$$
The vector $v_0:=(ft^{-r})v$ is of the zero weight, so
$$(\nu^{-1}(\rho)+\sum_{i=1}^{\ell} h_i h^i)v_0=0,\ \ e_{\alpha}^{(i)}v_0=0\ \ \text{ if }\alpha\not\leq (r\delta+\beta)$$
where $\leq$ is the standard partial order on $\fh^*$, i.e. $\alpha\leq \gamma$ 
if and only if $(\gamma-\alpha)\in\mathbb{Z}_{\geq 0}\Sigma$.
Thus 
$$(et^r)\hat{C} (ft^{-r})v=(et^r)\hat{C}v_0=2\sum_{0<\alpha\leq r\delta+\beta} \sum_i (et^r) e_{-\alpha}^{(i)}e_{\alpha}^{(i)} 
(ft^{-r})v.$$

For $\alpha=r\delta+\beta$ we can choose
 $e_{\alpha}^{(i)}=et^r$ and $e_{-\alpha}^{(i)}=ft^{-r}$ (by~\cite{Kbook}, Theorem 2.2), so
$$ (et^r)\sum_i e_{-r\delta-\beta}^{(i)} e_{r\delta+\beta}^{(i)} v_0=(et^r)(ft^{-r})(et^r)(ft^{-r})v=(\beta,\beta)^2 v.$$

If $\alpha$ is such that $0<\alpha<(r\delta+\beta)$, then $0<\alpha\leq r\delta$.
We conclude that 
 \begin{equation}
 \label{eq:rrhodeltav} 
 r(\rho,\delta)(\beta,\beta)v= \sum_{\alpha\in\Delta^+: \alpha\leq r\delta} 
 \sum_i (et^r) e_{-\alpha}^{(i)}e_{\alpha}^{(i)} (ft^{-r})v.
 \end{equation}

We write $\fg=\oplus_{i=0}^{r-1} \fg_i$ and let $\iota_{\pm}:\fg\to \hat{\fn}^{\pm}$ be the following embeddings (of vector spaces):
$$\iota_+(a)=\left\{\begin{array}{ll}
a & \text{ if } a\in \fn',\\
at^r & \text{ if } a\in ((\fn')^-\oplus \fh'),\\
at^i & \text{ if } a\in \fg_i\ \text{ for } 0<i<r\\
\end{array}\right.\ \ \ \  \ \ 
\iota_-(a)=\left\{\begin{array}{ll}
a & \text{ if } a\in (\fn')^-,\\
at^{-r} & \text{ if } a\in ((\fn')^+\oplus \fh'),\\
at^{-i} & \text{ if } a\in \fg_i\ \text{ for } 0<i<r\\
\end{array}\right.
$$
 
 For each $\alpha\in\Delta^+$ such that $0<\alpha\leq r\delta$ the root space $\hat{\fg}_{\pm\alpha}$ lies in the image of
$\iota_{\pm}$, so we have $e^{(i)}_{\alpha}=\iota_+(\dot{e}^{(i)}_{\alpha})$ and
$e^{(i)}_{-\alpha}=\iota_-(\dot{e}^{(i)}_{-\alpha})$ for some $\dot{e}^{(i)}_{\pm\alpha}\in\fg$.
If  $\alpha$ is such that $0<\alpha\leq r\delta$, then
$$\begin{array}{rl}
(et^r)  e_{-\alpha}^{(i)}e_{\alpha}^{(i)}  (ft^{-r})&\equiv (et^r)  e_{-\alpha}^{(i)}[e_{\alpha}^{(i)} , (ft^{-r})]v\\
& \equiv 
[(et^r),  e_{-\alpha}^{(i)}[e_{\alpha}^{(i)} , (ft^{-r})]]\equiv [(et^r),  e_{-\alpha}^{(i)}][e_{\alpha}^{(i)} , (ft^{-r})]\\
&\equiv
[[(et^r),  e_{-\alpha}^{(i)}],[e_{\alpha}^{(i)} , (ft^{-r})]] \mod \cU(\hat{\fg})\hat{\fn}^+\end{array}$$
which gives
$$(et^r)  e_{-\alpha}^{(i)}e_{\alpha}^{(i)}  (ft^{-r})v=[[(et^r),  e_{-\alpha}^{(i)}],[e_{\alpha}^{(i)} , (ft^{-r})]] v.$$
Observe that
$$[[(et^r),  e_{-\alpha}^{(i)}],[e_{\alpha}^{(i)} , (ft^{-r})]]=[[e,\dot{e}^{(i)}_{-\alpha}],[\dot{e}_{\alpha}^{(i)} , f]]\mod \mathbb{C}K.$$
Since $Kv=0$ we obtain  
$$[[(et^r),  e_{-\alpha}^{(i)}],[e_{\alpha}^{(i)} , (ft^{-r})]]v=[[e,\dot{e}^{(i)}_{-\alpha}],[\dot{e}_{\alpha}^{(i)} , f]]v.$$
Since $(at^i,bt^{-j})=\delta_{ij}(a,b)$, 
 $\{a^i\}_{i=1}^n:=\{\dot{e}^{(i)}_{\alpha}\}_{0<\alpha\leq r\delta}$ and
$\{a_i\}_{i=1}^n:=\{\dot{e}^{(i)}_{-\alpha}\}_{0<\alpha\leq r\delta}$ are dual bases of $\fg$. 
Combining~(\ref{eq:rrhodeltav}) and~\Lem{lem:Cxy} we obtain
$$ r(\rho,\delta)(\beta,\beta)v= \sum\limits_{j=1}^n [[e,a^j],[a_j,f]] v=
c[e,f]v=c(\beta,\beta)v,
$$
so $r(\rho,\delta)=c$ as required.
\qed

\subsection{Dual Coxeter number}\label{dualCoxeter}
Let $\fg$ be an indecomposable finite-dimensional Kac-Moody superalgebra.
We denote by $h^{\vee}_{(., .)}$ the $\frac{1}{2}$ of the eigenvalue of the corresponding
Casimir operator on $\fg$. The invariant bilinear form $(., .)$ is usually normalised by the following conditions:
$h^{\vee}_{(., .)}\geq 0$ and $(\alpha,\alpha)=2$ for a root of maximal square length.
 This gives a normalization in all cases except for $D(n+1|n)$ where 
$h^{\vee}_{(., .)}=0$; in this case we normalise by the condition $(\alpha,\alpha)=2$ 
for the roots in $D_{n+1}$. This bilinear form is called the
{\em normalised invariant bilinear form} on $\fg$.
The number $h^{\vee}_{(., .)}$ for this normalization 
 is called the {\em dual Coxeter number} of $\fg$ and is denoted by $h^{\vee}$.
 It is equal to $0$ if and only if the Killing form on $\fg$ is zero.
 The following table gives the values of $h^{\vee}$ (recall that $C(n+1)=D(1|n)$):
 $$\begin{array}{|c|c|c|c|c|c|c|c|c|}
 \hline
\fg &A(m|n)&B(m|n), m>n&B(m|n), m\leq n&D(m|n), m>n&D(m|n), m\leq n\\
 \hline
h^{\vee} &|m-n|& 2(m-n)-1& n-m+\frac{1}{2}&2(m-n-1)&n-m+1\\
 \hline\end{array}$$
and $h^{\vee}=4, 9,12,18,30$ for $\fg=G_2$, $F_4$, $E_6$, $E_7$, $E_8$  respectively, and 
 $h^{\vee}=0, 2,3$ for $\fg=D(2|1,a), G(3), F(4)$ respectively~\cite{Kbook},
Chapter VI,~\cite{KW94}, 4.

\subsubsection{} \label{extensionong(1)}
 We extend the invariant bilinear form $(-,-)$ from $\fg$  to $\fg^{(1)}$ 
 in such a way that $(d,d)=0$ and $(at^i,bt^j)=\delta_{ij}(a,b)$ 
 for all $a,b\in\fg$. It is easy to see that such extension is unique. Indeed,
 since $K$ is central we have $(K,u)=0$ for all $u\in [\fg^{(1)},\fg^{(1)}]$.
 Moreover, $i(d,at^i)=(d,[d,at^i])=0$ which gives 
 $(d,at^i)=0$ for all $a\in\fg$ and $i\not=0$. Finally,
 $$([a,b]+i(a,b)K,d)=   ([at^i,bt^{-i}],d)=(at^i, i bt^{-i})=i(a,b) $$
 so $(K,d)=1$ and $(a,d)=0$ for all $a\in\fg$.  
 
 \subsubsection{}
 We view  $\fg^{(r,\sigma)}$ as a subalgebra of $\fg^{(1)}$ and define
 the invariant  bilinear form $(-,-)$ on it as the restriction  of the  bilinear form $r(-,-)$ on $\fg^{(1)}$. We call this form the {\em normalised invariant bilinear form on $\fg^{(r,\sigma)}$.}
 Then
 $$(at^i, bt^j)=r\delta_{ij}(a,b)\ \ \text{ for all } a,b\in\fg\ \text{ such that } at^i,bt^j\in\fg^{(r,\sigma)}.$$
 %This defines the form on  $[\hat{\fg},\hat{\fg}]$. Note that $(K,[\hat{\fg},\hat{\fg}])=0$.
\Prop{prop:rhodelta} gives
\begin{equation}\label{eq:rhodelta}
(\rho,\delta)=h^{\vee}.
\end{equation}

 \subsubsection{}
 \begin{lem}{}
 Let $\hat{\fg}$ be the affine Lie algebra of type $X_N^{(r)}$ with the  invariant bilinear form normalised as above. 
 \begin{enumerate}
 \item
  The  bilinear form $(-,-)$ coincides with the  normalised invariant bilinear form in~\cite{Kbook}, Section 6.
  \item The maximal square length of a root is $2r$.
  \end{enumerate}
 \end{lem}
 \begin{proof}
Let $\hat{\Delta}$ be the root system of $\hat{\fg}$ and let $\hat{m}$ be the maximal square length of a root:
 $\hat{m}:=\max\{(\alpha,\alpha)|\ \alpha\in\hat{\Delta}\}$.
   Since all imaginary roots have  square length zero and $2\alpha$ is a root for any odd non-isotropic root, we have
  $\hat{m}=\max \{(\alpha,\alpha)|\ \alpha\in\hat{\Delta}^{\ree}_{\ol{0}}\}$.
  If $r=1$, then any $\hat{\alpha}\in\hat{\Delta}^{\ree}$ is of the form $\alpha+i\delta$ where $\alpha$ is a root of $\fg$, so $\hat{m}=2$ by~$\S$~\ref{dualCoxeter}. 

  Let $(-,-)'$ be the  normalised bilinear form 
 introduced in~\cite{Kbook}, Section 6; we use the same notation for the induced form on $\hat{\fh}^*$
 (where $\hat{\fh}$ is the Cartan subalgebra of $\hat{\fg}$).
Let $\{\alpha_i\}_{i=0}^{\ell}$ be the set of simple roots
 for $\hat{\fg}$ where the roots are enumerated as in Tables Aff 1--Aff 3 in~\cite{Kbook}, Section 4.
 The root spaces $\fhg_{\pm\alpha_i}$ for $i=1,\ldots,\ell$ generate a finite-dimensional simple
 Lie algebra, so the restrictions of $(-,-)$ and of $(-,-)'$ to this algebra 
 are proportional, i.e. there exists $c\in\mathbb{C}^*$ such that $(\alpha_i,\alpha_j)=c(\alpha_i,\alpha_j)'$
 for all $i,j=1,\ldots,\ell$. Since $\alpha_0$ lies in the linear combination on
 $\delta$ and $\alpha_i$ for $i=1,\ldots,\ell$ and $(\delta,\alpha_i)=0$ for all $i=0,\ldots,\ell$,
 we obtain $(\alpha_i,\alpha_j)=c(\alpha_i,\alpha_j)'$ for all $i,j=0,\ldots,\ell$.   
 Arguying as in~$\S$~\ref{extensionong(1)} and taking into account
 that  $(d,d)=(d,d)'=0$, 
 we conclude that  $(h_1,h_2)=c(h_1,h_2)'$ for all $h_1,h_2\in\fhh$. 
  From~\cite{Kbook}, Theorem 2.2 it follows that any  invariant form on 
$\hat{\fg}$ is uniquely defined by its restriction on $\fhh$. Therefore 
 it remains to verify that $c=1$.
 
 If $r=1$, then Corollary 6.4 in~\cite{Kbook} implies 
$\max \{(\alpha,\alpha)'|\ \alpha\in\hat{\Delta}\}=2$. By above, $\hat{m}=2$, so
$c=1$ as required. In particular, $(\rho,\delta)=(\rho,\delta)'$.
 
 Now take $r=2$ or $r=3$. 
 By~\cite{Kbook}, 6.1, $(\rho,\delta)'$ computed for $\fhg$ and $\fg^{(1)}$ are the same.
 By~(\ref{eq:rhodelta}), $(\rho,\delta)$ computed for $\fhg$ and $\fg^{(1)}$ are the same.
 By above this implies $(\rho,\delta)=(\rho,\delta)'$.
 Since  $\fg_{\ol{1}}=0$, we have $(\rho,\delta)\not=0$, 
   so $c=1$ as required.
  This completes the proof of (i).
  
  For (ii)  note that, since $\hat{\fg}$ is an affine Lie algebra, any real root is $W$-conjugated to one of the simple roots,
 so $\{\alpha_i\}_{i=0}^{\ell}$ contains the longest and the shortest root in $\hat{\Delta}^{\ree}$.
 Using (i) and~\cite{Kbook}, Chapter VI, 
 we obtain 
 $$(\alpha_i,\alpha_i)=(\alpha_i,\alpha_i)'=2a^{\vee}_i/a_i$$ where $\{a_i\}_{i=0}^{\ell}$,
 $\{a^{\vee}_i\}_{i=0}^{\ell}$
 are labels on the Dynkin diagrams of $\hat{\fg}$ and 
 $\hat{\fg}^{L}$ in Tables Aff in Chapter IV ~\cite{Kbook}
 (the  Dynkin diagram of $\hat{\fg}^{L}$ 
 is obtained from the  Dynkin diagram of $\hat{\fg}$ by reversing arrows;
 this notion is not defined for Kac-Moody superalgebras).  
 The Dynkin diagram of $\hat{\fg}$ lies in Table Aff $r$ ($r=1,2,3$).
 
 Let $\beta_0,\ldots,\beta_{\ell}$ be the simple roots of $\hat{\fg}^{L}$. By above,
 \begin{equation}\label{eq:4}
 (\beta_i,\beta_i)(\alpha_i,\alpha_i)=4\end{equation}
 for $i=0,\ldots,\ell$. If $r=1$, then  the square length of the longest root 
 is $2$. Assume that the Dynkin diagram of 
 $\hat{\fg}$ lies in Table Aff 2 or Aff 3
 and is not $A_{2n}^{(2)}$. Then this Dynkin diagram has a double arrow 
 (if $r=2$) or a triple arrow (if $r=3$) and 
 the Dynkin diagram of $\hat{\fg}^{L}$ lies in Table Aff 1. Therefore 
 the Dynkin diagram of $\hat{\fg}^{L}$  has a double arrow (if $r=2$) 
 or a triple arrow (if $r=3$).
 By above,  the square length of the longest root among $\beta_0,\ldots,\beta_{\ell}$
 is $2$ (since $\hat{\fg}^{L}$ is non-twisted), 
 so the square length of the shortest root among $\beta_0,\ldots,\beta_{\ell}$
 is $2/r$. By~(\ref{eq:4}), the square length of the longest root among 
 $\alpha_0,\ldots,\alpha_{\ell}$ is $2r$. The remaining case is $A_{2n}^{(2)}$. 
 In this case the formula $(\alpha_i,\alpha_i)=2a^{\vee}_i/a_i$
 gives $(\alpha_0,\alpha_0)=1$, $(\alpha_{\ell},\alpha_{\ell})=4$ and 
 $(\alpha_i,\alpha_i)=2$ for $1\leq i<\ell$.
\end{proof}

\section{The correspondence $L\mapsto L'$}\label{algebrag'}
 In this section  $\fg$ is an indecomposable symmetrizable Kac-Moody superalgebra  (i.e.,  $\fg$ is of type (An), (Fin) or (Aff)).

\subsection{Root system $\Delta_L$}\label{DeltaLdef}
For a symmetrizable Kac-Moody algebra  $\fg$ a set $R\subset {\Delta}^{\ree}$ is called a {\em root subsystem}  if
$s_{\alpha}\beta\in R$ for any $\alpha,\beta\in R$. 
For each non-critical weight $\lambda$ one can assign a root subsystem
$\Delta_{\lambda}$, called 
the  ``integral root system'' which is the set of real roots of a Kac-Moody algebra $\fg'=\fg(A')$,
see~\cite{Ku},\cite{MP}, \cite{KT98}
and other papers.
This definition can be naturally extended to
the Kac-Moody superalgebras of type (An). The situation is more complicated if $\fg$ is not of type (An).
In this case we define ``root subsystems'' using the axioms  listed in Appendix~\ref{sect:integralrootsystems}.
In~\Cor{cor:Deltalambda1} we show that the integral root system is, in fact, an invariant attached to an irreducible 
highest weight module. Each integral root system,
which is not of type $A(1|1)^{(2)}$, is 
 isomorphic to the set of real roots of some symmetrizable
 Kac-Moody superalgebra $\fg(A')$ of at most countable rank.
For type (An) this result is proved in~\cite{KT98}.
For types (Fin) and (Aff) we prove this result in ~\Thm{thm:Deltalambda} .
In types (Fin) and (Aff), each indecomposable block of the matrix $A'$ is of type (Fin) or (Aff)
(and is of finite size); moreover, the number of blocks is finite if $\lambda$ is non-critical.
We do not know whether this holds for type (An).

\subsubsection{}\label{Deltalambda}
For  $\lambda\in\fh^*$ we set
\begin{equation}\label{Dlambda2}\begin{array}{l}
\ol{R}_{\lambda}=\{\alpha\in \ol{\Delta}^{\ree}\ |\ 
2(\lambda,\alpha)\in j (\alpha,\alpha)\ \text{ where }
j\in\mathbb{Z}\ \text{ and $j$ is odd if $2\alpha\in\Delta$}\},\\
R_{\lambda}:=\ol{R}_{\lambda}\cup\{2\alpha|\ \alpha\in 
\ol{R}_{\lambda},\  2\alpha\in\Delta\}=\ol{R}_{\lambda}\cup\{2\alpha|\ \alpha\in 
(\ol{R}_{\lambda}\cap \Delta^{\an}_{\ol{1}})\}.
\end{array}
\end{equation}

Using~(\ref{eq:alphaveebeta}) and~(\ref{eq:rhoalpha}) we obtain the following useful formulas
\begin{equation}\label{eq:Rlambda+rho}
\begin{array}{l}
\ol{R}_{\lambda-\mu}\cap\Delta^{\an}=\ol{R}_{\lambda}\cap\Delta^{\an}\ \ \text{ for all }\ \mu\in\mathbb{Z}\Delta,\\
\{\alpha\in \Delta_{\ol{0}}|\ \langle\lambda,\alpha^{\vee}\rangle\in\mathbb{Z}\}=
R_{\lambda+\rho}.
\end{array}
\end{equation}

Let $\Delta_{\lambda}\subset \Delta^{\ree}$ be the minimal subset of $\Delta^{\ree}$
satisfying the following properties
\begin{itemize}
\item[(a)] $\ol{R}_{\lambda}\subset \Delta_{\lambda}$;
\item[(b)]   if $\alpha,\beta\in \Delta_{\lambda}$ are such that $(\alpha,\alpha)\not=0$, then $s_{\alpha}\beta\in \Delta_{\lambda}$;
\item[(c)] if $\alpha,\beta\in \Delta_{\lambda}$ are such that $(\alpha,\alpha)=0\not=(\beta,\alpha)$ 
then 
$\Delta_{\lambda}\cap \{\alpha\pm\beta\}=\Delta^{\ree}\cap \{\alpha\pm\beta\}$
(note that, by~Appendix~\ref{sect:integralrootsystems}, in this case
the set $\Delta\cap \{\alpha\pm\beta\}=\Delta^{\ree}\cap \{\alpha\pm\beta\}$
has cardinality one);
\item[(d)]   if $\alpha\in \Delta_{\lambda}$ and $2\alpha\in\Delta^{\ree}$, then
$2\alpha\in \Delta_{\lambda}$.
\end{itemize}
Note that, by (d),  $R_{\lambda}\subset\Delta_{\lambda}$.
By~$\S$~\ref{Xi} below, 
such minimal subset is unique.

%
%
%If $\fg$ is of type (An) we set $\Delta_{\lambda}:=R_{\lambda}$. 
%If $\fg$ is of types (Fin) or (Aff)
%we denote by $\Delta_{\lambda}$ the
%real root subsystem of $\Delta^{\ree}$ which is generated by $R_{\lambda}$. In other words,
%$\Delta_{\lambda}$ is the  minimal subset of 
%$\Delta^{\ree}$ which contains
%$R_{\lambda}$ and satisfies the axioms (GR0)--(GR2) listed in Appendix~\ref{sect:integralrootsystems}.
%
%
%
%It is easy to see that $\Delta_{\lambda}$ is wellndefined: if $\fg$ is finite-dimensional or affine Kac-Moody 
%superalgebra of type (An), then $R_{\lambda}$ satisfies the axioms (GR0)--(GR2).
%

\subsubsection{Remark}
Observe that $\Delta_{\lambda}$ does not depend on the normalization of $(-,-)$
(since $
R_{\lambda}$ does not depend on this normalization).

\subsubsection{}\label{Xi}
Using~(\ref{eq:alphaveebeta}) one can easily sees that $s_{\alpha}\beta\in R_{\lambda}$ 
if $\alpha,\beta\in R_{\lambda}$ are anisotropic. In particular, 
$\Delta_{\lambda}=R_{\lambda}$
if $(\lambda,\alpha)\not=0$ for all $\alpha\in\Delta^{\is}$.

In general, $\Delta_{\lambda}$ can be constructed by the following procedure:
\begin{itemize}
\item we set $X_0:=\ol{R}_{\lambda}$;
\item for odd $i>0$ we take $X_{i+1}$ to be the union of $X_i$, the sets
$s_{\alpha} X_i$ where 
$\alpha\in X_i\cap \Delta^{\an}$;
\item for even $i>0$ we take $X_{i+1}$ to be the union of $X_i$ and the sets
$\{\alpha\pm\beta\}\cap\Delta$ for $\alpha,\beta\in X_i$ satisfying
$(\alpha,\alpha)=0\not=(\alpha,\beta)$.

\end{itemize}
Then $X_0\subset X_1\subset\ldots $ and $X_i=-X_i$. Taking
$X:=\bigcup_{i=0}^{\infty} X_i$ we obtain 
$$\Delta_{\lambda}=X\cup \{2\alpha|\ \alpha\in X\ \text{ and } 2\alpha\in\Delta\}.$$
This implies the uniqueness of $\Delta_{\lambda}$
 and the inclusion  
 $\Delta_{\lambda}\subset \mathbb{Z}\ol{R}_{\lambda}$.

\subsubsection{}
\begin{lem}{lem:Xi}
\begin{enumerate}
\item If $\alpha\in\Delta_{\lambda}\cap \Delta_{\ol{0}}$
is such that  $\alpha/2\not\in\Delta$ and  $(\gamma,\gamma)\in\mathbb{Z}(\alpha,\alpha)$
for all $\gamma\in\ol{R}_{\lambda}$,  then $\alpha\in \ol{R}_{\lambda}$.
\item For any  $\gamma\in \Delta_{\lambda}\setminus R_{\lambda}$ there exists 
 $\alpha\in R_{\lambda}$ such that $(\gamma,\alpha)\not=0$.
\item
For any non-isotropic $\gamma\in \Delta_{\lambda}\setminus R_{\lambda}$ there exists isotropic $\alpha\in \Delta_{\lambda}$
such that $(\alpha,\gamma)\not=0$.
\end{enumerate}
\end{lem}
\begin{proof}
For (i) normalise $(-,-)$ by the condition $(\alpha,\alpha)=2$. Then
$(\lambda,\gamma)\in\mathbb{Z}$ for all $\gamma\in \ol{R}_{\lambda}$. Since  $\Delta_{\lambda}\subset \mathbb{Z}R_{\lambda}$ this gives
$(\lambda,\gamma)\in\mathbb{Z}$ for all $\gamma\in\Delta_{\lambda}$.
In particular, $(\lambda,\alpha)\in\mathbb{Z}$. Since $\alpha/2\not\in\Delta$
we have $\alpha\in\ol{\Delta}^{\ree}$. Therefore  
$\alpha\in\ol{\Delta}^{\ree}_{\ol{0}}$ and 
$\frac{2(\lambda,\alpha)}{(\alpha,\alpha)}=(\lambda,\alpha)\in\mathbb{Z}$. Hence
 $\alpha\in \ol{R}_{\lambda}$ as required.

 For (ii)
observe that $(\gamma,R_{\lambda})=0$ implies 
$(\gamma,\Delta_{\lambda})=0$ (since  $\Delta_{\lambda}$ is spanned by $R_{\lambda}$).
Hence for (ii), (iii)
it is enough to verify the existence of $\alpha\in\Delta_{\lambda}$
satisfying $(\gamma,\alpha)\not=0$ such that $\alpha$ is isotropic if $\gamma$ is anisotropic.
We retain notation of~$\S$~\ref{Xi}. Take  $\gamma\in \Delta_{\lambda}\setminus R_{\lambda}$. Without loss of generality we can assume that $\gamma\in X$, so
 $\gamma\in X_{i+1}\setminus X_i$ for some $i\geq 0$. 
We proceed by induction on $i$.

If $i$ is even, then $\gamma\in X_{i+1}\setminus X_i$
implies the existence of  $\alpha,\beta\in X_{i}$
such that $\gamma\in \{\alpha\pm\beta\}$ and $(\alpha,\alpha)=0\not=(\alpha,\beta)$.
Then $(\gamma,\alpha)\not=0$ which implies the required assertions
for the case when $i$ is even (in particular, for $i=0$).

Consider the case when $i$ is odd.
Then there exists $\beta_1,\beta_2\in X_{i}$ such that $(\beta_1,\beta_1)$, $(\beta_1,\beta_2)\not=0$ and
$\gamma=s_{\beta_1}\beta_2$. Then $(\gamma,\beta_1)\not=0$ and 
this completes the proof of (ii). 
For (iii) we assume that $\gamma$ is anisotropic.
Then  $\beta_2$ is anisotorpic as well.
If $\Delta_{\lambda}$ contains an isotropic root $\alpha_2$ 
such that $(\beta_2,\alpha_2)\not=0$, then $(\gamma,\alpha)\not=0$ for
$\alpha:=s_{\beta_1}\alpha_2$ and
 $\alpha\in\Delta_{\lambda}$ is isotropic, since
$\beta_2,\alpha_2\in\Delta_{\lambda}$ and $\alpha_2$ is isotropic.
Consider the remaining case when 
 $\Delta_{\lambda}$ does not contain an isotropic root $\alpha_2$ 
such that $(\beta_2,\alpha_2)\not=0$. By induction hypothesis this implies $\beta_2\in R_{\lambda}$.
Since $\gamma\not\in R_{\lambda}$, this forces
$\beta_1\not\in R_{\lambda}$. By induction hypothesis, there exists isotropic $\alpha\in \Delta_{\lambda}$ 
such that $(\beta_1,\alpha)\not=0$. By above, $(\beta_2,\alpha)=0$.  Therefore
$(\gamma,\alpha)=(s_{\beta_1}\beta_2,\alpha)\not=0$ (since $\gamma\not=\beta_2$).
This completes the proof of (iii).
\end{proof}

\subsubsection{}
\begin{lem}{lem:B11}
Suppose that $Z\subset\Delta^{\ree}$ satisfies properties (b)--(d) in~$\S$~\ref{Deltalambda} 
(with $\Delta_{\lambda}$ replaced by $Z$) and $Z$ contains two odd real roots $\alpha,\beta$ such that 
$(\beta,\beta)=0$, $(\alpha,\alpha)\not=0$, 
$\langle\beta, \alpha^{\vee}\rangle<0$.
Then 
$$(\mathbb{R}\alpha+\mathbb{R}\beta)\cap Z=\{\pm\alpha,\pm\beta,\pm 
2\alpha,\pm (\alpha+\beta), \pm (2\alpha+\beta)\}=(\mathbb{R}\alpha+\mathbb{R}\beta)\cap \Delta^{\ree},$$
 i.e. $(\mathbb{R}\alpha+\mathbb{R}\beta)\cap Z=(\mathbb{R}\alpha+\mathbb{R}\beta)\cap \Delta^{\ree}$ 
forms the root system
$B(1|1)$ \footnote{The  statement does not immediately follow from Serganova's
classification of generalized root systems in~\cite{VGRS} since we do not assume that
$(\mathbb{R}\alpha+\mathbb{R}\beta)\cap \Delta^{\ree}$ is finite.} and $\{\alpha,\beta\}$
is its base.
\end{lem}
\begin{proof}
Since $\beta$ is an isotropic real root, $\Delta^{\ree}$ is of type (Fin) or (Aff).

Renormalise the bilinear form by taking $(\alpha,\alpha)=1$.  Then
$b:=-(\alpha,\beta)>0$ and, by~(\ref{eq:alphaveebeta}),
$b\in \mathbb{Z}_{>0}$. 
 Recall that $\fg_{\pm \alpha}$ generate a subalgebra isomorphic to $\mathfrak{osp}(1|2)$. Since  $\beta$ and $s_{\alpha}\beta=\beta+2b\alpha$ are real roots,
 one has  $\alpha+\beta, 2\alpha+\beta\in \Delta^{\ree}$. One has
 $(\alpha+\beta,\alpha+\beta)=1-2b\not=0$,
so $\alpha+\beta$ is a real root and
$\frac{2(\alpha+\beta,\beta)}{(\alpha+\beta,\alpha+\beta)}\in\mathbb{Z}$ that is
$\frac{b}{2b-1}\in\mathbb{Z}$. Since $b$ is a positive integer, 
 this implies  $b=1$, so  $s_{\alpha}\beta=\beta+2\alpha$. 
Using  properties (b)--(d) 
 we obtain 
 $$(\mathbb{R}\alpha+\mathbb{R}\beta)\cap Z\supset Y\ \ \text{ for }
 Y:=\{\pm\alpha,\pm\beta,\pm 2\alpha,\pm (\alpha+\beta), \pm (2\alpha+\beta)\}.
$$

Let us show that $(\mathbb{R}\alpha+\mathbb{R}\beta)\cap \Delta^{\ree}=Y$.

 Set $\delta_1:=\alpha$, $\vareps_1:=\alpha+\beta$. Note that
 $\delta_1$ is an odd root and $\vareps_1$ is an even root. 
 Using $b=1$ we get $(\delta_1,\vareps_1)=0$, $(\vareps_1,\vareps_1)=-1$.

 Suppose the contrary, that $(\mathbb{R}\alpha+\mathbb{R}\beta)\cap \Delta^{\ree}\not=Y$.
 Take  $\gamma:=(x\vareps_1+y\delta_1)\in (\Delta^{\ree}\setminus Y)$. 
Then $s_{\vareps_1}\gamma$, $s_{\delta_1}\gamma\in \Delta^{\ree}$, so 
 $(\pm x\vareps_1\pm y\delta_1)\in \Delta^{\ree}$. Without loss
 of generality we can (and will) assume that $x,y\geq 0$.
 If $x=0$, then $\gamma\in\mathbb{R}\alpha$, so $\gamma\in\{\delta_1,2\delta_1\}\subset Y$.
 Similarly, if $y=x$, then $\gamma=\vareps_1+\delta_1\in Y$. Therefore 
 $x>0$, $y\geq 0$ and  $x\not=y$. Then $(\gamma,\gamma)\not=0$ and the
 conditions
 $$\langle \gamma,\vareps_1^{\vee}\rangle, 
\langle \vareps_1,\gamma^{\vee}\rangle, \langle \delta_1,\gamma^{\vee}\rangle\in\mathbb{Z},\ \ 
\langle \gamma,\delta_1^{\vee}\rangle\in 2\mathbb{Z}  $$
give $2x, y\in\mathbb{Z}$ and $\frac{2x}{x^2-y^2}, \frac{2y}{x^2-y^2}\in\mathbb{Z}$.
Therefore $2x\in\mathbb{Z}_{>0}$, $y\in\mathbb{Z}_{\geq 0}$, $x\not=y$ and
$\frac{2}{x+y}\in\mathbb{Z}$. This gives
$$(x,y)\in \{(\frac{1}{2},0), (1,0), (2,0)\}.$$
In particular, $(\frac{\vareps_1}{2}+\delta_1)\not\in\Delta$ (in this case
$x=\frac{1}{2}$, $y=1$).
If $(x,y)=(\frac{1}{2},0)$, then $\gamma=\frac{\vareps_1}{2}$, and
 $\delta_1$, $\vareps_1+\delta_1\in\Delta$ forces
$\frac{\vareps_1}{2}+\delta_1\in\Delta$, a contradiction.
If $(x,y)=(1,0)$, then $\gamma=\vareps_1\in Y$. If
$(x,y)=(2,0)$, then $\gamma=2\vareps_1\not\in \Delta^{\ree}$
since $\vareps_1$ is an even real root. Hence
$(\mathbb{R}\alpha+\mathbb{R}\beta)\cap \Delta^{\ree}=Y$. This completes the proof.
\end{proof}

\subsubsection{}
\begin{cor}{cor:Delta2beta}
If $\beta\in\Delta^{\ree}$ is such that 
$2\beta\in\Delta^{\ree}$, then $\beta\in\Delta_{\lambda}$
if and only if $2\beta\in \Delta_{\lambda}$.
\end{cor}
\begin{proof}
If $\beta\in\Delta_{\lambda}$, then $2\beta\in \Delta_{\lambda}$ by  property (d)
of $\Delta_{\lambda}$.

Suppose the contrary, that $\beta\in\Delta^{\ree}$ is such that 
$\beta\not\in\Delta_{\lambda}$ and $2\beta\in \Delta_{\lambda}$.
Since $2\beta$ is a real root, 
 $\beta$ is an odd non-isotropic root.
If $2\beta\in R_{\lambda}$, then, by definition of $R_{\lambda}$, we have $\beta\in \ol{R}_{\lambda}$,
 a contradiction. Therefore $2\beta\not\in R_{\lambda}$.
By~\Lem{lem:Xi} (iii), $\Delta_{\lambda}$ contains an isotropic root $\alpha$ such that 
$\langle \alpha, (2\beta)^{\vee}\rangle<0$. 
Then, by~\Lem{lem:B11},
 $\beta,\alpha$ form a base of the root subsystem 
$\Delta_{\lambda}\cap(\mathbb{Z}\beta+\mathbb{Z}\alpha)$
 which  is of type $B(1|1)$, and $(\alpha+\beta,\alpha+\beta)=-(\beta,\beta)$.
 
 Since $\alpha$ is isotropic, $\Delta$ is  of type (Fin) or (Aff).
From the Kac and van de Leur  classifications~\cite{Ksuper} and~\cite{vdLeur} we know that  each odd
 non-isotropic root in $\Delta$ is a short root. Renormalise the bilinear form $(-,-)$
 in such a way that $(\beta,\beta)=2$. Since $\beta$ is a short root,
 $(\gamma,\gamma)\in 2\mathbb{Z}$ for all $\gamma\in\Delta$,
 so $(\lambda,\gamma)\in \mathbb{Z}$ for all $\gamma\in \ol{R}_{\lambda}$ which implies
 $(\lambda,\gamma)\in \mathbb{Z}$ for all $\gamma\in\Delta_{\lambda}$.
 By above,  $(\alpha+\beta)$ is an even root with
  $(\alpha+\beta,\alpha+\beta)=-2$. Hence 
 $\alpha+\beta,\alpha\in\Delta_{\lambda}$ with
 $(\alpha,\alpha)=0\not=(\alpha,\beta)$. 
 Property (c) in~$\S$~\ref{Deltalambda} 
  gives $\beta\in \Delta_{\lambda}$, a contradiction.
 \end{proof}

\subsubsection{}
\begin{lem}{lem:nualpha}
If $\alpha$ is a real isotropic root and  $\nu\in\fh^*$ is such that
 $(\nu,\alpha)=0$, then $\Delta_{\nu}=\Delta_{\nu+\alpha}$.
 \end{lem}
 \begin{proof}
 Let us verify that $\ol{R}_{\nu}\subset\Delta_{\nu+\alpha}$. Take $\beta\in \ol{R}_{\nu}$. 
 
 If $(\beta,\alpha)=0$, then $(\nu,\beta)=(\nu+\alpha,\beta)$, so
$\beta\in \ol{R}_{\nu+\alpha}$. If
 $\beta$ is anisotropic, then, by~(\ref{eq:alphaveebeta}),
 $\langle\alpha,\beta^{\vee}\rangle$ is integral 
and is even if $\beta$ is odd; thus
 $\beta\in \ol{R}_{\nu+\alpha}$. This gives
\begin{equation}\label{Dlambda1}
\{\beta\in \ol{R}_{\nu}|\  
(\beta,\alpha)=0\}\subset \ol{R}_{\nu+\alpha},\ \ \ \ (\ol{R}_{\nu}\cap {\Delta}^{\an})\subset \ol{R}_{\nu+\alpha}.
\end{equation}

Now take $\beta\in (\ol{R}_{\nu}\cap \Delta^{\is})$ such that $(\beta,\alpha)\not=0$.
By~$\S$~\ref{Deltalambda} (c), the set 
 $\Delta_{\nu}\cap \{\alpha\pm\beta\}$ is non-empty. 
 We assume that $(\alpha+\beta)\in \Delta_{\nu}$
 (otherwise we  substitute
 $\beta$ by $-\beta$). 
Note that $\alpha+\beta$ is a real even root (since $(\alpha+\beta,\alpha+\beta)=2(\alpha,\beta)\not=0$).
By above, $(\nu,\alpha)=(\nu,\beta)=0$, so 
$(\nu,\alpha+\beta)=0$.

Assume that $(\alpha+\beta)/2\not\in\Delta$. Then
 $(\alpha+\beta)\in (\ol{R}_{\nu}\cap {\Delta}^{\an})$, so
 $(\alpha+\beta)\in \ol{R}_{\nu+\alpha}$
 by~(\ref{Dlambda1}). Therefore $\ol{R}_{\nu+\alpha}$ contains both $\alpha$
and $\alpha+\beta$. Since $(\alpha,\alpha+\beta)=(\alpha,\beta)\not=0$,
 $\S$~\ref{Deltalambda} (c) gives
 $\beta\in \Delta_{\nu+\alpha}$ as required.

 Now consider the remaining case $\alpha+\beta=2\gamma$ for  $\gamma\in\Delta$.
Then $\gamma$ is an odd non-isotropic root, $\alpha$
is an isotropic root and $(\gamma,\alpha)\not=0$. By~\Lem{lem:B11},
$\Delta\cap(\mathbb{R}\gamma+\mathbb{R}\alpha)$
  is the root system   of type $B(1|1)$. By above, $\Delta$ contains $\beta=\alpha-2\gamma$,
so, using the standard notation for $B(1|1)$,
we can assume $\gamma=\delta_1$ and $\alpha=\delta_1+\vareps_1$.
We have $(\nu,\alpha)=(\nu,\beta)=0$, so $(\nu,\vareps_1)=0$ and
$$\langle\nu+\alpha,\vareps_1^{\vee}\rangle=
\langle \delta_1+\vareps_1, \vareps_1^{\vee}\rangle=2.$$
Since $\vareps_1$ is an  even real root 
and $\vareps_1/2\not\in\Delta$, we obtain
$\vareps_1\in \ol{R}_{\nu+\alpha}$. 
Using  $\alpha=\delta_1+\vareps_1\in \ol{R}_{\nu+\alpha}$ and~$\S$~\ref{Deltalambda} (b) we get $s_{\vareps_1}\alpha\in \Delta_{\nu+\alpha}$. One has
$s_{\vareps_1}\alpha=\delta_1-\vareps_1=-\beta$.
Hence $\beta\in \Delta_{\nu+\alpha}$
as required.

By above, $\ol{R}_{\nu}\subset\Delta_{\nu+\alpha}$. This gives $\Delta_{\nu}\subset \Delta_{\nu+\alpha}$.
Since $(\nu+\alpha,-\alpha)=0$ we have $\Delta_{\nu+\alpha}\subset \Delta_{\nu}$. 
This completes the proof.
\end{proof}

\subsubsection{}
\begin{cor}{cor:Deltalambda1}
If $L(\lambda_1,\Sigma_1)=L(\lambda,\Sigma)$ for some  $\Sigma_1\in\Sp$
and $\rho,\rho_1$ are the Weyl vectors of $\Sigma$ and $\Sigma_1$ respectively, then 
$\Delta_{\lambda+\rho}=\Delta_{\lambda_1+\rho_1}$.  
\end{cor}
\begin{proof}
It is enough to verify the  assertion for $\Sigma_1=r_{\alpha}\Sigma$
where $\alpha\in  (\Sigma\cap \Delta^{\is})$. 

If  $(\lambda+\rho,\alpha)\not=0$, then $\lambda_1+\rho_1=\lambda+\rho$,
so $\Delta_{\lambda+\rho}=\Delta_{\lambda_1+\rho_1}$.

If $(\lambda+\rho,\alpha)=0$, then
$\lambda_1+\rho_1=\lambda+\rho+\alpha$ and the assertion follows from~\Lem{lem:nualpha}.
\end{proof}

\subsection{}
\begin{defn}{defn:DeltaL} For an irreducible module $L\in\CO_{\Sigma}(\fg)$ 
we set $\Delta_L:=\Delta_{\lambda+\rho}$ where $\lambda$ is the highest weight of $L$.
If $N\in{\CO}^{\infty}(\fg)$ is such that $\Delta_L$ is the same for all
irreducible subquotients of $N$, we denote this set by $\Delta_N$; in this
case we say that $\Delta_N$ is well defined.
%and denote by $W_N$ the subgroup of $W$ generated by the reflections
%$s_{\alpha}$, $\alpha\in\Delta_N$.
\end{defn}

\subsubsection{}
By~\Cor{cor:Deltalambda1} the set $\Delta_L$ does not depend on 
the choice of a base in the spine. In~$\S$~\ref{weightnu'} we will see that $\Delta_N$ is defined
for any indecomposable non-critical module $N$.

\subsubsection{}\label{KMsubsystem}
By~\Thm{thm:Deltalambda}, if $\fg(A,\tau)$ is of type (Fin) or (Aff), but different  from $\fg(A,\tau)\not=A(2m-1|2n-1)^{(2)}$,
and 
$L\in\CO_{\Sigma}(\fg)$ is a non-critical irreducible module, then
there exists a   symmetrizable Kac-Moody superalgebra $\fg(A',\tau')$ 
such that
\begin{itemize}
\item[(a)] the sets $\Delta^{\ree}(\fg(A',\tau'))$ and  $\Delta_L$ are isometric; denote by 
$\psi:\Delta^{\ree}(\fg(A',\tau'))\iso \Delta_L$ the isometry;
\item[(b)] for any $\alpha'\in \Delta^{\ree}(\fg(A',\tau'))$,
the roots $\alpha'$ and $\psi(\alpha')\in\Delta$
 have the same parity;
\item[(c)] any indecomposable submatrix $A_1'$  of $A'$ is of type (Fin) or (Aff); moreover, for such $A'$
 the restriction of $\psi$ gives a monomorphism
 $\mathbb{C}\Delta(\fg(A'_1,\tau_1'))\to\mathbb{C}\Delta$.
 \end{itemize}
For  $\fg(A,\tau)=A(2m-1|2n-1)^{(2)}$ the similar assertion holds
if $\Delta_L$ does not have an indecomposable component
of type $A(1|1)^{(2)}$.

For type (An) \Prop{prop:RR'an} implies the existence
of  $\fg(A',\tau')$, where $A'$ can be of infinite size, and the isometry
$\psi: \mathbb{C}\Delta(\fg(A',\tau'))\iso \mathbb{C}\Delta$ satisfying  (b).

From now on we always exclude the case when  $\fg(A,\tau)=A(2m-1|2n-1)^{(2)}$ and  $\Delta_L$ has an indecomposable component
of type $A(1|1)^{(2)}$.

\subsubsection*{Remark}
If $\fg(A,\tau)$ is of type  (Aff) and $L$ is critical, then 
a similar result holds,  but $\fg'$ might be a countable product 
of Kac-Moody superalgebras, with almost all factors isomorphic to $\fgl(1|1)$,
see~\Lem{lem:decompositionR}.

 \subsubsection{Example}
 Let $\fg=\fsl_4^{(1)}$ and $L=L(\lambda)$ for $\lambda:=\vareps_1-\vareps_2+\frac{\vareps_3-\vareps_4}{2}$.
 Then 
 $$\Delta_{\lambda+\rho}=\{\pm(\vareps_1-\vareps_2)+\mathbb{Z}\delta\}\cup
 \{\pm(\vareps_3-\vareps_4)+\mathbb{Z}\delta\}$$
 and $\fg(A')=\fsl_2^{(1)}\times \fsl_2^{(1)}$. Let
 $\{\alpha_1,\alpha_0\}$, $\{\beta_1,\beta_0\}$ be the sets of simple roots for two copies
 of $\fsl_2^{(1)}$ in $\fg(A')$. The isometry $\psi$ is given by
 $$\psi(\alpha_1):=\vareps_1-\vareps_2,\ \ \psi(\alpha_0)=\delta-(\vareps_1-\vareps_2),
 \ \ \psi(\beta_1)=\vareps_3-\vareps_4,\ \ \psi(\beta_0)=\delta-(\vareps_3-\vareps_4).$$
We see that the restriction of $\psi$  induces a bijection (of sets) between 
$\Delta^{\ree}(\fg(A'))$ and  $\Delta_L$, but the restriction of $\psi$  to $\Delta(\fg(A'))$ is not injective: $\psi$ maps
the minimal imaginary roots in both copies of $\fsl_2^{(1)}$ to the minimal imaginary root of $\fsl_4^{(1)}$.

\subsubsection{Example}\label{A112badexa}
Take $\fg=A(1|3)^{(2)}$
with $\Sigma=\{\delta-2\vareps_1,\vareps_1-\delta_1,\delta_1-\delta_2, \delta_2-\delta_3, 2\delta_3\}$ 
and $\lambda=\Lambda_0+\frac{3}{4}\delta_2+\frac{1}{4}\delta_3$. Then $\Delta_{\lambda+\rho}=A(1|1)^{(2)}\times 
A_1^{(1)}\times A(1|1)^{(1)}$,where 
$$
A(1|1)^{(2)}=\{\mathbb{Z}\delta\pm\vareps_1\pm\delta_1,\ 
2\mathbb{Z}\delta\pm 2\delta_1,\ 
(2\mathbb{Z}-1)\delta\pm 2\vareps_1\},
$$
and two copies of $A_1^{(1)}$ correspond to the subsystems
$A_1^{(1)}=\{\mathbb{Z}\delta\pm \alpha\}$
for $\alpha=\delta_2-\delta_3$ and $\alpha=\delta_2+\delta_3$.

\subsubsection{}
\begin{lem}{lem:KMbase}
Fix a base $\Sigma$ and set  $\Delta_L^+:=\Delta_L\cap\Delta^+$.
\begin{enumerate}
\item
The root system $\Delta(\fg(A',\tau'))$ admits a base $\Sigma'$ such that 
$$\psi\bigl((\Delta^{\ree}(\fg(A',\tau'))^+\bigr)=\Delta_L^+.$$
\item The set $\Sigma_L:=\psi(\Sigma')$ coincides with the set of indecomposable elements in $\Delta_L^+$
(i.e., the elements which can not be presented as a sum of several elements from  $\Delta_L^+$).
\end{enumerate}
\end{lem}
\begin{proof}
For type (An) (i) follows from~\cite{KT98}, Section 2; in this case
$$
\Sigma_L=\{\alpha\in \Delta_L|\ s_{\alpha}\bigl(\Delta_L^+\setminus 
\{\alpha,2\alpha\}\bigr)\subset \Delta_L^+\},\ \ \ \Sigma':=\psi^{-1}(\Sigma_L).$$

Let us   prove both statements for types (Fin) and (Aff).   
It is enough to check the statements for each indecomposable component of $\fg(A',\tau')$. 
Let  $\Delta'$ be the root system of
such component.  Write $\Sigma:=\{\alpha_i\}_{i=1}^{\ell}$ and let $a_1,\ldots, a_{\ell}$
  be $\mathbb{Q}$-linearly independent positive real numbers.  
Define a homomorphism $\omega: \mathbb{Z}\Delta\to \mathbb{R}$ by taking $\omega(\alpha_i):=a_i$.
  Note that $\omega(\mu)\not=0$ for all non-zero $\mu\in\mathbb{Q}\Delta$.
  In particular, $\omega(\alpha)\not=0$ for all $\alpha\in \Delta$ and 
  $\Delta^+\coprod \Delta^-$ is the triangular decomposition corresponding to $\omega$ in the sense
  of~$\S$~\ref{triangular}. Since $\Delta'$ lies in $\mathbb{Z}(\Delta')^{\ree}$ we have
  $\psi(\Delta')\subset \mathbb{Z}\Delta$ and property (c) in $\S$~\ref{KMsubsystem} implies
  $(\omega\circ \psi)(\mu')\not=0$ for all $\mu'\in\Delta'$.
    Therefore $\omega\circ \psi:\mathbb{Z}\Delta'\to\mathbb{R}$ determines triangular decomposition  
    $\Delta'=-(\Delta')^+\coprod (\Delta')^+$
in the sense  of~$\S$~\ref{triangular}.  If $\alpha'\in (\Delta')^+$, then 
$(\omega\circ\psi)(\alpha')>0$ that is  $\psi(\alpha')\in \Delta_L^+$.
Since $\Delta'$ is the root system of a Kac-Moody superalgebra 
of type (Fin) or (Aff), 
any triangular decomposition of $\Delta'$ admits a base (see
Theorem 0.4.3 in~\cite{Shay}). 
This completes the proof of (i).
Now  (ii) follows from (i) and the fact that $\Sigma'$ coincides with the set of indecomposable elements in 
$\Delta^{\ree}(\fg(A',\tau'))^+$.
\end{proof}

  \subsubsection{}\label{baseKMtypes}
Let $\fg(A',\tau')$ and $\psi$ be as in~$\S$~\ref{KMsubsystem} and set 
  $\Delta':=\Delta(\fg(A',\tau'))$. 
  Recall that 
$\psi$ gives a bijection $(\Delta')^{\ree}\iso \Delta_L$.
Using~\Lem{lem:KMbase} (i),
we assign to each   base $\Sigma_1$ of $\Delta$ a base  of  $\Delta'$, which we denote by $\Upsilon(\Sigma_1)$, which satisfies the property
   $$\psi\bigl((\Delta')^{\ree}\cap \Delta^+(\Upsilon(\Sigma_1))\bigr)=
   \Delta_L\cap\Delta^+(\Sigma_1) .$$

We claim that    for each isotropic $\alpha\in\Sigma$ we have 
\begin{equation}\label{eq:ralph}
\Upsilon(r_{\alpha}\Sigma)=\left\{\begin{array}{lcl}
\Upsilon(\Sigma) & & \text{ if } \ \alpha\not\in\psi(\Upsilon(\Sigma)),\\
r_{\psi^{-1}(\alpha)}\Upsilon(\Sigma) & & \text{ if }\  \alpha\in \psi(\Upsilon(\Sigma)). \\
\end{array}\right.\end{equation}
Indeed,  $\Delta^+\setminus (\Delta^+(r_{\alpha}\Sigma))= \{\alpha\}$
by~(\ref{eq:Delta+oddrefl}). Setting $\Sigma':=\Upsilon(\Sigma)$ we get
$$\begin{array}{lcl}
\Delta^+((r_{\alpha}\Sigma'))=\Delta^+(\Sigma') &  &\text{ if }\ \alpha\not\in\Delta_{\lambda+\rho},   \\
\Delta^+\setminus (\Delta^+(r_{\alpha}\Sigma)')=\Delta^+(\Sigma')\cap \mathbb{Z}\psi^{-1}(\alpha)&  &\text{ if }\ \alpha\in\Delta_{\lambda+\rho}.
\end{array}$$
 By~\Lem{lem:KMbase} (ii), the condition 
 $\alpha\not\in\Delta_L$ is equivalent to the condition
 $\alpha\not\in\psi(\Sigma')$. This implies~(\ref{eq:ralph}).

 Notice that, since $\psi$ preserves the bilinear 
form,  $\psi^{-1}(\alpha)$ is isotropic.
 Denote by $\Sp'$ the spine of $\Sigma'$. Using~(\ref{eq:ralph}) we obtain the correspondence
$\Upsilon:\Sp\to \Sp'$. Viewing $\Sp,\Sp'$
as graphs were vertices are the bases, and 
the edges correspond to the isotropic reflexions
$r_{\alpha}$, we obtain the following description of $\Upsilon$:
$\Upsilon(\Sigma)=\Sigma'$ and for every edge 
$r_{\alpha}:\Sigma_1\to\Sigma_2$ in $\Sp$  we have $\Upsilon(\Sigma_2):=\Sigma'_2$ 
if $\Sp'$ contains the edge $r_{\psi^{-1}(\alpha)}:\Sigma'_1\to\Sigma'_2$ and
$\Upsilon(\Sigma_2):=\Sigma'_1$ otherwise.

For example, for $\fsl(2|1)$ take $\Sigma=\{\alpha,\beta\}$ where $\beta$ is isotropic
and $\alpha$ is anisotropic.
Take $\lambda+\rho=a\beta$ for $a\not\in\mathbb{Z}$.
Then $\Delta_{\lambda+\rho}=\{\pm\beta\}$, so $\fg(A',\tau')\cong \fgl(1|1)$
with $\Sigma'=\{\beta'\}$ ($\beta'$ is isotropic). Then  $\Upsilon: \Sp\to \Sp'$ takes the following form 
$$\xymatrix{& \Sp & \Sigma\ar[r]^{r_{\beta}}\ar[d]^{\Upsilon} &\Sigma_1 \ar[r]^{r_{\alpha+\beta}}\ar[d]^{\Upsilon} &\Sigma_2\ar[d]^{\Upsilon}\\
&   & \Sigma'\ar[r]^{r_{\beta'}} &\Sigma'_1  &\Sigma'_1\\
}$$

\subsubsection{}\label{psidelta}
One can similarly introduce the correspondence between the skeleton $\Sk$
and the skeleton $\Sk'$.

Let $\fg''$ be an indecomposable component of $\fg'$. 
 Assume that $\fg''$ is of type (Aff) and let  
 $\delta''$ be the minimal imaginary root of $\Delta^+(\fg'')$. Then
 $\delta''$ is a  positive root for  any triangular decomposition
appearing in $\Sk'$, so $\psi(\delta'')$ lies in $\mathbb{Z}_{\geq 0}\Delta^+(\Sigma_1)$ for any 
$\Sigma_1\in \Sk$. By~\cite{GHS}, Proposition 7.6.1 this implies
$\psi(\delta'')\in \mathbb{Z}_{>0}\delta$.

\subsection{Module $L'$}\label{glambda}
Take $L=L(\lambda)$.
The following construction is a slight variation of the construction given in~\cite{GKadm}.
 Retain notation of~$\S$~\ref{baseKMtypes}.
 We set 
 $$\fg':=\fg(A',\tau'),\ \  \Delta':=\Delta(\fg(A',\tau')).$$
 Recall that $\psi$ gives a bijection $(\Delta')^{\ree}\iso \Delta_L$, see~$\S$~\ref{KMsubsystem}. We set 
$(\Delta')^{\perp}:=\{\mu'\in (\fh')^*|\ (\mu',\Delta')=0\}$ and 
  let  $\psi_*: \fh^*\to (\fh')^*/(\Delta')^{\perp}$ be the linear map satisfying
\begin{equation}\label{eq:psi*}
(\psi_*(\nu),\mu')=(\nu,\psi(\mu'))\ \ \text{ for all }\mu'\in\Delta'.\end{equation}

  \subsubsection{}
  Take $\Sigma'$ as in~\Lem{lem:KMbase}  and let 
  $\fg(A',\tau')=(\fn')^-\oplus\fh'\oplus (\fn')^+$ be the corresponding   triangular decomposition of 
$\fg(A',\tau')$; let $\rho'\in (\fh')^*$ be the corresponding 
 Weyl vector. Fix 
 $\lambda'\in (\fh')^*$ such that 
\begin{equation}\label{eq:lambdalphadef} 
 (\lambda'+\rho',\alpha'):=(\lambda+\rho,\psi(\alpha'))\ 
 \text{ for all }\alpha'\in\Sigma'.\end{equation}
 and set
\begin{equation}\label{eq:L'}
 L':=L_{\fg'}(\lambda').\end{equation}

 Note that the choice of $\lambda'$ is not unique:  $\lambda''$ satisfies~(\ref{eq:lambdalphadef})
 if and only if $\psi_*(\lambda')=\psi_*(\lambda'')$, or equivalently, if
  $L_{\fg'}(\lambda'')$ and $L'$ are isomorphic as $[\fg',\fg']$-modules.

\subsubsection{}
\begin{lem}{lem:RRlambda}
\begin{enumerate}
\item
For any $\mu'\in\mathbb{Z}\Delta'$ we have 
$$\psi(\ol{R}_{\lambda'+\rho'-\mu'})\subset \ol{R}_{\lambda+\rho-\psi(\mu')},\ \ \ 
\psi({\Delta}_{\lambda'+\rho'-\mu'})\subset \Delta_{\lambda+\rho-\psi(\mu')}.$$
\item We have $\psi(\ol{R}_{\lambda'+\rho'})=\ol{R}_{\lambda+\rho}$ 
 and $\Delta_{L'}=(\Delta')^{\ree}$.\end{enumerate}
\end{lem}
\begin{proof}
For any $\mu'\in\mathbb{Z}\Delta'$  and $\beta'\in\Delta'$ 
we have $(\lambda'+\rho'-\mu',\beta')=(\lambda+\rho-\psi(\mu'),\psi(\beta'))$ by~(\ref{eq:lambdalphadef}).
Since $\psi$ preserves $(-,-)$ and the parity,  this implies
 $\psi(\ol{R}_{\lambda'+\rho'-\mu'})\subset \ol{R}_{\lambda+\rho-\psi(\mu')}$. 
If $\fg$ is of type (An), then $\Delta_{\lambda'+\rho'-\mu'}=R_{\lambda'+\rho'-\mu'}$
and $\Delta_{\lambda+\rho-\mu}=R_{\lambda+\rho-\mu}$.
If $\fg$ is of type (Fin) or (Aff), then, by definition, $\Delta_{\lambda'+\rho'-\mu'}$ 
is the minimal subset of $(\Delta')^{\ree}$ which contains
$\ol{R}_{\lambda'+\rho'-\mu'}$ and 
satisfies  properties (b)--(d) in~$\S$~\ref{Deltalambda}. 
Therefore $\psi(\Delta_{\lambda'+\rho'-\mu'})$ is a subset of $\Delta^{\ree}$ which contains
$\psi(\ol{R}_{\lambda'+\rho'-\mu'})$ and 
satisfying   properties (b)--(d). By above, $\psi(\ol{R}_{\lambda'+\rho'-\mu'})\subset \ol{R}_{\lambda+\rho-\psi(\mu')}$, so 
$\psi({\Delta}_{\lambda'+\rho'-\mu'})\subset \Delta_{\lambda+\rho-\psi(\mu')}$. This establishes (i).

For (ii) take any $\beta\in\Delta_L$. Then $\beta':=\psi^{-1}(\beta)\in (\Delta')^{\ree}$
and $(\lambda'+\rho',\beta')=(\lambda+\rho,\beta)$ by~(\ref{eq:lambdalphadef}). Since
$\psi$ preserves $(-,-)$ and the parity,  we obtain 
$\beta\in \ol{R}_{\lambda+\rho}$ if and only if $\beta'\in \ol{R}_{\lambda'+\rho'}$. 
Therefore $\psi(\ol{R}_{\lambda'+\rho'})=\ol{R}_{\lambda+\rho}$.

If $\fg$ is of type (An), then $\Delta_{L}=R_{\lambda+\rho}$ and
$\Delta_{L'}=R_{\lambda'+\rho'}$ which gives $\psi(\Delta_{L'})=\Delta_L$ and implies
$\Delta_{L'}=(\Delta')^{\ree}$.

If $\fg$ is of type (Fin) or (Aff), then, by definition, $\Delta_{L'}$ is the minimal subset of $(\Delta')^{\ree}$ which contains
$\ol{R}_{\lambda'+\rho'}$ and 
satisfies  properties (b)--(d) in~$\S$~\ref{Deltalambda}. 
Therefore $\psi(\Delta_{L'})\subset\Delta$ contains
$\ol{R}_{\lambda+\rho}$ and 
satisfies  properties (b)--(d).  Therefore
$\psi(\Delta_{L'})$ contains $\Delta_L$ (since $\Delta_L$ is the minimal subset of $\Delta^{\ree}$ which contains
$\ol{R}_{\lambda+\rho}$ and 
satisfies (b)--(d)). Since the restriction of $\psi$ gives a bijection between
$(\Delta')^{\ree}$ and $\Delta_L$, we obtain $\Delta_{L'}=(\Delta')^{\ree}$ as required.
\end{proof}

\subsubsection{}\label{notaforlambda11}
We let  $\Sp'$ be the spine of $\Sigma'$ (if $\fg$ is of type (An), then $\Sp=\{\Sigma\}$ and
$\Sp'=\{\Sigma'\}$). Let $b_L: \Sp\to \fh^*$, $b_{L'}: \Sp'\to (\fh')^*$ be the maps given by
$$b_L(\Sigma_1):=\hwt_{\Sigma_1} L,\ \ \ b_{L'}(\Sigma'_1):=\hwt_{\Sigma'_1} L',$$
see~$\S$~\ref{hwtSpine} for the notation (one has
  $L=L(\hwt_{\Sigma_1} L-\rho_{\Sigma_1},\Sigma_1)$). 
Taking $\Upsilon:\Sp\to\Sp'$ as in~$\S$~\ref{baseKMtypes} and $\psi_{*}: (\fh')^*\to (\fh')^*/(\Delta')^{\perp}$
given by~(\ref{eq:psi*}) 
we obtain the following diagram:
\begin{equation}\label{diag1}
\xymatrix{\Sp\ar[r]^{b_L}\ar[d]^{\Upsilon}&\fh^*\ar[rd]^{\psi_*}\\
                     \Sp' \ar[r]^{b_{L'}}& (\fh')^*\ar[r]&(\fh')^*/(\Delta')^{\perp} }\end{equation}

\subsection{}
\begin{lem}{lem:lambda11} 
The diagram~(\ref{diag1}) is commutative. 
\end{lem}
\begin{proof}
We have to verify that 
\begin{equation}\label{eq:lambdalpha}
( \hwt_{\Sigma'_1} L' ,\mu)=( \hwt_{\Sigma_1} L,  \psi(\mu))\ \ \text{ for all }\ \ \mu\in \Delta'.
\end{equation}
For $\Sigma_1=\Sigma$  
the required formula  follows from the constriuction of $L'$, see~(\ref{eq:lambdalphadef}).
By the definition of $\Sp$, 
it is enough to verify that~(\ref{eq:lambdalpha})  holds for 
$\Sigma_1=r_{\alpha}\Sigma$
with $\alpha\in (\Sigma\cap\Delta^{\is})$.

Consider the case $(\hwt_{\Sigma}L,\alpha)=0$. Then 
$\alpha\in \ol{R}_{\lambda+\rho}$, so $\alpha\in \Delta_L$.  Since
$\alpha\in \Sigma$ we have $\psi^{-1}(\alpha)\in \Sigma'$ by~\Lem{lem:KMbase} (ii),
so~(\ref{eq:ralph}) gives
  $\Sigma_1'=r_{\psi^{-1}(\alpha)}\Sigma'$.  Formula~(\ref{eq:lambdalpha})  for $\Sigma$
gives   
$$(\hwt_{\Sigma'} L', \psi^{-1}(\alpha))=(\hwt_{\Sigma} L,\alpha)=0.$$
By~(\ref{eq:hwt}) we obtain
$\hwt_{\Sigma_1} L=\hwt_{\Sigma}L+\alpha$ and  $\hwt_{\Sigma'_1} L'=\hwt_{\Sigma'} L'+\psi^{-1}(\alpha)$. 
Thus
 $$( \hwt_{\Sigma'_1} L',\mu)=(\hwt_{\Sigma'} L'+\psi^{-1}(\alpha),\mu)=
 ( \hwt_{\Sigma} L,  \psi(\mu))+(\alpha,\psi(\mu))=
 (\hwt_{\Sigma_1} L,\psi(\mu))
 $$
as required.

Consider the case $(\hwt_{\Sigma}L ,\alpha)\not=0$. By~(\ref{eq:hwt})  we have
$\hwt_{\Sigma_1} L=\hwt_{\Sigma} L$. 
By~(\ref{eq:ralph}),
one has either  $\Sigma_1'=\Sigma'$ 
or  $\Sigma_1'=r_{\psi^{-1}(\alpha)}\Sigma'$ and $\alpha\in \psi(\Sigma')$.
In the latter case, using~(\ref{eq:lambdalpha}) for $\Sigma$, we get
$(\hwt_{\Sigma'} L', \psi^{-1}(\alpha))=(\hwt_{\Sigma} L,\alpha)\not=0$,
so~(\ref{eq:hwt})  gives
$\hwt_{\Sigma'_1} L'=\hwt_{\Sigma'} L'$.
Hence in both cases $\hwt_{\Sigma_1} L=\hwt_{\Sigma} L$ and $\hwt_{\Sigma'_1} L'=\hwt_{\Sigma'} L'$,
so  formula~(\ref{eq:lambdalpha})  for $\Sigma_1$ follows from the same formula for $\Sigma$. 
\end{proof}

\subsubsection{}\label{345}
The commutativity of the diagram~(\ref{diag1}) means that 
 the $[\fg',\fg']$-module $L'$ defined in~(\ref{eq:L'})
  does not depend on the choice of a base:
 if we apply the same procedure
to the pairs $(\lambda,\Sigma)$ and $(\lambda_1,\Sigma_1)$ such that $\Sigma_1\in\Sp$ and
$L=L(\lambda,\Sigma)=L(\lambda_1,\Sigma_1)$,
then the modules $L'=L_{\fg'}(\lambda',\Sigma')$  and $L_{\fg'}(\lambda_1',\Sigma'_1)$ 
are isomorphic as $[\fg',\fg']$-modules.

\subsection{}\label{relation:sim}
\begin{defn}{defn:sim}
Let $\sim$ be the equivalence relation on the set of isomorphism classes of
 irreducible modules in  $\CO_{\Sigma}(\fg)$
which is generated by the following relation: $L_1\sim L_2$
 if there exists $\Sigma_1\in \Sp$ and $\alpha$ such that 
\begin{equation}\label{eq:simb}
 \alpha\in (\Sigma_1\cap \Delta^{\an}),\ \ 
\alpha\not\in \Delta_{\hwt_{\Sigma_1} L_1}\ ,\ \ \ \hwt_{\Sigma_1} L_2=s_{\alpha} (\hwt_{\Sigma_1} L_1).
\end{equation}
\end{defn}

\subsubsection{}\label{gnu}
For $L=L(\lambda)$  we will use the following notation:
$$\fg'_{\lambda+\rho}:=\fg',\ \ \psi_{\lambda+\rho}:=\psi,\ \ \Sigma'_{\lambda+\rho}:=\Sigma'\subset \Delta(\fg'_{\lambda+\rho}),\ \ 
\Sigma_{\lambda+\rho}:=\psi(\Sigma')\subset \Delta_{\lambda+\rho}
$$
where  $\Sigma'$ is the base introduced in~\Lem{lem:KMbase}.
We  denote by $\lambda'$ 
 the weight given by~(\ref{eq:lambdalphadef}).

\subsubsection{}
\begin{lem}{lem:salphaSigma'}
If $\lambda+\rho=s_{\alpha}(\nu+\rho)$ for  $\alpha\in (\Sigma\cap\Delta^{\an})$ such that
 $\alpha\not\in \Delta_{\lambda+\rho}$, then 
\begin{enumerate}
\item 
 the algebra $\fg'_{\nu+\rho}$ can be identified
with $\fg':=\fg'_{\lambda+\rho}$ in such a way that  $\psi_{\nu+\rho}=s_{\alpha}\psi_{\lambda+\rho}$ and
 $\Sigma'_{\nu+\rho}=\Sigma'_{\lambda+\rho}$;
\item $\Delta_{\nu+\rho}^+=s_{\alpha} \Sigma_{\lambda+\rho}^+$ and 
 $\Sigma_{\nu+\rho}=s_{\alpha} \Sigma_{\lambda+\rho}$;
 \item for this identification $L_{\fg'}(\nu')$ and  $L_{\fg'}(\lambda')$ are isomorphic as $[\fg',\fg']$-modules.
\end{enumerate}
\end{lem}
\begin{proof}
One has 
$\ol{R}_{\nu+\rho}=s_{\alpha}
\ol{R}_{\lambda+\rho}$, so $\Delta_{\nu+\rho}=s_{\alpha}\Delta_{\lambda+\rho}$.
This allows to identify
$\fg_{\nu+\rho}$  with $\fg':=\fg_{\lambda+\rho}$ by taking 
$\psi_{\nu+\rho}:=s_{\alpha}\circ \psi_{\lambda+\rho}$. This gives (i).

 By~\Cor{cor:Delta2beta} we have $2\alpha\not\in \Delta_{\lambda+\rho}$.
Since $\alpha\in\Sigma$ and $\alpha, 2\alpha\not\in \Delta_{\lambda+\rho}$, 
the formula
 $\Delta_{\nu+\rho}=s_{\alpha}\Delta_{\lambda+\rho}$ forces
 $\Delta^+_{\nu+\rho}=s_{\alpha}\Delta^+_{\lambda+\rho}$.
 By~\Lem{lem:KMbase}. the sets $\Sigma_{\lambda+\rho}$, $\Sigma_{\nu+\rho}$ are
the sets of indecomposable elements in $\Delta^+_{\lambda+\rho}$, $\Delta^+_{\nu+\rho}$ respectively. This gives
   $\Sigma_{\nu+\rho}=s_{\alpha} \Sigma_{\lambda+\rho}$ and establishes (ii).
By definition~(\ref{eq:lambdalphadef}) for any $\mu'\in\Delta(\fg')$ we have
$$\begin{array}{rl}
(\nu'+\rho',\mu')&=(\nu+\rho,\psi_{\nu+\rho}(\mu) )=(\nu+\rho,s_{\alpha}\circ \psi_{\lambda+\rho}(\mu) )=
(s_{\alpha}(\nu+\rho), \psi_{\lambda+\rho}(\mu) )\\&=(\lambda+\rho,\psi_{\lambda+\rho}(\mu) )
=
(\lambda'+\rho',\mu).\end{array}$$
Therefore $(\lambda'-\nu',\Delta(\fg'))=0$, so $L_{\fg'}(\lambda')$
and $L_{\fg'}(\nu')$ are isomorphic as $[\fg',\fg']$-modules.
\end{proof}

Using~$\S$~\ref{345} we obtain 
\subsubsection{}
\begin{cor}{cor:FL}
If $L(\nu)\sim L(\lambda)$, then 
$\fg_{\nu+\rho}$  can be identified with   $\fg':=\fg_{\lambda+\rho}$
in such a way that  $L_{\fg'}(\nu')$ and  $L_{\fg'}(\lambda')$ are isomorphic as $[\fg',\fg']$-modules.
\end{cor}

\subsection{Set $U(\lambda)$}\label{Ulambda}
For $\lambda\in\fh^*$ we denote by $U(\lambda)$ the subset of $\fh^*$ consisting of weights 
$\nu$ such that 
 there exists  a chain 
$\nu+\rho=\mu_r<\mu_{r-1}<\ldots<\mu_0=\lambda+\rho$, 
where $\mu_{i+1}=\mu_i-m_i\gamma_i$ for some 
 $\gamma_i\in \ol{\Delta}^{\ree}\cap\Delta^+$ and  $m_i\in\mathbb{Z}_{>0}$
 such that $2(\mu_i,\gamma_i)=m_i(\gamma_i,\gamma_i)$, 
where $m_i=1$ is odd if $p(\gamma)=\ol{1}$, $m_i=1$ if
$\gamma$ is isotropic. Note that
$\mu_{i+1}=s_{\gamma_i}\mu_i$ if $\gamma_i$ is anisotropic.

\subsubsection{}
The following statement follows from~\cite{KK} for Lie algebras
and~\cite{GKvac} for Lie superalgebras.

\subsubsection{}
\begin{prop}{prop:KK}
Consider the expansion 
$De^{\rho}\ch L(\lambda)=\sum\limits_{\mu\in \mathbb{Z}_{\geq 0}\Sigma} a_{\mu} e^{\lambda+\rho-\mu}$ (with $a_{\mu}\in\mathbb{Z}$).
Assume that  $\mu\in \mathbb{Z}_{\geq 0}\Sigma$ is such that 
$[M(\lambda): L(\lambda-\mu)]\not=0$ or  $a_{\mu}\not=0$. Then
\begin{enumerate}
\item
If $\lambda$ is non-critical, then $(\lambda-\mu)\in U(\lambda)$.
\item If $\fg$ is of type (Aff) and $\lambda$ is critical, then 
$(\lambda-\mu)\in (U(\lambda)-\mathbb{Z}_{\geq 0}\delta)$ (where $\delta$ is the minimal imaginary root).
\end{enumerate}
\end{prop}

\subsubsection{}
\begin{lem}{lem:KK}
If $\nu\in U(\lambda)$, then
 $\Delta_{\nu+\rho}=\Delta_{\lambda+\rho}$ and $(\lambda-\nu)$ lies in $\mathbb{Z}_{\geq 0}\Delta_{\lambda+\rho}$.
\end{lem}
\begin{proof}
It is enough to verify the assertion for a chain of length one:
$\nu+\rho=\mu_1<\mu_0=\lambda+\rho$
(the general case follows by induction on the length of the chain).
For such a chain  we have
 $\nu+\rho=\lambda+\rho-m\gamma$ where  
$\gamma\in \ol{\Delta}^{\ree}\cap\Delta^+$ and  $m\in\mathbb{Z}_{>0}$
is such that $2(\lambda+\rho,\gamma)=m(\gamma,\gamma)$, 
and $m$ is odd if $p(\gamma)=\ol{1}$, $m=1$ if $\gamma$ is isotropic.
Notice that $\gamma\in \ol{R}_{\lambda+\rho}$, so $(\lambda-\nu)\in \mathbb{Z}_{\geq 0}\Delta_{\lambda+\rho}$. Let us verify that $\Delta_{\nu+\rho}=\Delta_{\lambda+\rho}$.

If $\gamma$ is anisotropic, then $\nu+\rho=s_{\gamma}(\lambda+\rho)$, so 
$\Delta_{\nu+\rho}=s_{\gamma}\Delta_{\lambda+\rho}=\Delta_{\lambda+\rho}$
(since $\gamma\in \Delta_{\lambda+\rho}$).

Consider the case when $\gamma$ is isotropic. Then 
$(\lambda+\rho,\gamma)=0$ and
$\nu=\lambda-\gamma$, so  the formula $\Delta_{\nu+\rho}=\Delta_{\lambda+\rho}$
follows from~\Lem{lem:nualpha}.
\end{proof}

\subsubsection{}\label{Olambda}
Arguing as in~\cite{DGK}, Section 4 we obtain the following decomposition theorem.

{\em Consider the equivalence relation $\lambda\approx \nu$ on $\fh^*$ which is generated
by $\nu\approx \lambda$ if $\nu\in U(\lambda)$  and
$(\lambda-\delta)\approx\lambda$ if $\fg$ is of type (Aff) and $\lambda$ is critical.

For any $N\in \CO_{\Sigma}(\fg)$ 
there exists a unique set of submodules $\{M_{\chi}\}_{\chi\in A}$ such that
\begin{itemize}
\item[$\bullet$]
$M=\oplus_{\chi\in A} M_{\chi}$;
\item[$\bullet$]
for each $\chi_1,\chi_2\in A$ and $\lambda,\nu$ such that 
$[M_{\chi_1}:L(\lambda)]$, $[M_{\chi_2}:L(\nu)]\not=0$
one has  $\lambda\approx \nu$ if and only if $\chi_1=\chi_2$.
\end{itemize}}

If $N\in\CO_{\Sigma}(\fg)$ is indecomposable, then $N=M_{\chi}$. 

\subsubsection{}
\begin{cor}{cor:Olambda}
 If $N\in\CO_{\Sigma}(\fg)$ is non-critical and indecomposable, then 
$$[N:L(\lambda)], [N:L(\nu)]\not=0\ \ \Longrightarrow\ \ 
\Delta_{\lambda+\rho}=\Delta_{\nu+\rho},\ \ 
(\lambda-\nu)\in \mathbb{Z}\Delta_{\lambda+\rho}.$$
\end{cor}
\begin{proof}
By~$\S$~\ref{Olambda} we have $\lambda\approx \nu$.
By~\Lem{lem:KK} this forces 
$\Delta_{\lambda+\rho}=\Delta_{\nu+\rho}$ and 
$(\lambda-\nu)\in \mathbb{Z}\Delta_{\lambda+\rho}$.
\end{proof}

\subsubsection{}\label{weightnu'}
Let $N\in\CO_{\Sigma}(\fg)$ be a non-critical indecomposable module.
Using~\Cor{cor:Olambda} we define $\Delta_N$ to be $\Delta_{\lambda+\rho}$
where $\lambda$ is such that  
$[N:L(\lambda)]\not=0$.

Take $\nu\in U(\lambda)$ and retain notation of~$\S$~\ref{gnu}.
By~\Prop{prop:KK} we can identify $\fg'_{\nu+\rho}$ with $\fg':=\fg'_{\lambda+\rho}$;
moreover, we can choose $\nu'$ in such a way that $\nu'\in U(\lambda')$
and $\psi(\lambda'-\nu')=\lambda-\nu$.

\subsection{Remark}\label{D'}
The character formulas for arbitrary irreducible non-critical modules
$L(\lambda)$ over symmetrizable  Kac-Moody  algebras were established in~\cite{KT99};
these formulas imply the following remarkable formula:  
 $De^{\rho}\ch L=D'e^{\rho'} \ch L'$,  where $D,\rho$ (resp., $D',\rho'$) are the Weyl denominator and
 the Weyl vector for $\Delta^+$ (resp., for $(\Delta')^+$).

More precisely, we denote by $\cR_0(\Sigma)$
the subalgebra of $\cR(\Sigma)$ consisting of 
$x\in\cR(\Sigma)$ satisfying $\supp(x)\subset 
-\mathbb{Z}_{\geq 0}\Sigma$.
The  linear monomorphism
$\psi_{\lambda+\rho}:\mathbb{Z}\Sigma'\to \mathbb{Z}\Sigma$
induces a ring monomorphism $\cR_0(\Sigma')\to \cR_0(\Sigma)$
which we also denote by $\psi_{\lambda+\rho}$. The algebra
$\cR_0(\Sigma')$ contains
 $D'$ and $ e^{-\lambda'}\ch L'$.
The formula  $De^{\rho}\ch L=D'e^{\rho'} \ch L'$ means 
\begin{equation}\label{eq:Fiebig}
De^{-\lambda} \ch L(\lambda)=
\psi_{\lambda+\rho}\bigl(D'e^{-\lambda'} \ch L'(\lambda')\bigr).
\end{equation}
For example, if a non-critical weight $\lambda$ 
is such that $m:=\langle \lambda+\rho,\alpha\rangle\in\mathbb{Z}_{>0}$
 for some $\alpha\in \Delta^{\ree +}$
 and $\langle \lambda+\rho,\beta\rangle\not\in\mathbb{Z}$
for all $\beta\in \Delta^{\ree+}\setminus\{\alpha\}$,
then $\fg_{\lambda+\rho}\cong \fsl_2$, $\Sigma'=\{\alpha'\}$
and $\psi_{\lambda+\rho}(\alpha')=\alpha$.
In this case
$$De^{-\lambda} \ch L(\lambda)=
\psi_{\lambda+\rho}\bigl(D'e^{-\lambda'} 
\ch L'(\lambda')\bigr)=\psi_{\lambda+\rho}(1-e^{-m\alpha'})=
1-e^{m\alpha}.$$

We conjecture that  formula~(\ref{eq:Fiebig})  holds for all symmetrizable Kac-Moody superalgebras, see~\cite{GKadm}. In~\cite{F}, P.~Fiebig established equivalence
 of the corresponding blocks
 in the categories $\CO_{\Sigma}(\fg)$ and $\CO_{\Sigma'}(\fg')$
 for an arbitrary symmetrizable Kac-Moody algebra $\fg$.

%\subsubsection{Remark}
%Observe that   
%\begin{equation}\label{eq:ggmod}
%L(\lambda)\cong L(\nu)\ \text{ as  $[\fg,\fg]$-modules }\ \Longleftrightarrow\  \ (\lambda-\nu,\Delta)=0\ \ \Longleftrightarrow\ \ \dim L(\lambda-\nu)=1.
%\end{equation}
%
%
%\subsubsection{Remark}\label{remark:conditionab}
%The conditions (a) and (b) allow to express $\c L_2$ via $\ch L_1$: 
%if $L(\lambda)\cong L(\nu)$ as  $[\fg,\fg]$-modules, then
%$\ch L(\lambda)=e^{\lambda-\nu}\ch L(\nu)$, and 
%for $L_1$, $L_2$ satisfying (b) we have
%$$e^{\frac{\alpha}{2}}(1-e^{-\alpha})\ch L_2=s_{\alpha}\bigl(e^{\frac{\alpha}{2}}(1-e^{-\alpha})   \ch L_1\bigr).$$
%where $s_{\alpha}$ acts naturally (see~$\S$~\ref{Weylgroup} for definition), see~\cite{GSsnow}, Theorems 2.3.1 and 3.2 (ii).

 \section{Quasi-admissible modules}
In this section $\fg$ is an indecomposable symmetrizable
 Kac-Moody superalgebra with a base $\Sigma$ and $\fg\not=\fgl(1|1)$.
  Recall that  $(-,-)$ is normalised in such a way that $(\alpha,\alpha)\in\mathbb{Q}_{>0}$ for some $\alpha\in\Delta^{\ree}$.
  The action of $w\in W$ on
  $\sum\limits_{\nu} m_{\nu}e^{\nu}$
    is always assumed to be the natural, i.e.
  $w(\sum\limits_{\nu} m_{\nu}e^{\nu}):=\sum\limits_{\nu} m_{\nu}e^{w\nu}$.

\subsection{Integrable and partially integrable modules}\label{partint}
Let $\Sigma_{\pr}$ be the set of principal roots (see~$\S$~\ref{Bpr} for definition).
In the case when $\fg=\dot{\fg}^{(1)}$ is the non-twisted affinization of $\dot{\fg}$ we let $\dot{\Delta}\subset\Delta$  be the root subsystem of $\dot{\fg}$.
By~\Rem{rem:prince} $\dot{\Sigma}_{\pr}:=\dot{\Delta}\cap {\Sigma}_{\pr}$ 
is the set of principal roots of $\dot{\Delta}$.

\subsubsection{}
 \begin{defn}{defn:pi-integrable}
For a subset $\pi\subset\Sigma_{\pr}$ we call a $[\fg,\fg]$-module $N$
{\em $\pi$-integrable} if  $\fg_{\pm\alpha}$ acts locally nilpotently
on $N$ for each $\alpha\in\pi$. 
 \end{defn}
 %Using~\Lem{lem:intmod} 
%  it is easy to see that if $[\fg,\fg]$-module is partially integrable, then
%  $\fg_{\alpha}$ acts locally nilpotently for each $\alpha\in\Delta^{\ree}$ such that
%  $(\alpha,\alpha)>0$.  
%  

 \subsubsection{}
\begin{lem}{lem:intmod}
Let $W[\pi]$ be the subgroup of $W$ generated by the reflections $s_{\alpha}$ with $\alpha\in\pi$.
If a $[\fg,\fg]$-module $N$ is  $\pi$-integrable, then 
the root space $\fg_{\alpha}$
acts locally nilpotently on $N$ 
for each $\alpha\in W[\pi](\pi)$. 
\end{lem}
\begin{proof}
Take $\alpha\in\pi$.
 We will prove that $\fg_{\pm w\alpha}$ acts locally nilpotently on $N$ 
 by induction on the length of $w\in W[\pi]$.
 If the length is zero, then, by definition, $\fg_{\pm \alpha}$ 
 acts locally nilpotently on $N$.
 Otherwise for some $\beta\in\pi$ we have $w=s_{\beta}w'$, where
 the length of $w'$ is less than the length of $w$. By induction, 
 $\fg_{\pm s_{\beta}\alpha}$ act locally nilpotently on $N$.
 By~\cite{Kbook}, Lemma 3.8, any $u\in\fg_{\pm w\alpha}$ can be written as  
 $u=(\exp \ad f)(\exp \ad(-e))(\exp \ad f)(u')$ for some
 $u'\in  \fg_{\pm w'\alpha }$, $e\in \fg_{\beta}$ and $f\in\fg_{-\beta}$. Let $E,F,U, U'$ be the images
 of $e,f,u,u'$ in $\End(N)$.  Then
 $$\begin{array}{rl}
 U&=(\exp \ad F)(\exp \ad(-E))(\exp \ad F)(U')\\
 &= (\exp F)(\exp -E)(\exp F)U'(\exp - F)(\exp E)(\exp -F),\end{array}$$
 since $(\exp a)b(\exp -a)=\exp(\ad a)(b)$.  For any $v\in N$ we have
 $$U^s(v)= (\exp F)(\exp -E)(\exp F)(U')^s(\exp - F)(\exp E)(\exp -F)(v).$$
 Since $E,F,U'$ acts locally nilpotenly,  $U^sv=0$ for $s>>0$, so $U$ acts locally nilpotently.
 \end{proof}

\subsubsection{Integrable modules}
Since $\Delta^{\an}_{\ol{0}}=W(\Sigma_{\pr})$ and $W=W[\Sigma_{\pr}]$, the $\Sigma_{\pr}$-integrable modules
are modules where  $\fg_{\alpha}$ acts locally nilpotently
 for each $\alpha\in\Delta^{\ree}$. Such modules are called {\em integrable}.

 \subsubsection{Remark}
 If $\fg=\fg(A,\tau)$ is a symmetrizable infinite-dimensional Kac-Moody superalgebra, then the following conditions are equivalent~\cite{KW94}:
 \begin{itemize}
\item[$\bullet$] the set $\Sigma_{\pr}$ is linearly independent;
 \item[$\bullet$] the Dynkin diagram of $\Sigma_{\pr}$ is connected;
 \item[$\bullet$] the subalgebra generated by vector space $\fh+\sum\limits_{\alpha\in\Sigma_{\pr}} \fg_{\pm \alpha}$  is
a Kac-Moody algebra in the sense of~$\S$~\ref{reflexion};
\item[$\bullet$] there exists an irreducible integrable module in
$\CO_{\Sigma}(\fg)$ which is not one-dimensional;
\item[$\bullet$] $\fg$ is  anisotropic, $\fsl(1|n)^{(1)}$,
 or $\mathfrak{osp}(2|2n)^{(1)}$.
  \end{itemize}

 \subsubsection{}\begin{rem}{rem:KKcat}
Consider the case when  $\fg$ is affine
and the Dynkin diagram of $\Sigma_{\pr}$ is not connected.
By above, $\CO_{\Sigma}(\fg)$ does not contain irreducible integrable modules
which are not one-dimensional. 
If we take $\pi$ to be a maximal proper subset of $\Sigma_{\pr}$
(i.e., $\Sigma_{\pr}\setminus \pi$ is of cardinality one), then
 $\CO_{\Sigma}(\fg)$ contains interesting $\pi$-integrable modules.
For example, if $\fg=\dot{\fg}^{(1)}$
 and $\pi=\dot{\Sigma}_{\pr}\cup\pi^{\#}$, where 
 $\pi^{\#}:=\{\alpha\in\Sigma_{\pr}|\ (\alpha,\alpha)>0\}$,
 then, for $k\in\mathbb{Z}_{>0}$, the irreducible $\pi$-integrable modules
 in $\CO_{\Sigma}(\fg)^k$ are irreducible modules in the 
 $KL_k$-category studied in~\cite{AKMPP}.
Important examples of $\pi$-integrable modules are irreducible vacuum modules
$L(k\Lambda_0)$ for suitable values of $k$.
The irreducible $\pi$-integrable modules in $\CO_{\Sigma}(\fg)$ 
were classified in~\cite{KW01}.

\subsubsection{}
 The simple affine vertex algebra $V_k(\fg)$ corresponding to 
  an anisotropic non-twisted affine superalgebra
  $\fg$ and an integrable $\fg$-module 
  $L(k\Lambda_0)$, 
  see~\cite{KWang} (and~\cite{FZ},\cite{DLM}).
   This result  is based on the fact that in this case 
   the integrability of  $L(k\Lambda_0)$ 
  implies that $V_k(\fg)$-modules are exactly 
  $V^k(\fg)$-modules which are integrable.

The proof of rationality of $V_k(\fg)$, 
given in~\cite{KWang}, Section 2.2, uses only the fact that the even part of the finite-dimensional Lie superalgebra $\dot{\fg}$ to which $\fg$ is associated, is simple.
It follows that the adjoint orbit of any non-zero element $a$ in $\dot{\fg}_{\ol{0}}$ spans it, hence $U(\dot{\fg})/(a^{k+1})$
and $S(\dot{\fg})/(a^{k+1})$ are finite-dimensional
(see~\cite{KWang}, Lemma 2.3).

\subsubsection{}
 If $\fg$ is a non-twisted affine superalgebra which is not 
 anisotropic, then $V_k(\fg)$ is not rational
 for $k\not=0$, since the category of $V_k(\fg)$-modules is not 
 completely reducible. However,  for $\fg\not=D(2|1,a)^{(1)}$,   
 the    $\pi^{\#}$-integrability of  $L(k\Lambda_0)$ 
  implies that $V_k(\fg)$-modules are exactly $V^k(\fg)$-modules 
  which are  $\pi^{\#}$-integrable,   
  see~\cite{GSsl1n} Theorem 5.3.1. 
\end{rem}

\subsubsection{}\label{newnot}
Let $\pi$ be a linearly independent subset of $\Sigma_{\pr}$. This means
that $\pi$ is a proper subset of $\Sigma_{\pr}$ if $\fg$ is affine,
not anisotropic and not equal to $\fsl(1|n)^{(1)}$, $\mathfrak{osp}(2|2n)^{(1)}$,
$D(2|1,a)^{(1)}$, $\pi$ lies in $A_1^{(1)}\coprod A_1\coprod A_1$
for $D(2|1,a)^{(1)}$, and $\pi$ is an arbitrary subset of $\Sigma_{\pr}$
in the rest of the cases.
We denote by $\fg_{\pi}$ the subalgebra generated by the vector space $\fh+\sum\limits_{\alpha\in\pi} \fg_{\pm \alpha}$.
 Since $\Sigma_{\pr}\subset\Delta^{\ree}_{\ol{0}}$, $\fg_{\pi}$
 is a subalgebra of $\fg_{\ol{0}}$.
 Using the arguments
 of~\cite{S}, Theorem 9.1 (see also~\cite{Shay}, 3.5)
 it is easy to show that $\fg_{\pi}$ is
 a  symmetrizable  Kac-Moody algebra
(in the sense of~$\S$~\ref{reflexion})
 with the Cartan subalgebra $\fh$,
  the set of simple roots $\pi$, and 
 $\pi^{\vee}:=\{\alpha^{\vee}| \ \alpha\in\pi\}$. 
 This Kac-Moody algebra is indecomposable if
 the Dynkin diagram of $\pi$ is connected.
 The restriction of $(-,-)$ to $\fg_{\pi}$ 
 is a non-degenerate invariant bilinear form.
 We set
 $$\Delta_{\pi}:=\Delta(\fg_{\pi})^{\ree}.$$

 \subsubsection{Example}
 Take   $\fg=\mathfrak{osp}(3|2)^{(1)}$.
 Then $\Sigma_{\pr}=\{\vareps_1,\delta-\vareps_1,
 2\delta_1,\delta-2\delta_1\}$. If $\pi$ is of cardinality three, then 
 $\fg_{\pi}\cong\fsl_2^{(1)}\times\fgl_2$.

 \subsubsection{}
 If $\fg$ is of type (Aff), then each indecomposable
 component of $\fg_{\pi}$ is of type (Fin) or (Aff). 
If $\pi^{\#}=\{\alpha\in\Sigma_{\pr}|\ (\alpha,\alpha)>0\}$
has a connected Dynkin diagram, then $\fg_{\pi^{\#}}$
is an affine Kac-Moody algebra. For $\pi=\dot{\Sigma}_{\pr}\cup\pi^{\#}$,
which appeared in~\Rem{rem:KKcat},
the algebra $\fg_{\pi}$ is the direct product of $\fg_{\pi^{\#}}$
and a reductive Lie algebra with the set
of simple roots $\{\alpha\in\dot{\Sigma}_{\pr}|\ (\alpha,\alpha)<0\}$.

 \subsubsection{Remark}
For any $N\in{\CO}^{\infty}(\fg)$ one has  $\Res^{\fg}_{\fg_{\pi}} N\in{\CO}^{\infty}(\fg_{\pi})$.
A $\fg$-module $N$ is $\pi$-integrable if and only if
  $\Res^{\fg}_{\fg_{\pi}} N$ is $\fg_{\pi}$-integrable.

  Note that for any $\pi$-integrable
   $L(\lambda)$ we have 
$\Delta_{\pi}\subset {R}_{\lambda+\rho}$  and 
$\langle \lambda,\alpha^{\vee}\rangle\in\mathbb{Z}_{\geq 0}$  for each $\alpha\in\pi$.

We will use the following result (which can be deduced from~\cite{GSsnow}, but we include a proof for the sake of completeness).

\subsubsection{}
\begin{lem}{lem:Nabc}
For any $N\in\CO_{\Sigma}(\fg)$ and $\gamma\in\pi$ 
the following conditions are equivalent:
\begin{itemize}
\item[(a)]  $\fg_{-\gamma}$ acts locally nilpotently on $N$;
\item[(b)] $s_{\gamma}((e^{\frac{\gamma}{2}}-e^{-\frac{\gamma}{2}})\ch N)=-
(e^{\frac{\gamma}{2}}-e^{-\frac{\gamma}{2}})\ch N$.
\item[(c)] $D_{\pi} e^{\rho_{\pi}} \ch N$ is naturally $s_{\gamma}$-anti-invariant
(where $D_{\pi}$ and $\rho_{\pi}$ are respectively the Weyl denominator and the Weyl vector for
$\Delta_{\pi}^+$).
\end{itemize}  
\end{lem}
\begin{proof}
Let $\ft$ be the subalgebra which is generated by $\fg_{\pm\gamma}$
and $\fh$. Then $\ft$ is the direct product of a copy of $\fsl_2$
and a subalgebra of codimension $1$ in $\fh$.
Set $M:=\Res^{\fg}_{\ft} N$. By~$\S$~\ref{categoryO} (iv) we have 
$M\in \CO_{\Sigma}(\ft)$ and
$$\ch N=\ch M=\sum\limits_{\mu\in \fh^*} [M: L_{\ft}(\mu)] \ch L_{\ft}(\mu).$$
Setting $P^+(\ft):=\{\mu|\ \dim L_{\ft}(\mu)<\infty\}$ we have
$$(e^{\frac{\gamma}{2}}-e^{-\frac{\gamma}{2}})\ch L_{\ft}(\mu)=\left\{
\begin{array}{ll}
(e^{\mu+\frac{\gamma}{2}}-s_{\gamma}e^{\mu+\frac{\gamma}{2}}) & \text{ if } \mu\in P^+(\ft)\\
e^{\mu+\frac{\gamma}{2}} & \text{ if } \mu\not\in P^+(\ft)
\end{array}\right.$$
which gives
$$
(e^{\frac{\gamma}{2}}-e^{-\frac{\gamma}{2}})\ch N=
\sum\limits_{\mu\in P^+(\ft)} [M: L_{\ft}(\mu)] 
(e^{\mu+\frac{\gamma}{2}}-s_{\gamma}e^{\mu+\frac{\gamma}{2}})+
\sum\limits_{\mu\not\in P^+(\ft)} [M: L_{\ft}(\mu)] e^{\mu+\frac{\gamma}{2}}.$$
Therefore (b) is equivalent to
$$\sum\limits_{\mu\not\in P^+} [M: L_{\ft}(\mu)] s_{\gamma} e^{\mu+\frac{\gamma}{2}}=
-\sum\limits_{\mu\not\in P^+} [M: L_{\ft}(\mu)]  e^{\mu+\frac{\gamma}{2}}.$$
Observe that all coefficients in the left-hand side are non-negative and in
the right-hand side are non-positive. Thus (b) is equivalent to 
$[M: L_{\ft}(\mu)]=0$ for all  $\mu\not\in P^+$ which means that $M$ is a direct sum of
finite-dimensional $\ft$-modules. This establishes the equivalence (a) $\Longleftrightarrow$ (b).

For the equivalence (b) $\Longleftrightarrow$ (c) note that 
 $\rho_{\pi}-\frac{\gamma}{2}$ and $\Delta(\fg_{\pi})^+\setminus\{\gamma\}$ are 
$s_{\gamma}$-invariant, so 
\begin{equation}\label{eq:sgammainv}
\begin{array}{l}D_{\pi}e^{\rho_{\pi}}(e^{\frac{\gamma}{2}}-e^{-\frac{\gamma}{2}})^{-1}
=e^{\rho_{\pi}-\frac{\gamma}{2}}\prod\limits_{\alpha\in \Delta(\fg_{\pi})^+\setminus\{\gamma\}}
(1-e^{-\alpha})\ \text{ is $s_{\gamma}$-invariant} \\
  D^{-1}_{\pi}e^{-\rho_{\pi}}(e^{\frac{\gamma}{2}}-e^{-\frac{\gamma}{2}})
 =e^{-\rho_{\pi}+\frac{\gamma}{2}}\prod\limits_{\alpha\in \Delta(\fg_{\pi})^+\setminus\{\gamma\}}
(1-e^{-\alpha})^{-1}\ \text{ is $s_{\gamma}$-invariant}    \end{array}\end{equation}
 (see~\cite{GKadm}, Section 2
on invariance of infinite products). This implies the  equivalence (b) $\Longleftrightarrow$ (c) 
and completes the proof.
\end{proof}

\subsection{Quasi-admissible modules in type (An)}\label{snowflake}
Below we recall some  definitions and results of~\cite{GSsnow}.
 Let $\fg$ be of type (An). In this case
 $$\Sigma_{\pr}=\{\alpha\in \Sigma|\ p(\alpha)=\ol{0}\}\cup \{2\alpha|\ \alpha\in \Sigma,\ p(\alpha)=\ol{1}\}$$
 and $\Sigma_{\pr}$ is connected.
 Recall that for $L=L(\lambda)$ we have
 $$\Delta_L=\Delta_{\lambda+\rho}=
 R_{\lambda+\rho},$$ 
(see~\Defn{defn:DeltaL}).
We take $\pi:=\Sigma_{\pr}$. The Lie algebra $\fg_{\pi}$ is a "maximal Kac-Moody 
subalgebra" of $\fg_{\ol{0}}$. For example, for $\fg$ of type $C_2^{(2)}$,
$\fg_{\ol{0}}$ is not a Kac-Moody algebra: in this case $\Sigma=\{\alpha,\delta-\alpha\}$
where $p(\alpha)=p(\delta-\alpha)=\ol{1}$ and $\delta$ is the minimal
imaginary roots; $\fg_{\pi}$ is of type $A_1^{(1)}$ with 
the simple roots $\{2\alpha,2\delta-2\alpha\}$, 
so $\fg_{\delta}\cap \fg_{\pi}=0$.

\subsubsection{}\label{521}
 By~\cite{GSsnow}, Theorem 4.5.1, for  $N\in {\CO}(\fg)$ such that $\Delta_N$ is well defined
 (for instance, if $N$ is indecomposable),
 the following conditions
 are equivalent:
\begin{enumerate}
\item 
the expression $De^{\rho}\ch N$ is (naturally)
$W[\Delta_N]$-anti-invariant, i.e. 
$$s_{\alpha}(De^{\rho}\ch N)=-
De^{\rho}\ch N\ \ \text{ for all }\alpha\in \Delta_N$$
\item the expression $D_{\pi}e^{\rho_{\pi}}\ch N$ is (naturally)
$W[\Delta_N]$-anti-invariant.
\end{enumerate}

\subsubsection{}
\begin{defn}{defn:Snow} If $N\in {\CO}(\fg)$ is such that $\Delta_N$ is well defined, 
we say that $N$ is
{\em quasi-admissible}\footnote{In~\cite{GSsnow} quasi-admissible modules are called ``snowflake''.} if $\ch N$ satisfies the  equivalent conditions (i) and (ii).
\end{defn}

\subsubsection{Remark}\label{rem:NN}
The equivalence in~$\S$~\ref{521}  implies that   
\begin{equation}\label{eq:equivNN}
N \text{ is quasi-admissible }\ \Longleftrightarrow\ \ 
\Res^{\fg}_{\fg_{\pi}} N \text{ is quasi-admissible}.\end{equation}

\subsubsection{Theorem}\label{thm:snow} (\cite{KW88}; \cite{GSsnow}, Section 4).
Let $\fg$ be a symmetrizable  anisotropic affine Kac-Moody superalgebra.
\begin{enumerate}
\item A $\fg$-module  $N$ is quasi-admissible if and only if every irreducible subquotient of $N$
  is quasi-admissible.
 \item  The following conditions are equivalent
\begin{itemize}
 \item[$\bullet$] $L:=L(\lambda)$ is quasi-admissible;
\item[$\bullet$]  $L'$
  is an integrable $\fg'$-module (see~$\S$~\ref{glambda} for notation).
  \end{itemize}
  \item If $L(\lambda)$ is   non-critical quasi-admissible, then
  $De^{\rho}\ch L(\lambda)=\!\!\sum\limits_{w\in W_L}\! (\sgn w)\  e^{w(\lambda+\rho)}$.
  \item If $L(\nu)$ is a non-critical quasi-admissible subquotient of $M(\lambda)$, then $\lambda=\nu$. 
 \item All  non-critical quasi-admissible modules are 
 completely reducible.
\end{enumerate}

\subsection{Remark}\label{snowabc}
Take $L:=L(\lambda)$ and $L'=L_{\fg'}(\lambda')$ as in~$\S$~\ref{glambda}.
By~\Lem{lem:KMbase} (ii), the set $\Sigma_L:=\psi(\Sigma')$ 
is the set of indecomposable elements in $\Delta_L^+$.
Since $\langle\lambda+\rho,\alpha^{\vee}\rangle=\langle\lambda'+\rho',\psi^{-1}(\alpha)^{\vee}\rangle$, 
 the following conditions are equivalent:
\begin{itemize}
\item[(a)]  $L_{\fg'}(\lambda')$ is integrable;
\item[(b)] 
 $\langle\lambda+\rho,\alpha^{\vee}\rangle\in\mathbb{Z}_{>0}$ for all $\alpha\in \Sigma_L$.
   \end{itemize}
   
 \subsection{Quasi-admissible modules in general case}\label{g>0}
 We retain notation of~$\S$~\ref{newnot}.

  \subsubsection{Notation}\label{Wnatural}
  Recall that $\pi$ is a linearly independent subset of 
 $\Sigma_{\pr}$, $\fg_{\pi}$ is the corresponding  Kac-Moody subalgebra of 
 $\fg_{\ol{0}}$  and $\Delta_{\pi}=\Delta(\fg_{\pi})^{\ree}$. 
 We denote by $W[\pi]$ the subgroup of $W$ generated by 
 $\{s_{\alpha}\}_{\alpha\in\pi}$. This
   is the Weyl group of $\fg_{\pi}$ and 
   $\Delta_{\pi}=W[\pi](\pi)$. As before $D_{\pi}$, $\rho_{\pi}$ denote the corresponding Weyl denominator
   and Weyl vector respectively.

Recall that for any subset $X$ of $\Delta^{\ree}$
the notation $W[X]$ stands for the subgroup of $W$ generated
by $s_{\alpha}$ with $\alpha\in X$.

If $L$ is fixed, we 
retain notation of~$\S$~\ref{glambda} and 
denote by $\Sigma'_{\pr}$ the set of principal roots for $\Delta'$; we set
$$\pi':=\{\alpha'\in\Sigma'_{\pr}|\ \psi(\alpha)\in\Delta_{\pi}\}. $$

We suggest the following definition.

\subsubsection{Definition}\label{defn:Snowgen}
If $N\in {\CO}(\fg)$ is such that $\Delta_N$
 is well defined we say that $N$ is
 {\em $\pi$-quasi-admissible} if  $D_{\pi}e^{\rho_{\pi}}\ch N$ is naturally
 $W[\Delta_L\cap \Delta_{\pi}]$-anti-invariant, i.e. 
$$s_{\alpha}(D_{\pi}e^{\rho_{\pi}}\ch N)=-D_{\pi}e^{\rho_{\pi}}\ch N\ \ \text{ for all }
\alpha\in \Delta_L\cap \Delta_{\pi}.$$

\subsubsection{Remark}\label{rem:choice}
Note that $\Delta_{\ol{0}}^+$ does not depend on
   the choice of a base in $\Sp$, so the expression 
   $D_{\pi}e^{\rho_{\pi}}\ch N$ does not depend on this choice
as well.

\subsubsection{}
In the anisotropic case~\Defn{defn:Snowgen} is equivalent to~\Defn{defn:Snow}.

In~\Thm{thm:criterionN} below we partially extend~\Thm{thm:snow} (i) to general $\fg$;
in~\Thm{thm:criterion} below we  extend~\Thm{thm:snow} (ii).
The results which are similar  of~\Thm{thm:snow} (iii)--(v) and the 
 equivalence  (a) $\Longleftrightarrow$ (b) in~$\S$~\ref{snowabc}
do not hold for atypical quasi-admissible modules, since these assertions do not hold for
atypical $\pi$-integrable modules which are particular cases of $\pi$-quasi-admissible
modules.
  
\subsubsection{}\begin{exa}{exa:piDeltaL}
By~\Lem{lem:Nabc}, if $N\in\CO_{\Sigma}(\fg)$ is such that $\Delta_N$ is well defined and 
 $\pi\subset \Delta_N$, then
the $\pi$-quasi-admissibility of $N$ is equivalent to
the $\pi$-integrability.
\end{exa}

%Now take $N\in \CO_{\Sigma}(\fg)_{\Lambda}$ and consider $M:=\Res^{\fg}_{\fg_{\pi}} N$. 
%The set $\Delta_M$ is well defined and coincides
%with 
%$\{\alpha\in\Delta_{\pi}| \ \langle\Lambda,\alpha^{\vee}\rangle \in\mathbb{Z}\}$;
%this is the integral root subsystem of $\Delta_{\pi}$.

\subsection{Main results}
In this section we will prove the following theorems.

\subsubsection{}
\begin{thm}{thm:criterionN}
Let $N\in\CO^{\fin}(\fg)$ be such that $\Delta_N$ is well-defined.
Then the following conditions are equivalent:
\begin{itemize}
\item[(a)]
$N$ is $\pi$-quasi-admissible;
\item[(b)] all irreducible
subquotients of $N$ are  $\pi$-quasi-admissible;
\item[(c)] all irreducible
subquotients of 
$M:=\Res^{\fg}_{\fg_{\pi}} N$ are quasi-admissible
$\fg_{\pi}$-modules, $\Delta_M$ is well defined and    $\Delta_{\pi}\cap\Delta_N=\Delta_M$. 
\end{itemize}
\end{thm}

The proof will be given in~$\S$~\ref{proof:thm:criterionN}.

\subsubsection{}
\begin{cor}{cor:critN}
If $N\in\CO^{\fin}(\fg)$ is  $\pi$-quasi-admissible and
$M:=\Res^{\fg}_{\fg_{\pi}} N$ is non-critical, then
$M$ is a completely reducible $\fg_{\pi}$-module.
\end{cor}

The proof will be given in~$\S$~\ref{proof:cor:critN}.

\subsubsection{}
\begin{thm}{thm:criterion}
If $L:=L(\lambda)$ is non-critical, then the following conditions are equivalent:
\begin{itemize}
\item[(a)]
$L$ is  $\pi$-quasi-admissible;
\item[(b)]  $L'$ is  $\pi'$-integrable, where $\pi'=\{\alpha'\in\Sigma'_{\pr}| \
\psi(\alpha')\in\Delta_{\pi}\}$.
\end{itemize}
\end{thm}

The proof will be given in~$\S$~\ref{proof:thm:criterion}.

%\subsubsection*{Remark}
%Let us show $\pi'$ is linearly independent. Suppose the contrary.
%Then $\Sigma'_{\pr}$ is not  linearly independent, so
%$\fg'$ is affine. 

%Consider the case  $\fg'\not=D(2|1,a)^{(1)}$.
%Since $\pi'$ is linearly dependent,
%$\pi'=\Sigma'_{\pr}$ and the minimal imaginary root
%$\delta'$ lies in $\mathbb{Z}_{\geq 0}\{\alpha\in\pi'|\ (\alpha,\alpha)>0\}$
%and in $\mathbb{Z}_{\geq 0}\{\alpha\in\pi'|\ (\alpha,\alpha)<0\}$.
%Then $\psi(\delta')$ lies in $\mathbb{Z}_{\geq 0}\{\alpha\in\pi|\ (\alpha,\alpha)>0\}$
%and in $\mathbb{Z}_{\geq 0}\{\alpha\in\pi|\ (\alpha,\alpha)<0\}$.

\subsubsection{Remark}
By~(\ref{eq:Rlambda+rho}) we have
$\{\alpha\in\Delta_{\pi}|\ \langle \lambda,\alpha^{\vee}\rangle\in\mathbb{Z}\}=R_{\lambda+\rho}\cap\Delta_{\pi}$.
Combining~\Thm{thm:criterionN} and~$\S$~\ref{DeltaMN}  we obtain
\begin{equation}\label{eq:DeltaLM}
L(\lambda)\ \text{ is  $\pi$-quasi-admissible }\ \Longrightarrow\ \ \ 
\Delta_{\lambda+\rho}\cap\Delta_{\pi}=R_{\lambda+\rho}\cap\Delta_{\pi}.
\end{equation}

\subsection{The set $\Delta_M$ for $M:=\Res^{\fg}_{\fg_{\pi}} N$}\label{DeltaMN}
Recall that $\Delta_{\pi}\subset\fh^*$ is the set of real roots of the Kac-Moody algebra
$\fg_{\pi}$.

The set $\Delta_M$ is well defined if and only if
the set 
$\{\alpha\in\Delta_{\pi}| \ \langle\nu,\alpha^{\vee}\rangle \in\mathbb{Z}\}$
is the same for 
 all $\nu$ such that  $[M: L_{\fg_{\pi}}(\nu)]\not=0$; moreover
 $\Delta_M$ coincides with these sets. (Notice that  
 $\{\alpha\in\Delta_{\pi}| \ \langle\nu,\alpha^{\vee}\rangle \in\mathbb{Z}\}$
 is the integral root subsystem 
of $\Delta_{\pi}$ corresponding to the weight 
$\nu$.)
Take $\nu$ such that $[M: L_{\fg_{\pi}}(\nu)]\not=0$. Then
there exists $\lambda\in\fh^*$ such that $L(\lambda)$ is a subquotient of $N$ and
$L(\lambda)_{\nu}\not=0$, so
$\lambda-\nu\in\mathbb{Z}\pi$. Therefore
$$\{\beta\in\Delta_{\pi}|
\ \langle \nu,\beta^{\vee}\rangle\in\mathbb{Z}\}=\{\beta\in\Delta_{\pi}|
\ \langle \lambda,\beta^{\vee}\rangle\in\mathbb{Z}\}.$$

 We conclude that 
\begin{itemize}
\item[(i)] $\Delta_M$ is well defined if and only if
the set $\{\beta\in\Delta_{\pi}|
\ \langle \lambda,\beta^{\vee}\rangle\in\mathbb{Z}\}$
is the same for all $\lambda$
such that $[N: L(\lambda)]\not=0$;
\item[(ii)] if $\Delta_M$ is well defined, then 
$\Delta_M=\{\beta\in\Delta_{\pi}|
\ \langle \lambda,\beta^{\vee}\rangle\in\mathbb{Z}\}$
for any $\lambda$
such that $[N: L(\lambda)]\not=0$ and  $\Delta_M\subset (\Delta_N\cap \Delta_{\pi})$;
\item[(iii)] 
$\{\alpha\in\Delta_{\pi}| \ \langle\nu,\alpha^{\vee}\rangle \in\mathbb{Z}\}\subset (\Delta_N\cap\Delta_{\pi})$ for any $\nu$ such that $N_{\nu}\not=0$.
\end{itemize}

\subsubsection{}\label{pi'}
The important example for the next sections is $L=L(\lambda)$,
 $M:=\Res^{\fg}_{\fg_{\pi}} L$.
 Then $$\Delta_L\cap\Delta_{\pi}=\Delta_{\pi}\cap \Delta_{\lambda+\rho}.$$
 Combining 
(i) and~(\ref{eq:Rlambda+rho}) we get
$$\Delta_M=\{\beta\in\Delta_{\pi}|
\ \langle \lambda,\beta^{\vee}\rangle\in\mathbb{Z}\}=R_{\lambda+\rho}\cap\Delta_{\pi}.$$

The condition  $\Delta_L=\Delta_M\cap\Delta_{\pi}$  holds in many cases, for example if
\begin{itemize}
\item[$\bullet$] $\Delta_{\pi}$ consists of short roots: in this case~\Lem{lem:Xi}  (i) gives
$\Delta_{\pi}\cap R_{\lambda+\rho}=
\Delta_{\pi}\cap \Delta_{\lambda+\rho}$.
\item[$\bullet$] if $L(\lambda)$ is typical (in this case 
$\Delta_L=R_{\lambda+\rho}$).
\end{itemize}

Introduce $R_{\lambda'+\rho'}\subset\Delta'$ as in~$\S$~\ref{Deltalambda}.
By~\Lem{lem:RRlambda}, 
$\psi$ maps  $R_{\lambda'+\rho'}$ to $R_{\lambda+\rho}$, and  $\Delta_{L'}=(\Delta')^{\ree}$. 
Therefore $\Delta_{L'}\cap \Delta_{\pi'}=\Delta_{\pi'}$ which
gives the following useful equivalence
\begin{equation}\label{eq:DeltaRDelta'R'}
\Delta_{\pi}\cap\Delta_{\lambda+\rho}=\Delta_{\pi}\cap R_{\lambda+\rho}\ \ 
\Longleftrightarrow\ \ \pi'\subset R_{\lambda'+\rho'}.
\end{equation}

\subsubsection{Example}\label{B13}
 The following example shows that  $\Delta_{\pi}\cap R_{\lambda+\rho}$  can be a proper subset of
 $\Delta_{\pi}\cap\Delta_{\lambda+\rho}$ (i.e. $\Delta_M$ 
is a proper subset of $\Delta_{\pi}\cap\Delta_L$).

Take $\fg$ of type $B(1|3)$ with $\Sigma=\{\delta_1-\delta_2, \delta_2-\delta_3,  \delta_3-\vareps_1, 
\vareps_1\}$. One has $p(\vareps_1)=\ol{0}$ and $p(\delta_i)=\ol{1}$.
We have $\Sigma_{\pr}=\{\delta_1-\delta_2, \delta_2-\delta_3,  
 \delta_3,\vareps_1\}$ and we take $\pi:=\{\delta_1-\delta_2, \delta_2-\delta_3,  
 \delta_3\}$. 
 
 For
 $\lambda+\rho:=3\delta_1+2\delta_2+\delta_3-\vareps_1$ we have
${R}_{\lambda+\rho}=\{\pm\vareps_1;\pm \delta_i\pm\delta_j;\pm (\delta_3-\vareps_1)\}$.
Therefore
  $\Delta_{\lambda+\rho}=\Delta$ and $\Delta_{\pi}\cap\Delta_{\lambda+\rho}= \Delta_{\pi}$ is 
  the root system of  type $B_3$, and
 $R_{\lambda+\rho}\cap \Delta_{\pi}=\{\pm \delta_i\pm\delta_j\}$ is the root system of type 
$D_3$.

\subsubsection{}
\begin{lem}{lem:W'L}
\begin{enumerate}
\item The set
 of indecomposable elements in $\Delta_L\cap\Delta^+_{\ol{0}}$
 coincides with $\psi(\Sigma'_{\pr})$. One has  
 $\Delta_L\cap\Delta_{\ol{0}}=W[\Delta_L](\psi(\Sigma'_{\pr}))$
and  $W[\Delta_L]$ is the Coxeter group 
generated by $s_{\alpha}$ with $\alpha\in \psi(\Sigma'_{\pr})$. 

\item The set of indecomposable elements in $\Delta_L\cap \Delta^+_{\pi}$ 
coincides with $\psi(\pi')$. One has  
$\Delta_L\cap\Delta_{\pi}=W[\Delta_L\cap \Delta_{\pi}](\psi(\pi'))$
and   $W[\Delta_L\cap \Delta_{\pi}]$ is the Coxeter group 
generated by $s_{\alpha}$ with $\alpha\in\psi(\pi')$.
\end{enumerate}
\end{lem}
\begin{proof}
The map $s_{\alpha'}\mapsto s_{\psi(\alpha')}$ for $\alpha'\in \Delta'$ defines
a group homomorphism 
$\psi_W:\ W(\fg')\to W[\Delta_L]$
which satisfies
$\psi(w\alpha')=\psi_W(w)\psi(\alpha')$ for all $w\in W(\fg'),\ \alpha'\in\Delta'$.
Since
$\mathbb{C}\Delta'$ is a faithful representation of $W(\fg')$, the map $\psi_W$
is bijective, so 
$$\psi_W: W(\fg')\iso W[\Delta_L].$$
By~$\S$~\ref{Bpr}, $(\Delta'_{\ol{0}})^{\ree}$
is a real root system with the set of simple roots $\Sigma'_{\pr}$ and
 $W(\fg')$ is the Coxeter group generated by $s_{\alpha'}$
with $\alpha'\in\Sigma'_{\pr}$. This implies (i).

For (ii) take any $\beta\in (\Delta_L\cap \Delta^+_{\pi})$
and set $\beta':=\psi^{-1}(\beta)$. Then
$\beta'$ is a positive even real root in $\Delta'$.
By~\Lem{lem:standard}, 
there exist $\alpha'_1,\ldots,\alpha'_j\in \pi'$
such that $\beta'=s_{\alpha'_j}\ldots s_{\alpha'_2}\alpha'_1$
and $\beta'=\sum\limits_{i=1}^j k_i\alpha'_i$ with $k_i>0$.
Set $\alpha_i:=\psi(\alpha'_i)$ and notice that $\alpha_i\in \psi(\Sigma'_{\pr})$.
Therefore
$$\beta=\sum\limits_{i=1}^j k_i\alpha_i\ \ \text{ with }\ k_i\in\mathbb{Z}_{>0},\ \ 
\alpha_i\in \psi(\Sigma'_{\pr}).$$
Let us show that $\alpha_i\in\psi(\pi')$ for all $i=1,\ldots,j$.
Since $\Sigma_{\pr}$
is the set of indecomposable elements in $\Delta^{\ree +}_{\ol{0}}$ we have 
$$\alpha_i=\sum\limits_{\gamma\in\Sigma_{\pr}} m^i_{\gamma} \gamma\ \text{
for some }\ m^i_{\gamma}\geq 0.$$
Therefore
$$\beta=\sum\limits_{\gamma\in\Sigma_{\pr}} n_{\gamma}\gamma,\ \ \text{ for }
\ n_{\gamma}:=\sum\limits_{i=1}^j k_i m^i_{\gamma}.$$
 If $\alpha_i\not\in\psi(\pi')$, then
$m^i_{\gamma}\not=0$ for some $\gamma\not\in\pi$, so
$n_{\gamma}\not=0$ and thus 
$\beta\not\in \Delta_{\pi}$, a contradiction. Hence $\alpha_1,\ldots, \alpha_j\in \psi(\pi')$. Therefore 
$\beta\in\mathbb{Z}_{\geq 0}\psi(\pi')$
and 
$\beta=s_{\alpha_j}\ldots s_{\alpha_2}\alpha_1$. This proves that
$\psi(\pi')$ is the set of indecomposable elements in $\Delta_L\cap \Delta^+_{\pi}$
and establishes the formula
$\Delta_{L}\cap\Delta_{\pi}=W[\psi(\pi')](\psi(\pi'))$ which implies
$W[\Delta_{L}\cap\Delta_{\pi}]=W[\psi(\pi')]$. Since $\psi(\pi')$
is a subset of $\psi(\Sigma'_{\pr})$ and, by (i),  $s_{\alpha}$
with $\alpha\in \psi(\Sigma'_{\pr})$ are Coxeter generators of 
$W[\Delta_L]=W[\psi(\Sigma'_{\pr})]$,  the set $s_{\alpha}$
with $\alpha\in \psi(\pi')$ are Coxeter generators of 
$W[\psi(\pi')]$.
This completes the proof of (ii).
\end{proof}

\subsection{Enright functor}\label{Enright}
The proofs are based on  the following construction.

\subsubsection{}\label{EnrightM}
Let $\alpha$ be a positive even real root and let $\ft\subset \fg$
be any subalgebra containing $\fg_{\pm \alpha}$.
Let $\cM^e(\ft)$ be the full subcategory of the category of $\ft$-modules
consisting of the objects where $\alpha^{\vee}$ acts semisimply with non-integral eigenvalues
and $\fg_{\alpha}$ acts locally nilpotently. By~\cite{IK};\cite{GSsnow}, Section 2, the category $\cM^e(\ft)$
admits the Enright functor  $\mathcal{C}:\cM^e(\ft)\iso\cM^e(\ft)$ 
which is an equivalence of categories
such that  $\mathcal{C}^2(N)\cong N$ for all $N\in \cM^e(\ft)$.
By construction, the functor $\mathcal{C}$ commutes with the restriction functor, i.e.
$\mathcal{C}( \Res^{\fg}_{\ft} (N))$ is (canonically) isomorphic to 
$\Res^{\fg}_{\ft}(\mathcal{C}(N)) $ for any $N\in \cM^e(\ft)$.

\subsubsection{}\label{EnrightO}
Let $\pi\subset \Sigma_{\pr}$ be as in~$\S$~\ref{Wnatural}.
Assume that $\alpha\in\pi$ is such that 
$\Sigma$ contains $\alpha$ or $\alpha/2$.
Let $\CO_{\alpha}$ be the full subcategory 
of $\CO^{\fin}(\fg)$ whose objects 
 are  modules  lying in $\cM^e(\fg)$.
 By~\cite{IK}; \cite{GSsnow}, Theorems 2.3.1 and 3.2  the Enright functor  gives
 the equivalence of categories
$\mathcal{C}:\CO_{\alpha}\iso\CO_{\alpha}$, such that 
\begin{equation}\label{eq:Enright}\begin{array}{l}
\mathcal{C}(M(\nu))=M(s_{\alpha}(\nu+\rho)-\rho),
\ \ \ 
\mathcal{C}(L(\nu))=L(s_{\alpha}(\nu+\rho)-\rho),\\
(e^{\frac{\alpha}{2}}-e^{-\frac{\alpha}{2}})
\ch \mathcal{C}(N)=s_{\alpha}\bigl(
(e^{\frac{\alpha}{2}}-e^{-\frac{\alpha}{2}})  \ch N\bigr)
\end{array}
\end{equation}
for all $L(\nu), N\in \CO_{\alpha}$. 
Notice that $L(\nu)\in \CO_{\alpha}$
if and only if $\langle \nu,\alpha^{\vee}\rangle\not\in\mathbb{Z}$.

Since $\alpha\in\pi$,
the product $e^{\rho_{\pi}-\alpha/2}D_{\pi}(1-e^{-\alpha})^{-1}$ 
is naturally $s_{\alpha}$-invariant (see~(\ref{eq:sgammainv})).
Now the second formula in~(\ref{eq:Enright}) gives 
$$\begin{array}{ll}
D_{\pi}e^{\rho_{\pi}}\ch N&=
\frac{D_{\pi}}{1-e^{-\alpha}}\cdot \frac{e^{\rho_{\pi}}}{e^{\alpha/2}}\cdot 
(e^{\frac{\alpha}{2}}-e^{-\frac{\alpha}{2}})\ch N\\ 
&=\frac{D_{\pi}}{1-e^{-\alpha}}\cdot \frac{e^{\rho_{\pi}}}
{e^{\alpha/2}}\cdot s_{\alpha}\bigl( 
(e^{\frac{\alpha}{2}}-e^{-\frac{\alpha}{2}})  
\ch \mathcal{C}(N)\bigr)\\
&=s_{\alpha}\bigl( \frac{D_{\pi}}{1-e^{-\alpha}}\cdot
\frac{e^{\rho_{\pi}}}{e^{\alpha/2}}\cdot 
(e^{\frac{\alpha}{2}}-e^{-\frac{\alpha}{2}}) 
\ch \mathcal{C}(N)\bigr)
=s_{\alpha}\bigl(D_{\pi}e^{\rho_{\pi}}
\ch \mathcal{C}(N)\bigr)\end{array}$$
for the natural action of $s_{\alpha}$, which implies
\begin{equation}\label{eq:EnrightD}
D_{\pi}e^{\rho_{\pi}}\ch \mathcal{C}(N)
=s_{\alpha}
\bigl(D_{\pi}e^{\rho_{\pi}}
\ch N\bigr)\ \ \text{ for all } N\in\CO_{\alpha}
\end{equation}
for the natural action of $s_{\alpha}$.
%(Similarly we have $De^{\rho}\ch %N=s_{\alpha}\bigl(D e^{\rho}\ch 
%\mathcal{C}(N)\bigr)$.)

\subsubsection{}\begin{lem}{lem:Enright}
For any $N\in\CO_{\alpha}$ and $\gamma\in\Delta^{\ree}$ one has
$$D_{\pi}e^{\rho_{\pi}}\ch N=-
s_{\gamma}\bigl(D_{\pi}e^{\rho_{\pi}}\ch N\bigr)
\ \ \Longleftrightarrow\ \ 
D_{\pi}e^{\rho_{\pi}}\ch \mathcal{C}(N)=
-s_{s_{\alpha}\gamma}\bigl((D_{\pi}
e^{\rho_{\pi}}\ch \mathcal{C}(N)\bigr).
$$
\end{lem}
\begin{proof}
Since $\mathcal{C}^2(N)\cong N$ it is enough to verify 
the implication $\Longrightarrow$.
Combining the formulas
$D_{\pi}e^{\rho_{\pi}}\ch N=-
s_{\gamma}\bigl(D_{\pi}e^{\rho_{\pi}}\ch N\bigr)$
and~(\ref{eq:EnrightD}) we get
$$\begin{array}{rl}
D_{\pi}e^{\rho_{\pi}}
\ch \mathcal{C}(N)&=
s_{\alpha}\bigl(D_{\pi}e^{\rho_{\pi}}\ch N\bigr)
=-s_{\alpha}s_{\gamma}\bigl(D_{\pi}e^{\rho_{\pi}}
\ch N\bigr)=
-s_{s_{\alpha}\gamma}s_{\alpha}
\bigl(D_{\pi}e^{\rho_{\pi}}\ch N\bigr)\\
&=-s_{s_{\alpha}\gamma}
\bigl(D_{\pi}e^{\rho_{\pi}}
\ch \mathcal{C}(N)\bigr)
\end{array}$$
as required.
\end{proof}

\subsubsection{}
Let $N\in\CO^{\fin}(\fg)$ be such that $\Delta_N$
is well defined. We claim that
\begin{equation}\label{eq:Oalpha}
 \alpha\not\in \Delta_N\ \ \Longrightarrow\ \ N\in \CO_{\alpha}. 
\end{equation}
Indeed, if $\alpha\not\in \Delta_N$, then
by~(\ref{eq:Rlambda+rho}), 
$\langle \lambda,\alpha^{\vee}\rangle\not\in\mathbb{Z}$
for any $\lambda$ such that $[N:L(\lambda)]\not=0$; 
this implies $\langle \mu,\alpha^{\vee}\rangle\not\in\mathbb{Z}$
for all $\mu$ such that $N_{\mu}\not=0$, so $N\in\CO_{\alpha}$.

\subsubsection{}
We will use the following formula
\begin{equation}\label{eq:multi}
[\Res^{\fg}_{\fg_{\pi}} \mathcal{C}(N): L_{\fg_{\pi}}(s_{\alpha}(\nu+\rho_{\pi})-\rho_{\pi})]=[\Res^{\fg}_{\fg_{\pi}} N: L_{\fg_{\pi}}(\nu)] \ \text{ for  any }\ N\in \CO_{\alpha}.
\end{equation}
Indeed, take $N\in \CO_{\alpha}$.
Then $\mathcal{C}(N)\in \CO_{\alpha}$. Set $M:=\Res^{\fg}_{\fg_{\pi}} N$. 
(If $M\in\CO^{\fin}(\fg_{\pi})$ the formula follows 
 immediately from~$\S$~\ref{EnrightO}, but $M$
might be not finitely generated.)
By~$\S$~\ref{categoryO},
$M\in \CO_{\Sigma}(\fg_{\pi})$
and this module admits a weak composition series at $\nu$:
$$0=M_0\subset M_1\subset \ldots\subset M_r=M.   $$
Observe that $M_1,\ldots, M_r\in \cM^e(\fg_{\pi})$. In the light of~$\S$~\ref{EnrightM}  we obtain a finite filtration
$$0=\mathcal{C}(M_0)\subset \mathcal{C}(M_1)\subset \ldots \mathcal{C}(M_r)=\mathcal{C}(M).$$
Using~(\ref{eq:Enright}) for the algebra $\fg_{\pi}$ we get
$\mathcal{C}(M_i/M_{i-1})\cong L_{\fg_{\pi}}(s_{\alpha}(\nu+\rho_{\pi})-\rho_{\pi})$
 if and only if $M_i/M_{i-1}\cong L_{\fg_{\pi}}(\nu)$. This gives the inequality $\geq $
 in~(\ref{eq:multi}). The opposite inequality follows from 
 the formula $\mathcal{C}^2(N)\cong N$. This establishes~(\ref{eq:multi}).

 \subsection{Proof of~\Thm{thm:criterionN}}\label{proof:thm:criterionN}
Let $N\in\CO^{\fin}(\fg)$ be such that $\Delta_N$ is well defined.
We want to prove the equivalence of the conditions
\begin{itemize}
\item[(a)]
$N$ is $\pi$-quasi-admissible;
\item[(b)] all irreducible
subquotients of $N$ are  $\pi$-quasi-admissible;
\item[(c)] all irreducible
subquotients of 
$M:=\Res^{\fg}_{\fg_{\pi}} N$ are quasi-admissible
$\fg_{\pi}$-modules, and   $\Delta_{\pi}\cap\Delta_N=\Delta_M$. 
\end{itemize}

\subsubsection{}
\begin{lem}{lem:gammanotinpi}
Fix  any $\lambda$ 
and introduce $\pi'$ for $L(\lambda)$ as in~$\S$~\ref{Wnatural}.
If $\gamma\in\psi(\pi')$ is such that $\gamma\not\in\pi$, then there exists
$\alpha\in\pi$ such that 
$s_{\alpha}\gamma=\gamma-j\alpha$ for some $j\in\mathbb{Z}_{>0}$
and $s_{\alpha}\gamma\in\Delta^+_{\pi}$. Moreover,
 $\alpha\in \Delta_{\lambda+\rho}$.
 \end{lem}
 \begin{proof}
Recall that $\pi\subset\Sigma_{\pr}$ and $\Delta_{\pi}=W[\pi](\pi)$.
Using~\Lem{lem:standard} for $\Delta_{\pi}$, we conclude that there exists
$\alpha\in\pi$ such that $\gamma-s_{\alpha}\gamma=j\alpha$ for some $j\in\mathbb{Z}_{>0}$
and $s_{\alpha}\gamma\in\Delta^+_{\pi}$.
If $\alpha\in \Delta_N$, then 
$s_{\alpha}\gamma\in \Delta_N$. Therefore
 $\gamma=s_{\alpha}\gamma+j\alpha $ with $s_{\alpha}\gamma, \alpha$
 in $ \Delta_N\cap\Delta^+_{\pi}$
 which contradicts to~\Lem{lem:W'L} (since $\gamma\in\psi(\pi')$).
Therefore $\alpha\not\in \Delta_N$. 
\end{proof}

\subsubsection{}
\begin{lem}{}Fix  any $\lambda$ such that $[N:L(\lambda)]\not=0$
and introduce $\pi'$ as in~$\S$~\ref{Wnatural}.
Take $\gamma\in\psi(\pi')$ and $\nu$ such that $[\Res^{\fg}_{\fg_{\pi}} N:
L_{\fg_{\pi}}(\nu)]\not=0$.
If $D_{\pi} e^{\rho_{\pi}} \ch N$ is  $s_{\gamma}$-anti-invariant, then
$D_{\pi} e^{\rho_{\pi}} \ch L_{\fg_{\pi}}(\nu)$ is  $s_{\gamma}$-anti-invariant
and $\langle\nu,\gamma^{\vee}\rangle\in\mathbb{Z}$.
\end{lem}
 \begin{proof}
We proceed by induction on    the partial order on $\Delta_{\pi}$ given
by $\alpha\geq \beta$ if $\alpha-\beta\in\mathbb{Z}_{\geq 0}\pi$.

If $\gamma\in\pi$, then, by~\Lem{lem:Nabc}, $\fg_{-\gamma}$ acts locally nilpotently on $N$ and 
thus acts locally nilpotently on  $L_{\fg_{\pi}}(\nu)$. 
By~\Lem{lem:Nabc},
$R_{\pi}e^{\rho_{\pi}}\ch L_{\fg_{\pi}}(\nu)$ is 
$s_{\gamma}$-anti-invariant.

Take $\gamma\not\in\pi$ and $\alpha$ as in~\Lem{lem:gammanotinpi}.
Since $\alpha\in\pi$,
either $\alpha$ or $\alpha/2$ lies in a certain base in the spine.
Using~\ref{rem:choice}
we conclude that, without loss of generality, we can (and will) assume that
$\Sigma$ contains $\alpha$ or $\alpha/2$. Let $\mathcal{C}$
be the Enright functor for $\alpha$. Since  $\alpha\not\in\Delta_{\lambda+\rho}=\Delta_N$
we have $\langle \nu,\alpha^{\vee}\rangle\not\in\mathbb{Z}$ by
$\S$~\ref{DeltaMN} (iii). By~(\ref{eq:multi}) 
we have
$$[\Res^{\fg}_{\fg_{\pi}}(\mathcal{C}(N)):  L_{\fg_{\pi}}(\mu)]\not=0\ \ \text{ for }\mu:=s_{\alpha}(\nu+\rho_{\pi})-\rho_{\pi}.$$
Combining~\Lem{lem:Enright}
and~(\ref{eq:Oalpha}) we conclude that 
the assumption that 
 $D_{\pi} e^{\rho_{\pi}} \ch N$ 
is $s_{\gamma}$-anti-invariant implies
the
 $s_{s_{\alpha}\gamma}$-anti-invariance
of $D_{\pi} e^{\rho_{\pi}}\ch \mathcal{C}(N)$.

By~\Lem{lem:salphaSigma'} we have
$\psi_{s_{\alpha} (\lambda+\rho)-\rho}(\gamma')=s_{\alpha}\gamma$.
By~$\S$~\ref{EnrightO}, the module $L(s_{\alpha}(\lambda+\rho)-\rho)$
is a subquotient of $\mathcal{C}(N)$.
Since $s_{\alpha}\gamma<\gamma$, by induction hypothesis,  the $s_{s_{\alpha}\gamma}$-anti-invariance
of $D_{\pi} e^{\rho_{\pi}}\ch \mathcal{C}(N)$ implies 
 $s_{s_{\alpha}\gamma}$-anti-invariance
of $D_{\pi} e^{\rho_{\pi}}\ch L_{\fg_{\pi}}( \mu)$
and  $\langle \mu, (s_{\alpha}\gamma)^{\vee}\rangle\not\in\mathbb{Z}$.
In particular,
$$\langle \mu, (s_{\alpha}\gamma)^{\vee}\rangle=\langle s_{\alpha}(\nu+\rho_{\pi}), (s_{\alpha}\gamma)^{\vee}\rangle
=\langle \nu+\rho_{\pi}, \gamma^{\vee}\rangle\not\in\mathbb{Z}
$$
so $\langle \nu, \gamma^{\vee}\rangle\not\in\mathbb{Z}$.
Since  $L_{\fg_{\pi}}(\mu)=\mathcal{C}(L_{\fg_{\pi}}(\nu))$,
the $s_{s_{\alpha}\gamma}$-anti-invariance
of $D_{\pi} e^{\rho_{\pi}}\ch L_{\fg_{\pi}}( \mu)$ implies
the $s_{\gamma}$-anti-invariance
of $D_{\pi} e^{\rho_{\pi}}\ch L_{\fg_{\pi}}(\nu)$. This completes the proof.
\end{proof}

\subsubsection{}
Assume that $N$ is $\pi$-quasi-admissible.
Combining the above lemma and~$\S$~\ref{DeltaMN} (iii) we get
$$\psi(\pi')\subset \{\beta\in\Delta_{\pi}|\ \langle \nu,\beta^{\vee}\rangle\in\mathbb{Z}\}\subset (\Delta_N\cap\Delta_{\pi}).$$
Using~\Lem{lem:W'L} we conclude $(\Delta_{\pi}\cap \Delta_N)=\{\beta\in\Delta_{\pi}|\ \langle \nu,\beta^{\vee}\rangle\in\mathbb{Z}\}$. Since this formula holds for all $\nu$ such that 
$[M: L_{\fg_{\pi}}(\nu)]\not=0$, the set $\Delta_M$ is well defined
and coincides with  $\Delta_{\pi}\cap \Delta_N$, see~$\S$~\ref{DeltaMN} (ii).
Combining the above lemma and~\Lem{lem:W'L} we conclude that
$L_{\fg_{\pi}}(\nu)$ is quasi-admissible. This establishes the implication (a) $\Longrightarrow$
(c).

Since $\ch N$ is the sum of 
the characters of irreducible subquotients of $N$, (b) implies (a).
Similarly, $\ch N$ is the sum of 
the characters of irreducible subquotients of $M$,  so (c) implies (a).
Hence (a) $\Longleftrightarrow$  (c) and (b) $\Longrightarrow$  (a).  

Now assume that (c) holds. 
Let $L$ be any subquotient of $N$. We have $\Delta_N=\Delta_L$. Any irreducible subquotient $L_1$ of
$\Res^{\fg}_{\fg_{\pi}} L$ is a subquotient of $M$, so, by (c),
it is quasi-admissible $\fg_{\pi}$-module. Moreover, 
and $\Delta_{\pi}\cap \Delta_L=\Delta_{\pi}\cap \Delta_N=\Delta_{L_1}$.
Using the implication (c) $\Longrightarrow$ (a) for $L$ we conclude that
$L$ is $\pi$-quasi-admissible. This completes the proof of
the implication (c) $\Longrightarrow$ (b) and the proof of~\Thm{thm:criterionN}.

 \subsection{Proof of~\Cor{cor:critN}}\label{proof:cor:critN}
 Let $N$ be a $\pi$-quasi-admissible
module and $M:=\Res^{\fg}_{\fg_{\pi}}(N)$. Assume that $M$ is non-critical.
Fix $v\in M$ and let $M_1$ be a cyclic module generated by $v$.
Then $M_1\in\CO^{\fin}(\fg_{\pi})$.
By~\Thm{thm:criterionN},
any irreducible subquotient of $M_1$
is quasi-admissible. Since $M_1$ is non-critical, $M_1$
is completely reducible, see~\Thm{thm:snow}. Hence any cyclic submodule of $M$ is completely reducible.
Thus $M$ is a sum of irreducible modules, so $M$ is completely reducible
(see~\cite{Lang}, Chapter XVII).

\subsection{Proof of~\Thm{thm:criterion}}\label{proof:thm:criterion} 
By~\Lem{lem:W'L}, (a) is equivalent to  $s_{\alpha}(D_{\pi} e^{\rho_{\pi}} \ch L)=-D_{\pi} e^{\rho_{\pi}} \ch L$ 
for all $\alpha\in \psi(\pi')$. 
On the other hand, (b) is equivalent to 
the fact that for all $\gamma'\in\pi'$ the root space $\fg'_{-\gamma'}$
acts locally nilpotently on $L'$. Using~\Lem{lem:Nabc}
we conclude that  it is enough to verify that for $\gamma'\in\pi'$  
the following conditions are equivalent 

\begin{itemize}
\item[(a')] $s_{\gamma}(D_{\pi} e^{\rho_{\pi}} \ch L)=-D_{\pi} e^{\rho_{\pi}} \ch L\ $ 
for  $\gamma:=\psi(\gamma')$;

\item[(b')] $s_{\gamma'}(D_{\pi'} e^{\rho_{\pi'}} \ch L')=-D_{\pi'} e^{\rho_{\pi'}} \ch L'$.
\end{itemize}

This  equivalence is established in the lemma below.

\subsubsection{}
\begin{lem}{lem:abequivnew}
For $\gamma'\in\pi'$ the assertions (a') 
 and (b') are equivalent.
\end{lem}
\begin{proof}
We proceed by induction on    the partial order on $\Delta_{\pi}$ given
by $\nu\geq \mu$ if $\nu-\mu\in\mathbb{Z}_{\geq 0}\pi$.
For the induction base we
assume that $\gamma\in\pi$.  
Then either $\gamma$ or $\gamma/2$ lies in a certain base in the spine.
Using~$\S\S$~\ref{345} and~\ref{rem:choice}
we conclude that, without loss of generality, we can (and will) assume that
$\Sigma$ contains $\gamma$ or $\gamma/2$. 
Then $\Sigma'$ contains $\gamma'$ or $\gamma'/2$ respectively.
By~\Lem{lem:Nabc} (a') and (b') are equivalent to locally nilpotence of actions 
of $\fg_{-\gamma}$ and $\fg'_{\gamma'}$ on $L$ and $L'$ respectively.
It is well-known that  $\fg_{-\gamma}$ acts locally nilpotently on $L$ 
if and only if  $\frac{2(\lambda+\rho,\gamma)}{(\gamma, \gamma)}$
is a positive integer which is odd if $\gamma/2\in\Sigma$.
Using the formula 
$\frac{2(\lambda'+\rho',\gamma')}{(\gamma', \gamma')}=
\frac{2(\lambda+\rho,\gamma)}{(\gamma, \gamma)}$.

Now assume that $\gamma\not\in\pi$. 
Take $\gamma\not\in\pi$ and $\alpha$ as in~\Lem{lem:gammanotinpi}.
Since $\alpha\in\pi$,
either $\alpha$ or $\alpha/2$ lies in a certain base in the spine.
Using~$\S\S$~\ref{345} and~\ref{rem:choice}
we conclude that, without loss of generality, 
we can (and will) assume that
$\Sigma$ contains $\alpha$ or $\alpha/2$ and thus
$\Sigma'$ contains $\alpha'$ or $\alpha'/2$ respectively.
Take $\nu:=s_{\alpha}(\lambda+\rho)-\rho$.  
Combining~\Lem{lem:Enright}
and~(\ref{eq:Oalpha}) we conclude that (a')  is equivalent to
$$D_{\pi}e^{\rho_{\pi}}\ch L(\nu)=-s_{s_{\alpha}\gamma}\bigl((D_{\pi}
e^{\rho_{\pi}}\ch L(\nu)\bigr).$$
By~\Lem{lem:salphaSigma'} we have
$\psi_{\nu}(\gamma')=s_{\alpha}\gamma$. Since $s_{\alpha}\gamma<\gamma$,
by induction hypothesis,  (b') is equivalent to  $s_{s_{\alpha}\gamma}$-anti-invariance
of $D_{\pi} e^{\rho_{\pi}}\ch L(\nu)$. Hence (a') is equivalent to (b') for $\gamma$.
This completes the proof.
\end{proof}

\section{Admissible weights}
In this section $\fg\not=D(2|1,a)^{(1)}$ is an indecomposable symmetrizable affine Kac-Moody superalgebra,
 which is not anisotropic (i.e., $\Delta$ contains a real isotropic 
root).

\subsection{Notation}\label{notataff}
We normalise the bilinear form $(-,-)$ as in~$\S$~\ref{dualCoxeter} and in
Appendix~\ref{affinebases}. For a weight $\lambda\in\fh^*$ we call $k:=(\lambda,\delta)$ the {\em level} of $\lambda$.

We let $\dot{\fg}\subset \fg$ be the finite-dimensional
Kac-Moody superalgebra described  in Appendix~\ref{affinebases}.
Note that
 $\fg=\dot{\fg}^{(1)}$ if $\fg$ is non-twisted.
We denote by $\dot{\Delta}$ the root system of $\dot{\fg}$.
 We let $\Lambda_0\in\fh^*$ be such that
$(\Lambda_0,\delta)=1$ and 
$(\Lambda_0,\dot{\Delta})=(\Lambda_0,\Lambda_0)=0$. We let
$\cl$ be the canonical map  
$\cl:\mathbb{Q}\Delta\to \mathbb{Q}\Delta/\mathbb{Q}\delta=\mathbb{Q}\dot{\Delta}$.
By~\Rem{rem:prince},
$\dot{\Sigma}_{\pr}:=\Sigma_{\pr}\cap \dot{\Delta}$  is the set of principal roots
for $\dot{\Delta}$. We set
$$\begin{array}{l}
\pi^{\#}:=\{\alpha\in\Sigma_{\pr}|\ (\alpha,\alpha)>0\},\ \ \ \Delta^{\#}:=\Delta_{\pi^{\#}}=
\{\alpha\in\Delta_{\ol{0}}|\ (\alpha,\alpha)>0\}\\
\pi:=\pi^{\#}\cup \dot{\Sigma}_{\pr}=
\{\alpha\in\Sigma_{\pr}|\ \alpha\in\dot{\Delta} \text{ or }
(\alpha,\alpha)>0\}.\end{array}
$$
The assumption $\fg\not=D(2|1,a)^{(1)}$ ensures that  $\pi$
is linearly independent.  If  $\fg=\dot{\fg}^{(1)}$ 
is a  non-twisted affine Kac-Moody superalgebra,
then  $\fg_{\pi^{\#}}$ is the direct product of an abelian Lie algebra and
a non-twisted affine Kac-Moody algebra,
which is the affinization of the "largest component" $\dot{\fg}^{\#}$ of $\dot{\fg}_{\ol{0}}$, since the root system 
$\Delta(\fg_{\pi^{\#}})$ is equal to $\Delta(\dot{\fg}^{\#})+\mathbb{Z}\delta$.

Since we use the normalised invariant form, the longest root in $\Delta^{\#}$ is of square length $2$ and $L(k\Lambda_0)$
is $\pi^{\#}$-integrable if and only if $k\in\mathbb{Z}_{\geq 0}$.

\subsubsection{Vacuum modules}\label{vacmod}
For each $k\in\mathbb{C}$ we introduce the vacuum module $V^k$ of level $k$ in the usual way:
we set $\fg_+:=\fh+\sum\limits_{\alpha\in\Delta: (\alpha,\Lambda_0)\geq 0} \fg_{\alpha}$
and let $\mathbb{C}_k$ the one-dimensional $\fg_+$-module
where each  $h\in\fh$ acts by multiplication by $\langle k\Lambda_0,h\rangle$; we set
$$V^k:=\Ind^{\fg}_{\fg_+}\mathbb{C}_k.$$

Notice that $L(k\Lambda_0)$ is the unique irreducible quotient of $V^k$;
we call this module a {\em simple vacuum module}; this term will be also used
for the corresponding $[\fg,\fg]$-module. 

\subsubsection{}
\begin{rem}{rem:Vkandpi}
Since $V^k$ is a cyclic module generated by $\mathbb{C}_k$
and $\dot{\fg}\mathbb{C}_k=0$, $V^k$ is $\dot{\Sigma}_{\pr}$-integrable. Therefore any subquotient of $V^k$
is $\dot{\Sigma}_{\pr}$-integrable and any
$\pi^{\#}$-integrable subquotient is $\pi$-integrable.
Moreover, any $\pi^{\#}$-quasi-admissible subquotient of $V^k$
is $\pi$-quasi-admissible.
\end{rem}

\subsection{The module $L'$ for $L:=L(k\Lambda_0)$}\label{6.2}
Consider the case $\fg=\dot{\fg}^{(1)}$ or $D(m+1|n)^{(2)}$.
Let  $L:=L(k\Lambda_0)$ for  $k\in\mathbb{Q}$.  
We let $\fg':=\fg_{k\Lambda_0}$
and let $\psi:\mathbb{Z}\Delta(\fg')\to\mathbb{Z}\Delta$
be the linear map
as in~$\S$~\ref{glambda}. It is easy to see that  
$\dot{\Delta}\subset \Delta_{k\Lambda_0+\rho}$ and $\fg':=\fg_{k\Lambda_0+\rho}$
is of type (Aff).
 Let $(-,-)'$ be the normalised invariant form on $\Delta(\fg')$.

\subsection{}
\begin{lem}{lem:Bmn1Dmn2etc}
Let $\fg$ be a non-twisted affine Kac-Moody superalgebra  or $D(m+1|n)^{(2)}$.
Let $k\in\mathbb{Q}$, $\fg',\psi$ be as~$\S$~\ref{6.2}.
\begin{enumerate}
\item 
One has $\fg'\cong \fg$ if $\fg\not=B(m|n)^{(1)}$ or  $D(m+1|n)^{(2)}$.
\item
If $\fg$ is  $B(m|n)^{(1)}$ or  $D(m+1|n)^{(2)}$,
then $\fg'$ is also $B(m|n)^{(1)}$ or $D(m+1|n)^{(2)}$.
\item  One has 
$(\psi(v_1),\psi(v_2))=(v_1,v_2)'$ for $v_1,v_2\in \Delta(\fg')$
and $\psi^{-1}(\dot{\Delta})$ coincides with $\dot{\Delta}'$ as introduced in Appendix~\ref{dotgclg}.
\item  For $L:=L(k\Lambda_0)$
the module $L'$ is a simple vacuum $[\fg',\fg']$-module (we denote its level by $k'$).
\end{enumerate}
\end{lem}
\begin{proof}
 We have
 $\dot{\Delta}\subset \cl(\Delta_{k\Lambda_0+\rho})\subset\cl(\Delta)$.
 Since $\fg=\dot{\fg}^{(1)}$ or $D(m+1|n)^{(2)}$, we have 
  $\cl(\Delta)=\dot{\Delta}\cup\{0\}$, so
$\cl(\Delta_{k\Lambda_0+\rho})=\dot{\Delta}$.   
Using the table in Appendix~\ref{dotgclg},
we conclude that 
  $\fg'\cong \fg$  if $\fg$ is non-twisted and 
 $\fg\not=B(m|n)^{(1)}$. If $\fg=B(m|n)^{(1)}$
or $D(m+1|n)^{(2)}$, then $\fg'$ is also $\fg=B(m|n)^{(1)}$
or $D(m+1|n)^{(2)}$.  This establishes (i), (ii)  and
the formula  $\psi^{-1}(\dot{\Delta})=\dot{\Delta}'$.
The restriction 
of $(-,-)'$ to $\cl(\Delta')\cong \dot{\Delta}$ is the normalised invariant form, so 
the restriction of $(-,-)'$ to $\mathbb{Z}\Delta(\fg')$
coincides with the restriction of the  normalised invariant form. This establishes (iii).
 For (iv)
set $\lambda:=k\Lambda_0$ and introduce $\lambda'$
as in~$\S$~\ref{glambda}. For $\alpha\in\dot{\Delta}$ we have 
$$(\lambda'+\rho',\psi^{-1}(\alpha))=(\lambda+\rho,\alpha)=(\dot{\rho},\alpha)$$
where $\dot{\rho}$ is the Weyl vector for $\dot{\Delta}^+$.
Since $\psi^{-1}(\dot{\Delta})=\dot{\Delta}'$, 
this gives $(\lambda',\alpha')=0$
for all $\alpha'\in \dot{\Delta}'$. Therefore
 $L_{\fg'}(\lambda')$ is a vacuum module.
\end{proof}

\subsection{Admissibility}\label{admlevel}
We retain notation of~$\S$~\ref{notataff}.

\subsubsection{Definition}\label{def:admissible}

We call $N\in\CO_{\Sigma}(\fg)$ {\em admissible} if $N$ is non-critical, $\pi^{\#}$-quasi-admissible,
and $\mathbb{Q} (\Delta_{N}\cap \Delta^{\#})=\mathbb{Q} \Delta^{\#}$. 
We call a  weight $\lambda$ {\em admissible} if $L(\lambda)$ 
is admissible (i.e., if $L(\lambda)$ is $\pi^{\#}$-quasi-admissible
and $\mathbb{Q}(\Delta^{\#}\cap\Delta_{\lambda+\rho})=\mathbb{Q}\Delta^{\#}$).
 The level $k$ is called
admissible if the weight $k\Lambda_0$ is admissible.
An admissible level $k$  is called 
{\em principal admissible }
 if 
 $\Delta_{k\Lambda_0+\rho}$ is isometric to $\Delta^{\ree}$, and 
 is called {\em subprincipal admissible}
otherwise.  An admissible weight
$\lambda$ is called {\em principal 
(resp., subprincipal) admissible } if 
$(\lambda,\delta)$ is a principal 
(resp., subprincipal) admissible level, and 
$\Delta^{\#}\cap\Delta_{\lambda+\rho}$ is isometric to
$\Delta^{\#}\cap \Delta_{k\Lambda_0+\rho}$.

\subsubsection{}
We claim that if $\lambda$ is a principal or a subprincipal 
admissible weight, then $\Delta_{\lambda+\rho}$
is isometric to the set of real roots of a
Kac-Moody superalgebra. 

Indeed, by~\Thm{thm:Deltalambda} (iii), (iv),
if $\Delta_{\lambda+\rho}\not\cong \Delta^{\ree}(\fg')$
for a non-critical $\lambda\in\fh^*$, then
$\fg=A(2m-1|2n-1)^{(2)}$ and $\Delta_{\lambda+\rho}$
has an indecomposable component of type  $A(1|1)^{(2)}$.
Then $\Delta^{\#}\cap\Delta_{\lambda+\rho}$  
has an indecomposable component of type 
$\{\alpha\in A(1|1)^{(2)}|\ (\alpha,\alpha)>0\}\cong A_1^{(1)}$,
so $\cl(\Delta^{\#}\cap\Delta_{\lambda+\rho})$  
has an indecomposable component of type  $A_1$.
However, if $\lambda$ is a principal or a subprincipal 
admissible weight, then $\cl(\Delta^{\#}\cap\Delta_{\lambda+\rho})$ is isometric to
$\cl(\Delta^{\#}\cap \Delta_{k\Lambda_0+\rho})$.
Without loss of generality we take $m\leq n$. Then $n>1$. 
Since $\dot{\Delta}\subset \Delta_{k\Lambda_0+\rho})$ we have
$$D_n=\dot{\Delta}^{\#}\subset 
\cl(\Delta^{\#}\cap \Delta_{k\Lambda_0+\rho})\subset \cl(\Delta^{\#})=C_n.$$
Therefore $\cl(\Delta^{\#}\cap\Delta_{\lambda+\rho})$  
does not have an indecomposable component of type  $A_1$.

In Example~\ref{A112badexa} 
$\lambda$ is an admissible weight and  $\Delta_{\lambda+\rho}$
is not isometric to the set of real roots of a
Kac-Moody superalgebra.

\subsubsection{}
\Lem{lem:Bmn1Dmn2etc} implies the following

\subsubsection{}
\begin{cor}{cor:admnotBmn} If
$\fg$ is a non-twisted affine Kac-Moody superalgebra, which is not an affine Kac-Moody
algebra, $B(m|n)^{(1)}$, $D(2|1,a)^{(1)}$,
then any admissible level  is principal admissible.
\end{cor}

\subsubsection{}
\begin{lem}{lem:admadm}
Let  $\fg$ be a  non-twisted affine Kac-Moody superalgebra.
\begin{enumerate}
\item
If $\lambda$ is admissible, then $(\lambda+\rho,\delta)\in\mathbb{Q}_{>0}$.
 \item Any (principal) admissible level for $\fg$ is a (principal)  admissible level for 
 the non-twisted affine Kac-Moody algebra $\fg_{\pi^{\#}}$.
\end{enumerate}
\end{lem}
\begin{proof}
We set 
$(\pi^{\#})':=\{\alpha'\in \Sigma'_{\pr}|\ \psi(\alpha')\in\Delta^{\#}\}$.
Then
$$(\pi^{\#})'=\{\alpha'\in \Sigma'_{\pr}|\ \psi(\alpha')\in\Delta^{\#}\}=
\{\alpha'\in \Sigma'_{\pr}|\ (\alpha',\alpha')>0\}.
$$
The condition $\mathbb{Q}\Delta^{\#}_{\lambda+\rho}=\mathbb{Q}\Delta^{\#}$
implies the existence of $q\in\mathbb{Z}_{>0}$ such that $q\delta\in \Delta^{\#}_{\lambda+\rho}$.
Then $\psi^{-1}(q\delta)$ is a positive imaginary root in $\Delta(\fg')$.
By~\Thm{thm:criterion}, $L_{\fg'}(\lambda')$ is $(\pi^{\#})'$-integrable,
so $(\lambda'+\rho',\psi^{-1}(q\delta))\in\mathbb{Z}_{\geq 0}$
 which gives $(\lambda+\rho,q\delta)\in\mathbb{Z}_{\geq 0}$.
By definition, $\lambda$ is non-critical, thus $(\lambda+\rho,\delta)\in\mathbb{Q}_{>0}$.
This proves (i).

For (ii) let $k$ be a (principal) admissible weight. 
Combining~\Thm{thm:criterionN} (iii) and~$\S$~\ref{DeltaMN} (ii), 
we conclude that $L_{\fg_{\pi^{\#}}}(k\Lambda_0)$
 is quasi-admissible (since this module is a subquotient of $\Res_{\fg_{\pi^{\#}}}^{\fg} L(k\Lambda_0)$) and 
$\{\alpha\in \Delta^{\#}|\ \langle \lambda,\alpha^{\vee}\rangle\in\mathbb{Z}\}=
\Delta^{\#}\cap R_{k\Lambda_0+\rho}$, so 
the $Q$-span of this set coincides
with $\mathbb{Q}\Delta^{\#}$. In order to prove that $k$ is a (principal)  admissible level for  $\fg_{\pi^{\#}}$,
it remains to verify that 
$k$ is not  the critical level for $\fg_{\pi^{\#}}$.
Let $\rho^{\#}$ be  the Weyl vector
for the Kac-Moody algebra $\fg_{\pi^{\#}}$.  Since
$\fg$ is non-twisted, 
$\delta$  is the minimal imaginary root of $\fg_{\pi^{\#}}$ and
the restriction of $(-,-)$
to $\Delta^{\#}$ is the  normalised invariant form. Thus
 $(\rho^{\#},\delta)=h^{\vee,\#}$ is the dual Coxeter
number for $\Delta^{\#}$. 
Using the table in $\S$~\ref{dualCoxeter} we see that 
$h^{\vee}\leq h^{\vee,\#}$ (the equality holds if and only if  $\fg$ is a Lie algebra,
that is $\fg=\fg_{\pi^{\#}}$). Therefore
$(\rho^{\#}-\rho,\delta)\geq 0$
and thus
$$(k\Lambda_0+\rho^{\#},\delta)=(k\Lambda_0+\rho,\delta)+(\rho^{\#}-\rho,\delta)>0,
$$
since, by (i), $(k\Lambda_0+\rho,\delta)>0$.
 \end{proof}

\subsection{}
\begin{prop}{prop:Vksubq}
\begin{enumerate}
\item
Let $\fg\not=D(2|1,a)^{(1)}$ be of type (Aff).
Assume that $k\in\mathbb{Z}_{\geq 0}$ is such that $k+h^{\vee}\not=0$.
Then $L(k\Lambda_0)$ is the unique $\pi^{\#}$-integrable subquotient of $V^k$.
\item Let $\fg\not=D(2|1,a)^{(1)}$ be a non-twisted affine Kac-Moody superalgebra.
If  $k$ is an admissible level, then  $L(k\Lambda_0)$ is
the unique  $\pi^{\#}$-quasi-admissible subquotient of $V^k$.
\end{enumerate}
\end{prop}
\begin{proof}
We fix a base $\Sigma$ as in  Appendix~\ref{affinebases}; then
the set
$\dot{\Sigma}=\Sigma\cap \dot{\Delta}$ is a base of $\dot{\Delta}$.

For (i) let $L(k\Lambda_0-\mu)$ be a $\pi$-integrable subquotient of the Verma module $M(k\Lambda_0)$
and $\mu\not=0$.
Then  $\langle k\Lambda_0-\mu,\alpha^{\vee}\rangle\geq 0$ for all $\alpha\in\pi$
 and $2(k\Lambda_0+\rho,\mu)=(\mu,\mu)$.
 By~\Prop{prop:Uklambda0} this gives $(k\Lambda_0-\mu)\in\mathbb{Z}\dot{\Sigma}$,
 so the $(k\Lambda_0-\mu)$-weight space in $V^k$ is zero. Hence
 $L(k\Lambda_0-\mu)$ is not a subquotient of $V^k$.

For (ii) let $L(\nu)$ be a $\pi^{\#}$-quasi-admissible subquotient of $V^k$
and $\nu\not=k\Lambda_0$. 
By~\Prop{prop:KK} (i), $\nu\in U(k\Lambda_0)$. 
Set $\lambda:=k\Lambda_0$. By~\Lem{lem:Bmn1Dmn2etc} (iv),
we have $\lambda'=k'\Lambda'_0$. Take $\nu'$ as in~$\S$~\ref{gnu}.
By~$\S$~\ref{weightnu'} we can choose
$\nu'\in U(\lambda')$. 
By~\Thm{thm:criterion}, $L_{\fg'}(\nu')$
 is $\pi'$-integrable where $\pi':=\{\alpha'\in\Sigma'_{\pr}|\ \psi(\alpha')\in\Delta_{\pi}\}$.
By~\Lem{lem:Bmn1Dmn2etc} (iii), we have
$$\pi':=\{\alpha'\in \Sigma'_{\pr}|\ (\alpha',\alpha')>0\}\cup
(\Sigma'_{\pr}\cap \dot{\Delta}').$$
Using~\Prop{prop:Uklambda0} we obtain
$\lambda'-\nu'\in\mathbb{Z}\dot{\Delta}'$.
Now~$\S$~\ref{weightnu'} gives $\lambda-\nu\in\mathbb{Z}\dot{\Delta}$, 
so the $\nu$-weight space in $V^k$ is zero  (since $\nu\not=k\Lambda_0$).
Hence $L(\nu)$ is not a subquotient of $V^k$.
Now both assertions follow from~\Rem{rem:Vkandpi}.
\end{proof}

\section{$V_k(\fg)$-modules}
In this section $\fg=\dot{\fg}^{(1)}$ is an indecomposable  non-twisted affine Kac-Moody 
superalgebra and $\fg\not=D(2|1,a)^{(1)}$.

\subsection{Notation}\label{notvertex}
 We denote by 
$V^k(\fg)$, $V_k(\fg)$ the universal  and simple
affine vertex superalgebras of level $k$ respectively.
We always assume that $k$ is non-critical.
We denote by $\CO_{\Sigma}(\fg)^k$, ${\CO}^{\inf}(\fg)^k$
 the full subcategories of $\CO_{\Sigma}(\fg)$,  ${\CO}^{\inf}(\fg)$
consisting of the modules of level $k$ (i.e., the modules
annihilated by $(K-k)$).
Any module in  ${\CO}^{\inf}(\fg)^k$  has the natural structure of $V^k(\fg)$-module.

We retain notation of~$\S\S$~\ref{newnot} and~\ref{g>0}.
Recall that $\fg_{\pi^{\#}}$ is   a subalgebra of $\fg$
generated by $\fh$ and $\fg_{\pm\alpha}$ for $\alpha\in\pi^{\#}$;
this is the direct product of 
 a  subalgebra of $\fh$ and
 a non-twisted affine Kac-Moody algebra, see~$\S$~\ref{notataff}, which we denote 
 $\fg^{\#}$. The restriction of $(-,-)$ to $\fg^{\#}$
 is the normalised non-degenerate invariant bilinear form. As in~\Lem{lem:admadm}
we denote by $\rho^{\#}$ and $(h^{\#})^{\vee}$ the Weyl vector and
the dual Coxeter number of $\fg^{\#}$ respectively.

By~\Lem{lem:admadm} any admissible level for $\fg$ is admissible for $\fg^{\#}$.

\subsection{Anisotropic case}\label{ArakawaThm}
Arakawa's Theorem states that for a non-twisted affine Kac-Moody algebra $\fg$
and an admissible level $k$,
any $V_k(\fg)$-module in the category $\CO$ is completely reducible and
irreducible modules in this category are 
admissible $\fg$-modules $L(\lambda)$ such that
 $\Delta_{\lambda+\rho}$ is isometric to $\Delta_{k\Lambda_0+\rho}$. 

In~\cite{GSsnow} this assertion was extended to
$\fg=\mathfrak{osp}(1|2n)^{(1)}$ (which is 
the only anisotropic  non-twisted affine Kac-Moody superalgebra
with $\fg_{\ol{1}}\not=0$). It was also shown that any $V_k(\fg)$-module
in the category ${\CO}^{\inf}(\fg)^k$ is completely reducible (see~\cite{GSsnow}, Corollary 5.4.1).

\subsection{Main results}
The main results of this section are the following theorems.
\subsubsection{}
\begin{thm}{thm:arak}
Let $\fg$ be an indecomposable non-twisted affine Kac-Moody superalgebra, which is not 
 of type
$D(2|1;a)^{(1)}$.  Let $k$ be an admissible level. Then a $V^k(\fg)$-module $N$
is  a $V_k(\fg)$-module if and only if 
 $\Res^{\fg}_{\fg^{\#}} N$ is a $V_k(\fg^{\#})$-module (then $k$ is
 an admissible level for $\fg^{\#}$).
\end{thm}

For the integral level $k$ the above assertion is Theorem 5.3.1 in~\cite{GSsl1n}.

\subsubsection{}
 \begin{thm}{thm:cor:arak}
Let $\fg$ and $k$ be as in~\Thm{thm:arak}, and  let $N\in\CO_{\Sigma}(\fg)^k$  be
 an indecomposable $\fg$-module. 
 If  $N$ is admissible and 
 $\Delta_{N}\cap \Delta^{\#}$ is isometric to $ \Delta_{k\Lambda_0+\rho}\cap\Delta^{\#}$, then 
 $N$ is a $V_k(\fg)$-module. The converse holds for $\fg\not=B(1|1)^{(1)}$
 and for $\fg=B(1|1)^{(1)}$ if $k$ is principal admissible.
  \end{thm}

\subsubsection{Remark}
We do not know whether the last statement holds  for $B(1|1)^{(1)}$ for
 subprincipal level $k$.

\subsubsection{}
\begin{cor}{cor:arak} Let $\fg$ and $k$ be as in~\Thm{thm:arak}.
A module $N\in {\CO}^{\inf}(\fg)^k$  is
 a $V_k(\fg)$-module if and only if all irreducible subquotients of $N$
are $V_k(\fg)$-modules.
\end{cor}
\begin{proof}
By~\Thm{thm:arak}, $N$ is a $V_k(\fg)$-module if and only if 
$\Res^{\fg}_{\fg^{\#}} N$ is a $V_k(\fg^{\#})$-module. 
If $N\in {\CO}^{\inf}(\fg)^k$, then $M:=\Res^{\fg}_{\fg^{\#}} N$
lies in ${\CO}^{\inf}(\fg^{\#})^k$ (here we have to use $\CO^{\inf}$
since the weight spaces 
in $M$ might be infinite-dimensional), so
$M$ is completely reducible (by Corollary 5.4.1 in~\cite{GSsnow}).
Thus  $N\in {\CO}^{\inf}(\fg)^k$  is
 a $V_k(\fg)$-module if and only if all irreducible subquotients of
 $M$ are $V_k(\fg^{\#})$-modules. Using~\Thm{thm:arak} again
we obtain the required statement. 
\end{proof}

\subsection{Proof of~\Thm{thm:arak}}
We denote by $V^{\#,k}$ the vacuum module for $\fg^{\#}$ and view it as a submodule
of $V^k$. We denote by $J$ and by $J^{\#}$ the maximal proper submodules of $V^k$ and $V^{\#,k}$
respectively (then
 $L(k\Lambda_0)=V^k/J$ and 
$L_{\fg^{\#}}(k\Lambda_0)=V^{\#,k}/J^{\#}$).
A $V^k(\fg)$-module $N$
is a $V_k(\fg)$-module if and only if $Y(v,z)N=0$ for all $v\in J$.
Similarly, a $V^k(\fg^{\#})$-module $N$
is a $V_k(\fg^{\#})$-module if and only if $Y(v,z)N=0$ for all $v\in J^{\#}$.
Thus it is enough to verify that the submodule $J$ is generated by $J^{\#}$,
which is established in the following lemma.

\subsubsection{}
\begin{prop}{prop:VkVk}
If $k$ is an admissible level, then $J$ is generated by $J^{\#}$.
\end{prop}
\begin{proof}
Since $V^{\#,k}$ is a submodule
of $\Res^{\fg}_{\fg^{\#}}V^k$, the module  $V^{\#,k}/( V^{\#,k}\cap J)$
is a submodule of $\Res^{\fg}_{\fg^{\#}} L(k\Lambda_0)$. Since $k$ is an admissible level, $L(k\Lambda_0)$ is admissible, so, 
by~\Thm{thm:criterionN}, any irreducible subquotient  of $V^{\#,k}/( V^{\#,k}\cap J)$
is quasi-admissible.  By~\Prop{prop:Vksubq} (ii)  this implies
$V^{\#,k}/( V^{\#,k}\cap J)=L_{\fg^{\#}}(k\Lambda_0)$. Thus
$V^{\#,k}\cap J=J^{\#}$. 

Denote by $J_1$  the $\fg$-submodule of $V^k$  generated by $J^{\#}$.
By above, $J_1\subset J$.
Let $V_1:=V^k(\fg)/J_1$ be the quotient vertex algebra. Then the $V_1$-modules are the $V^k(\fg)$-modules which are annihilated by 
$Y(v,z)$ for all $v\in J^{\#}$. Considering the adjoint representation of $V_1$
we obtain $Y(v,z) V_1=0$
for all $v\in J^{\#}$.  Then  
$M:=\Res^{\fg}_{\fg^{\#}} V_1$
 is a $V^k(\fg^{\#})$-module. By above, $Y(v,z) M=0$ for all $v\in J^{\#}$, so
 $M$ is a $V_k(\fg^{\#})$-module. Since $k$ is a $\fg^{\#}$-admissible level
  and $M\in \CO_{\Sigma}(\fg^{\#})$, 
  Arakawa's Theorem for $\fg^{\#}$  implies that $M$ 
  is an admissible $\fg^{\#}$-module. 

 Let 
 $L(\nu)$ be   an irreducible subquotient of $V^k/J_1$. 
 Let us verify that $L(\nu)$ if $\pi^{\#}$-quasi-admissible.
 By~\Thm{thm:criterionN} it is enough to check that
any irreducible subquotient of $M_1:=\Res^{\fg}_{\fg^{\#}} L(\nu)$ is
  a quasi-admissible $\fg^{\#}$-module and that $\Delta_{M_1}=\Delta_{\nu+\rho}\cap\Delta^{\#}$.
Since $M_1$ is a subquotient of $M$,  any irreducible subquotient of $M_1$ is quasi-admissible.
 By~$\S$~\ref{DeltaMN} we have
$$\Delta_{M_1}=R_{\nu+\rho}\cap\Delta^{\#}=R_{\lambda+\rho}\cap\Delta^{\#}=
\Delta_{\lambda+\rho}\cap\Delta^{\#}$$
 (the second equality follows from~(\ref{eq:Rlambda+rho})  and the last equality
 follows from~(\ref{eq:DeltaLM}) and the fact that $L(\lambda)$
is admissible ). By~\Lem{lem:KK}
  we have $\Delta_{\nu+\rho}=\Delta_{\lambda+\rho}$. Therefore
$\Delta_{M_1}=\Delta_{\nu+\rho}\cap\Delta^{\#}$.
Hence $L(\nu)$ is  quasi-admissible. By~\Prop{prop:Vksubq} (ii),
  $\nu=k\Lambda_0$. Hence $J_1=J$ as required.
\end{proof}

\subsection{Zhu algebra of $V_k(\fg)$}
Let $k$ be an admissible level.
Since $k$ is an admissible level for $\fg^{\#}$,
$J^{\#}$ is generated by a $\fg^{\#}$-singular vector $a\vac$ (where $a\in U(\fn^{\#-})$),
see~\cite{KW88}, Theorem 1, and the Zhu algebra for $V_k(\fg^{\#})$ is 
the quotient of 
$U(\dot{\fg}^{\#})$ by the principal ideal
$U(\dot{\fg}^{\#})F(a)$
where $F$ is the isomorphism given by the formula
(3.1.1) in~\cite{FZ}.
By~\Prop{prop:VkVk}, the  Zhu algebra for $V_k(\fg)$ is 
the quotient of 
$U(\dot{\fg}^{\#})$ by the principal ideal
$U(\dot{\fg})F(a)$.

\subsubsection{Remark}\label{rem:qing}
By~\cite{KW88}, Theorem 1, 
if $\dot{\fg}$ is a Lie algebra, then
$J$ is
generated by a singular vector of weight
$s_{\alpha}.k\Lambda_0$, where
$\Sigma_{k\Lambda_0+\rho}=\{\alpha\}\cup\dot{\Sigma}$.
If $k$ is a principal
admissible level, then 
 $\alpha=q\delta-\theta$, where $\theta$
 is the maximal root of $\dot{\fg}$, and
 if $k$ is a subprincipal
admissible level, then 
 $\alpha=\frac{q}{u}\delta-\theta_s$, where 
$u$ is the lacety and  $\theta_s$
 is the maximal short (=non-long) root of $\dot{\fg}$).
By~\Prop{prop:VkVk}, in general case,
if $k$ is admissible, then $J$ is generated by 
a $\fg^{\#}$-singular vector of weight $s_{\alpha}(k\Lambda_0+\rho^{\#})-\rho^{\#}$
where $\alpha\in {\Delta}^{\#}$
is such that the root subsystem
$\Delta^{\#}\cap \Delta_{k\Lambda_0+\rho}$ has the base
$\dot{\Sigma}^{\#}\cup\{\alpha\}$. A natural question is whether
$J$ is 
generated by a singular vector of weight
$s_{\gamma}.k\Lambda_0$ 
for certain choice of $\Sigma$ and some $\gamma\in \Delta^{\#}\cap \Delta_{k\Lambda_0+\rho}$?

The answer is positive if $q=1$. 
 In this case we choose $\Sigma$ containing $\gamma:=\delta-\theta$,
 where $\theta$ is the maximal root of $\dot{\fg}^{\#}$.
 Then $M(k\Lambda_0)$ contains a singular vector $v$ of weight
 $s_{\gamma}.k\Lambda_0$ and  $M(k\Lambda_0)/U(\fg)v$ is
$\fg_{\pm\gamma}$-integrable. Let $\ol{v}$ be the image
of $v$ in $V^k$. Then $V^k/U(\fg)\ol{v}$ is a quotient of
 $M(k\Lambda_0)/U(\fg)v$, so $V^k/U(\fg)\ol{v}$ is 
$\fg_{\pm\gamma}$-integrable. Since $V^k$ is 
$\dot{\fg}_{\ol{0}}$-integrable, $V^k/U(\fg)\ol{v}$ is $\fg^{\#}$-integrable. Using~\Prop{prop:Vksubq}  (i)
we conclude that $V^k/U(\fg)\ol{v}$ is simple.

 By~\cite{GSsnow}, Corollary 4.3, the answer is positive for 
 $\mathfrak{osp}(1|2\ell)^{(1)}$. Using this corollary and the description of $\Sigma_{k\Lambda_0+\rho}$ given in~$\S$~\ref{admB0n} below we obtain
that for an admissible level
 $k=-h^{\vee}+\frac{p}{q}$ one has
 $\gamma=q\delta-\theta$ if $pq$ is odd, and
$\gamma=q\delta-\theta_s$ if $pq$ is even, where
 $\theta_s$
is the maximal non-long root for $\osp(1|2\ell)$
(i.e., $\theta_s=\vareps_1+\vareps_2$ for $\ell>1$ and
$\theta_s=\vareps_1$ for $\ell=1$); in particular,
$\gamma=\alpha$ except for the case $\ell=1$ with even $pq$.
In Section~\ref{sect:sl21}
we show that the answer is positive for $\fsl(2|1)^{(1)}$
(in this case $\alpha=q\delta-\theta$).

\subsection{Proof of~\Thm{thm:cor:arak}}
Let $k$ be an admissible level and  $N\in\CO_{\Sigma}(\fg)^k$
be an indecomposable module. Consider the properties
\begin{itemize}
\item[(a)] $N$ is admissible and $\Delta^{\#}\cap R_{\lambda+\rho}$ is isometric to $\Delta^{\#}_{k\Lambda_0+\rho}$;
\item[(b)] $\Res^{\fg}_{\fg^{\#}} N$ is a $V_k(\fg^{\#})$-module.
\end{itemize}

Our goal is to prove that (a) implies (b), and that (b) implies (a)
for all cases except for $B(1|1)^{(1)}$ with  a subprincipal level $k$.

\subsubsection{}
By~\Thm{thm:criterion},  (a) is equivalent to

\begin{itemize}
\item[(a')] $\Res^{\fg}_{\fg^{\#}} N$
is  admissible, $\Delta^{\#}\cap R_{\lambda+\rho}=
\Delta^{\#}\cap \Delta_{\lambda+\rho}$, and $\Delta^{\#}\cap \Delta_{\lambda+\rho}$
is isometric to $\Delta^{\#}\cap \Delta_{k\Lambda_0+\rho}$.
\end{itemize}

Let $L(\lambda)$
be a subquotient of $N$. Then
$\Delta_{\lambda+\rho}\cap \Delta^{\#}=\Delta_{k\Lambda_0+\rho}\cap\Delta^{\#}$.
By~$\S$~\ref{DeltaMN} we have
$\{\alpha\in\Delta^{\#}|\ \langle \nu,\alpha^{\vee}\rangle\in\mathbb{Z}\}=\Delta^{\#}\cap R_{\lambda+\rho}$
for any $\nu$ such that  $L_{\fg^{\#}}(\nu)$
is a subquotient of $\Res^{\fg}_{\fg^{\#}} N$. Since $k$ is a $\fg^{\#}$-admissible level,
 Arakawa's Theorem implies that  (b) is equivalent to

\begin{itemize}
\item[(b')]  $\Res^{\fg}_{\fg^{\#}} N$
is  admissible
and $\Delta^{\#}\cap R_{\lambda+\rho}$ is isometric to
$\Delta^{\#}\cap\Delta_{k\Lambda_0+\rho}$.
\end{itemize}

Since (a') implies (b'), 
 (a) implies (b). It remains to show that (b') implies (a')
except for the case $\fg=B(1|1)^{(1)}$ when $k$ is subprincipal admissible.
It is enough to verify that in these cases
one has
\begin{equation}\label{eq:b'a'}
(\Delta^{\#}\cap R_{\lambda+\rho})\text{ is isometric to } (\Delta^{\#}\cap\Delta_{k\Lambda_0+\rho})\ \ \
\Longrightarrow\ \ \ \Delta^{\#}\cap R_{\lambda+\rho}=
\Delta^{\#}\cap \Delta_{\lambda+\rho}\end{equation}
if $\lambda$ is of the level $k$.

\subsubsection{}
\begin{lem}{lem:ii}
If $\Delta^{\#}\cap R_{\lambda+\rho}$ is isometric to $(\Delta^{\#})^{\ree}$ and $\Delta^{\#}$ contains a short even root of $\Delta$, then~(\ref{eq:b'a'}) holds.
\end{lem}
\begin{proof}
By~\Lem{lem:Xi} (i), $\Delta_{\lambda+\rho}$  and $R_{\lambda+\rho}$
contain the same short even roots. Thus
$(\Delta^{\#}\cap R_{\lambda+\rho})\subset 
(\Delta^{\#}\cap \Delta_{\lambda+\rho})$ are two affine root systems,
which share the same short roots and the smaller one is non-twisted
(since $\Delta^{\#}\cap R_{\lambda+\rho}$ is isometric to $(\Delta^{\#})^{\ree}$).
 Using the description of affine root systems
given in~\cite{Kbook}, Proposition 6.3, 
we conclude that these root systems coincide.
\end{proof}

\subsubsection{}
Let $k$ be a principal admissible level. By~\Lem{lem:admadm}, 
$k$ is  principal admissible for $\fg^{\#}$, so $\Delta^{\#}\cap R_{k\Lambda_0+\rho}$
is isometric to $(\Delta^{\#})^{\ree}$.
Thus the first assumption of the above lemma is satisfied.
From the classification of simple finite-dimensional 
Kac-Moody superalgebras~\cite{Ksuper}, it follows that 
$\Delta^{\#}$ contains a short even root of $\Delta$
if $\fg\not=B(m|n)^{(1)}$ with $n\geq m$. 

By~\Cor{cor:admnotBmn}, for $\fg\not=B(m|n)^{(1)}$  any  admissible level $k$  is  principal. Hence the above lemma establishes~(\ref{eq:b'a'}) 
in all cases except for $B(m|n)^{(1)}$ with $n\geq m$,
or $m>n$ and $k$ being subprincipal. This completes the proof of~\Thm{thm:cor:arak}
for these cases.  The remaining cases  
 will be treated  in Section~\ref{classadmBmn} (see~$\S$~\ref{endproofarak}).

\subsection{Classification of admissible levels for $\fg\not=B(m|n)^{(1)}$}\label{classadmneBmn}
We denote by $r^{\vee}$  the lacety of $\fg^{\#}$, i.e.
the ratio of square lengths of a long root and of
a short root of
 $\Delta^{\#}$. (This number coincides with the dual tier number used in~\cite{KW89}).

Let $k$ be an admissible level.  
By~\Lem{lem:admadm},  we have $k+h^{\vee}\in\mathbb{Q}_{>0}$;
we write  $k+h^{\vee}=\frac{p}{q}$ where $p,q$ are positive coprime integers.
By~\Lem{lem:Bmn1Dmn2etc} (iv), for $L=L(k\Lambda_0)$ the $\fg'$-module $L'$
is a simple vacuum module, i.e. $L'=L_{\fg'}(k'\Lambda'_0)$.

\subsubsection{Admissible levels for affine Kac-Moody algebras}
\label{admnonsuper}
This case is well-known~\cite{KW08}.
Indeed, we have $\Sigma=\dot{\Sigma}\cup\{\delta-\theta\}$ where $(\theta,\theta)=2$.
If  $\gcd(q,r^{\vee})=1$, then $\Sigma_{k\Lambda_0+\rho}=\dot{\Sigma}\cup\{q\delta-\theta\}$. If  
$\gcd(q,r^{\vee})=r^{\vee}$, then $\Sigma_{k\Lambda_0+\rho}=\dot{\Sigma}\cup\{q\delta-\theta'\}$ 
where $\theta'$ is the maximal short root; in this case 
$\Delta_{k\Lambda_0}$ is the root system of a twisted affine Kac-Moody algebra 
of type $X_{\ell}^{(r^{\vee})}$. As a result, $k$ is principal admissible if and only if
$\gcd(q,r^{\vee})=1$ and $p\geq h^{\vee}$, and $k$ is subprincipal admissible if and only if
$\gcd(q,r^{\vee})\not=1$ and $p$ is greater or equal to the dual Coxeter number
of  $X_{\ell}^{(r^{\vee})}$.

\subsubsection{}
\begin{prop}{prop:pradm}
If $\fg$  is an indecomposable affine Kac-Moody superalgebra, 
which is not of type $B(m|n)^{(1)}$, and is
not an affine Kac-Moody algebra, then
 $k$ is admissible if and
 only if $\gcd(q,r^{\vee})=1$  and 
 $p-h^{\vee}\geq 0$ (then $k'=p-h^{\vee}$).
\end{prop}
\begin{proof}
By~\Cor{cor:admnotBmn}, for $\fg$ in question  any  admissible level $k$ 
 is  principal, so, by~\Lem{lem:admadm}, $k$ is  principal admissible for $\fg^{\#}$. By~$\S$~\ref{admnonsuper}
this implies $\gcd(q,r^{\vee})=1$. 
It is easy to see that $\gcd(q,r^{\vee})=1$ implies 
$\Delta_{k\Lambda_0+\rho}=\{\mathbb{Z}q\delta+\dot{\Delta}\}$.
Therefore $\psi(\delta')=q\delta$  
(where $\psi$ is as in~$\S$~\ref{glambda}), so 
$$k':=qk-(\rho',\delta')=p-h^{\vee},$$
since $\fg\cong\fg'$.
By~\Thm{thm:criterion}, $k\Lambda_0$ is quasi-admissible if and only if
$L(k'\Lambda'_0)$ is $\pi^{\#}$-integrable, that is $k'\in\mathbb{Z}_{\geq 0}$.
\end{proof}

The remaining case $B(m|n)^{(1)}$ will be considered in the next section.

\section{Admissible levels for $B(m|n)^{(1)}$}\label{classadmBmn}
In this section we consider  $\fg:=B(m|n)^{(1)}$.
One has $\fg^{\#}=B_m^{(1)}$ if $m>n$ and $\fg^{\#}=C_n^{(1)}$
if $n\geq m$.  

Recall that  $k+h^{\vee}$ is a positive rational number, if $k$ is an admissible level,
see~\Lem{lem:admadm} (i). Fix a rational number $k$.
For $m>n$ 
we write $k+h^{\vee}=\frac{p}{q}$ where $p,q$ are non-zero coprime integers and $q>0$;
for $n\geq m$ we write $k+h^{\vee}=\frac{p}{2q}$ where $p,q$ are non-zero coprime integers and $q>0$.
For example, for $k\in\mathbb{Z}$ we have $q=1$ and $p=k+m-n-1$ for $m>n$, and
$p=2(k+n-m)+1$ for $n\geq m$.

By~\Lem{lem:Bmn1Dmn2etc} (iv), for $L=L(k\Lambda_0)$ we have 
$L'=L_{\fg'}(k'\Lambda_0)$
for some $k'$. 

In this section we will show that for $0\leq m\leq n$, $k$ is principal admissible
if and only if $p,q$ are odd and  $\frac{p}{2}\geq h^{\vee}$; in this case,
$\fg'\cong\fg$ and $k'=\frac{p}{2}-h^{\vee}$.
For $m>n>0$, $k$ is principal admissible
if and only if $q$ is odd and  $p\geq h^{\vee}$; in this case,
$\fg'\cong\fg$ and $k'=p-h^{\vee}$.

Furthermore, the subprincipal admissible levels exist only for the cases
$m=0$ when $p$ or $q$ is even,  $m=n=1$ when $p$ is even, and $m>n$
when $q$ is even. In all these cases $k$ is a subprincipal admissible level
if and only if $\frac{p}{2}\geq (h')^{\vee}$ and $p\not=0$, where
$(h')^{\vee}$ is the dual Coxeter number for 
$\fg'$, which is of type $A(0|2n-1)^{(2)}$ if $m=0$ and of type $D(m+1|n)^{(2)}$ if $m>0$.
In all these cases $k'=\frac{p}{2}-(h')^{\vee}$.

\subsection{Case $B(0|\ell)^{(1)}$}\label{admB0n}
In this case $\fg^{\#}=\fg_{\ol{0}}\cong C_{\ell}^{(1)}$.
We have $h^{\vee}=\ell+\frac{1}{2}$ and $\Sigma=\dot{\Sigma}\cup\{\delta-\theta\}$
where $\theta=2\delta_1$. Recall that $(\theta,\theta)=2$.

 For $\alpha\in\dot{\Delta}$
 we have 
 $$(i\delta+\alpha)\in \ol{R}_{k\Lambda_0+\rho}\ \ \Longleftrightarrow\ \ 
 \left\{\begin{array}{lcl}
\frac{ip}{q}\in\mathbb{Z}, &  & (\alpha,\alpha)=\frac{1}{2}, 1\\
\frac{ip}{q}\in 2\mathbb{Z}+1, & &
(\alpha,\alpha)=2,\  \ \text{and $i$ is odd}
\end{array}\right.$$
(the condition $i$ is odd in the last line reflects the fact that $i\delta+\alpha\not\in\ol{\Delta}$ if $i$ is even). This gives
$$\Sigma_{k\Lambda_0+\rho}=\left\{
\begin{array}{lcl}
\dot{\Sigma}\cup\{q\delta-\theta\} & & \text{ if $p,q$ are odd},\\
\dot{\Sigma}\cup\{q\delta-\theta_s\} & & \text{otherwise},
\end{array}\right.$$
where $\theta_s$ is the maximal among non-long roots:
$\theta_s=\delta_1$ for $\ell=1$ and $\theta_s=\delta_1+\delta_2$ for $\ell>1$.
Therefore $\psi(\delta')=q\delta$, so 
$$k'=(\lambda'+\rho',\delta')-(\rho',\delta')=(\lambda+\rho,q\delta)-(h')^{\vee}=\frac{p}{2}
-(h')^{\vee},$$
where $(h')^{\vee}$ is the dual Coxeter number of $\fg'$.
This gives
\begin{itemize}
\item[$\bullet$]
If $pq$ is odd, then $\fg'$ is of type $B(0|\ell)^{(1)}$ and 
$k'=\frac{p}{2}-h^{\vee}$.
\item[$\bullet$] If  $pq$ is even, then  $\fg'$ is of type $A(0|2\ell-1)^{(2)}$
and  $k'=\frac{p}{2}-(h')^{\vee}$ with $(h')^{\vee}=\ell-\frac{1}{2}$ 
 (note that $A(0|1)^{(2)}=D(1|1)^{(2)}$ since $\fsl(1|2)=\osp(2|2)$).
\end{itemize}

By~\Thm{thm:criterion}, $k$ is admissible if and only if $L_{\fg'}(k'\Lambda_0)$
is integrable. If $pq$ is odd, then $\fg'\cong\fg$, so $L_{\fg'}(k'\Lambda_0)$
is integrable if and only if $k'\in\mathbb{Z}_{\geq 0}$. If $pq$ is even,
then $\fg'$ is of type $A(0|2\ell-1)^{(2)}$ with the  normalised invariant bilinear form
(the square length of the long root is $2$). The set of simple roots for $A(0|1)^{(2)}$ 
is $\{\delta-\delta_1,\delta_1\}$ (both roots are odd) and $L_{\fg'}(k'\Lambda_0)$
is integrable if and only if $2k'=p- 2(h')^{\vee}\in \mathbb{Z}_{\geq 0}$ i.e.
$p\geq 2(h')^{\vee}=1$. For $\ell>1$ the 
set of simple roots for $A(0|2\ell-1)^{(2)}$ 
is $\{\delta-\delta_1-\delta_2,\delta_1-\delta_2,\ldots,
\delta_n\}$ (the last root is odd) and $L_{\fg'}(k'\Lambda_0)$
is integrable if and only if $k'=\frac{p}{2}-(h')^{\vee}\in \mathbb{Z}_{\geq 0}$
which means that $p$ is odd and $p\geq 2(h')^{\vee}$.
We conclude that
\begin{itemize}
\item[$\bullet$] $k$ is principal admissible if and only if
$p,q$ are odd 
and $p\geq 2h^{\vee}=2\ell+1$;
\item[$\bullet$] $k$ is subprincipal admissible if and only if  $p\geq 2\ell-1$, 
and, in addition, $pq$ is even for $\ell=1$, and  $q$ is even for $\ell>1$; 
if $k$ subprincipal admissible, then $\fg'\cong A(0|2\ell-1)^{(2)}$.
\end{itemize}

\subsubsection{}
Let $\lambda$ be an admissible weight of $B(0|\ell)^{(1)}$   satisfying 
$\cl(\Delta_{\lambda+\rho})=\dot{\Delta}_{\ol{0}}=C_{\ell}$. Retain notation of~$\S$~\ref{gnu}.
By~\cite{K78}, $\fg'$ is one of the following affine  Kac-Moody superalgebras: 
$C_{\ell}^{(1)}$, $A_{2\ell-1}^{(2)}$, $B(0|\ell)^{(1)}$, $A(0|2\ell-1)^{(2)}$,
or $D(1|\ell)^{(2)}=C(\ell+1)^{(2)}$
(notice that $D(1|1)^{(2)}=C(2)^{(2)}=A(0|1)^{(2)}$).

Recall that $\alpha\in\Delta^{\ree}$ is even if $(\alpha,\alpha)=1,2$,
and is odd if $(\alpha,\alpha)=1/2$. If $(\alpha,\alpha)=1, 1/2$,
then $\alpha\in \Delta_{\lambda+\rho}$ is equivalent to $\alpha\in\ol{R}_{\lambda+\rho}$,
so $\alpha,\alpha+j\delta\in \Delta_{\lambda+\rho}$ is equivalent to $\frac{jp}{q}\in\mathbb{Z}$.
This gives
\begin{equation}\label{eq:B0l>1}
(\alpha+\mathbb{Z}\delta)\cap \Delta_{\lambda+\rho}=\alpha+\mathbb{Z}q\delta
 \ \text{ if }\alpha\in \Delta_{\lambda+\rho}\ \text{ is s.t. } (\alpha,\alpha)=1,
1/2.
\end{equation}

Since $\cl(\Delta_{\lambda+\rho})=\dot{\Delta}_{\ol{0}}\ $,
$\cl(\Delta(\fg')_{\ol{0}})=C_{\ell}$ and the square length of the longest root
in $\Delta(\fg')$ is $2$.
Let $\delta'$ be the minimal imaginary root of $\Delta(\fg')$. 
Recall that $\psi(\delta')=u\delta$ for some $u\in\mathbb{Z}_{>0}$.
By~$\S$~\ref{alphas+delta} if 
 $\beta$ is a short root of $\Delta(\fg')$, then $(\beta+j\delta')\in\Delta(\fg')$
is equivalent to $j\in\mathbb{Z}$, so
$$\{\psi(\beta)+\mathbb{Z}\delta\}\cap \Delta_{\lambda+\rho}=\psi(\beta)+\mathbb{Z}\psi(\delta').$$

If $\fg'\not\cong C_1^{(1)}$, then $(\beta,\beta)\not=2$, so (\ref{eq:B0l>1}) implies
 $\psi(\delta')=q\delta$. Hence 
\begin{equation}\label{eq:B0ndelta'}
\psi(\delta')=q\delta\ \text{ for } \fg'\not\cong C_1^{(1)}.
\end{equation}
Therefore $(\lambda',\delta')=(\lambda'+\rho',\delta')-(h')^{\vee} =(\lambda+\rho,\delta)
 -(h')^{\vee}$ where  $(h')^{\vee}$ is the dual Coxeter number of $\fg'$. 
 Hence 
\begin{equation}\label{eq:B0nk'}
(\lambda',\delta')=\frac{p}{2}-(h')^{\vee} \text{ for } \fg'\not\cong C_1^{(1)}.
 \end{equation}

If  $\fg'\cong C_1^{(1)}$, then $\ell=1$. For
$\alpha\in\Delta_{\lambda+\rho}$ one has  $(\alpha,\alpha)=2$
and $\alpha/2\not\in\Delta$, so $(\lambda+\rho,\alpha)\in\mathbb{Z}$.
Therefore $(\alpha+j\delta)\in \Delta_{\lambda+\rho}$ is equivalent to 
$s(k+h^{\vee})\in\mathbb{Z}$ that is $s\in 2\mathbb{Z}q$. Therefore
\begin{equation}\label{eq:B01delta'}
\psi(\delta')=2q\delta,\ \ (\lambda',\delta')=p-2
 \ \text{ for } \fg'\cong C_1^{(1)}.
\end{equation}

 Let us show that $\fg'\not\cong D(1|\ell)^{(2)}$ for $\ell>1$. Suppose the contrary.
 Then $\Delta(\fg')_{\ol{0}}$ contains $\alpha$ such that $(\alpha,\alpha)=1$
 and $\alpha+\delta'\not\in \Delta(\fg')$. Then $\psi(\alpha)\in \Delta_{\lambda+\rho}$
 and $\psi(\alpha+\delta')=\psi(\alpha)+q\delta\not \in \Delta_{\lambda+\rho}$
 which contradicts to (\ref{eq:B0l>1}). Hence 
\begin{equation}\label{eq:B0nDeltalambda}
\fg'\cong C_{\ell}^{(1)},\ B(0|\ell)^{(1)}, \ A_{2\ell-1}^{(2)}, \text{or } A(0|2\ell-1)^{(2)}.
\end{equation}

\subsubsection{}
\begin{cor}{cor:VkforB0n}
Let $\fg=B(0|\ell)^{(1)}$,
$k$ be an admissible level, and  $L(\lambda)$ be a $V_k(\fg)$-module. 
\begin{enumerate}
\item If $k$ is principal admissible, then 
 $\Delta_{\lambda+\rho}\cong \Delta_{k\Lambda_0+\rho}$ and $\fg'\cong B(0|\ell)^{(1)}$.
 In this case $L_{\fg'}(\lambda')$ has  level $\frac{p}{2}-h^{\vee}$.
\item Let $k$ be a  subprincipal admissible level. Then 
either $\fg'\cong A(0|2\ell-1)^{(2)}$,
or $\fg'\cong A_{2\ell-1}^{(2)}$ for $\ell>1$,
or  $\fg'\cong C_{1}^{(1)}$ for $\ell=1$.
If  $\fg'\not\cong C_{1}^{(1)}$, then 
$L_{\fg'}(\lambda')$ has  level $\frac{p}{2}-(h')^{\vee}$
 where $(h')^{\vee}$ is
 the dual Coxeter number of  $\fg'$.
 If $\fg'\cong C_{1}^{(1)}$, then $L_{\fg'}(\lambda')$ has  level 
 $p-2$.
\end{enumerate}
\end{cor}
\begin{proof}
By Arakawa's Theorem (established in~\cite{GSsnow} for this case),
$\Delta_{\lambda+\rho}\cap \Delta_{\ol{0}}$ is isometric to
$\Delta_{k\Lambda_0+\rho}\cap \Delta_{\ol{0}}$.

Consider the case $\ell>1$ and 
 $q$ is even.   Then, by~$\S$~\ref{admB0n}, $\fg'_{\ol{0}}$ is of type  $A_{2\ell-1}^{(2)}$, 
 and, by~\cite{K78},  $\fg'$ is $A(0|2\ell-1)^{(2)}$
or $A_{2\ell-1}^{(2)}$.

Consider the case when $\ell>1$ and $q$ is odd.
By~(\ref{eq:B0ndelta'}),
$\psi$ maps the minimal imaginary root of $\Delta'_{\ol{0}}\cong C_{\ell}^{(1)}$  to $q\delta$.
Hence
$$(\gamma+\mathbb{Z}\delta)\cap \Delta_{\lambda+\rho}=\gamma+\mathbb{Z}q\delta \ \text{ for all }
\gamma\in \Delta_{\lambda+\rho}\cap \Delta_{\ol{0}}.$$
Take $\gamma\in \Delta_{\lambda+\rho}$ such that $(\gamma,\gamma)=2$,
that is $\gamma=s\delta\pm 2\delta_i$. By above, $(s+q)\delta\pm 2\delta_i$
lies in $\Delta_{\lambda+\rho}$. Since $s$ or $s+q$ is even, 
$\Delta_{\lambda+\rho}$ contains an element of the form $2\beta$
where $\beta$ is odd. Thus  $\beta\in \Delta_{\lambda+\rho}$, so $\fg'$ contains an odd root.
By~$\S$~\ref{admB0n}, $\fg'_{\ol{0}}$ is of type $C_{\ell}^{(1)}$.
Hence, by~\cite{K78},  $\fg'\cong B(0|\ell)^{(1)}$.

Now  consider the case $\ell=1$. Then  $\fg'_{\ol{0}}\cong C_1^{(1)}$.
We have the following possibilities
\begin{itemize}
\item[(a)] $\fg'\cong B(0|1)^{(1)}$ with 
$\Sigma_{\lambda}=\{s_1\delta\pm 2\delta_1,s_2\delta\mp\delta_1\}$, where $s_1$ is odd;
\item[(b)] $\fg'\cong A(0|1)^{(2)}$ with
$\Sigma_{\lambda}=\{s_1\delta-\delta_1,s_2\delta+\delta_1\}$; 
\item[(c)] $\fg'\cong C_1^{(1)}$
with $\Sigma_{\lambda}=\{s_1\delta-2\delta_1,s_2\delta+2\delta_1\}$, where $s_1,s_2$ are odd.
\end{itemize}

If $\fg'\cong B(0|1)^{(1)}$, then
 $(\lambda+\rho,s_1\delta\pm 2\delta_1)\in\mathbb{Z}$
and $4(\lambda+\rho,s_2\delta\mp \delta_1)\in 2\mathbb{Z}+1$, so
$\frac{(s_1+2s_2)p}{q}\in 2\mathbb{Z}+1$.
Thus $p,q$ are odd.

If $\fg'\cong A(0|1)^{(2)}$, then 
$(s\delta+\delta_1)\in\Delta_{\lambda+\rho}$ for some $s$
and $((2s+2i+1)\delta+2\delta_1)\not\in \Delta_{\lambda+\rho}$
for all  $i$. Thus  $\frac{p(2i+1)}{q}\not\in 2\mathbb{Z}+1$ for all $i$, so $pq$ 
is even.

If $\fg'\cong C_1^{(1)}$, then
  $(s\delta+2\delta_1)\in\Delta_{\lambda+\rho}$ for some odd $s$
and $(i\delta+\delta_1)\not\in \Delta_{\lambda+\rho}$
for all  $i$. This implies $\frac{jp}{q}\not\in 2\mathbb{Z}+1$
for all odd $j$, so $pq$ is even. 

The rest of the assertions follow from~(\ref{eq:B01delta'}) 
and~(\ref{eq:B0nk'}).
\end{proof}

\subsection{Admissible levels for $\fg:=B(m|n)^{(1)}$ with $m,n\geq 1$}
Let $k$ be as in the beginning of Section~\ref{classadmBmn}.
By~\Lem{lem:Bmn1Dmn2etc} (ii), $\fg'$ is either $B(m|n)^{(1)}$
or $D(m+1|n)^{(2)}$. Thus, if $k$ is admissible, then
$k$ is principal if $\fg'$ is $B(m|n)^{(1)}$, and
is subprincipal if $\fg'$ is $D(m+1|n)^{(2)}$.

We denote by $\delta'$ the minimal imaginary root of $\fg'$; recall that
$\psi(\delta')=u\delta$ for some $u\in\mathbb{Z}_{>0}$. One has
$(k'\Lambda'_0+\rho',\delta')=(k\Lambda_0+\rho,\psi(\delta))$, so
$$k'=\frac{up}{2q}-(h')^{\vee} \text{ if } n\geq m, \ \ \ \ 
k'=\frac{up}{q}-(h')^{\vee} \text{ if } n<m,$$
where $(h')^{\vee}:=(\rho',\delta')$ is the dual Coxeter number of $\fg'$.

Note that the lacety $r^{\vee}$ of $\fg^{\#}$ is equal to $1$ if $m=n=1$
and is equal to $2$ otherwise.

\subsubsection{}\label{iran}
By~\Lem{lem:Bmn1Dmn2etc} (iii),  the  bilinear form on $\Delta'$
is the normalised invariant form.
If $\fg'\cong B(m|n)^{(1)}$, then $L_{\fg'}(k'\Lambda_0)$
is $(\pi')^{\#}$-integrable if and only if $k'\in\mathbb{Z}_{\geq 0}$.
If $\fg'\cong D(m+1|n)^{(2)}$, then $L_{\fg'}(k'\Lambda_0)$
is $(\pi')^{\#}$-integrable if and only if $2k'\in\mathbb{Z}_{\geq 0}$.

\subsubsection{}
\begin{prop}{prop:Bmnlevels}
Let $\fg=B(m|n)^{(1)}$ with $m,n\geq 1$. Recall that $h^{\vee}=n-m+\frac{1}{2}$, $(h')^{\vee}=
\frac{n-m}{2}$ if
$n\geq m$, and $h^{\vee}=2m-2n-1$, $(h')^{\vee}=m-n$ if $m>n$.

\begin{enumerate}

\item Let $m=n=1$. 
If $q$ is even, then $k$ is not admissible.
If $p,q$ are odd, then $\fg'\cong B(1|1)^{(1)}$
and $k'=\frac{p}{2}-h^{\vee}$, so $k+h^{\vee}=\frac{k'+h^{\vee}}{q}$.
If $p$ is even, then $\fg'\cong D(2|1)^{(2)}$
and $k'=\frac{p}{2}-(h')^{\vee}=\frac{p}{2}$, so $k+h^{\vee}=\frac{k'}{q}$.
In both cases 
$k$ is admissible if and only if $k'\geq 0$.

\item Let $1\leq n<m$. 
If $q$ is odd, then $\fg'\cong B(m|n)^{(1)}$, and
$k$ is admissible (and principal) if and only if 
$k'=p-h^{\vee}\geq 0$. 
If  $q$ is even, then $\fg'\cong D(m+1|n)^{(2)}$, and
$k$ is admissible (and subprincipal) if and only if 
$k'=\frac{p}{2}-(h')^{\vee}\geq 0$
(notice that $k'\in\mathbb{Z}+\frac{1}{2}$).

\item If  $1\leq m\leq n\not=1$, then $k$ is not admissible except for the case when 
$p$ and $q$ are odd. If $p,q$ are odd, then  $\fg'\cong B(m|n)^{(1)}$, $k'=\frac{p}{2}-h^{\vee}$,
and 
$k$ is admissible 
(and principal) if and only if $k'\geq 0$. In this case $k+h^{\vee}=\frac{k'+h^{\vee}}{q}$.
\end{enumerate}
\end{prop}

\subsubsection{}
\begin{cor}{}
\begin{enumerate}
\item
The level  $k$ is principal admissible for $\fg=B(m|n)^{(1)}$ with $m\geq 0$, $n\geq 1$
if and only if $p,q$ are odd and  $\frac{p}{2}\geq h^{\vee}$
for  $m\leq n$, and  $q$ is odd and  $p\geq h^{\vee}$
for $m>n$.
\item
The level $k$ is subprincipal admissible for $\fg=B(m|n)^{(1)}$ with   $m\geq 0$, $n\geq 1$
in the following cases: $\frac{p}{2}\geq (h')^{\vee}$, $p\not=0$ and 
\begin{itemize}
\item[$\cdot$] $m=0$, $n=1$ when $pq$ is even;
\item[$\cdot$]  $m>n\geq 1$ or $m=0$, $n>1$ when
$q$ is even;
\item[$\cdot$] $m=n=1$ when
$p$ is even for $m=n=1$. 
\end{itemize}
One has $\fg'=D(m+1|n)^{(2)}$ (so
$(h')^{\vee}=n-\frac{1}{2}$
for $m=0$, $(h')^{\vee}=m-n$ for $m>n$ and $m=n=1$).\end{enumerate}
\end{cor}

\subsubsection{Proof of~\Prop{prop:Bmnlevels}}
We retain the usual notation for the roots in $B(m|n)$.
By abuse of notation, we use the same notation for  the set of real roots of affine Kac-Moody superalgebra
$\fg$ and $\fg$ itself (for example, $\Delta_{k\Lambda_0+\rho}\cong B(m|n)^{(1)}$ means
that $\Delta_{k\Lambda_0+\rho}$ is isomorphic to the set of real roots of $B(m|n)^{(1)}$).
It is convenient to use the normalization $(\vareps_1,\vareps_1)=1$
(which coincides with the original normalization if $m>n$). Then 
$(k\Lambda_0+\rho,\delta)=\frac{p_1}{q}$ where $p_1=p$ for $m>n$
and $p_1=-p$ if $n\geq m$. 

By~\Thm{thm:criterion}, if $k$ is admissible, then it satisfies
\begin{equation}\label{eq:iran}
\Delta_{k\Lambda_0+\rho}\cap\Delta^{\#}\subset {R}_{k\Lambda_0+\rho}.
\end{equation}
In the light of~\Thm{thm:criterion} and~$\S$~\ref{iran},
it is enough to verify  that for 
$u:=\frac{q}{\gcd(2,q)}$ (i.e., $u=q$ is $q$ is odd and $u=q/2$
otherwise) we have:
\begin{itemize}
\item[(a)]
if $n\geq m$ and $p,q$ are odd, or $m>n$
and $q$ is odd, then~(\ref{eq:iran}) holds and 
\begin{equation}\label{eq:Bmnq}
\Delta_{k\Lambda_0+\rho}=\dot{\Delta}+\mathbb{Z}u\delta\cong B(m|n)^{(1)},
\end{equation}
\item[(b)] if $m=n=1$ and $p$ is even, or $m>n$
and $q$ is even,  then~(\ref{eq:iran}) holds  and 
\begin{equation}\label{eq:Dmn2q}
\Delta_{k\Lambda_0+\rho}=\{\dot{\Delta}+2\mathbb{Z}u\}\cup\{\alpha+\mathbb{Z}u\delta|\ \alpha\in\dot{\Delta},
\ (\alpha,\alpha)=\pm 1\}\cong D(m+1|n)^{(2)},
\end{equation}
\item[(c)] formula~(\ref{eq:iran}) 
does not hold in all other cases, i.e. for $m=n=1$ if $q$ is even,
and for $n\geq m$, $n\not=1$ if $p$ or $q$ are even.
\end{itemize}

For $\alpha\in\dot{\Delta}$
we have 
$$(\alpha+i\delta)\in \ol{R}_{k\Lambda_0+\rho}\ \ \Longrightarrow\ \ 
\left\{\begin{array}{lcl}
i=0 & & \text{ if } (\alpha,\alpha)=0\\
\frac{pi}{q}\in \mathbb{Z}  & & \text{ if } (\alpha,\alpha)=\pm 2,-1\\
\frac{2pi}{q}\in \mathbb{Z}  & & \text{ if } (\alpha,\alpha)=1,\\
i,\frac{pi}{q}\in  2\mathbb{Z}+1 & & \text{ if } (\alpha,\alpha)=-4.
\end{array}
\right.$$

If $p,q$ are odd, then 
$\Delta_{k\Lambda_0+\rho}=\dot{\Delta}+\mathbb{Z}q\delta\cong B(m|n)^{(1)}$
and $\Delta_{k\Lambda_0+\rho}\cap \ol{\Delta}^{\an}=\ol{R}_{k\Lambda_0+\rho}$, so~(\ref{eq:iran}) 
 holds.  This proves (a).

In our normalization the short even roots (i.e., $s\delta\pm\vareps_i$)
are of  square length $1$.
By~\Lem{lem:Xi} (i), the root
$si\delta\pm\vareps_i)$ lies in $\Delta_{k\Lambda_0+\rho}$
if and only if it lies in $\ol{R}_{k\Lambda_0+\rho}$, that is
$(s\delta\pm\vareps_i)\in \Delta_{k\Lambda_0+\rho}$ if and only if $2s$ is divisible by 
 $q$.
Using~\Lem{lem:Bmn1Dmn2etc} (ii), we conclude that $\Delta_{k\Lambda_0+\rho}$
satisfies one of the formulas~(\ref{eq:Bmnq}), (\ref{eq:Dmn2q}) for  
$u=\frac{q}{\gcd(2,q)}$.

If $q$ is even, then  the set given by  formula~(\ref{eq:Dmn2q}) contains 
$\ol{R}_{k\Lambda+\rho}$, so $\Delta_{k\Lambda_0+\rho}$ is 
as in~(\ref{eq:Dmn2q}) with $u=\frac{q}{2}$.
In this case $\fg'$
is of type $D(m+1|n)^{(2)}$. For $m>n$  formula~(\ref{eq:iran}) 
 holds. 
If $n\geq m$, then $(q\delta+2\delta_1)\in \Delta_{k\Lambda_0+\rho}\setminus R_{k\Lambda_0+\rho}$,
so~(\ref{eq:iran}) 
does not hold. 

Now consider the remaining case when $p$ is even (and  $q$ is odd).
For $B(1|1)$ we have 
$$\Delta_{k\Lambda_0+\rho}=\{2\mathbb{Z}q\delta\pm \vareps_1\pm\delta_1,
\mathbb{Z}q\delta\pm\vareps_1, 
\mathbb{Z}q\delta\pm\delta_1, 2\mathbb{Z}q\delta\pm 2\delta_1 \}\cong D(2|1)^{(2)},$$
so $\Delta_{k\Lambda_0+\rho}\cap\Delta^{\#}=
\{2\mathbb{Z}q\delta\pm 2\delta_1\}=
{R}_{k\Lambda_0+\rho}\cap\Delta^{\#}$ and~(\ref{eq:iran}) 
 holds.

For $(m,n)\not=(1,1)$ the set
  $\dot{\Delta}$
contains $\alpha$ such that $(\alpha,\alpha)\in \{\pm 2\}$ and
$\alpha+q\delta\in \ol{R}_{\lambda+\rho}$. Thus
$\Delta_{k\Lambda_0+\rho}\cong B(m|n)^{(1)}$ is given by  formula~(\ref{eq:Bmnq}).
 If $m>n$, 
$${R}_{k\Lambda_0+\rho}\cap\Delta^{\#}=\{\mathbb{Z}q\delta\pm \dot{\Delta}^{\#} \}
=\Delta_{k\Lambda_0+\rho}\cap\Delta^{\#},
$$
so~(\ref{eq:iran}) 
 holds. 
If $n\geq m$, then  $q\delta+2\delta_1\in\Delta_{k\Lambda_0+\rho}\setminus
{R}_{k\Lambda_0+\rho}$, so~(\ref{eq:iran}) 
does not hold. This completes the proof.
\qed

\subsection{Completion of the proof of~\Thm{thm:cor:arak}}\label{endproofarak}
 It remains to verify~(\ref{eq:b'a'}) for 
$B(m|n)^{(1)}$ with $n\geq m$ and $k$ being prinicpal admissible, 
and  for $m>n$ and $k$ being subprincipal admissible. 
The  first case is treated in the following lemma.

\subsubsection{}
\begin{lem}{lem:Bngeqm0}
Let $\fg=B(m|n)^{(1)}$ with $n\geq m$ and $k$ such that $k+h^{\vee}=\frac{p}{2q}$
where $p,q$ are odd. Let $\lambda$ be a weight of level $k$
such that $R_{\lambda+\rho}\cap\Delta^{\#}$ is 
isometric to $(\Delta^{\#})^{\ree}$. Then 
$\Delta_{\lambda+\rho}\cap\Delta^{\#}=R_{\lambda+\rho}\cap\Delta^{\#}$.
\end{lem}
\begin{proof}
We will use  the normalization $(\delta_1,\delta_1)=1$.
For this normalization
$(k\Lambda_0+\rho,\delta)=\frac{p}{q}$ 
and all roots are of integral square length, so
$2(\lambda+\rho,\gamma)$ is integral for all  $\gamma\in\Delta_{\lambda+\rho}$.

We will use the following observation: 
$(s\delta+2\delta_i)\in  R_{\lambda+\rho}$ if and only if
$(\lambda+\rho,s\delta+2\delta_i)$
is odd if $s$ is even, and $(\lambda+\rho,s\delta+2\delta_i)$
is even if $s$ is odd.

The assumption that 
$R_{\lambda+\rho}\cap\Delta^{\#}$ is isometric to $(\Delta^{\#})^{\ree}$
implies that for each $i=1,\ldots,n$
there exists $s$ such that $(s\delta+2\delta_i)\in  R_{\lambda+\rho}$.
If $s$ is odd, then the above observation gives 
$$(s\delta+2\delta_i)\in  R_{\lambda+\rho}\ \ \Longleftrightarrow\ \ 
(\lambda+\rho,s\delta+2\delta_i)\in 2\mathbb{Z}\ \Longleftrightarrow\ \ 
(\lambda+\rho,(s+q)\delta+2\delta_i)\in 2\mathbb{Z}+1$$
since $p$ is odd.  Thus $(s\delta+2\delta_i)\in  R_{\lambda+\rho}$
for some even $s$. Therefore
for each $i=1,\ldots,n$ there exists
$p_i$ such that $(2p_i\delta+2\delta_i)\in  R_{\lambda+\rho}$,
which means that 
$\beta_i:=p_i\delta+\delta_i$ lies in $\ol{R}_{\lambda+\rho}$, that is 
$$a_i:=(\lambda+\rho, \beta_i)\in \mathbb{Z}+\frac{1}{2}.$$

Take $\gamma\in (\Delta_{\lambda+\rho}\cap \Delta^{\#})$ and set
$b:=(\lambda+\rho,\gamma)$. 
Since $\gamma\in\Delta^{\#}$ we have $(\gamma,\gamma)$ is $2$ or $4$, and
$\gamma=s\delta\pm\beta_i\pm\beta_j$
for some  $s\in\mathbb{Z}$, $1\leq i,j\leq n$, so 
$$b=(\pm a_i\pm a_j)+\frac{sp}{q}.
$$
Since $\gamma\in\Delta_{\lambda+\rho}$, we have $2b\in\mathbb{Z}$.
Since $(\pm a_i\pm a_j)$ is integral, $\frac{2sp}{q}$ is integral, so 
$s$ is divisible by $q$ (since $q$ is odd). Therefore 
$b\in \mathbb{Z}$.
Hence $\gamma\in R_{\lambda+\rho}$ if $(\gamma,\gamma)=2$.
If   $(\gamma,\gamma)=4$, then 
$i=j$ and 
$b=\pm 2a_i+\frac{sp}{q}$. Thus $b$
is even if $s$ is odd and is odd if $s$
is even (since $p$ is odd).  Using the above observation we obtain
 $\gamma\in R_{\lambda+\rho}$ as required.
 \end{proof}

\subsubsection{}
It remains to verify~(\ref{eq:b'a'}) for 
$B(m|n)^{(1)}$ with $m>n$ and $k$ being subprincipal admissible. 
In the light of~\Prop{prop:Bmnlevels}
this corresponds to the  case when
 $q$ is even.

 \subsubsection{}\label{grouptildeT}
 Set
$P:=\{\nu\in \mathbb{C}\dot{\Delta}|\ (\nu,\alpha)\in\mathbb{Z}\ \ 
\text{ for all }\alpha\in\dot{\Delta}\}$.
We consider the map $t: P\to  GL(\fh^*)$
given by
\begin{equation}\label{tnu}
t_{\nu}(\mu):=\mu+(\mu,\delta)\nu-((\mu,\nu)+\frac{(\mu,\delta)}{2}
(\nu,\nu))\delta.\end{equation}
The map
 $t$ is a group monomorphism; we denote by $\tilde{T}$ the image 
of this map. Then $\tilde{T}\in O(\fh^*)$ (i.e., $t_{\nu}$
preserves the bilinear form $(-,-)$ on $\fh^*$). Note that for $\gamma\in\Delta$ we have
$$t_{\nu}(\gamma)=\gamma-(\gamma,\nu)\delta.$$
In particular, $t_{\nu}(\delta)=\delta$ and 
 and $t_{\nu}(\Delta)=\Delta$. We denote by $\tilde{W}$ the subgroup of
 $O(\fh^*)$ generated by  $\dot{W}$ and $\tilde{T}$
 (by~\cite{Kbook}, Proposition 6.5, $\tilde{W}=\dot{W}\ltimes \tilde{T}$, but we won't use this fact). For each $w\in\tilde{W}$ we have 
 \begin{equation}\label{eq:tildew}
 w\Delta=\Delta,\ \ w\delta=\delta,\ \ R_{w(\lambda+\rho)}=w R_{\lambda+\rho},\ \ \ \Delta_{w(\lambda+\rho)}=w\Delta_{\lambda+\rho}\ \text{ for }\lambda\in\fh^*.\end{equation}

\subsubsection{}
Recall that $(\vareps_1,\vareps_1)=1$ for $m>n$. 
For this normalization $P=\mathbb{Z}\Delta$ ,
and all roots are of integral square length, so
$2(\lambda+\rho,\gamma)$ is integral for all  $\gamma\in\Delta_{\lambda+\rho}$.

\subsubsection{}
We want to prove that for $m>n$ we have
$$\Delta^{\#}\cap R_{\lambda+\rho}\ \text{ is isometric to }\ \Delta^{\#}\cap \Delta_{k\Lambda_0+\rho}\ \ \Longrightarrow\ \ 
\Delta^{\#}\cap R_{\lambda+\rho}=\Delta^{\#}\cap \Delta_{\lambda+\rho}$$
if $\lambda$ is of level $k$, where $k+h^{\vee}=\frac{p}{q}$ and $q$ is even.
Using~(\ref{eq:tildew}) we conclude that
the above implication for $\lambda$ is equivalent to the similar implication for 
the weight $w(\lambda+\rho)-\rho$ for $w\in \tilde{W}$.

\subsubsection{}
We set 
$$\begin{array}{l}
I_{\vareps}:=\{i|\ \vareps_i\in\cl(R_{\lambda+\rho})\}=
\{i|\ \exists s\in \mathbb{Z}\ \text{ s. t.}\ 2(\lambda+\rho,s\delta+\vareps_i)\in\mathbb{Z}\},\\
I_{\delta}:=\{j|\ \delta_j\in \cl(R_{\lambda+\rho})\}=
\{j|\ \exists s\in \mathbb{Z}\ \text{ s. t.}\ 2(\lambda+\rho,s\delta+\delta_j)\in 2\mathbb{Z}+1\}.\end{array}$$

The assumption that 
$\Delta^{\#}\cap R_{\lambda+\rho}$ is isometric to $ \Delta^{\#}\cap \Delta_{k\Lambda_0+\rho}$ 
implies $\dot{\Delta}^{\#}\subset \cl(R_{\lambda+\rho})$, 
so  $I_{\vareps}=\{1,\ldots,m\}$.

Since $2(\lambda,\alpha+\frac{q}{2}\delta)=2(\lambda,\alpha)+p$ and $p$ is odd
we have
$$\begin{array}{l}
I_{\vareps}=\{i|\ \exists s\in \mathbb{Z}\ \text{ s. t.}\ 
2(\lambda+\rho,s\delta+\vareps_i)\in 2\mathbb{Z}+1\}\\
I_{\delta}=\{j|\ \exists s\in \mathbb{Z}\ 
\text{ s. t.}\ 2(\lambda+\rho,s\delta+\delta_j)\in \mathbb{Z}\}.
\end{array}$$

\subsubsection{}
If $R_{\lambda+\rho}$ does not contain isotropic roots, then
$R_{\lambda+\rho}= \Delta_{\lambda+\rho}$. Thus we may (and will)
assume that $R_{\lambda+\rho}$
contains  an isotropic root, that is $(\lambda+\rho,s\delta+\vareps_i-\delta_j)=0$
for some $s,i,j$. Then $\delta_j\in I_{\delta}$. Acting by a suitable element of $\dot{W}$
we can assume that $I_{\delta}=\{1,\ldots,n_1\}$ for some $n_1$ such that $1\leq n_1\leq n$ and that 
\begin{equation}\label{eq:aibjBm>n}
\begin{array}{l}
a_i:=(\lambda+\rho,\vareps_i)\in \mathbb{Z}+\frac{1}{2}\ \ \text{ for }i=1,\ldots,m;\\
b_j:=(\lambda+\rho,\delta_j)\in \mathbb{Z}+\frac{1}{2}\ \ \text{ for }j=1,\ldots,n_1;\\
a_n=b_1.\end{array}\end{equation}
We let  $B(m|n_1)^{(1)}$ be
the corresponding natural copy  in 
$\Delta(\fg)$.
The intersection of $\Delta_{\ol{0}}^{\ree}$
with the span of $\delta$ and $\{\delta_j\}_{j>n_1}$
 is isomorphic to $C_{n-n_1}^{(1)}$
 and we denote this intersection  by $C_{n-n_1}^{(1)}$.

For $j>n_1$ we have $j\not\in I_{\delta}$, so 
 $R_{\lambda+\rho}$ 
does not contain the roots of the form $s\delta\pm \delta_j$,
$s\delta\pm \vareps_i\pm\delta_j$ and 
$s\delta\pm \delta_r\pm\delta_j$ for $i=1,\ldots,m$ and $r=1,\ldots,n_1$.
This means that $R_{\lambda+\rho}$ lies in $ B(m|n_1)^{(1)}\coprod C_{n-n_1}^{(1)}$.
Hence
$$\Delta_{\lambda+\rho}\subset \bigl( B(m|n_1)^{(1)}\coprod C_{n-n_1}^{(1)}\bigr).$$
The set $\Delta^{\#}\cap \Delta_{\lambda+\rho}$ lies
$\Delta_{\lambda+\rho}\cap B(m|n_1)^{(1)}$. 
Therefore we have to verify 
\begin{equation}\label{eq:inclus}
\bigl(\Delta^{\#}\cap (\Delta_{\lambda+\rho} \cap B(m|n_1)^{(1)})\bigr)\subset 
\bigl( R_{\lambda+\rho}\cap B(m|n_1)^{(1)}\bigr).\end{equation}
Note that 
$\Delta_{\lambda+\rho}\cap B(m|n_1)^{(1)}$ 
is the minimal subset of $B(m|n_1)^{(1)}$ which contains the set
$R_{\lambda+\rho}\cap B(m|n_1)^{(1)}$ and satisfies properties (a)--(c) 
of~$\S$~\ref{Deltalambda}. 

Let us compute the set $R_{\lambda+\rho}\cap B(m|n_1)^{(1)}$.
For a non-isotropic $\alpha \in B(m|n_1)$  we have
\begin{equation}\label{eq:mn1}
(s\delta+\alpha)\in \ol{R}_{\lambda+\rho} \ \ 
\Longleftrightarrow\ \ \left\{\begin{array}{ll}
q|s & \text{ if } (\alpha,\alpha)=-1,\pm 2\\
\frac{q}{2}|s & \text{ if } (\alpha,\alpha)=1\\
\end{array}\right.\end{equation}
(The roots $s\delta\pm 2\delta_1$ do not lie in $\ol{R}_{\lambda+\rho}$, 
since $\frac{ps}{q}\pm 2b_j\not\in 2\mathbb{Z}$ for odd $s$).
If $s\delta-\vareps_i\pm \delta_j$ lies in $\ol{R}_{\lambda+\rho}$,
then $\frac{ps}{q}-a_i\pm b_j=0$, so $s=\frac{q(a_i\mp b_j)}{p}$ is divisible by $q$.
We conclude that $R_{\lambda+\rho}\cap B(m|n_1)^{(1)}$ lies in the set
$$\{\alpha+q\delta|\ \alpha\in B(m|n_1)\}\cup \{\alpha+\frac{q}{2}\delta|\ \alpha\in B(m|n_1),
\ (\alpha,\alpha)=\pm 1\}\cong 
D(m+1|n_1)^{(2)}.$$
Hence $\Delta_{\lambda+\rho}\cap B(m|n_1)^{(1)}$ lies the copy
$D(m+1|n_1)^{(2)}$  given by the formula above (in fact,
$\Delta_{\lambda+\rho}\cap B(m|n_1)^{(1)}$ coincides with this set).
Using~(\ref{eq:mn1}) we obtain~(\ref{eq:inclus}) and complete the proof.

\subsection{Example: $B(1|1)^{(1)}$ when $p$ is even}
Assume that
$$\Delta^{\#}\cap R_{\lambda+\rho}\ \text{ is isometric to }
\ \Delta^{\#}\cap R_{k\Lambda_0+\rho}\ \ \text{ and }\ \ 
(\Delta^{\#}\cap \Delta_{\lambda+\rho})\not\subset R_{\lambda+\rho}.$$

Then $R_{\lambda+\rho}$ contains an isotropic root and a root of the form
 $s\delta+2\delta_1$.
 
We will use  the normalization $(\delta_1,\delta_1)=-1$.
For this normalization
$(k\Lambda_0+\rho,\delta)=\frac{p}{q}$.

Retain notation of~$\S$~\ref{grouptildeT}. If $s$ is even, then,
acting by a suitable element of $\tilde{W}$,
we can (and will) assume that $R_{\lambda+\rho}$ contains $\vareps_1-\delta_1$ and
 $2\delta_1$. Since $R_{\lambda+\rho}$ contains $\vareps_1-\delta_1$,
 one has $(\lambda+\rho,\vareps_1-\delta_1)=0$ and
 ${R}_{\lambda+\rho}$ does not contain $j\delta+\vareps_1-\delta_1$ for $j\not=0$ 
 (since $(\lambda+\rho,\delta)\not=0$).
Set 
$$a:=(\lambda+\rho,\delta_1).$$
Since $(\lambda+\rho,\vareps_1-\delta_1)=0$ we have $a=(\lambda+\rho,\vareps_1)$.
 Since $2\delta_1\in R_{\lambda+\rho}$, $2a$ is an odd integer.
In this case $\ol{R}_{\lambda+\rho}\cap\Delta^{\an}
 =\{\mathbb{Z}q\delta\pm\vareps_1;\mathbb{Z}q\delta\pm\delta_1\}$.
 If $j\delta+\vareps_1+\delta$ lies in $\ol{R}_{\lambda+\rho}$,
 then $\frac{jp}{q}=2a$ which is impossible since $p$ is even.
 Hence 
 $\ol{R}_{\lambda+\rho}
 =\{\mathbb{Z}q\delta\pm\vareps_1;\mathbb{Z}q\delta\pm\delta_1; \pm (\vareps_1-\delta_1)\}$, 
 so 
 $$\Delta_{\lambda+\rho}=\{\mathbb{Z}q\delta\pm\vareps_1;\mathbb{Z}q\delta\pm\delta_1\}\cup
 \{2\mathbb{Z}q\pm\vareps_1\pm\delta_1,2\mathbb{Z}q\delta\pm 2\delta_1\}\cong D(2|1)^{(2)}
$$
and $\Delta_{\lambda+\rho}\cap\Delta^{\#}=\{2\mathbb{Z}q\delta\pm 2\delta_1\}
\subset R_{\lambda+\rho}$.

Now consider the case 
$R_{\lambda+\rho}$ contains an isotropic root $\beta$
and a root of the form
 $s\delta+2\delta_1$ for an odd $s$. 
 Using the normalization $(\delta_1,\delta_1)=1$ we have
$$(\lambda+\rho,s\delta+2\delta_1)\in 2\mathbb{Z},\ \ (\lambda+\rho,\beta)=0$$
for some isotropic root $\beta$. Then $2\beta=s\delta+2\delta_1+j\delta\pm 2\vareps_1$
(for some odd $j$), so
$(\lambda+\rho, \frac{j}{2}\delta\pm \vareps_1)$ is integral and thus
$(\lambda+\rho, \frac{j+q}{2}\delta\pm \vareps_1)$ is integral.
Since $j+q$ is even, we get $\vareps_1\in\cl(R_{\lambda+\rho})$.
Hence $\cl(\Delta_{\lambda+\rho})$ contains $\pm\vareps_1$, $\pm 2\delta_1$
and an isotropic root $\cl(\beta)$. Since $\cl(\Delta_{\lambda+\rho})\subset B(1|1)$
and $s_{\alpha} \cl(\Delta_{\lambda+\rho})$
for any non-isotropic $\alpha\in \cl(\Delta_{\lambda+\rho})$
we obtain $\cl(\Delta_{\lambda+\rho})=B(1|1)$, so
$\Delta_{\lambda+\rho}$ is isomorphic to $B(1|1)^{(1)}$
or to $D(2|1)^{(1)}$. If $\Delta_{\lambda+\rho}\cong D(2|1)^{(1)}$, then
for any $\alpha\in (\Delta_{\lambda+\rho}\cap\Delta^{\#})$
one has $\alpha/2\in \Delta_{\lambda+\rho}$; since 
$(s\delta+2\delta_1)/2\not\in\Delta$ we conclude 
$\Delta_{\lambda+\rho}\not\cong D(2|1)^{(1)}$. Hence 
$\Delta_{\lambda+\rho}\cong B(1|1)^{(1)}$. Let us show that
$$\Delta_{\lambda+\rho}\cap \Delta^{\#}\not\subset R_{\lambda+\rho}.$$

 Since $\Delta_{\lambda+\rho}\cong B(1|1)^{(1)}$, the set
 $\Delta_{\lambda+\rho}\cap \Delta^{\#}$
contains a root of the form $2\alpha$ where $\alpha\in \Delta_{\lambda+\rho}$,
that is $\alpha=r\delta\pm \delta_1$. If $2\alpha\in R_{\lambda+\rho}$, then
$(\lambda+\rho, 2r\delta\pm 2\delta_1)$ is odd.
Since $(\lambda+\rho,s\delta+2\delta_1)$ is even,
 $\frac{(2r\pm s)p}{q}$ is an odd integer which is impossible since $p$ is even.

\subsection{Subprincipal levels}\label{subprince}  
Summarizing, we obtain: if $\fg$ is an indecomposable
non-twisted affine Kac-Moody superalgebra which is 
not a Lie algebra and 
$k$ is a subprincipal admissible level,
then $\fg$ is of  type $B(m|n)^{(1)}$ with $m>n$, or $m=n=1$, or $m=0$.
For $B(0|n)^{(1)}$ the subprincipal admissible levels
are of the form $k=-h^{\vee}+\frac{p}{2q}$, where  $p,q$ are coprime positive integers,
$pq$ is even and $p\geq 2n-1$ (in this case $k'=\frac{p-2n+1}{2}$).
For $B(m|n)^{(1)}$ with $m>n$  the subprincipal admissible levels
are of the form $k=-h^{\vee}+\frac{p}{q}$,
 where $p,q$ are coprime positive integers
$q$ is even  and $p>2(m-n)$ (in  this case $k'=\frac{p}{2}-(m-n)$
lies in $\mathbb{Z}_{\geq 0}+\frac{1}{2}$).
For $B(1|1)^{(1)}$ the subprincipal admissible levels
are of the form $k=-h^{\vee}+\frac{p}{2q}$,
 where $p,q$ are positive coprime integers and $q$ is odd
 (in this case $k'=p)$.

\section{The set $\Delta_{L}$ for $V_k(\fg)$-module $L$ when $k$ is principal admissible }\label{Delta''}

Let $k$ be a principal admissible level. As before, we write $k+h^{\vee}=\frac{p}{q}$
for $\fg\not=B(m|n)^{(1)}$ with $n\geq m$ and $k+h^{\vee}=\frac{p}{2q}$
for $\fg=B(m|n)^{(1)}$ with $n\geq m$, where $p,q$ are positive coprime integers.
Let $L(\lambda)$ be a $V_k(\fg)$-module.

\subsection{Notation}\label{fg''}
By~\Thm{thm:cor:arak} one has
$\Delta_{\lambda+\rho}\cap\Delta^{\#}=R_{\lambda+\rho}\cap\Delta^{\#}$  is isometric to $\Delta^{\#}$.

Retain notation of~$\S$~\ref{glambda}.  By~\Thm{thm:Deltalambda} 
all components of $\fg'$ are finite-dimensional or affine Kac-Moody superalgebras.

 Since $\Delta_{\lambda+\rho}\cap\Delta^{\#}$ is isometric to 
 $(\Delta^{\#})^{\ree}$,  this 
is an indecomposable root system, so
$\psi^{-1}(\Delta_{\lambda+\rho}\cap\Delta^{\#})$
lies in the root system of an indecomposable component of $\fg'$.
We denote this component by $\fg''$ and its  root system 
by $\Delta''$.  

Since $(\Delta^{\#})^{\ree}$ is infinite,  $\Delta''$ is infinite, so
$\fg''$ is an affine Kac-Moody superalgebra.
We denote by $\delta''$
the minimal imaginary root of $\fg''$. 
Recall that $\psi(\delta'')=u\delta$ for some $u\in\mathbb{Z}_{>0}$, see~$\S$~\ref{psidelta}.

\subsection{Example}
Let $\fg$ be $A(2|2)^{(1)}$ with 
$$\Sigma=\{\vareps_1-\vareps_2,\vareps_2-\vareps_3,\vareps_3-\delta_1,\delta_1-\delta_2,\delta_2-\delta_3, \delta-\vareps_1+\delta_3\}.$$
Let $\lambda=\Lambda_0+a(\delta_2+\delta_3)$ where $a\not\in\mathbb{Z}$.
Then $\fg'=\fg''\times\fg'''$, where $\fg''=A(2|1)^{(1)}$ with the simple roots 
$\{\vareps_1-\vareps_2,\vareps_2-\vareps_3,\vareps_3-\delta_1, \delta-\vareps_1+\delta_1\})$
and $\fg'''=A_1^{(1)}$ with the simple roots $\{\delta_2-\delta_3,\delta-\delta_2+\delta_3\}$.

\subsection{}
\begin{prop}{prop:DeltaL}
\begin{enumerate}
\item 
If $\fg'''\not=\fg''$  is a  component of $\fg'$,
 then for any real root $\alpha$ of $\fg'''$ we have 
$(\alpha,\alpha)<0$  (in particular, $\fg'''$ is anisotropic). 
\item
The Kac-Moody algebra $\fg''$ is the non-twisted affinization of $\dot{\fg}''$, where

 $\dot{\fg}''=A(m|n_1)$ if $\dot{\fg}=A(m|n)$ with $m\geq n$,
 
$\dot{\fg}''=G_2$ or $G(3)$ for $\dot{\fg}=G(3)$, 

$\dot{\fg}''=B_3$ or $F(4)$ for $\dot{\fg}=F(4)$, and 
$$\begin{array}{|c||c|c| c|c|c|c|c|c|c|}
 \hline
\dot{\fg} & B(m|n),\, m>n & B(m|n),\, n\geq m & D(m|n),\, m>n &D(m|n),\, n\geq m\\
 \hline
\dot{\fg}''  & B(m|n_1) & B(m_1|n)& D(m|n_1) &   D(m_1|n)\\
 \hline
 \end{array}$$
where $0\leq n_1\leq n$, $0\leq m_1\leq m$. Moreover, $\fg''\cong \fg^{\#}$
if $\lambda$ is typical, and $\fg''$ is not anisotropic
if $\lambda$ is atypical.
\item One has $\psi(\delta'')=q\delta$.
\end{enumerate}
\end{prop}

\subsubsection{}
\begin{cor}{cor:Delta''}
If the dual Coxeter number of $\fg''$
is equal to $h^{\vee}$, then   $\fg''\cong \fg$  and  $\Delta'=\Delta''$.
If $\fg'\not\cong \fg$, then the dual Coxeter number of $\fg''$
 is greater than $h^{\vee}$. 
\end{cor}
\begin{proof}
 From
 the table for the dual Coxeter numbers in~$\S$~\ref{dualCoxeter} we see
 that the dual Coxeter number of $\fg''$
 is greater than $h^{\vee}$ if $\fg''\not\cong\fg$. If $\fg''\cong\fg$, then 
 $\cl(\Delta'') \cong\cl(\Delta)$. But $\cl(\psi(\Delta''))$ is a subset of 
 $\dot{\Delta}=\cl(\Delta^{\ree})$
 (and $\dot{\Delta}$ is finite), so  $\cl(\psi(\Delta''))=\dot{\Delta}$. This implies
 $\fg'=\fg''$ (since, if $\fg'''\not=\fg''$  is a component of $\fg'$,
 then $\Delta(\fg''')$ is orthogonal to $\Delta(\fg'')$, so
 $\psi(\Delta(\fg'''))\subset\dot{\Delta}$ is orthogonal to $\psi(\Delta(\fg''))=\dot{\Delta}$,
a contradiction).
\end{proof}

\subsection{Proof of~\Prop{prop:DeltaL}}
The case  $\dot{\fg}=B(0|n)$ is treated  in~\Cor{cor:VkforB0n}. 
Now we assume that $\dot{\fg}$ is not anisotropic.

Using $\psi$ we view $\Delta''$ as a subset of $\Delta$.  

\subsubsection{Proof of (i)}
\label{clDelta''}
By construction, $\Delta''$ contains $\psi^{-1}(\Delta_{\lambda+\rho}\cap \Delta^{\#})$
which is isometric to $\Delta^{\#}$, so 
$\cl\bigl((\Delta'')^{\ree}\bigr)$ is a subset of 
$\cl\bigl((\Delta)^{\ree}\bigr)=\dot{\Delta}$ which contains a subset isometric
to $\cl(\Delta^{\#})=\dot{\Delta}^{\#}$. Since $\dot{\Delta}$
is finite, we obtain 
\begin{equation}\label{eq:dotDeltacl}
\dot{\Delta}^{\#}\subset \cl\bigl((\Delta'')^{\ree}\bigr)\subset \dot{\Delta}.
\end{equation}

Take $\alpha\in\Delta_{\lambda+\rho}$ such that $(\alpha,\alpha)\geq 0$.
We have to verify that $\alpha\in\Delta''$. Suppose the contrary.
Then $(\alpha,\Delta'')=0$.
The set $\dot{\Delta}^{\#}$ contains a root  $\dot{\gamma}$ such that
 $(\cl(\alpha),\dot{\gamma})\not=0$ (one has $\mathbb{R}\dot{\Delta}=\mathbb{R}\dot{\Delta}^{\#}
 \oplus V'$, where 
 the restriction of $(-,-)$ to $V'$ is negative definite;
 if  $(\cl(\alpha),\dot{\Delta}^{\#})=0$,
 then $\cl(\alpha)\in V'$, so $(\alpha,\alpha)=(\cl(\alpha),\cl(\alpha))< 0$).
  By~(\ref{eq:dotDeltacl}), 
  $\Delta''$ contains $\gamma$
such that $\cl(\gamma)=\dot{\gamma}$ and thus
$(\gamma,\alpha)=0$, a contradiction.

\subsubsection{Proof of (ii) for  $\fg\not=B(m|n)^{(1)}$, $n\geq m$}
The assumption $\fg\not=B(m|n)^{(1)}$ with $n\geq m$ gives 
$$\Delta^{\#}=\{\alpha\in\Delta|\ (\alpha,\alpha)>0\}.$$

If $\lambda$ is typical,
then $\Delta_{\lambda+\rho}$ does not contain 
isotropic roots and, by (i), 
\begin{equation}\label{eq:99}
(\Delta'')^{\ree}=\{\alpha\in\Delta_{\lambda+\rho}|\ (\alpha,\alpha)>0\}.
\end{equation}
The assumption $\fg\not=B(m|n)^{(1)}$ with $n\geq m$ gives 
$\Delta^{\#}=\{\alpha\in\Delta|\ (\alpha,\alpha)>0\}$, so
$$\{\alpha\in (\Delta'')^{\ree}|\ (\alpha,\alpha)>0\}=(\Delta_{\lambda+\rho}\cap \Delta^{\#})
\text{ 
is isometric to }(\Delta^{\#})^{\ree}.$$

If $\lambda$ is atypical, then $\Delta_{\lambda+\rho}$
contains isotropic roots and these roots lie in $\Delta''$
by (i).  Using~(\ref{eq:dotDeltacl}) and 
the table  in Appendix~\ref{dotgclg} we obtain the statement for $\dot{\fg}\not=B(m|n)$.

For $\dot{\fg}=B(m|n)$ the same argument implies that $\fg''$ is of type $B(m|n_1)^{(1)}$
or $D(m+1|n_1)^{(2)}$ for some $0<n_1\leq n$.
If $m>n$ and $\fg''=D(m+1|n_1)^{(2)}$, then 
$$\{\alpha\in (\Delta'')^{\ree}|\ (\alpha,\alpha)>0\}=D_{m=1}^{(2)}$$
is not isometric to $(\Delta^{\#})^{\ree}=B_m^{(1)}$, a contradiction.
This completes the proof of (ii) for $\fg\not=B(m|n)$, $n\geq m$.

\subsubsection{Proof of (iii) for  $\fg\not=B(m|n)^{(1)}$, $n\geq m$}
Recall that 
$\Delta_{\lambda+\rho}\cap\Delta^{\#}=R_{\lambda+\rho}\cap\Delta^{\#}$
is isometric to $(\Delta^{\#})^{\ree}$. Since, 
for $\fg$ in question, $\alpha/2\not\in\Delta$  for all  $\alpha\in\Delta^{\#}$
we have 
$$\ol{R}_{\lambda+\rho}\cap\Delta^{\#}={R}_{\lambda+\rho}\cap\Delta^{\#}.$$
Since this set is isometric to $(\Delta^{\#})^{\ree}$, it contains $\alpha$
such that $(\alpha,\alpha)=2$. Since 
$\alpha\in \ol{R}_{\lambda+\rho}$, so $(\lambda+\rho,\alpha)\in\mathbb{Z}$ and
$$(\lambda+\rho, i\delta+\alpha)=i(k+h^{\vee})+(\lambda+\rho,\alpha)=\frac{ip}{q}+
(\lambda+\rho,\alpha).$$
Therefore $(i\delta+\alpha)\in \ol{R}_{\lambda+\rho}$ if and only if $i$
is divisible by $q$. Since $\ol{R}_{\lambda+\rho}\cap\Delta^{\#}$
is isometric to $(\Delta^{\#})^{\ree}$, this gives
$$(\ol{R}_{\lambda+\rho}\cap\Delta^{\#})+i\delta=\ol{R}_{\lambda+\rho}\cap\Delta^{\#}\ \ 
\Longleftrightarrow\ \ q|i$$
so  $u=q$ and thus  $\psi(\delta'')=q\delta$ as required.

\subsubsection{}
For the remaining case $\fg=B(m|n)^{(1)}$ with $n\geq m$, (ii) and (iii)
is proved in~\Lem{lem:Bngeqm} below.\qed

\subsection{}
\begin{lem}{lem:Bngeqm}
Let $\fg=B(m|n)^{(1)}$ with $n\geq m$ and $k$ such that $k+h^{\vee}=\frac{p}{2q}$
where $p,q$ are odd coprime positive integers. Let $\lambda$ be a weight of level $k$
such that $R_{\lambda+\rho}\cap\Delta^{\#}$ is 
isometric to $(\Delta^{\#})^{\ree}$. Then 
 $\fg''\cong B(m_1|n)^{(1)}$ for some $0\leq m_1\leq m$,
 where $m=0$ if and only if $\lambda$ is typical. Moreover,
  $\psi(\delta'')=q\delta$.
\end{lem}
\begin{proof}

We will use  the normalization $(\delta_1,\delta_1)=1$.
For this normalization
$(k\Lambda_0+\rho,\delta)=\frac{p}{q}$ 
and all roots are of integral square length, so
$2(\lambda+\rho,\gamma)$ is integral for all  $\gamma\in\Delta_{\lambda+\rho}$.

As in the proof of~\Lem{lem:Bngeqm0}
we will use the following observation: 
$(s\delta+2\delta_i)\in  R_{\lambda+\rho}$ if and only if
$(\lambda+\rho,s\delta+2\delta_i)$
is odd if $s$ is even, and $(\lambda+\rho,s\delta+2\delta_i)$
is even if $s$ is odd.
The assumption that 
$R_{\lambda+\rho}\cap\Delta^{\#}$ is isometric to $(\Delta^{\#})^{\ree}$
implies that $(s\delta+2\delta_1)\in  R_{\lambda+\rho}$ for some $s$. 
Since $p$ is odd, the above observation implies
 $((s+q)\delta+2\delta_1)\in  R_{\lambda+\rho}$.
Since $q$ is odd, one of the numbers $s$ and $s+q$ is even and another is odd.
Without loss of generality we assume that $s$ is even.
Then  $(s\delta+2\delta_1)\in  R_{\lambda+\rho}$
implies $\beta\in {R}_{\lambda+\rho}$ for $\beta:=\frac{s}{2}\delta+\delta_1$.
Moreover, $R_{\lambda+\rho}$ contains $\gamma=((s+q)\delta+2\delta_1)$
such that $\gamma/2\not\in R_{\lambda+\rho}$
(since $s+q$ is odd) and
 $\cl(\gamma)/2=\cl(\beta)\in \cl(R_{\lambda+\rho})$.

Set $X:=\{\alpha\in R_{\lambda+\rho}|\ (\alpha,\alpha)>0\}$.
Notice that 
  $\beta, \gamma\in X$.

Since $R_{\lambda+\rho}\cap \Delta^{\#}$ is isometric to $(\Delta^{\#})^{\ree}$,
$\cl(X)$ contains $\dot{\Delta}^{\#}$.
 Observe that $X$ is a real root subsystem 
(this means that
$s_{\alpha}\gamma\in X$ for all $\alpha,\gamma\in X$) in the root system
$\{\alpha\in \Delta|\ (\alpha,\alpha)>0\}$ which is of type $B(0|n)^{(1)}$.
By~\Prop{prop:RR'an}, $X$ is isomertic to the set of real roots 
of an  anisotropic affine Kac-Moody superalgebra $\ft$. By above,
 $X$ contains an odd non-isotropic root $\beta$ and $\cl(X)$ contains $\dot{\Delta}^{\#}$
 (of type $C_n$).
Using
the classification in~\cite{K78}, we conclude that 
$\ft$ is of type $B(0|n)^{(1)}$ or $D(1|n+1)^{(2)}$.
By above,  $X$
contains  $\gamma$
 such that $\cl(\gamma)/2\in \cl(X)$ and $\gamma/2\not\in X$.
The root system of $D(1|n+1)^{(2)}$ does not contain a root with these properties, so
 $\ft$ is of type $B(0|n)^{(1)}$.

One has $(\beta+j\delta)\in R_{\lambda+\rho}$
if and only if $j$ is divisible by $q$. Thus
$X+j\delta=X$ if and only if $j$ is divisible by $q$. This gives
$\psi(\delta'')=q\delta$. 

If $\lambda$ is typical, then
$\Delta''=X$ by~(\ref{eq:99}), so $\fg''=\ft=B(0|n)^{(1)}$ as required.
If $\lambda$ is atypical, then 
$\Delta''$ is a root system which contains isotropic roots and 
odd non-isotropic root; moreover,
$\{\alpha\in \Delta''|\ (\alpha,\alpha)>0\}=X$ is of type $B(0|n)^{(1)}$.
Using the table in Appendix~\ref{dotgclg} we obtain
that $\fg''\cong B(m_1|n)^{(1)}$. Since $\cl(\Delta'')$
is a root subsystem of $B(m|n)$ one has  $m_1\leq m$
as required.
\end{proof}

\section{Boundary admissible levels}\label{boundary}
Let $\fg$ be an indecomposable non-twisted affine Kac-Moody superalgebra,
$\fg\not=D(2|1,a)^{(1)}$. 
Retain notation of~$\S$~\ref{notataff}.

Let $k$ be an admissible level. Introduce 
$k'$  as in~\Lem{lem:Bmn1Dmn2etc} (iv). We call $k$ {\em boundary admissible level}
if $k'=0$, see~\cite{KWboundary}.

\subsection{}
\begin{lem}{lem:boundary}
\begin{enumerate}
\item
The boundary admissible levels are  
\begin{itemize}
\item[$\cdot$] principal: $k+h^{\vee}=\frac{h^{\vee}}{q}$, where 
$q$ is a positive integer, $\gcd(q,r^{\vee})=\gcd(q,h^{\vee})=1$
(if $h^{\vee}$ is not integral (i.e.,  $\fg=B(m|n)^{(1)}$ with $n\geq m$
and $h^{\vee}=n-m+\frac{1}{2}$), then $\gcd(q,h^{\vee})=1$ stands for $\gcd (2q, 2h^{\vee})=1$);
\item[$\cdot$] subprincipal: this happens only for $\fg=B(0|n)^{(1)}$; in this case
$k+h^{\vee}=\frac{2n-1}{2q}$, where $q$ is an even positive integer which is coprime with $2n-1$
(and $h^{\vee}=n+\frac{1}{2}$).
\end{itemize}
\item
If $h^{\vee}=0$, then $\fg$
 does not admit boundary admissible levels.
 \end{enumerate}
\end{lem}
\begin{proof}
Let us  verify that boundary  subprincipal admissible levels
exist only for $\fg=B(0|n)^{(1)}$. Let $k$ be a boundary subprincipal admissible  level.
By~$\S$~\ref{subprince}, $\fg$ is either a 
non-twisted affine Lie algebra or $B(m|n)^{(1)}$ with $m>n$, or  $m=n=1$, or $m=0$.

Since $k$ is boundary admissible, $k'=0$. If $\fg=B(m|n)^{(1)}$ with $m>n$ or $m=n=1$,
this is impossible  by~$\S$~\ref{subprince}. 
If $\fg$ is a non-twisted affine Lie algebra, 
then $q$ is divisible by $r^{\vee}>1$ (since $k$ is subprincipal)
and $k'=0$ means that 
$p$  is equal to the Coxeter number of $\fg$. By the table in~$\S$~\ref{commentintro},
 $p$   is divisible by
$r^{\vee}>1$, so $p$ and $q$ 
are not coprime, a contradiction.  Hence $\fg=B(0|n)^{(1)}$.
 By~$\S$~\ref{subprince}, 
if $k$ is subprincipal and 
$k'=0$, then $k+h^{\vee}=\frac{2n-1}{2q}$
and $(2n-1)q$ is even, so $q$ is even. 
This establishes (i).

For (ii) assume that $h^{\vee}=0$. Then $\fg\not=B(0|n)^{(1)}$, so by (i),
if $k$ is boundary admissible, then $k=0$, which means that  $k$ is a critical level.
By definition, critical level is not admissible.
\end{proof}

\subsection{}
\begin{prop}{prop:boundary}
Let $k$ be a  boundary admissible level.
\begin{enumerate}
\item If $L(\lambda)$ is a $V_k(\fg)$-module, then 
$\Delta_{\lambda+\rho}$ is isometric to 
$\Delta_{k\Lambda_0+\rho}$, $\Delta_{\lambda+\rho}+q\delta= \Delta_{\lambda+\rho}$,
and $\dim L_{\fg'}(\lambda')=1$ (see~$\S$~\ref{glambda} for notation).

\item The number of isomorphism classes
of irreducible $V_k(\fg)$-modules in $\CO$  is finite and any
$V_k(\fg)$-module $N\in\CO^{\inf}(\fg)^k$ 
is completely reducible.
\end{enumerate}
\end{prop}

For the case $\fsl(2|1)^{(1)}$ the proposition
was established recently in~\cite{Qing}.

\begin{proof}
Retain notation of~$\S$~\ref{gnu}.

For (i) let $L(\lambda)$ be a $V_k(\fg)$-module. 

We start from the case when $k$ is boundary principal admissible.
Combining Theorems~\ref{thm:cor:arak} and~\ref{thm:criterion} we conclude that $L_{\fg'}(\lambda')$ is $\pi'$-integrable  for 
$$\pi':=\{\alpha'\in\Sigma'_{\pr}|\ (\alpha,\alpha)>0\}$$ and
$$R_{\lambda+\rho}\cap\Delta^{\#}=\Delta_{\lambda+\rho}\cap\Delta^{\#}\ \text{
is isometric to }\ 
(\Delta_{k\Lambda_0+\rho}\cap\Delta^{\#}).$$

Retain notation of Section~\ref{Delta''}. 
Notice that $\pi'\subset \Delta''$.
Therefore $L_{\fg''}(\lambda')$ is $\pi'$-integrable.
By~\Prop{prop:DeltaL} the minimal imaginary root of $\Delta''$
is $\delta''=\psi^{-1}(q\delta)$. Therefore
$$(\lambda'+\rho',\delta'')=(\lambda+\rho,q\delta)=h^{\vee}.$$
Then $(\lambda',\delta'')=h^{\vee}-h^{\vee,''}$
where $h^{\vee,''}$ is the dual Coxeter number of $\Delta''$.
Since $L_{\fg''}(\lambda')$ is $\pi'$-integrable, $h^{\vee}-h^{\vee,''}\geq 0$.
In the light of~\Cor{cor:Delta''} we obtain $h^{\vee}=h^{\vee,''}$ 
and $\Delta'=\Delta''$ is isometric to $\Delta$. 
Therefore $L_{\fg'}(\lambda')$ is $\pi'$-integrable
module of level zero. By~\Lem{lem:zerolevel} below, $\dim L_{\fg'}(\lambda')=1$.
Finally, the formula 
 $\psi(\delta'')=q\delta$ implies $\Delta_{\lambda+\rho}+q\delta= \Delta_{\lambda+\rho}$.

Now consider the case when  $k$
is a boundary subprincipal admissible  level. 
In this case $\fg=B(0|n)^{(1)}$ and
$k+n+\frac{1}{2}=\frac{2n-1}{2q}$ where $q$ is even.  
Using~\Cor{cor:VkforB0n} (ii) we obtain the following.
If $\fg'\cong C_1^{(1)}$, then $n=1$  and 
$\lambda'$ has the central level $-1$, if $\fg'\cong A_{2n-1}^{(2)}$
for $n>1$, then 
$\lambda'$ has the central level $\frac{2n-1}{2}-n=-\frac{1}{2}$,
and,  if $\fg'\cong A(0|2n-1)^{(2)}$ for $n>1$, then 
$\lambda'$ has the central level $0$.
The first two  cases are impossible  since 
$L_{\fg'}(\lambda')$ is integrable, so the central level should be positive.
Therefore $\fg'\cong A(0|2n-1)^{(2)}$
and $\lambda'$ has the zero central level. 
Since $L_{\fg'}(\lambda')$ is integrable, $\dim L_{\fg'}(\lambda')=1$,
by~\Lem{lem:zerolevel} below. Since $\fg'\cong A(0|2n-1)^{(2)}$ is indecomposable, $\fg''=\fg'$
and $\psi(\delta'')=q\delta$ by~(\ref{eq:B0ndelta'}). This 
 implies $\Delta_{\lambda+\rho}+q\delta= \Delta_{\lambda+\rho}$ and completes the proof of 
 (i).

For (ii)  let $L(\lambda)$, $L(\nu)$ be  $V_k(\fg)$-modules.
Assume that $\Delta_{\lambda+\rho}=\Delta_{\nu+\rho}$.
By (i)
we have 
$$(\lambda+\rho,\alpha)=(\lambda'+\rho', \psi^{-1}(\alpha))=
= (\rho',\psi^{-1}(\alpha'))$$
 for all $\alpha\in\Delta_{\lambda+\rho}$. Similarly,
 $(\nu+\rho,\alpha)=(\rho',\psi^{-1}(\alpha'))$, so
$(\lambda-\nu,\alpha)=0$ for all $\alpha\in\Delta_{\lambda+\rho}$.
Since $\Delta_{\lambda+\rho}$ is isometric to $\Delta^{\ree}$, 
 the span of $\Delta_{\lambda+\rho}$ contains $\Delta$. Hence
$(\lambda-\nu)\in\mathbb{C}\delta$, so $L(\lambda)$
and $L(\nu)$ are isomorphic as  $V_k(\fg)$-modules.

Thus the number of isomorphism classes
of irreducible $V_k(\fg)$-modules in $\CO$  is not greater than the number
of the subsets $X\subset\Delta^{\ree}$
satisfying $X+q\delta=X$. Each subset of this form corresponds
to a subset of the  set  $\Delta^{\ree}/\mathbb{Z}q\delta$
which is finite. Hence  the number of isomorphism classes
of irreducible $V_k(\fg)$-modules in $\CO$  is finite.

Now let $N$ be an indecomposable $V_k(\fg)$-module in the 
category $\CO_{\Sigma}(\fg)^k$.
Set 
$$Y:=\{\lambda|\ [N:L(\lambda)]\not=0\}.$$ 
Take $\lambda,\nu\in Y$.
By~\Cor{cor:Olambda}
we have $\Delta_{\lambda+\rho}=\Delta_{\nu+\rho}$.
By above, this gives $(\lambda-\nu)\in\mathbb{C}\delta$.
Using the action of
the Casimir operator 
we obtain $(\lambda+\rho,\lambda+\rho)=(\nu+\rho,\nu+\rho)$
which forces $\lambda=\nu$ (since $(\lambda+\rho,\delta)\not=0$).
Hence $Y=\{\lambda\}$.
Since $N\in\CO_{\Sigma}(\fg)$, $\fh$ acts diagonally, so
$N=L(\lambda)$. Now let $N$ be a $V_k(\fg)$-module in the category 
$\CO^{\inf}(\fg)^k$. A non-zero vector $v\in N$ 
generates a submodule which lie in $\CO_{\Sigma}(\fg)^k$;
by above, this submodule is completely reducible.
Hence $N$ is a sum of irreducible modules, so, by~\cite{Lang}, Chapter XVII,
$N$ is completely reducible. This completes the proof for the case when $k$
is a boundary principal admissible level.
\end{proof}

\subsection{}
\begin{lem}{lem:zerolevel}
Let $\fg$ be an affine Kac-Moody superalgebra and  let $L(\lambda)$ be a $\pi^{\#}$-integrable
module with
 $(\lambda,\delta)=0$. Then $\lambda\in\mathbb{C}\delta$ and $\dim L(\lambda)=1$.
\end{lem}
\begin{proof}
We will check that $(\lambda,\alpha)=0$ for all $\alpha\in\Sigma$.
If $\fg$ is anisotropic, the assertion immediately follows from the fact that
$\delta$ is a positive linear combination of simple roots. 

For the case
when $\fg$ is not anisotropic, 
take $\Sigma$ as in Appendix~\ref{listgoodbases}. Since the assertion holds for 
$\fg^{\#}$ we have $(\lambda,\alpha)=0$
for all $\alpha\in\Delta_{\ol{0}}$ such that $(\alpha,\alpha)>0$. Take $\alpha\in\Sigma$.
If $(\alpha,\alpha)>0$, then $(\lambda,\alpha)=0$ by above. Assume that
$(\alpha,\alpha)\leq 0$.
By~$\S$~\ref{pairs}  there exists $\beta\in\Sigma$
such that 
\begin{equation}\label{eq:zerolevel}
(\beta,\beta)=0<(\alpha+\beta,\alpha+\beta).
\end{equation}
Note that $\alpha+\beta$ lies in $r_{\beta}\Sigma$.
Therefore $(\lambda,\alpha+\beta)=0$ and $(\lambda_1,\alpha+\beta)=0$
where $\lambda_1$ is the highest weight 
of $L(\lambda)$ with respect to $r_{\beta}\Sigma$.
Thus $(\lambda-\lambda_1,\alpha+\beta)=0$.
If $(\lambda,\beta)\not=0$, then $\lambda-\lambda_1=\beta$, so
$(\beta,\alpha+\beta)=0$ which contradicts to~(\ref{eq:zerolevel})
(since $(\alpha,\alpha)\leq 0$). Thus $(\lambda,\beta)=0$, so
$(\lambda,\alpha)=0$ as required.
\end{proof}

%\subsubsection{Remark}
%It would be nice to know whether $\Ext^1_{[\fg,\fg]} %(L(\lambda_1),L(\lambda_1))=0$.

\subsection{Character formula}
If $k$ is a boundary level, then $\ch L(k\Lambda_0)$
is given by a formula similar to formula (7) in~\cite{KW17}, see below.

\subsubsection{}
Let $k$ be a boundary admissible level. 
 We retain notation of~$\S$~\ref{6.2}, \ref{D'}.
Using~\cite{GKadm}, Theorem 11.3.1  for $\Sigma$ chosen in Appendix~\ref{listgoodbases} 
we can write the character formula for $L(k\Lambda_0)$
in the following way:
$$e^{-k\Lambda_0}D \ch L(k\Lambda_0)=\psi(D'),$$
where $D$ is the Weyl denominator for
$\Delta^+(\Sigma)$  (see~$\S$~\ref{Weyldenominator}),
$D'$ is  the Weyl denominator for $\Delta^{+}(\fg')$,
and  $\psi:\mathbb{Z}\Delta(\fg')\to\mathbb{Z}\Delta$
is as in~$\S$~\ref{KMsubsystem}, so that 
$$\psi(D')=\frac{\prod\limits_{\alpha\in \Delta^{+'}_{\ol{0}}} 
(1-e^{-\psi(\alpha)})^{\dim\fg_{\psi(\alpha)}}}
{\prod\limits_{\alpha\in\Delta^{+'}_{\ol{1}}} 
(1+e^{-\psi(\alpha)})^{\dim\fg_{\psi(\alpha)}}}.$$

Set $z:=e^{-\delta}$. Then
$$D=\dot{D}\prod\limits_{i=1}^{\infty} \bigl((1-z^i)^s
\prod\limits_{\alpha\in \dot{\Delta}_{\ol{0}}} (1-z^i e^{\alpha})
\prod\limits_{\alpha\in \dot{\Delta}_{\ol{1}}} (1+z^i e^{\alpha})^{-1}
\bigr),$$
where $\dot{D}$ is the Weyl denominator for $\dot{\Delta}$
and $s:=\dim\fg_{\delta}$ (i.e., $s=\dim\dot{\fh}$
for $\fg\not=A(n|n)^{(1)}$ and $s=2n$
for  $\fg=A(n|n)^{(1)}$).

\subsubsection{}
If $k$ is a principal boundary admissible level, i.e. $k+h^{\vee}=\frac{h^{\vee}}{q}$
with $\gcd(q,r^{\vee})=\gcd(q,h^{\vee})=1$, then
$$\psi(D')=\dot{D}\prod\limits_{i=1}^{\infty} \bigl( (1-z^i)^s
\prod\limits_{\alpha\in \dot{\Delta}_{\ol{0}}}  (1-z^{qi} e^{\alpha})
\prod\limits_{\alpha\in \dot{\Delta}_{\ol{1}}}  (1+z^{qi} e^{\alpha})^{-1}\bigr),
$$
which gives

$$e^{-k\Lambda_0} 
\ch L(k\Lambda_0)=\prod\limits_{m=1,\  q\, \nmid\, m}^{\infty} \bigl((1-z^m)^s
\prod\limits_{\alpha\in \dot{\Delta}_{\ol{0}}}  (1-z^{qm} e^{\alpha})
\prod\limits_{\alpha\in \dot{\Delta}_{\ol{1}}}  (1+z^{qm} e^{\alpha})^{-1}\bigr).
$$

\subsubsection{}
If $k$ is a subprincipal boundary admissible level, i.e. $\fg=B(0|n)^{(1)}$ and 
$k+n+\frac{1}{2}=\frac{2n-1}{2q}$,
where $q$ is even and $\gcd(q,2n-1)=1$, then
$$D=\dot{D}\prod\limits_{m=1}^{\infty} \bigl((1-z^m)^{n}
\prod\limits_{1\leq i\not=j\leq n}
(1-z^{m} e^{\pm \delta_i\pm\delta_j})(1-z^m e^{\pm 2\delta_i})
(1+z^m e^{\pm \delta_i})^{-1}\bigr).$$

By~$\S$~\ref{admB0n} we have $\fg'\cong A(0|2n-1)^{(2)}$
with $\psi(\delta')=q\delta$. This gives
$$\psi(D')=\dot{D}\prod\limits_{m=1}^{\infty} \bigl((1-z^{mq})^{n}
\prod\limits_{1\leq i\not=j\leq n}
(1-z^{qm} e^{\pm \delta_i\pm\delta_j})(1-z^{2qm} e^{\pm 2\delta_i})
(1+z^{qm} e^{\pm \delta_i})^{-1}\bigr).$$
Hence

$$\begin{array}{l}
e^{-k\Lambda_0} 
\ch L(k\Lambda_0)=\\
\prod\limits_{m=1,  \ q\, \nmid\, m }^{\infty} \bigl((1-z^m)^n
\prod\limits_{1\leq i\not=j\leq n}
(1-z^{m} e^{\pm \delta_i\pm\delta_j})
(1+z^{qm} e^{\pm \delta_i})^{-1}\bigr) \prod\limits_{i=1}^n
\prod\limits_{m=1,\  2q\,\nmid\, m}^{\infty} 
(1-z^{m} e^{\pm 2\delta_i}).
\end{array}$$

\section{Example: $\fsl(2|1)^{(1)}$}\label{sect:sl21}
In this case 
all admissible levels are principal admissible by~\Cor{cor:admnotBmn}. Moreover, the admissible
weights are principal admissible, see~\S~\ref{sl21pradm} below.
The admissible levels are $k=-1+\frac{p}{q}$  for 
coprime $p,q\in\mathbb{Z}_{>0}$ by~\Prop{prop:pradm}; the boundary admissible levels 
are $k=-1+\frac{1}{q}$.  In this section we prove the following proposition.

\subsection{}\begin{prop}{prop:sl21}
For an admissible level $k=-1+\frac{p}{q}$ the maximal proper submodule of $V^k$
is  generated 
by a singular vector of weight 
$k\Lambda_0-p(q\delta-\alpha)$, where 
$\alpha$ is the simple root of $\fsl_2$.
\end{prop}

\subsection{Notation}
Let $\pi^{\#}:=\{\alpha, \alpha_0:=\delta-\alpha\}$ be the base of $\fsl_2^{(1)}\subset \fg_{\ol{0}}$
and let $\{\beta,\alpha-\beta\}$ be a base of $\Delta(\fsl(2|1))$  
($\beta$ is isotropic). Then
$$\Sigma:=\{\alpha_0,\beta,\alpha-\beta\}, \ \ \ 
r_{\beta}\Sigma=\{\alpha_0+\beta,-\beta,\alpha\}$$
 are bases of $\Delta$. We let $k=-1+\frac{p}{q}$ to be an admissible level and set
 $$\lambda:=k\Lambda_0+\rho,\ \ M(w):=M(w\lambda-\rho)\ \text{ for } w\in W.$$ 
The 
integral root subsystem $\Delta^{\#}\cap \Delta_{\lambda}$ 
has the base $\{\alpha,q\delta-\alpha\}$
(this is the set of indecomposable elements in
 $\Delta^{\#}\cap \Delta_{\lambda}^+$), and we have:
$$k\Lambda_0-p(q\delta-\alpha)=s_{q\delta-\alpha}.k\Lambda_0=
s_{q\delta-\alpha}\lambda+\rho.$$

 \subsection{Proof of~\Prop{prop:sl21}}
The case $q=1$ is explained in~\Rem{rem:qing}, so
from now on we assume that $q\geq 2$.
 
The module $M(e)=M(k\Lambda_0)$  has singular vectors  $v_1$, $v_2$ of
 weights $k\Lambda_0-\beta$ and $k\Lambda_0-(\alpha-\beta)$ respectively;
 the vacuum module
  $V^k$ is the quotient of ${M}(\Id)$ by the submodule generated  by $v_1$ and $v_2$. 
  Since $(\lambda,q\delta-\alpha)=p>0$, 
 the module $M(e)$  has  a singular vector  $v$ of weight
 $s_{q\delta-\alpha}\lambda-\rho_{\Sigma}$.
 We denote by $N$ the submodule of $M(e)$ which is generated by $v_1$, $v_2$ and $v$ 
 and by  $\ol{v}$ the image of $v$ in $V^k$ (apriori, ${v}$ might be zero).
 We want to show that 
 $V^k/U(\fg)\ol{v}$ is irreducible (in particular, this would give $\ol{v}\not=0$).   By~\Prop{prop:Vksubq}  (ii) a 
 $\pi^{\#}$-quasi-admissible quotient  of $V^k$ is irreducible. 
 Thus
 it is enough to verify that $M(e)/N\cong V^k/U(\fg) \ol{v}$ is $\pi^{\#}$-quasi-admissible. Since $\{\alpha,q\delta-\alpha\}$ 
is the set of indecomposable elements in $\Delta^{\#}\cap \Delta_{\lambda}^+$, by~\Lem{lem:W'L}  (ii) it suffices to check
\begin{equation}\label{eq:desired}
s_{\gamma} (D_{\pi}e^{\rho_{\pi}} \ch (M(e)/N)=-D_{\pi}e^{\rho_{\pi}}
\ch (M(e)/N) \end{equation}
 for $\gamma\in \alpha,q\delta-\alpha$ (and the natural action of $s_{\gamma}$).
Since $\fg_{\pm\alpha}$ acts locally finitely on $V^k$,
$\fg_{\pm\alpha}$ acts locally finitely on  $V^k/U(\fg) \ol{v}\cong M(e)/N$, so~(\ref{eq:desired}) holds for $\gamma=\alpha$.
Thus it remain to verify~(\ref{eq:desired}) for $\gamma=q\delta-\alpha$.  The proof 
 is based on the following key observation.
 
 \subsubsection{}
 \begin{lem}{lem:atyp}
 If $w\in W$ is such that $w\not=\Id,s_{\alpha}$, then 
 $$M(w):=M(w\lambda-\rho_{\Sigma};\Sigma)=
 M(w\lambda-\rho_{r_{\beta}\Sigma},r_{\beta}\Sigma)$$
 (see~$\S$~\ref{hwtSpine} for notation).
 \end{lem}
 \begin{proof}
 Let  $w\in W$. By~$\S$~\ref{hwtSpine},
 it is enough to show that $(w\lambda,\beta)\not=0$ for $w\not=\Id,s_{\alpha}$.
 
 One has  $w\lambda=\lambda-a\delta+b\alpha$
for some $a,b\in\mathbb{C}$. The condition
$ (w\lambda ,\beta)=0$ implies $b=0$, that is
  $w\lambda=\lambda-a\delta$. Using 
  $(w\lambda,w\lambda)=(\lambda,\lambda)$
  we get $a=0$. Hence  $w\lambda=\lambda$.
  It is well-known that
 $\Stab_{W} \lambda$ is generated by 
  $s_{\gamma}$ for $\nu\in\Delta_{\pi^{\#}}$ satisfying $(\nu,\gamma)=0$.
  Since $\Delta_{\pi^{\#}}=\mathbb{Z}\delta\pm \alpha$ and
  $$(\lambda,j\delta\pm\alpha)=((k+1)\Lambda_0,j\delta\pm \alpha)=(k+1)j,$$
 we have $\Stab_{W} \lambda=\{\Id,s_{\alpha}\}$.
 This completes the proof.
  \end{proof}

\subsubsection{}
For $i\geq 0$ we have
 $$\begin{array}{lcl}
 (s_{\alpha_0}s_{\alpha})^i \alpha=-2i\delta+\alpha, & &
 (s_{\alpha_0}s_{\alpha})^i \alpha_0=(2i+1)\delta-\alpha,\\
 (s_{\alpha_0}s_{\alpha})^is_{\alpha_0} \alpha=(2i+2)\delta-\alpha, 
 & &
 (s_{\alpha_0}s_{\alpha})^i s_{\alpha_0} \alpha_0=-(2i+1)\delta+\alpha.\\
 \end{array}$$

We set $w:= (s_{\alpha}s_{\alpha_0})^j$ if $q=2j+1$ and
$w:= s_{\alpha_0}(s_{\alpha}s_{\alpha_0})^j$ if $q=2j+2$.
Then 
\begin{equation}\label{eq:wy}\begin{array}{l}
 w(q\delta-\alpha)\in \pi^{\#},\\
\pi^{\#}\cap  \Delta_{y\lambda}=\emptyset\ \text{ for }
y\in \{(s_{\alpha}s_{\alpha_0})^i,
s_{\alpha_0}(s_{\alpha}s_{\alpha_0})^i\}_{i=0}^j
\setminus\{w\}.\end{array}\end{equation}
Retain notation of~$\S$~\ref{Enright}.
Let  $\cC_w$ be the corresponding composition of Enright functors, that is
$$\begin{array}{ccllcl}
 & w=(s_{\alpha}s_{\alpha_0})^j, &   &  \cC_w:=(\cC_{\alpha}\circ \cC_{\alpha_0})^j & & \text{ if } q=2j+1 \\
& w=s_{\alpha_0}(s_{\alpha} s_{\alpha_0})^j, & &  \cC_w:=\cC_{\alpha_0}(\cC_{\alpha}\circ 
\cC_{\alpha_0})^j  & & \text{ if } q=2j+2  .\end{array}$$

\subsubsection{}
Since $\alpha_0\in\Sigma$ and $\alpha\in r_{\beta}\Sigma$,
combining~\Lem{lem:atyp} and~(\ref{eq:Enright}), 
we obtain, for $y\in W$, that 
$\cC_{\alpha_0}(M(y))=M(s_{\alpha_0}y)$ if $\alpha_0\not\in\Delta_{y\lambda}$,
 and $\cC_{\alpha}(M(y))=M(s_{\alpha}y)$ if 
 $y\not=\Id,s_{\alpha}$ and $\alpha\not\in \Delta_{y\lambda}$.
 
Since $v\in N$, we have non-zero homomorphisms
\begin{equation}\label{homom1}
M(s_{q\delta-\alpha})\longrightarrow N\ \hookrightarrow M(e).
\end{equation}
By~(\ref{eq:wy}), applying $\cC_w$, we obtain non-zero homomorphisms
$$\cC_w(M(s_{q\delta-\alpha}))\longrightarrow \cC_w(N)\ \hookrightarrow 
\cC_w(M(e))$$
with
$$\cC_w(M(e))=M(w),\ \ 
\cC_w(M(s_{q\delta-\alpha}))=
M(ws_{q\delta-\alpha})=M(s_{w(q\delta-\alpha)}w).$$
Thus we have
 non-zero homomorphisms
$$M(s_{w(q\delta-\alpha)}w)\longrightarrow \cC_w(N)\ \hookrightarrow M(w).$$
Since the root  $w(q\delta-\alpha)$ lies in $\pi^{\#}$,
this root lies in $\Sigma$ or in $r_{\beta}\Sigma$.
By~\Lem{lem:atyp}, $M(w)/M(s_{w(q\delta-\alpha)}w)$
is $\fg_{\pm w(q\delta-\alpha)}$-integrable, so
$\cC_w(M(e))/\cC_w(N)$ is $\fg_{\pm w(q\delta-\alpha)}$-integrable, which gives
$$s_{w(q\delta-\alpha)} (D_{\pi}e^{\rho_{\pi}} \ch 
\bigl(\cC_w(M(e)/N)\bigr)
=-D_{\pi}e^{\rho_{\pi}}
\ch \bigl(\cC_w(M(e)/N)\bigr).$$

Using~\Lem{lem:Enright} we deduce~(\ref{eq:desired}) for
$\gamma=q\delta-\alpha$, as required.\qed

\subsection{Remark}
\label{sl21pradm}
It is easy to see 
that any admissible weight is  principal admissible.
Indeed, let
$\lambda$ be an admissible weight.  This means that
$\lambda$ is non-critical, $L(\lambda)$ is 
$\pi^{\#}$-quasi-admissible and 
$\mathbb{Q}(\Delta_{\lambda+\rho}\cap \Delta^{\#})=
\mathbb{Q} \Delta^{\#}$. Since $\Delta^{\#}$ is the set of real 
roots of $\fsl_2^{(1)}$, the last formula
implies $\Delta_{\lambda+\rho}\cap \Delta^{\#}\cong \Delta^{\#}$.
Since $L(\lambda)$ is $\pi^{\#}$-quasi-admissible, we have
$(\lambda+\rho,q_{\pm}\delta\pm \alpha)\in\mathbb{Z}_{\geq 0}$
for some $q_-\in\mathbb{Z}_{>0}$ and $q_+\in\mathbb{Z}_{\geq 0}$.
Then $q:=q_++q_-\in\mathbb{Z}_{>0}$ and
$p:=(\lambda+\rho,q\delta)\in\mathbb{Z}_{\geq 0}$.
Since  $\lambda$ is non-critical, $p\not=0$, so 
$(\lambda,\delta)=-1+\frac{p}{q}$ is a principal admissible level.
Hence $\lambda$ is a principal admissible weight.

\appendix
\section{Root systems}\label{affinebases}
In this section $\dot{\Delta}$ or $\Delta$ is the set of roots of  an indecomposable symmetrizable  finite-dimensional or affine  Kac-Moody superalgebra
which is not a Lie algebra and is not $D(2|1,a)$ or $D(2|1,a)^{(1)}$. Our goal is to prove~\Prop{prop:Uklambda0}, 
which was used in the proof of~\Prop{prop:Vksubq}.
\Prop{prop:Uklambda0} also allows to simplify proofs of some character formulas
in~\cite{GKadm}.
By abuse of notation, we use the same notation for  the set of real roots of finite-dimensional or affine Kac-Moody superalgebra
$\fg$ and the superalgebra  itself.

\subsection{}
For finite root systems we will use the standard notation of~\cite{Ksuper}, taking into account
that $D(1|n)=C(n+1)$, and 
for the affine root systems we  use the notation of~\cite{vdLeur}.  
Note that $C(n+1)^{(2)}=D(1|n)^{(2)}$ is  anisotropic. 
Thus $\dot{\Delta}$ and
 $\Delta$  are from the following list of  root systems
$$\begin{array}{l}
A(m|n)\ (m\geq n\geq 0),\  B(m|n)\ (m\geq 0, n\geq 1),\   D(m|n) \ (m, n\geq 1, (m,n)\not=(2,1)), \\
A(m|n)^{(1)},\  B(m|n)^{(1)},\  D(m|n)^{(1)},\\
 G(3), F(4),  G(3)^{(1)}, F(4)^{(1)},\\
 A(2m|2n-1)^{(2)},\ D(m+1|n)^{(2)} \ (m\geq 0, n\geq 1),\\
 A(2m|2n)^{(4)}\ (m\geq n\geq 1), \ \ A(2m-1|2n-1)^{(2)} \ (m\geq n\geq 1, (m,n)\not=(1,1)).
 \end{array}$$
 In the second line the restrictions on $m,n$ are the same as in the first line.

\subsubsection{}\label{dotgclg}
If $\Delta$ is affine, then
we denote by $\delta$  the minimal
imaginary root and
 by $\dot{\Delta}$ the subset  of  $\Delta$ where
$\delta$ appears with zero coefficient, and
let
$\cl: \mathbb{Q}\Delta\to \mathbb{Q}\Delta/\mathbb{Q}\delta$
be the canonical map. We have

$$\begin{array}{|c||c|c|c|c|c|c|c|c|}
 \hline
\Delta &A(2m|2n-1)^{(2)}& A(2m-1|2n-1)^{(2)} &   A(2m|2n)^{(4)} &D(m+1|n)^{(2)} \\
 \hline
\dot{\Delta} & B(m|n) & D(m|n) & B(m|n)  & B(m|n) \\
 \hline
 \cl(\Delta) & BC(m|n)\cup\{0\} & C(m|n)\cup\{0\} & BC(m|n)\cup\{0\} & B(m|n)\cup\{0\}\\
 \hline
 \end{array}$$
where $BC(m|n):=B(m|n)\cup\{\pm 2\vareps_i\}_{i=1}^m$ and $C(m|n):=D(m|n)\cup\{\pm 2\vareps_i\}_{i=1}^m$ are the "weak generalized root systems" introduced in~\cite{VGRS}.

\subsubsection{}
As before,
we denote by $\Lambda_0$ the weight satisfying
$(\Lambda_0,\delta)=1$, $(\Lambda_0,\dot{\Delta})=(\Lambda_0,\Lambda_0)=0$. We will use notation
$\rho$ and $\Delta^{\#}$ 
introduced in~$\S\S$~\ref{Weylvector} and~\ref{notataff}.

As before,
we normalise the bilinear form $(-,-)$
in such a way that  
 $(\alpha,\alpha)=2$ for  a longest root $\alpha\in \Delta^{\#}$.
For the cases $A(n|n)$ and $A(n|n)^{(r)}$ the above conditions
determine the  bilinear form up to a sign (the root system admits an automorphism
which multiplies the bilinear form on $-1$); we take $(\vareps_1,\vareps_1)=1$
for  $A(n|n)^{(1)}$ and $(\delta_1,\delta_1)=1/2$
for $A(2n-1|2n-1)^{(2)}$, $A(2n|2n)^{(4)}$.

\subsubsection{Compatibility of $(-,-)$}
If the bilinear forms on $\Delta$ and on $\dot{\Delta}$ are normalised as above, then
 the restriction of $(-,-)$ from $\Delta$ to $\dot{\Delta}$
 coincides with the bilinear form on $\dot{\Delta}$ except for the pair
 $\Delta=A(2n|2n-1)^{(2)}$, $\dot{\Delta}=B(n|n)$, considered in~\ref{Bnn} below.

\subsubsection{}
For finite root systems we obtain

$$\begin{array}{|c||l|l|l|l|l|l|l|l|}
 \hline
\dot{\Delta} & A(m|n), m\geq n& B(m|n),D(m|n), m>n  &   B(m|n), D(m|n), m\leq n \\
 \hline
\dot{\Delta}^{\#} & A_m & B_m,\ \ \ \ \ D_m & C_n\\
 \hline\end{array}$$
and $\Delta^{\#}=G_2,B_3$ for $G(3), F(4)$ respectively.

For $\Delta=\dot{\Delta}^{(1)}$ we have $\Delta^{\#}=(\dot{\Delta}^{\#})^{(1)}$, for example,
 $\Delta^{\#}$ is $B_3^{(1)}$ for $B(3|1)^{(1)}$. In the remaining (i.e., twisted) cases
we have
$$\begin{array}{|c||l|l|l|l|l|l|l|l|}
 \hline
\Delta & A(m|n)^{(2)}, A(m|n)^{(4)}& D(m+1|n)^{(2)}, m>n  &   D(m+1|n)^{(2)}, m\leq n \\
 \hline
\Delta^{\#} & A_{\max(m,n)}^{(2)} & D_{m+1}^{(2)} &  C_n^{(1)}\\
 \hline\end{array}$$

In the cases $A(m|n)^{(4)}$ and $D(m+1|n)^{(2)}$ the minimal imaginary root
in $\Delta^{\#}$ is $2\delta$ (note that 
$\Delta^{\#}$ is "less twisted": the order of the automorphism for $\Delta$ is twice the order of the automorphism for $\Delta^{\#}$).
In all other cases $\delta$ is the minimal imaginary root
in $\Delta^{\#}$.

Notice that
the assumption $\Delta\not=D(2|1,a)$, $D(2|1,a)^{(1)}$ ensures
that  $\Delta^{\#}$ is  indecomposable.

\subsubsection{Nice pairs $(\Sigma,S)$}\label{pairs}
Consider the
pairs $(\Sigma,S)$, where 
$\Sigma$ is a base of simple roots and $S$ is a subset of  $\Sigma\cap \dot{\Delta}$
satisfying $(S,S)=0$ and the following properties:
\begin{itemize}
\item[(a)] $\dot{\Sigma}:=\Sigma\cap \dot{\Delta}$
is a base for $\dot{\Delta}$;
\item[(b)]
 for all $\alpha\in\Sigma\setminus S$ one has  $(\alpha,\alpha)>0$
or $(\alpha+\beta,\alpha+\beta)>0$ for some $\beta\in S$;
\item[(c)]  $(\alpha,\alpha)\geq 0$ for all $\alpha\in \Sigma$.
\end{itemize}
 They appear in many papers (including~\cite{Gnon-zero}, 
where the pair for $B(n+1|n)^{(1)}$ was missed). 

Note that (c) implies that for the Weyl 
vector $\rho$ we have $(\rho,\alpha)\geq 0$ 
for all $\alpha\in\Sigma$, so 
$$ (c)\ \ \Longrightarrow\ \ \ 
(\rho,\mu)\geq 0 \text{ for all }\mu\in \mathbb{Z}_{\geq 0}\Sigma.$$
If there is no $\Sigma$ satisfying (c) 
it is useful to have the property
\begin{itemize}
\item[(d)] ${\rho}$ can be written as
$\rho=h^{\vee}\Lambda_0+
\sum\limits_{\alpha\in\dot{\Delta}^+_{\ol{0}}} 
k_{\alpha}\alpha$ with 
$k_{\alpha}(\alpha,\alpha)\in\mathbb{Q}_{\leq 0}$.
\end{itemize}

\begin{defn}{}
The pair $(\Sigma, S)$
satisfying properties (a), (b), and (c) or (d) 
is called a {\em nice pair}. The pair $(\dot{\Sigma},\dot{S})$ 
is called a
{\em nice pair} if it satisfies properties 
similar to (b) and (c), where $\Sigma$ is replaced
by $\dot{\Sigma}$.
\end{defn}

In~$\S$~\ref{listgoodbases} below
we will give examples of nice pairs  for all affine root systems except for $D(n+1|n)^{(1)}$ and $A(2n-1|2n-1)^{(2)}$:
\begin{itemize}
\item[$\bullet$]  we give examples of pairs $(\Sigma,S)$ satisfying (a), (b), (c)
for all cases apart from  $D(n+1|n)^{(1)}$,
$A(2n-1|2n-1)^{(2)}$,
$D(n+1|n)^{(2)}$, and  $A(2n|2n)^{(4)}$;

\item[$\bullet$] for $D(n+1|n)^{(2)}$ and $A(2n|2n)^{(4)}$
we give examples of  pairs $(\Sigma,S)$ satisfying (a), (b) and (d).

\end{itemize}

For  $\Delta=D(n+1|n)^{(1)},  A(2n-1|2n-1)^{(2)}$ we give examples of  $(\Sigma,S)$
satisfying (a) and (c). In both cases the Dynkin diagram takes the form
$$
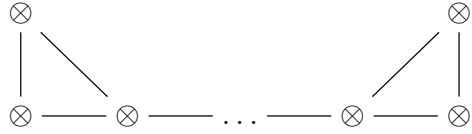
\xymatrix{
&\otimes\ar@{-}[rd]\ar@{-}[d] & & & & 
\otimes\ar@{-}[ld]\ar@{-}[d]\\
&\otimes\ar@{-}[r]&\otimes\ar@{-}[r]&
\ldots&\otimes\ar@{-}[l]&\otimes\ar@{-}[l]\\
}$$
where the number of nodes is even for $D(n+1|n)^{(1)}$
and is odd  for $A(2n-1|2n-1)^{(2)}$. In both cases for any $\alpha\in\Sigma$
there exists $\beta\in\Sigma$ such that $(\alpha+\beta,\alpha+\beta)=2(\alpha,\beta)>0$.

Note that, if (c) holds, then (b) follows from the following property:
for any isotropic $\beta'\in\Sigma\setminus S$ there exists $\beta\in S$ such that
$(\beta,\beta')>0$.

\subsection{List of nice pairs $(\Sigma,S)$}\label{listgoodbases}
For  affine root systems with $\dot{\Delta}\not=B(n|n)$ we
 take the pair $(\Sigma,S)$ of the form
$\Sigma=\dot{\Sigma}\cup\{\alpha_0\}$ (we list $\alpha_0$ for all cases) and $S=\dot{S}$
(where $(\dot{\Sigma},\dot{S})$ is a nice pair for $\dot{\Delta}$).

In the non-exceptional cases $\vareps_i$ for $i=1,\ldots,m$,
and $\delta_j$ for $j=1,\ldots,n$
are pairwise orthogonal and $\vareps_i-\delta_j$
are isotropic roots, so $(\vareps_i,\vareps_i)=-(\delta_j,\delta_j)$.

\subsubsection{$\dot{\Delta}=A(m-1|n-1)$ for $m\geq n$}
In this case $(\vareps_i,\vareps_i)=1$ for all $i$. We take 
$$\begin{array}{l}
\dot{\Sigma}_{A(n-1|n-1)}:=\{\vareps_1-\delta_1,\delta_1-\vareps_2,\ldots, \delta_{n-1}-\vareps_n\},\\
\dot{\Sigma}_{A(m-1|n-1)}:=\dot{\Sigma}_{A(n-1|n-1)}\cup \{\delta_n-\vareps_{n+1},\ldots, 
\vareps_{m-1}-\vareps_m\}, \ \text{ for } m>n ,\end{array}$$  
and $S:=\{\vareps_i-\delta_i\}_{i=1}^n$.
For $A(m-1|n-1)^{(1)}$ we have 
$\alpha_0=\delta-\vareps_1+\delta_n$ if $m=n$ and
$\alpha_0=\delta-\vareps_1+\vareps_m$ if $m>n$, with the same $S$. Properties (a), (b) and (c)  hold.

\subsubsection{$\dot{\Delta}=B(m|n)$ for $m\geq n+2$, $\dot{\Delta}=D(m|n)$, $m\geq n+3$}
In this case $(\vareps_i,\vareps_i)=1$ for all $i$.
We take
$$\begin{array}{rl}
\dot{\Sigma}_{B(m|n)}:=&\{\vareps_1-\vareps_2,\vareps_2-\vareps_3,
\ldots,\vareps_{m-n-1}-\vareps_{m-n},\vareps_m\}\cup\\
&
\{\beta_1:=\vareps_{m-n}-\delta_1,
\beta_2:=\delta_1-\vareps_{m-n+1},\ldots, \beta_{2n}:=\delta_n-\vareps_{m}\},\\
\dot{\Sigma}_{D(m|n)}:=&\{\vareps_1-\vareps_2,\vareps_2-\vareps_3,
\ldots,\vareps_{m-n-2}-\vareps_{m-n-1}\}\cup\\
&
\{\beta_1:=\vareps_{m-n-1}-\delta_1,
\beta_2:=\delta_1-\vareps_{m-n},\ldots, \beta_{2n}:=\delta_n-\vareps_{m-1},
\vareps_{m-1}+\vareps_m\}
\end{array}$$
and $S:=\{\beta_{2i-1}\}_{i=1}^n$.
For the corresponding affine root systems $\alpha_0$ is given by the following table
$$\begin{array}{|c|l|}
\hline 
\Delta & \alpha_0\\
\hline
 B(m|n)^{(1)}\ m>n+1,\ \  D(m|n)^{(1)}\ m>n+1 & \delta-\vareps_1-\vareps_2\\
 \hline
A(2m|2n-1)^{(2)}\ m>n,\ \  A(2m-1|2n-1)^{(2)}\ m>n+1 & \delta-2\vareps_1\\
\hline
A(2m|2n)^{(4)}\ m>n,\ \  D(m+1|n)^{(2)}\  m>n  & \delta-\vareps_1\\
\hline
\end{array}$$
  Note that 
  $\delta$ is odd for $A(2m|2n)^{(4)}$ and even in other cases.

Properties (a) and (c) obviously hold.
The simple isotropic roots are  $\beta_1,\ldots,\beta_{2n}$
with $\beta_i\in S$  and $(\beta_i,\beta_{i+1})>0$ if $i$ is odd.
Thus property (b) holds.

\subsubsection{$\dot{\Delta}=B(m|n)$, $D(m|n)$ for $m< n$}
In this case $(\delta_j,\delta_j)=\frac{1}{2}$ for all $j$.
We take
$$\begin{array}{rl}
\dot{\Sigma}_{B(m|m+1)}, \dot{\Sigma}_{D(m|m+1)}:=&
\{\beta_1:=\delta_1-\vareps_1,\beta_2:=\vareps_1-\delta_2,\ldots, \beta_{2m}:=\vareps_{m}-\delta_{m+1}; a\delta_{m+1}\}\\
\dot{\Sigma}_{B(m|n)}, \dot{\Sigma}_{D(m|n)}:=&
\{\delta_1-\delta_2,\delta_2-\delta_3,\ldots,
\delta_{n-m-1}-\delta_{n-m},a\delta_n\}\cup\\
&\{
\beta_1:=\delta_{n-m}-\vareps_1,\beta_2:=\vareps_1-\delta_{n-m+1},\ldots,
\beta_{2m}:=\vareps_m-\delta_n\},\\
S:&=\{\beta_{2i}\}_{i=1}^m.
\end{array}$$
for $m<n-1$,  where $a=1$ for $B(m|n)$ and $a=2$ for $D(m|n)$.
For the corresponding affine root systems $\alpha_0$ is given by the following table
$$\begin{array}{|c|l|}
\hline 
\Delta & \alpha_0\\
\hline
 B(m|n)^{(1)},\ D(m|n)^{(1)} &  \delta-2\delta_1\\
 \hline
 D(m+1|n)^{(2)} & \delta-\delta_1\\
 \hline
 A(2m|2n-1)^{(2)}, m\not=n-1  & \delta-\delta_1-\delta_2 \\
 \hline
 A(2m|2m+1)^{(2)} & \delta-\vareps_1-\delta_1\\
 \hline\end{array}$$
(where $m<n$ in all cases).

Properties (a) and (c) obviously hold.
The simple isotropic roots  are 
are  $\beta_1,\ldots,\beta_{2m}$, and $\delta-\vareps_1-\delta_1$
for $A(2n-2|2n-1)^{(2)}$. Since
 $\beta_i\in S$, $(\beta_i,\beta_{i-1})>0$ if $i$ is even, and
 $(\delta-\vareps_1-\delta_1,\beta_2)>0$, property (b) holds.

For example, for $(m,n)=(1,2)$ we obtain  the following sets of simple roots
$$\begin{array}{lcl}
D(1|2)^{(1)} & & \{\delta-2\delta_1, \delta_1-
\vareps_1,\vareps_1-\delta_2, 2\delta_2\}\\
B(1|2)^{(1)} & & \{\delta-2\delta_1, \delta_1-
\vareps_1,\vareps_1-\delta_2,\delta_2\}\\
A(2|3)^{(2)} & & \{\delta-\delta_1-\vareps_1, \delta_1-
\vareps_1,\vareps_1-\delta_2,\delta_2\}\\
D(2|2)^{(2)} & & \{\delta-\delta_1, \delta_1-
\vareps_1,\vareps_1-\delta_2,\delta_2\}\\
\end{array}$$
with $S=\{\vareps_1-\delta_2\}$.

\subsubsection{$\dot{\Delta}=B(n+1|n), D(n+2|n)$}
In this case $(\vareps_i,\vareps_i)=1$ for all $i$.
We set 
$$\begin{array}{l}
\beta_1:=\vareps_{1}-\delta_1,\ 
\beta_2:=\delta_1-\vareps_{2},\ldots, \beta_{2n}:=\delta_n-\vareps_{n+1},\\
\dot\Sigma_{B(n+1|n)}:=\{\beta_i\}_{i=1}^{2n}\cup \{\vareps_{n+1}\},\ \ \
\dot\Sigma_{D(n+2|n)}:=\{\beta_i\}_{i=1}^{2n}\cup \{
\vareps_{n+1}\pm\vareps_{n+2}\}\\
S:=\{\beta_{2i}\}_{i=1}^n.
\end{array}$$
For the corresponding affine root systems $\alpha_0$ is given by the following table
$$\begin{array}{|c|l|}
\hline 
\Delta & \alpha_0\\
\hline
 B(n+1|n)^{(1)},\ \  D(n+2|n)^{(1)} & \delta-\vareps_1-\delta_1\\
 \hline
A(2n+2|2n-1)^{(2)}, \ \  A(2n+3|2n-1)^{(2)} & \delta-2\vareps_1\\
\hline
A(2n+2|2n)^{(4)},\ \  D(n+2|n)^{(2)}  & \delta-\vareps_1\\
\hline
\end{array}$$
  Note that 
  $\delta$ is odd for $A(2n+2|2n)^{(4)}$ and even in other cases.

Properties (a) and (c) obviously hold.
The simple isotropic roots  are 
are  $\beta_1,\ldots,\beta_{2m}$, and $\delta-\vareps_1-\delta_1$
for $B(n+1|n)^{(1)}$, $D(n+2|n)^{(2)}$. Since 
 $\beta_i\in S$, $(\beta_i,\beta_{i+1})>0$ if $i$ is even, 
 and  $(\delta-\vareps_1-\delta_1,\beta_2)>0$ property (b) holds.

\subsubsection{$\dot{\Delta}=D(n+1|n)$}\label{Dn+1n}
In this case $(\vareps_i,\vareps_i)=1$ for all $i$.
We take
$$\dot{\Sigma}_{D(n+1|n)}=\{\beta_1:=\vareps_1-\delta_1,\beta_2:=\delta_1-\vareps_2,\ldots, 
\beta_{2n}:=\delta_n-\vareps_{n+1},
\beta_{2n+1}:=\delta_n+\vareps_{n+1},
\}$$
and $S:=\{\beta_{2i-1}\}_{i=1}^n$.
One has $\alpha_0=\delta-\vareps_1-\delta_1$
for $D(n+1|n)^{(1)}$ and $\alpha_0=\delta-2\vareps_1$
for $A(2n+1|2n-1)^{(2)}$.
Properties (a) and  (c) hold in both cases, and (b) holds for $A(2n+1|2n-1)^{(2)}$, so the pair $(\Sigma, S)$ is nice  in this case (and is not nice for
$D(n+1|n)^{(1)}$).

\subsubsection{$\dot{\Delta}=B(n|n)$, $\Delta\not=A(2n|2n-1)^{(2)}$}\label{Bnn}
In this case $(\delta_j,\delta_j)=\frac{1}{2}$ 
for all $j$. We take
$$\dot{\Sigma}_{B(n|n)}=\{\beta_1:=\vareps_1-\delta_1,\beta_2:=\delta_1-\vareps_2,\ldots, 
\beta_{2n-1}:=\vareps_n-\delta_n;\delta_n\}$$
and $S:=\{\beta_{2i-1}\}_{i=1}^n$.
Then $\alpha_0=\delta-\vareps_1-\delta_1$
for $B(n|n)^{(1)}$ and $\alpha_0=\delta-\vareps_1$ for 
$A(2n|2n)^{(4)}$, $D(n+1|n)^{(2)}$. Properties (a), (b) 
hold, since $\beta_i\in S$ and $\beta_i+\beta_{i-1}$
if $i$ is odd, and $(\delta-\vareps_1-\delta_1,\beta_1)>0$.
Property (c) holds for $B(n|n)$
and $B(n|n)^{(1)}$.

Note that in all affine cases we have 
$\rho=h^{\vee}\Lambda_0-
\frac{1}{2}\sum\limits_{\beta\in S}\beta$
with $h^{\vee}=\frac{1}{2}$ for $B(n|n)^{(1)}$,  and $h^{\vee}=0$ for $D(n+1|n)^{(2)}$ and
$A(2n|2n)^{(4)}$. If (c) does not hold, then $h^{\vee}=0$, so 
$\rho=\frac{1}{2}\sum\limits_{i=1}^n (\vareps_i-\delta_i)$.
Since $\vareps_i,\delta_i\in\dot{\Delta}^+$
and $(\delta_i,\delta_i)=-(\vareps_i,\vareps_i)>0$,
property (d) holds. 

\subsubsection{
$\Delta=A(2n|2n-1)^{(2)}$}
In this case $(\vareps_i,\vareps_i)=\frac{1}{2}$ for all $j$.
We take
$$\dot{\Sigma}:=\{\delta_1-\vareps_1,\vareps_1-\delta_2,\ldots, 
\delta_n-\vareps_n,\vareps_n\},\  \alpha_0=\delta-\delta_1-\vareps_1,\ \ 
S=\{\delta_i-\vareps_i\}_{i=1}^n.$$
Properties (a), (b) and (c) hold.
Note that 
$\rho=h^{\vee}\Lambda_0-
\frac{1}{2}\sum\limits_{\beta\in S}\beta$
with   $h^{\vee}=1$.

\subsubsection{$\dot{\Delta}=D(n|n)$}\label{A2n-12n-1}
In this case $(\delta_j,\delta_j)=\frac{1}{2}$ for all $j$.
We take
$$\dot{\Sigma}_{D(n|n)}=\{\delta_1-\vareps_1,\vareps_1-\delta_2,
\ldots, \delta_n-\vareps_n,\delta_n+\vareps_n\}$$
and $S=\{\delta_i-\vareps_i\}_{i=1}^n$. 

For $D(n|n)^{(1)}$  we have 
$\alpha_0=\delta-2\delta_1$  and the
properties (a)--(c) are satisfied, so $(\Sigma,S)$
is a nice pair. For $A(2n-1|2n-1)^{(2)}$ one has
$\alpha_0=\delta-\vareps_1-\delta_1$ and (b) does not hold
(property (c) holds), so 
the pair $(\Sigma, S)$ is not nice. For example, for $m=n=2$ we have $S=\{\vareps_i-\delta_i\}_{i=1}^2$ and
$$\begin{array}{lcl}
D(2|2)^{(1)} & & \{\delta-2\delta_1, \delta_1-
\vareps_1,\vareps_1-\delta_2,\delta_2\pm\vareps_2\}\\
A(3|3)^{(2)} & & \{\delta-\vareps_1-\delta_1, \delta_1-
\vareps_1,\vareps_1-\delta_2,\delta_2\pm\vareps_2\}
\end{array}$$
For $D(2|2)^{(1)}$
the simple roots with non-positive square length are $\beta_i:=\delta_i-\vareps_i$ 
for $i=1,2$ and $\beta'_1=\vareps_1-\delta_2$, $\beta'_2=\delta_2+\vareps_2$.
Since $\beta_i\in S$ and $\beta_i+\beta'_i$
is of positive square length,  property (b) holds.

\subsubsection{Exceptional cases}
We take 
$$\begin{array}{l}
\dot{\Sigma}_{G(3)}:=\{\vareps_3,\delta_1-\vareps_3,-\delta_1+\vareps_2\},\\
\dot{\Sigma}_{F(4)}:=\{\vareps_3,\frac{1}{2}(\vareps_1+\vareps_2-\vareps_3-\delta_1),
\frac{1}{2}(-\vareps_1+\vareps_2-\vareps_3-\delta_1),\vareps_1-\vareps_2\},
\end{array}$$
and $\alpha_0=\delta+\vareps_1-\vareps_2$ for $G(3)^{(1)}$,
and $\alpha_0=\delta-\vareps_1-\vareps_2$ for $F(4)^{(1)}$.
In these cases $\Sigma$ contains exactly two isotropic roots
$\beta_1,\beta_2$ and $(\beta_1,\beta_2)>0$; all 
 other roots are of positive square length. Hence the pair $(\Sigma,S)$ for $S=\{\beta_1\}$
satisfies properties (a)--(c), so it is nice.

\subsection{}
\begin{prop}{prop:Uklambda0}
Let $\Delta$ be such that  $\dot{\Sigma}$
  and  $\Sigma$ are as in~\ref{listgoodbases}. Let $\rho$ be the Weyl vector for $\Sigma$. Assume that 
$k\in\mathbb{Z}_{\geq 0}$ and $\nu\in \mathbb{Z}_{\geq 0}\Sigma$ are such that
\begin{itemize}
\item[1.] $(k\Lambda_0+\rho,\delta)\not=0$;
\item[2.]
$\langle k\Lambda_0-\nu,\alpha^{\vee}\rangle\geq 0$ for all $\alpha\in\dot{\Delta}^+_{\ol{0}}$
and for all $\alpha\in\Delta^+$ such that $(\alpha,\alpha)>0$;
\item[3.]  $2(k\Lambda_0+\rho,\nu)=(\nu,\nu)$. 
\end{itemize}
Then $\nu\in \mathbb{Z}_{\geq 0}\dot{\Sigma}$, and, moreover,
$\nu\in\mathbb{Z}_{\geq 0}S$ if $\fg\not=A(2n-1|2n-1)^{(2)}, D(n+1|n)^{(1)}$.
\end{prop}

The proof occupies~$\S\S$~\ref{pf1}--\ref{pfA2n-12n-1} below.

\subsubsection{}\label{pf1}
Assume that $\mu,\xi$ are such that 
\begin{equation}\label{eq:muxi}\begin{array}{l}
\nu=\mu+\xi,\ \ \mu\in \mathbb{Q}_{\geq 0}\Delta^{\#+},\ \ \xi\in\mathbb{Z}\dot{\Sigma},\ \ 
(\xi,\rho)=0,\ \ (\xi,\xi)\leq 0,\\
\xi\in\mathbb{Z}_{\geq 0}S \ \text{ if } \ \fg\not=A(2n-1|2n-1)^{(2)}, D(n+1|n)^{(1)}.
\end{array}
\end{equation}

Assumption 3. on $k$ gives 
$2(k\Lambda_0+\rho,\mu)=(\mu+\xi,\mu+\xi)$, so
$$2(k\Lambda_0+\rho-\nu,\mu)=(\mu+\xi,\mu+\xi)-2(\mu+\xi,\mu)
=(\xi,\xi)-(\mu,\mu).$$
Since $\mu\in \mathbb{Q}_{\geq 0}\Delta^{\#+}$, 
the second assumption
implies $(k\Lambda_0-\nu,\mu)\geq 0$, so
\begin{equation}\label{eq:mumu}
(\mu,\mu)+2(\rho,\mu)\leq (\xi,\xi)\leq 0.
\end{equation}

Recall that $(\Sigma,S)$ satisfies (c) or (d).
If  (c) holds, then $(\rho,\mu)\geq 0$.
The second assumption gives $(\alpha,\alpha)(\nu,\alpha)\leq 0$ 
for all $\alpha\in\dot{\Delta}^+_{\ol{0}}$, so (d) forces
$(\nu,\rho)\geq 0$ which again gives $(\mu,\rho)\geq 0$.

Since $(\rho,\mu)\geq 0$, the inequality~(\ref{eq:mumu}) gives
$(\mu,\mu)\leq 0$. Since $\Delta^{\#}$ is  the  root system 
of  affine algebra, this forces $\mu=s\delta$ for  
$s\geq 0$. Then  $2(k\Lambda_0+\rho,\mu)=(\mu+\xi,\mu+\xi)$
implies $s(k\Lambda_0+\rho,\delta)\leq 0$. Since $k\geq 0$
and $(k\Lambda_0+\rho,\delta)\not=0$, this forces $s=0$.
Therefore $\mu=0$ and $\nu=\xi\in \mathbb{Z}\dot{\Sigma}$ as required.

It remains to verify the existence  $\mu,\xi$
satisfying~(\ref{eq:muxi}). We do this
in~$\S\S$~\ref{pf2}--\ref{pfA2n-12n-1}.

\subsubsection{}\label{pf2}
Consider the case 
$\Delta\not=A(2n-1|2n-1)^{(2)}, D(n+1|n)^{(1)}$.
Then $(\Sigma,S)$ satisfies properties (a), (b).
We claim that (b) implies the existence of a linear map
$p_{\#}:\ \mathbb{Z}\Delta\to \mathbb{Q}\Delta^{\#}$
with the properties
\begin{itemize}
\item[$\bullet$] $\Ker p_{\#}=\mathbb{Z}S$;
\item[$\bullet$] $p_{\#}(\mathbb{Z}_{\geq 0}\Delta^+) \subset\mathbb{Q}_{\geq}\Delta^{\#+}$.
\end{itemize}

We construct  $p_{\#}$ as follows: for $\alpha\in S$ we set
$p_{\#}(\alpha):=0$; for $\alpha\in\Sigma$ such that
$(\alpha,\alpha)>0$ we take $p_{\#}(\alpha):=\alpha$.
In the remaining case $\alpha\in\Sigma\setminus S$ such that 
$(\alpha,\alpha)\leq 0$,
 there exists $\beta\in S$ such that $(\alpha+\beta,\alpha+\beta)>0$.
In particular, $(\alpha,\beta)\not=0$, so  $\alpha+\beta$
or $\alpha-\beta$ lies in $\Delta$. Since $\alpha,\beta\in\Sigma$
we obtain $\alpha+\beta\in\Delta^+$. 
Set $p_{\#}(\alpha):=\alpha+\beta$.
If $\alpha$ is odd, then
$p_{\#}(\alpha)\in\Delta^{\#+}$;
if $\alpha$ is even, then $2p_{\#}(\alpha)\in\Delta^{\#+}$.
Extending this map by linearity we obtain $p_{\#}$ satisfying the above properties.

Now set $\mu:=p_{\#}(\nu)$ and $\xi:=\nu-\mu$. The properties of $p_{\#}$ give
$\mu\in  \mathbb{Q}_{\geq}\Delta^{\#+}$
and $\xi\in\mathbb{Z}S$. Property (b) gives $\xi\in\mathbb{Z}\dot{\Delta}$ and
$(\xi,\rho)=(\xi,\xi)=0$.
Therefore $(\mu,\xi)$ satisfies~(\ref{eq:muxi}) as required.

\subsubsection{}\label{pfDn+1n}
Consider the  case  $D(n+1|n)^{(1)}$. We choose 
$$\Sigma=\{\delta-\vareps_1-\delta_1,\vareps_1-\delta_1,\delta_1-\vareps_2,\ldots,\vareps_n-\delta_n,\delta_n\pm\vareps_{n+1}\}.$$ 
We draw the diagram for $\Sigma$ marking the edge between $\beta_i$ and $\beta_j$  by $(\beta_i,\beta_j)$. 
$$\xymatrix{
&\beta_0\ar@{-}[rd]^1\ar@{-}[d]_{-2} & & & & 
\beta_{2n+1}\ar@{-}[ld]_1\ar@{-}[d]^{-2}\\
&\beta_1\ar@{-}[r]^1&\beta_2\ar@{-}[r]^{-1}&
\ldots&\beta_{2n-1}\ar@{-}[l]_{-1}&\beta_{2n}\ar@{-}[l]_1\\
}$$
Since all $\beta_i$ are isotropic, we have $\rho=0$. Moreover,
 $(\beta_i+\beta_j)\in\Delta^{\#+}$ if $(\beta_i,\beta_j)>0$, so
 $\Delta^{\#+}$ contains $\beta_0+\beta_2$, $\beta_{2n-1}+\beta_{2n+1}$ and 
 $\beta_{2i-1}+\beta_{2i}$ for $i=1,\ldots,n$.

We have $\nu=\sum\limits_{i=0}^{2n+1} k_i\beta_i$ where $k_i\in\mathbb{Z}_{\geq 0}$ for all $i$.
If $k_{2n}> k_{2n-1}$ we take
$$\begin{array}{l}
\mu:=k_0(\beta_0+\beta_2)+\sum\limits_{i=1}^n k_{2i-1}(\beta_{2i-1}+\beta_{2i})\\
\xi:=(k_2-k_1-k_0)\beta_2+\sum\limits_{i=2}^{n} (k_{2i}-k_{2i-1})\beta_{2i}+
k_{2n+1}\beta_{2n+1}
\end{array}$$
Then $\nu=\mu+\xi$ and 
$(\xi,\xi)=-4(k_{2n}-k_{2n-1})k_{2n+1}\leq 0$;
the pair $(\mu,\xi)$ satisfies~(\ref{eq:muxi}).

If  $k_{2n}\leq k_{2n-1}$ we take
$$\begin{array}{rl}\mu:=&
 k_0(\beta_0+\beta_2)+\sum\limits_{i=1}^{n-1} k_{2i-1}(\beta_{2i-1}
 +\beta_{2i})+k_{2n}(\beta_{2n-1}+\beta_{2n})\\
 &\ \ \ \ \ \ \ \ 
 +(k_{2n-1}-k_{2n})(\beta_{2n-1}+\beta_{2n+1})\\
\xi:=&(k_2-k_1-k_0)\beta_2+\sum\limits_{i=2}^{n-1} (k_{2i}-k_{2i-1})\beta_{2i}+(k_{2n+1}+k_{2n}-k_{2n-1})\beta_{2n+1}\end{array}$$
 We have $\nu=\mu+\xi$ and
$(\xi,\xi)=0$; the pair $(\mu,\xi)$ satisfies~(\ref{eq:muxi}).

\subsubsection{}\label{pfA2n-12n-1}
Consider the  remaining case  $A(2n-1|2n-1)^{(2)}$ for $n>1$. We choose 
$$\Sigma=\{\delta-\vareps_1-\delta_1,\delta_1-\vareps_1,\vareps_1-\delta_2,\ldots,
\vareps_{n-1}-\delta_n,\delta_n\pm\vareps_{n}\}.$$ 
We draw the diagram for $\Sigma$ marking the edge between $\beta_i$ and $\beta_j$  by $(\beta_i,\beta_j)$. 
$$\xymatrix{
&\beta_0\ar@{-}[rd]^{1/2}\ar@{-}[d]_{-1} & & & & 
\beta_{2n}\ar@{-}[ld]_{-1/2}\ar@{-}[d]^{1}\\
&\beta_1\ar@{-}[r]^{1/2}&\beta_2\ar@{-}[r]^{-1/2}&
\ldots&\beta_{2n-2}\ar@{-}[l]_{1/2}&\beta_{2n-1}\ar@{-}[l]_{-1/2}\\
}$$
Since all $\beta_i$ are isotropic, we have $\rho=0$. Moreover,
 $(\beta_i+\beta_j)\in\Delta^{\#+}$ if $(\beta_i,\beta_j)>0$, so
 $\Delta^{\#+}$ contains $\beta_0+\beta_2$ and 
 $\beta_{2i-1}+\beta_{2i}$ for $i=1,\ldots,n$.

We have $\nu=\sum\limits_{i=0}^{2n} k_i\beta_i$ where $k_i\in\mathbb{Z}_{\geq 0}$ for all $i$.
Take
$$\begin{array}{l}
\mu_+:=k_0(\beta_0+\beta_2)+\sum\limits_{i=1}^{n} k_{2i-1}(\beta_{2i-1}+\beta_{2i})\\
\xi_+=\nu-\mu_+=(k_2-k_1-k_0)\beta_2+
\sum\limits_{i=2}^{n-1} (k_{2i}-k_{2i-1})\beta_{2i}+(k_{2n}-k_{2n-1})\beta_{2n}\\
\mu_-:=k_0(\beta_0+\beta_2)+\sum\limits_{i=1}^{n-1} k_{2i-1}(\beta_{2i-1}+\beta_{2i})+k_{2n}(\beta_{2n-1}+\beta_{2n})\\
\xi_-=\nu-\mu_-=(k_2-k_1-k_0)\beta_2+
\sum\limits_{i=2}^{n-1} (k_{2i}-k_{2i-1})\beta_{2i}
+(k_{2n-1}-k_{2n})\beta_{2n-1}.
\end{array}
$$
The pairs $(\mu_+,\xi_+)$, $(\mu_-,\xi_-)$
satisfy~(\ref{eq:muxi}) if 
$(\xi_+,\xi_+)\leq 0$, $(\xi_-,\xi_-)\leq 0$
respectively. Since $(\xi_+,\xi_+)=-(\xi_-,\xi_-)$
at least one of these pairs satisfies~(\ref{eq:muxi}).

This completes the proof of~\Prop{prop:Uklambda0}.
\qed

\subsection{Simplified proof of character formulas}
\Prop{prop:Uklambda0} allows to simplify several proofs of character formulas for $(\Sigma,S)$ listed in~$\S$~\ref{listgoodbases}.

\subsubsection{}
Note that 
$$(\alpha,\alpha)\geq 0\ \text{ for all } \alpha\in\dot{\Sigma}.$$
The sets 
$$\Delta^{\#}:=\{\alpha\in\Delta_{\ol{0}}|\ 
(\alpha,\alpha)>0\}, \ \ 
\dot{\Delta}^{\#}:=\Delta^{\#}\cap \dot{\Delta}$$
are the root systems of a 
Kac-Moody algebra (affine or finite-dimensional).
(Notice that for $\dot{\Delta}=B(n|n)$,
 $\dot{\Delta}^{\#}$ is of type $C_n$ except for the case when
 $\Delta=A(2n|2n-1)^{(2)}$, where $\dot{\Delta}^{\#}$ is of type $B_n$).
Let $\pi^{\#}$, $\dot{\pi}^{\#}$ be the set of simple roots for
$\Delta^{\#+}$, $\dot{\Delta}^{\#+}$, and
$W^{\#}$, $\dot{W}^{\#}$  be the Weyl groups
of $\Delta^{\#}$, $\dot{\Delta}^{\#}$.

\subsubsection{}
\begin{lem}{lem:AAA}
\begin{enumerate}
\item
If $\fg$ is finite-dimensional, then for $w\in W^{\#}$
one has $\rho-w\rho\in\mathbb{Z}_{\geq 0} \Sigma$.
\item Let $\fg$ be affine and 
$k\in\mathbb{Z}_{\geq 0}$ be such that $k\not=-h^{\vee}$.
For $\lambda:=k\Lambda_0+\rho$ and any  $w\in W^{\#}$
one has $\lambda-w\lambda\in\mathbb{Z}_{\geq 0} \pi^{\#}$.
\end{enumerate}
\end{lem}
\begin{proof}
The group  $W^{\#}$ is   
generated by $s_{\gamma}$ for
$\gamma\in\pi^{\#}$.  A standard reasoning shows that
for $\nu\in\fh^*$ such that $(\nu,\gamma)\in \mathbb{Z}_{\geq 0}$ 
for all $\gamma\in\pi^{\#}$ and any $w\in W^{\#}$,
 $\nu-w \nu\in \mathbb{Z}_{\geq 0}\pi^{\#}$.

If (c) holds (i.e., $(\alpha,\alpha)\geq 0$ for all
$\alpha\in\Sigma$),
then $(\rho,\alpha)\geq 0$ for all $\alpha\in\Delta^+$, so
$\lambda-w\lambda\in \mathbb{Z}_{\geq 0}\pi^{\#}$
for $\nu=\rho$, and $\nu=k\Lambda_0+\rho$
if $\fg$ is affine and $k\in\mathbb{Z}_{\geq 0}$.

If (c) does not hold, then 
  $\fg=D(n+1|n)^{(2)}$, $A(2n|2n)^{(4)}$.  Since 
$(\alpha,\alpha)\geq 0$ for all $\alpha\in\dot{\Sigma}$,
we have  $(\rho,\gamma)\geq 0$ for all $\gamma\in \dot{\Sigma}$.
One has $\pi^{\#}\setminus \dot{\Sigma}=\{\delta-\delta_1\}$ and
 $(\rho, \delta-\delta_1)=-\frac{1}{4}$. Since 
 $h^{\vee}=0$ the assumptions on $k$ give $k\geq 1$, so 
 $(k\Lambda_0+\rho,\delta-\delta_1)>0$.
 Therefore $(k\Lambda_0+\rho,\alpha)\geq 0$ for all $\alpha\in\pi^{\#}$.
Hence $\lambda-w\lambda\in \mathbb{Z}_{\geq 0}\pi^{\#}$
for $\lambda=k\Lambda_0+\rho$.
\end{proof}

\subsubsection{}
The character formula~(\ref{vacform}) below was established in~\cite{GKadm}.
In~\cite{GKadm}   we considered two separate cases: 
the proof in Section 4 works for $h^{\vee}\not=0$
and for $A(n|n)^{(1)}$, and a more complicated arguments
in Section 6 work for the rest of the cases.
The proof below is a slight simplification
of the proof in Section 4  and  is very different 
from the proof in Section 6 loc. cit..
We think that a similar method may simplify
other proofs in~\cite{GKadm} and might allow to generalize
some of the results.

\subsubsection{}
\begin{prop}{prop:Vk}
Let $\fg\not=D(2|1,a)^{(1)}$ be 
a symmetrizable affine Kac-Moody superalgebra and 
$k\in\mathbb{Z}_{\geq 0}$ be such that $k+h^{\vee}\not=0$.
Then 
\begin{equation}\label{vacform}
De^{\rho}\ch L(k\Lambda_0)=\sum\limits_{w\in W^{\#}} sgn(w)\ \frac{e^{w(k\Lambda_0+\rho)}}{\prod\limits_{\beta\in S} (1+e^{-w\beta})},
\end{equation}
where $W^{\#}$ is the Weyl group of $\Delta^{\#}$,
$(\Sigma, S)$ is as in Appendix~\ref{affinebases} and
$(1+e^{-\alpha})^{-1}=\sum\limits_{i=0}^{\infty} (-1)^i 
e^{-i\alpha}$, $(1+e^{\alpha})^{-1}=\sum\limits_{i=1}^{\infty} (-1)^{i-1} 
e^{-i\alpha}$
for $\alpha\in\Delta^+$.
\end{prop}
\begin{proof}
By~\Lem{lem:AAA} (ii) we have
\begin{equation}\label{eq:wlambda}
 w(k\Lambda_0+\rho)\in (k\Lambda_0+\rho)-
\mathbb{Z}_{\geq 0}\pi^{\#}\end{equation}
for any $w\in W^{\#}$. We set
$$\begin{array}{l}
P(\Sigma):=\prod\limits_{\alpha\in \Delta^+_{\ol{0}}}(1-e^{-\alpha})
^{\dim \fg_{\alpha}}e^{\rho_{\pi}},\ \ Q(\Sigma):=
\prod\limits_{\alpha\in \Delta^+_{\ol{1}}}
(1+e^{-\alpha})^{\dim \fg_{\alpha}}e^{\rho_{\pi}-\rho},\\
Z_1:=De^{\rho}\ch L(k\Lambda_0),\ \ \ 
Z_w:=\sgn(w)\frac{e^{w(k\Lambda_0+\rho)}}
{\prod\limits_{\beta\in S} (1+e^{-w\beta})},\ \ 
Z_2:=\sum\limits_{w\in W^{\#}} Z_w\\
Y_i:=Q(\Sigma)Z_i\ \text{ for }i=1,2.
\end{array}$$
One easily sees that $P(\Sigma)$, $Q(\Sigma)$ and
$Z_1$ lie in $\cR(\Sigma)$ (see~$\S$~\ref{triangular}
for notation) and
$\supp(Z_1)\subset (k\Lambda_0+\rho-\mathbb{Z}_{\geq 0}\Sigma)$.
Moreover, $Z_2\in \cR(\Sigma)$ by~\cite{GKadm}, Lemma 2.2.8.
One has
\begin{equation}\label{eq:suppZw}
\supp (Z_w)\subset w(k\Lambda_0+\rho)+\sum\limits_{\beta\in S: w\beta\in-\Delta^+} w\beta\subset w(k\Lambda_0+\rho)-\mathbb{Z}_{\geq 0}\Sigma.\end{equation}
Using~(\ref{eq:wlambda}) we get 
 $\supp(Z_2)\subset (k\Lambda_0+\rho-\mathbb{Z}_{\geq 0}\Sigma)$,
so 
\begin{equation}\label{eq:Z1Z2}
\supp(Z_1-Z_2)\subset (k\Lambda_0+\rho-\mathbb{Z}_{\geq 0}\Sigma).
\end{equation}

Let
$$\pi:=\pi^{\#}\cup \dot{\Sigma}_{\pr}=
\{\alpha\in\Sigma_{\pr}|\ \alpha\in\dot{\Delta} \text{ or }
(\alpha,\alpha)>0\}$$
and let  $W[\pi]$  be the  group generated by $s_{\gamma}$ for $\gamma\in\pi$. 
We claim that 
$P(\Sigma)$
is naturally $W[\pi]$-anti-invariant and 
$Q(\Sigma)$ is  naturally $W[\pi]$-invariant.
Since $W[\pi]$ is generated by $s_{\gamma}$ 
for $\gamma\in\pi$, it is enough to verify 
that for each $\gamma\in\pi$
$P(\Sigma)$, $Q(\Sigma)$ are naturally
$s_{\gamma}$-anti-invariant and $s_{\gamma}$-invariant
respectively. It is easy to see that this holds
if $\gamma$ or $\gamma/2$ lie in $\Sigma$.
Since $\pi\subset\Sigma_{\pr}$, for every
$\gamma\in\pi$ there exists  $\Sigma_1\in\Sp$ containing
 $\gamma$ or $\gamma/2$. One readily sees that 
 $P(\Sigma_2)=P(r_{\beta}\Sigma_2)$ and $Q(\Sigma_2)=Q(r_{\beta}\Sigma_2)$
 for any $\Sigma_2\in \Sp$  and an odd reflexion 
 $r_{\beta}$. Therefore 
$P(\Sigma)=P(\Sigma_1)$ and $Q(\Sigma)=Q(\Sigma_1)$.
This implies the claim.

Since the module $L:=L(k\Lambda_0)$ is  $\pi$-integrable,
$\ch L$ is naturally $W[\pi]$-invariant. Moreover,
  $De^{\rho}=\frac{P(\Sigma)}{Q(\Sigma)}$, so
$Y_1=Q(\Sigma)Z_1=P(\Sigma)\ch L$
is  naturally $W[\pi]$-anti-invariant.

Set 
$$\dot{Z}_2:=\sum\limits_{w\in W[\dot{\pi}]} Z_w.$$
Since $\rho-\dot{\rho}=h^{\vee}\Lambda_0$ is $W[\dot{\pi}^{\#}]$-invariant we have
$$\dot{Z}_2=e^{\rho-\dot{\rho}}\sum\limits_{w\in W[\dot{\pi}^{\#}] } \sgn(w)\ \frac{e^{w(k\Lambda_0+\dot{\rho})}}{\prod\limits_{\beta\in S} (1+e^{-w\beta})}.
$$

One has
$\dot{\Sigma}_{\pr}:=\dot{\pi}^{\#}\coprod \dot{\Sigma}_-$, 
where 
$\dot{\Sigma}_{-}:=
\{\alpha\in \dot{\Sigma}_{\pr}| \ (\alpha,\alpha)<0\}$.
This gives
$$\pi=\pi^{\#}\cup\dot{\Sigma}_{\pr}=\pi^{\#}\coprod
\dot{\Sigma}_-.$$
The restriction of $(-,-)$ to $\dot{\fg}$
coincides with the normalised invariant bilinear form
on $\dot{\fg}$, except for the case $\fg=A(2n|2n-1)^{(2)}$
where $\dot{\fg}=B(n|n)$ (in the latter case
these forms are proportional with the coefficient $-1$).
By~\cite{Gdenomfin}, Section 7 in all cases we have
$$\sum\limits_{w\in W[\dot{\pi}^{\#}] } \sgn(w)\ \frac{e^{w(k\Lambda_0+\dot{\rho})}}{\prod\limits_{\beta\in S} (1+e^{-w\beta})}=j_0^{-1} \sum\limits_{w\in W[\dot{\Sigma}_{\pr}]}
\sgn(w)\ \frac{e^{w(k\Lambda_0+\dot{\rho})}}
{\prod\limits_{\beta\in S} (1+e^{-w\beta})}=\dot{D}e^{\dot{\rho}},
$$
where $j_0$ is the cardinality of $W[\dot{\Sigma}_-]$
and $\dot{D}\in\cR(\dot{\Sigma})$ is the Weyl denominator for $\dot{\Delta}^+$. 
Therefore 
\begin{equation}\label{eq:dotZ2}
\dot{Z}_2=
j_0^{-1} \sum\limits_{w\in W[\dot{\Sigma}_{\pr}]}
\sgn(w)\ \frac{e^{w(k\Lambda_0+\rho)}}
{\prod\limits_{\beta\in S} (1+e^{-w\beta})}=\dot{D}e^{\rho}.
\end{equation}
We have 
 $Z_2=\sum\limits_{w\in W^{\#}} Z_w=\sum\limits_{y\in W^{\#}/W[\dot{\pi}^{\#}]} y\bigl( \sum\limits_{w\in W[\dot{\pi}^{\#}]} Z_w\bigr)$. Since $W[\pi]/W[\dot{\Sigma}_{\pr}]=W^{\#}/W[\dot{\pi}^{\#}]$ this gives
$$Z_2=j_0^{-1}\sum\limits_{y\in W^{\#}/W[\dot{\pi}^{\#}]} y\bigl( \sum\limits_{w\in W[\dot{\Sigma}_{\pr}]} Z_w\bigr)=j_0^{-1} \sum\limits_{w\in W[\pi]} Z_w.$$
Since
$Q(\Sigma)$
is  naturally $W[\pi]$-invariant, we have
$$\begin{array}{rl}
Y_2&=j_0^{-1} Q(\Sigma)\sum\limits_{w\in W[\pi]} 
Z_w=
j_0^{-1}\sum\limits_{w\in W[\pi]} 
\sgn(w)\ \prod\limits_{\alpha\in \Delta^+_{\ol{1}}}
(1+e^{-w\alpha})e^{w(\rho_{\pi}-\rho)}\frac{e^{w(k\Lambda_0+\rho)}}
{\prod\limits_{\beta\in S} (1+e^{-w\beta})}\\
&=
j_0^{-1}\sum\limits_{w\in W[\pi]} 
\sgn(w)\ e^{w(k\Lambda_0+\rho_{\pi})}
\prod\limits_{\beta\in \Delta^+_{\ol{1}}\setminus S} 
(1+e^{-w\beta})\end{array}
$$
By~\cite{GKadm} $\S$ 2.2.5, the last sum is naturally 
$W[\pi]$-anti-invariant. Therefore $Y_2$ is  
naturally $W[\pi]$-anti-invariant.

Suppose the contrary, that $Z_1\not=Z_2$. Let $\nu+\rho$ 
be a maximal element in $\supp(Z_1-Z_2)$ 
 with respect to the partial order
$\nu_1\geq \nu_2$ if $\nu_1-\nu_2\in\mathbb{Z}_{\geq 0}\Sigma$.
Then $\nu+\rho_{\pi}$ is a maximal element in $\supp(Y_1-Y_2)$ 
 with respect to this partial order. By above,
 $Y_1, Y_2$ are  naturally $W[\pi]$-anti-invariant. Thus
$Y_1-Y_2$ is  naturally $W[\pi]$-anti-invariant, so
$\langle \nu+\rho_{\pi},\alpha^{\vee}\rangle\not\in\mathbb{Z}_{\leq 0}$
for all $\alpha\in \pi$. By~(\ref{eq:Z1Z2}), $\nu\in \mathbb{Z}_{\geq 0}\Sigma$. Since $\nu\in\supp(Z_1)\cup\supp(Z_2)$
we have $(k\Lambda_0+\rho,k\Lambda_0+\rho)=(\nu+\rho,\nu+\rho)$
by~$\S$~\ref{Ulambda}.  By~\Prop{prop:Uklambda0} this gives
$\nu\in\mathbb{Z}_{\geq 0}\dot{\Sigma}$.
One readily sees that for $\mu\in \mathbb{Z}_{\geq 0}\dot{\Sigma}$
the coefficient of $e^{k\Lambda_0+\rho-\mu}$ in $Z_1$ is equal to
the coefficient of $e^{{\rho}-\mu}$ in 
$\dot{D}e^{{\rho}}=\dot{Z}_2$.  Therefore it remains to show
$k\Lambda_0+\rho-\nu\not\in \supp(Z_2-\dot{Z}_2)$.

It remains to verify that for any $w\in W^{\#}$ 
such that $(k\Lambda_0+\rho-\nu)\in \supp (Z_w)$
we have $w\in W[\dot{\pi}^{\#}]$.
Combining~(\ref{eq:suppZw}) and~(\ref{eq:wlambda}) we conclude that
$$
\begin{array}{l}
(k\Lambda_0+\rho)-w(k\Lambda_0+\rho)\in \mathbb{Z}_{\geq 0}\dot{\pi}^{\#},\\
w\beta \in -\mathbb{Z}_{\geq 0} S
\text{ for each }\beta\in S\text{ such that }w\beta\in-\Delta^+.
\end{array}$$

Recall that  $W^{\#}$
is the Weyl group generated by $s_{\gamma}$ for
$\gamma\in\pi^{\#}=\dot{\pi}\cup\{\alpha^{\#}_0\}$.
We will use the following property of $W^{\#}$:
if $\lambda\in\fh^*$ is such that $(\lambda,\gamma)\geq 0$
for all $\gamma\in \pi^{\#}$ and $(\lambda,\alpha_0^{\#})>0$,
then $\lambda-y\lambda\in \mathbb{Z}\dot{\pi}$
for $y\in W^{\#}$
is equivalent $y\in W[\dot{\pi}]$. 
Using this property we obtain $w\in  W[\dot{\pi}^{\#}]$
if $(k\Lambda_0+\rho,\alpha^{\#}_0)>0$.

If $\rho=0$, then $h^{\vee}=0$, so 
$k\Lambda_0+\rho=k\Lambda_0$ for $k\geq 1$.
Therefore $(k\Lambda_0+\rho,\alpha^{\#}_0)>0$.

If $\rho\not=0$, then  property (c) holds, 
so  $(\rho,\alpha)\geq 0$   for all $\alpha\in\Delta^+$. Since
   $\alpha_0^{\#}\not\in\dot{\Delta}$
we have $\alpha_0^{\#}-\alpha_0\in {\Delta}^+$,
so 
$$0=(k\Lambda_0+\rho,\alpha_0^{\#})\geq k+ (\rho,\alpha^{\#})\geq k+(\rho,\alpha_0)=k+\frac{1}{2}(\alpha_0,\alpha_0).$$
 This gives  $k=(\alpha_0,\alpha_0)=0$
 and $h^{\vee}\not=0$ (since $k\not=-h^{\vee}$).

 This completes the proof for all cases except 
 for 
  $k=0$ and $\Delta$ from the following list:
$$B(n|n)^{(1)},\ B(n+1|n)^{(1)},\ D(n+2|n)^{(1)},\ 
A(2n|2n-1)^{(2)}, \ A(2m|2m+1)^{(2)}.$$
For the remaining cases the required formula is
reduced to the implication
\begin{equation}\label{eq:123}
w\in \Stab_{W^{\#}}\rho, \ \ wS\subset \Delta^+\ \ \Longrightarrow
w\in W[\dot{\pi}^{\#}]
\end{equation}
which can be easily verified case-by-case (see
 an argument for $B(n+1|n)^{(1)}$ below).

For $k=0$ the required formula is
a particular case of Denominator Identity established 
in~\cite{Gnon-zero}, \cite{Reifdenom} and  
the implication~(\ref{eq:123}) 
was verified in all cases except for the case $B(n+1|n)^{(1)}$
which was not properly treated.
Below we give an argument for this case.

For $\Delta=B(n+1|n)^{(1)}$ we have 
$$S=\{\delta_i-\vareps_{i+1}\}_{i=1}^n,\ \ 
\rho=\Lambda_0+\frac{1}{2}\sum\limits_{i=1}^{n+1}\vareps_i-
\frac{1}{2}\sum\limits_{i=1}^{n}\delta_i$$
and ${\pi}^{\#}=\{\vareps_1-\vareps_2, \vareps_2-\vareps_3,\ldots, \vareps_n-\vareps_{n+1}, \vareps_{n+1};\delta-\vareps_1-\vareps_2\}$
has the Dynkin diagram of type $B_{n+1}^{(1)}$.
The group $\Stab_{W^{\#}}\rho$ is generated
by $s_{\alpha}$ for $\alpha\in \pi^{\#}\setminus\{\vareps_{n+1}\}$.
The Dynkin diagram corresponding to
$\pi^{\#}\setminus\{\vareps_{n+1}\}$ 
is of type $D_{n+1}$: 
$$\pi^{\#}\setminus\{\vareps_{n+1}\}=
\{\epsilon_i-\epsilon_{i+1}\}_{i=1}^n\cup\{\epsilon_n+\epsilon_{n+1}\}\ \ \text{ for }  \epsilon_i:=\frac{\delta}{2}-\vareps_{n+2-i}
\ (i=1,\ldots,n+1).$$
Thus the group $\Stab_{W^{\#}}\rho$ 
acts on $\{\epsilon_i\}_{i=1}^{n+1}$
by signed permutations changing even number of signs.
Fix  $w\in \Stab_{W^{\#}}\rho$  is such that $wS\subset \Delta^+$.
If $w$ changes some signs, then 
$w(\epsilon_{n+2-i})=-\epsilon_{n+2-j}$ for some $i\not=1$ 
(since $w$ changes at least two signs), which gives
 $w\vareps_{i}=\delta-\vareps_j$ that is
$w(\delta_{i-1}-\vareps_i)=\delta_{i-1}+\vareps_j-\delta\in\Delta^-$,
a contradiction.
Therefore $w$ lies the group of permutations of 
 $\{\epsilon_i\}_{i=1}^{n+1}$. If $w\not=\Id$, then 
 $w(\epsilon_{n+2-i})=\epsilon_{n+2-j}$ for some $j<i$ which gives
 $w(\delta_{i-1}-\vareps_i)=\delta_{i-1}-\vareps_j\in\Delta^-$,
 a contradiction. Hence $w=\Id$. This completes the proof.
\end{proof}

\section{}\label{sect:integralrootsystems}
The goal of this section is to justify the definition 
of integral root system $\Delta_{\lambda}$ given in~$\S$~\ref{Deltalambda}.
The main result of this section 
 is~\Thm{thm:Deltalambda}.

In this section $\fg$ is a symmetrizable Kac-Moody superalgebra
and  $\Delta^{\ree}(\fg)$ is the set of real roots. As in Appendix~\ref{affinebases}
we will use the same notation for $\Delta^{\ree}$ and for $\fg$
(for example, $A(2|2)$ stands for $\Delta(\fgl(3|3))$. As above,
if $\fg$ is affine, we denote by $\delta$ the minimal imaginary root and
let $\cl$ be the canonical map
from $\mathbb{C}\Delta^{\ree}(\fg)$ to 
$\mathbb{C}\Delta^{\ree}(\fg)/\mathbb{C}\delta$.
If $\fg$ is finite-dimensional, we let $\cl$ be the identity
map on $\mathbb{C}\Delta^{\ree}(\fg)$.

\subsection{}
First consider the case when $\fg$ is a 
 symmetrizable Kac-Moody algebra.
 For non-critical $\lambda\in\fh^*$ 
 the integral root systems $\Delta_{\lambda}=\{\alpha\in {\Delta}^{\ree}|\ \langle \lambda,\alpha^{\vee}\rangle\in\mathbb{Z}\}$ are well-known, see for example,~\cite{Ku},\cite{MP}, \cite{KT98}. 
In this case $R\subset {\Delta}^{\ree}$ is called a {\em root subsystem}  if
$s_{\alpha}\beta\in R$ for any $\alpha,\beta\in R$. 
 Clearly, $\Delta_{\lambda}$ is a root subsystem.
 
If $R$ is a root subsystem, then setting $R^+:=R\cap \Delta^+$ and 
\begin{equation}\label{eq:SigmaR}
\Sigma(R):=\{\alpha\in R|\ s_{\alpha}(R^+\setminus\{\alpha\})\subset R^+\},
\end{equation}
we have 
\begin{itemize}
\item[(a)]
$\Sigma(R)$ is the set of indecomposable elements in $R^+$, see, for example,~\cite{KT98}, Lemma 2.2.8.
\item[(b)] The matrix $(\langle\alpha^{\vee},\beta\rangle)_{\alpha,\beta\in\Sigma(R)}$ is a symmetrizable generalized Cartan matrix, possibly of countable rank,
loc. cit.  Lemma 2.2.10.
\end{itemize}
(The set $\Sigma(R)$ is a base of $R^+$ if $\Delta$ is finite; if $\Delta$ is affine, $\Sigma(R)$ could be linearly dependent).

By (b), $(\langle\alpha^{\vee},\beta\rangle)_{\alpha,\beta\in\Sigma(R)}$ is the Cartan matrix of 
a symmetrizable Kac-Moody algebra (possibly of countable rank), which we denote by $\fg_{\lambda}$.
From~\cite{KK} it follows that  if $\lambda-\rho$ is a non-critical weight, then 
$$\begin{array}{llll}
& (*) & & \text{
 $L(\nu-\rho)$ lies in the same block as }L(\lambda-\rho)\ \ \Longleftrightarrow\ \ \nu\in W[\Delta_{\lambda}]\lambda\end{array}$$
(where $W[\Delta_{\lambda}]$ is the subgroup of $GL(\fh^*)$ generated by $s_{\alpha}$ for $\alpha\in \Delta_{\lambda}$).
 P.~Fiebig proved that 
 the structure of non-critical block in the category $\CO_{\Sigma}(\fg)$  depends only on $W[\Delta_{\lambda}]$
and its action on the highest weights, see~\cite{F}, Theorem 4.1.

A similar construction
works for symmetrizable Kac-Moody superalgebras 
in the case when $\lambda$ is typical, i.e. $(\lambda+\rho,\alpha)\not=0$ for all $\alpha\in\Delta^{\is}$;
in particular, this works if $\fg$ is anisotropic (however, to the best of our knowledge,
Fiebig's results have not been  extended to anisotropic Kac-Moody superalgebras).

In this section we consider  
indecomposable symmetrizable  Kac-Moody superalgebras.
By~\cite{Hoyt}, \cite{S}
all these superalgebras are anisotropic,  
finite-dimensional or affine.   
In order to  define $\Delta_{\lambda}$ rigorously
we introduce the notion of a real 
root system, using the axioms given by V.~Serganova for 
generalized finite root systems in~\cite{VGRS},
and then take $\Delta_{\lambda}$ to be 
the minimal real root system containing
real roots which are "integral" for $\lambda$. 
By contrast with the anisotropic case, 
it is not true that,  for affine Kac-Moody superalgebras,
a real root subsystem $R$ is 
isomorphic to the set of real roots
of a Kac-Moody superalgebra. 
By Serganova's classification~\cite{VGRS}
this is "almost true" if $R$ is finite 
(then $R$ is isomorphic to the set of real roots
of a Kac-Moody superalgebra or $\mathfrak{psl}(n|n)$
for $n\geq 3$)-- this can be easily deduced
from Serganova's classification~\cite{VGRS}.
However, if a finite real root system $R$
is a subset of the root system of a Kac-Moody superalgebra $\fg$,
then $R$ is isomorphic to the set of real roots
of a Kac-Moody superalgebra $\fg'$, see~\Prop{prop:RsubsetDelta} (i) below
(i.e., $R$ is not isomorphic to the root system of $\mathfrak{psl}(n|n)$).

Unfortunately in the affine case
the set of real roots of the Kac-Moody superalgebra of type 
$A(2|1)^{(2)}$ contains
a subsystem $R$ of type $A(1|1)^{(2)}$ 
which is not a contragredient Lie superalgebra, see~$\S$~\ref{A112} below.
This example indicates that
the  real roots of the Kac-Moody superalgebra
could not be singled out by a  set of "inheritable" axioms (see~$\S$~\ref{inheritable}) since
$R$ is the intersection of the set of real roots 
$\Delta^{\ree}(\mathfrak{sl}(2|4)^{(2)})$
with a subspace (spanned by $\vareps_1,\delta_1$
and $\delta$).

In this section we prove ~\Thm{thm:Deltalambda}
asserting that $\Delta_{\lambda}$
is isomorphic to the set of real roots of 
a symmetrizable Kac-Moody superalgebra $\fg_{\lambda}$ (possibly of countable rank) if all indecomposable components of $\Delta_{\lambda}$
are not isomorphic to $A(1|1)^{(2)}$ 
(this automatically holds 
if $\fg\not=A(2m-1|2n-1)^{(2)}$, see~\ref{DeltalambdaA112} below).
 Using~\cite{KK},  property (*)
can be generalized to symmetrizable Kac-Moody superalgebras: by~\cite{Gdepth}, Theorem 4.8,
 if $\lambda-\rho$ is a non-critical weight, then 
$$\begin{array}{llll}
&  & & \text{
 $L(\nu-\rho)$ lies in the same block as }L(\lambda-\rho)\ \ \Longleftrightarrow\ \ 
 \nu\in W[\Delta_{\lambda}](\lambda+\mathbb{Z}S_{\lambda}),\end{array}$$
where $S_{\lambda}$ is an arbitrary maximal iso-set orthogonal to $\lambda+\rho$, see~$\S$~\ref{iso-set} below.

\subsubsection{Remark}
Let $\fg$ be a symmetrizable Kac-Moody algebra and $E\subset \mathbb{C}\Delta^{\ree}(\fg)$ be a vector subspace
such that $R:=E\cap \Delta^{\ree}(\fg)$ is non-empty. Then
$R$
satisfies axioms (GR0), (GR1) from~\ref{GR03} below. Consider the subalgebra $\fg_R$
generated by $\fg_{\gamma}$ for $\gamma\in R$.

Let $\Sigma(R)$ be as in~(\ref{eq:SigmaR}). Set
$A':=(\langle\alpha^{\vee},\beta\rangle)_{\alpha,\beta\in\Sigma(R)}$.
Consider the Kac-Moody algebra $\fg(A')$ with the set of simple roots $\Sigma'$. Let $\phi$ be the natural bijection of $\Sigma'$
and $\Sigma(R)$.
The algebra $\ft:=[\fg(A'),\fg(A')]$ is generated by
$\fg(A')_{\pm \alpha}$ for $\alpha\in\Sigma'$
which leads to an algebra homomorphism $\psi: \ft\to \fg_R$
which maps $\fg(A')_{\pm \alpha'}$ for $\alpha'\in\Sigma'$
to $\fg_{\pm \phi(\alpha')}$.

By~\cite{KT98},  $\Sigma_R$ coincides with the set of indecomposable elements in $R\cap \Delta^+$  and 
 $R=W[\Sigma_R]\Sigma_R$. The latter gives
 $\psi(\ft_{\gamma'})\subset \fg_{\phi(\gamma')}$.
Since $\dim \fg_{\gamma}=1$ for all $\gamma\in R$,
the map $\psi: \ft\to \fg_R$ is surjective.

The above argument can be adapted to 
 anisotropic  symmetrizable Kac-Moody superalgebras, see~\Prop{prop:RR'an}. 
 For  symmetrizable Kac-Moody superalgebras with isotropic roots
 we can not define $\Sigma(R)$ by  formula~(\ref{eq:SigmaR}), but 
 we can take $\Sigma(R)$ consisting of the
 indecomposable elements in $R\cap \Delta^+$. 
 Taking $A'$  as above and  $\tau'$ corresponding to the set of odd roots in $\Sigma(R)$, we 
 can construct a Kac-Moody superalgebra $\fg(A',\tau')$,
and  introduce $\ft$, $\psi$. From~\Prop{prop:RsubsetDelta}
below it follows that
the map $\psi: \ft\to \fg_R$ is surjective
except for the case when $\fg$ is of the type $A(2m-1|2n-1)^{(2)}$
and $R\cong A(1|1)^{(2)}$ (for example, we can take $R$ spanned by $\delta$, $\vareps_1$ and $\delta_1$).
In the latter case $\fg(A',\tau')\cong \fgl(2|2)$, so
$\ft\cong \fsl(2|2)$ is finite-dimensional while 
$\fg_R$ is infinite-dimensional.

\subsection{Real root systems}\label{defnGR123}
In this section we  will consider various generalizations of the notion of a root system.
We consider   a complex vector space $L$  endowed by a complex 
valued symmetric bilinear form $(-,-)$ (our main example is $L=\mathbb{C}\Delta$).

\subsubsection{}\label{basismn}
We denote by 
 $L_{m,n}$ the vector space spanned by mutually orthogonal vectors
$\vareps_i,\delta_j,\delta$  for $i=1,\ldots, m$,
$j=1,\ldots,n$ satisfying
$$(\vareps_i,\vareps_i)=-(\delta_i,\delta_i)\not=0,\ \ \ 
(\delta,\delta)=0$$ 
(thus the kernel of the bilinear form $(-,-)$ 
is spanned by $\delta$).

\subsubsection{Definition}\label{def:bijquot}
We say that $R'\subset L'$ is a 
{\em quotient} of $R\subset L$ if there exists a linear map $\psi: \mathbb{C}R\to \mathbb{C} R'$, which is compatible with the bilinear forms 
and $\psi(R)=R'$. We say that
$R'$ is a {\em bijective quotient}\footnote{In the main text we say that $R'$ is isometric to $R$.} of $R$
if the restriction of $\psi$ to $R$ gives a bijection between $R$ and $R'$, and
we say that  $R$ and $R'$ are  {\em isomorphic}
if $\psi:\mathbb{C}R\to \mathbb{C} R'$ is bijective. 

Note that, under this convention the  root systems of $\mathfrak{sl}(n|n)$
and of $\mathfrak{gl}(n|n)$
are isomorphic; moreover,  $\Delta(\fpsl(n|n))$ is a bijective quotient
of $\Delta(\fgl(n|n))$ for $n\geq 3$ ($\Delta(\fpsl(2|2))$ is not 
a bijective quotient
of $\Delta(\fsl(2|2))$ since the preimage of each isotropic root contains two roots).

\subsubsection{}\label{GR03}
We will consider the following  axioms
for a subset $R$ of $L$ (for  $\alpha,\beta\in R$):
\begin{itemize}
\item[(GR0)]  $R\not=\emptyset$, $0\not\in R$,  $-R=R$, and $\mathbb{Z}R$ is a free module of finite rank;
\item[(GR1)] 
 if  $(\alpha,\alpha)\not=0$, then $\frac{2(\alpha,\beta)}{(\alpha,\alpha)}\in \mathbb{Z}$
 and $(\beta-\frac{2(\alpha,\beta)}{(\alpha,\alpha)}\alpha)\in R$;
\item[(GR2)] if $(\alpha,\alpha)=0\not=(\beta, \alpha)$, then 
the set  $\{\beta+\alpha,\beta-\alpha\}\cap R$ has cardinality one;
  \item[(GR3)]
for any $\alpha\in R$ there exists $\beta\in R$ such that $(\alpha,\beta)\not=0$.
\end{itemize}

 If $R$ satisfies some of the above axioms, then any bijective quotient
 satisfies the same axioms.

We call a set $R$ satisfying (GR0)--(GR2) a {\em real root system}.
 
If $R$ satisfies (GR0)--(GR3), then any quotient of $R$ satisfies
(GR0), (GR1),  (GR3)
and the following weaker form of (GR2)
\begin{itemize}
\item[(WGR2)]
if $\alpha,\beta\in R$ are  such that $(\alpha,\alpha)=0\not=(\beta, \alpha)$, then   
$\{\beta+\alpha,\beta-\alpha\}\cap R_\not=\emptyset$.
\end{itemize}

\subsubsection{Inheritable properties}\label{inheritable}
We call a property of $R\subset L$
{\em inheritable} if for any $R$ satisfying this property and
any vector space $L_1\subset L$
having non-empty intersection with $R$, the set $R\cap L_1$
satisfies this property. 
The properties (GR0)--(GR2) and (WGR2) are inheritable.

\subsubsection{Group $W[R]$}
For  $v\in L$ such that $(v,v)\not=0$ we let $s_v\in \End(L)$ be the reflection $s_v(u)=u-\frac{2(u,v)}{(v,v)}v$.
For any $R\subset L$  we denote 
 by $W[R]$ the subgroup of $\End(L)$ which is generated by
 the reflections $s_v$  for  $v\in R$ such that $(v,v)\not=0$.

If $R\subset L$ satisfies (GR1), then 
 $W[R]R=R$.

\subsubsection{Notation}
For each $R\subset L$ we set
$$R^{\an}:=\{v\in R|\ (v,v)\not=0\},\ \ \ 
R^{\is}:=\{v\in R|\ (v,v)=0\}.$$

\subsubsection{Definition}\label{def:indec}
A nonempty subset $R\subset L$ satisfying (GR0) is called
{\em indecomposable}
if $R$ can not be presented as a disjoint union
of two nonempty mutually orthogonal sets satisfying (GR0).

\subsubsection{Remarks}
The only indecomposable set 
satisfying (GR0)--(GR2) which does not satisfy (GR3) is $R=\{\pm\beta\}$
with $(\beta,\beta)=0$, i.e. $R=\Delta(\fgl(1|1))$.

We will use the following simple lemma.

\subsubsection{}
\begin{lem}{lem:chainiso}
Let $R\subset L $ be an indecomposable set satisfying (GR0) and (GR1)
and let $\beta\in R^{\is}$. Then
\begin{enumerate}
\item 
For any $\gamma\in R\setminus\{\pm \beta\}$
there exists a chain of isotropic elements 
$\beta=\gamma_1$, $\gamma_2$, ..., $\gamma_s$ such that 
$(\gamma_i,\gamma_{i+1})\not=0$ and $(\gamma_s,\gamma)\not=0$.
\item If $R$ satisfies also (WGR2),
then $R^{\is}=-W[R]\beta\cup W[R]\beta$.
\end{enumerate}
\end{lem}
\begin{proof}
Since $R$ is indecomposable, there exists 
a chain $\beta=\gamma_1$, $\gamma_2$, ..., $\gamma_s$
such that $(\gamma_i,\gamma_{i+1})\not=0$ and  $(\gamma_s,\gamma)\not=0$.
Choosing a chain of the minimal
length we obtain $(\gamma_i,\gamma_j)=0$ for $j-i>1$. Now, substituting
$\gamma_i$ by $s_{\gamma_i}\gamma_{i-1}$ if $\gamma_i$ is anisotropic,
we obtain a chain of isotropic vectors with the required properties.
This establishes (i).  

For (ii) take $\gamma\in R^{\is}$. 
Consider the case when $(\beta,\gamma)\not=0$.
By (WGR2) $R$ contains  $\gamma+\beta$
or $\gamma-\beta$. Without loss
of generality we can assume $\alpha=\gamma-\beta\in R$. Since
$\beta,\gamma$ are isotropic we have $(\alpha,\alpha)=2(\beta,\gamma)\not=0$,
so $s_{\alpha}\in W(R)$, and $s_{\alpha}\beta=\gamma$.
Hence $\gamma\in W[R]\beta$ if  $(\beta,\gamma)\not=0$.
The general case follows from (i).
\end{proof}

\subsection{Examples}\label{abcd}
Let $\fg(A,\tau)$ be a symmetrizable  Kac-Moody superalgebra
of possibly infinite rank 
with  the Cartan subalgebra $\fh$.
The set 
$\Delta^{\ree}$  satisfies (GR0)--(GR2) and the
isotropic roots  are isotropic elements in $\Delta^{\ree}$;
if  $\fg(A,\tau)$ is indecomposable, then
$\Delta^{\ree}$ satisfies (GR3) if and only if 
$\fg(A,\tau)\not=\fgl(1|1)$.

\subsubsection{}
As in~$\S$~\ref{Deltalambda} for $\lambda\in\fh^*$ we introduce
$$\begin{array}{l}
\ol{R}_{\lambda}=\{\alpha\in \ol{\Delta}^{\ree}\ |\ 
2(\lambda,\alpha)\in j (\alpha,\alpha)\ \text{ where }
j\in\mathbb{Z}\ \text{ and $j$ is odd if $2\alpha\in\Delta$}\},\\
R_{\lambda}:=\ol{R}_{\lambda}\cup\{2\alpha|\ \alpha\in 
\ol{R}_{\lambda},\  2\alpha\in\Delta\}=\ol{R}_{\lambda}\cup\{2\alpha|\ \alpha\in 
(\ol{R}_{\lambda}\cap \Delta^{\an}_{\ol{1}})\}
\end{array}$$
and let $\Delta_{\lambda}\subset \Delta^{\ree}$ be the minimal subset of $\Delta^{\ree}$
satisfying the following properties: 
\begin{itemize}
\item[(a)] $\ol{R}_{\lambda}\subset \Delta_{\lambda}$;
\item[(b)]   if $\alpha,\beta\in \Delta_{\lambda}$ are such that $(\alpha,\alpha)\not=0$, then $s_{\alpha}\beta\in \Delta_{\lambda}$;
\item[(c)] if $\alpha,\beta\in \Delta_{\lambda}$ are such that $(\alpha,\alpha)=0\not=(\beta,\alpha)$ 
then 
$\Delta_{\lambda}\cap \{\alpha\pm\beta\}=\Delta^{\ree}\cap \{\alpha\pm\beta\}$;
\item[(d)]   if $\alpha\in \Delta_{\lambda}$ and $2\alpha\in\Delta^{\ree}$, then
$2\alpha\in \Delta_{\lambda}$.
\end{itemize}
Since $\Delta^{\ree}$ satisfies (GR0)--(GR3), the set
 $\Delta_{\lambda}$ also satisfies  properties (GR0)--(GR3).

\subsubsection{Serganova's classification}\label{finR}
Let $R\subset L$ be  a finite indecomposable set  satisfying (GR0), (GR1) and (WGR2),
and such that   the restriction of $(-,-)$ to $\mathbb{C}R$ is non-degenerate.
By~\cite{VGRS}, if $R$ satisfies (GR2), then $R$ is isomorphic to the root system 
of a finite-dimensional Kac-Moody superalgebra or $\mathfrak{psl}(n|n)$
for $n\geq 3$. Moreover, if $R$ 
does not satisfy (GR2), then $R$ is of type $C(m|n)$
or $BC(m|n)$ (with $m,n\geq 1$), defined by
$$C(m|n)=\{\pm 2\vareps_i;\pm 2\delta_j;
\pm \vareps_i\pm\delta_j\}_{1\leq i\leq m}^{1\leq j\leq n},\ \ \ \ 
BC(m|n):=C(m|n)\cup \{\pm \vareps_i;\pm \delta_j\}_{1\leq i\leq m}^{1\leq j\leq n},$$
where $\{\vareps_i\}_{i=1}^m\cup\{\delta_j\}_{j=1}^n$ is an orthogonal basis in 
$\mathbb{C}{R}$ such that 
with $(\vareps_i,\vareps_i)=-(\delta_j,\delta_j)\not=0$ for all $i,j$.
The root system $\Delta(\mathfrak{psl}(2|2))$ is of type $C(1|1)$.

\subsubsection{}
For other examples of sets satisfying (GR0)--(GR2) see~\ref{C11} below.

 \subsection{Iso-sets}\label{iso-set}
Let $R$ satisfy (GR0), (GR1) and (WGR2). We  call a  subset $S\subset R$ {\em iso-set} if
$S=-S$ and $(S,S)=0$.
Iso-sets are ordered by inclusion.
For $v\in L$ we denote by $A_v^{\max}$ the collection of maximal iso-sets
orthogonal to $v$.

\subsubsection{}
\begin{lem}{lem:trans-isosets}
Let $R$ be a finite set satisfying (GR0), (GR1) and (WGR2), and
let $S,S'\in A^{\max}_v$. Then there exists $w\in\Stab_{W[R]}(v)$ such that
$S'=w(S)$.
\end{lem}
\begin{proof}We proceed by induction in the cardinality of $S'\setminus S$. Let $\alpha\in S'\setminus S$. Then $S\cup\{\alpha\}$ is not an iso-set, so there exists $\beta\in S$ such that $(\alpha|\beta)\ne 0$. This implies that $\alpha+\beta$ or $\alpha-\beta$ lies in $R$; without loss
of generality we can assume $\gamma=\alpha-\beta\in R$. Obviously,
$(\gamma,\gamma)\not=0$ so $s_\gamma\in W[R]$. The formula
$
s_\gamma(x)=x-2 (x,\gamma)/(\gamma|\gamma)
$
implies that $s_\gamma$ preserves $v$ and $S\cap S'$, and
$s_\gamma(\alpha)=\beta$.
This implies the claim.
\end{proof}

We will use the following corollary.

\subsubsection{}
\begin{cor}{cor:moddelta}
Let $\Delta^{\ree}$ be the set of real root of an indecomposable
symmetrizable finite-dimensional or affine Kac-Moody superalgebra.
 Let 
 $\beta_1,\ldots,\beta_s\in \Delta^{\ree}$ be mutually orthogonal isotropic roots such 
 that  $(\cl(\beta_i)\pm\cl(\beta_j))\not=0$
 for all $1\leq i<j\leq s$. Then  the set $\{\cl(\beta_i)\}_{i=1}^s$
  is linearly independent.
 \end{cor}
\begin{proof}
If $s=1$ the claim holds, so we may (and will) assume that $s>1$.

The set $\cl(\Delta^{\ree})$
is a finite set satisfying (GR0), (GR1) and (WGR2).  

The set $B:=\{\pm \cl(\beta_i)\}_{i=1}^s$
is an iso-set containing $2s$ elements. Using~\Lem{lem:trans-isosets}
for $v=0$ we conclude that $B$
 is $W[R]$-conjugated to
a subset of any maximal iso-set $S$ in $\cl(\Delta^{\ree})$.
The set $\cl(\Delta^{\ree})$ is not of types $D(2|1,a)$, $G(3)$ and $F(4)$
(for these types any two non-proportional isotropic roots are non-orthogonal, 
which contradicts to $s>1$).
Using the standard notation for $A(m-1|n-1)$, $B(m|n)$, $D(m|n)$ and
the notation  of~$\S$~\ref{finR}
for $BC(m|n)$, $C(m|n)$,   we can choose a maximal iso-set $S$ of the form
$\{\pm (\vareps_i-\delta_i)\}_{i=1}^d$ for $d:=\min(m,n)$.
Since the vectors $\{(\vareps_i-\delta_i)\}_{i=1}^d$ are linearly independent,
the set $\{\cl(\beta_i)\}_{i=1}^s$ is linearly independent.
\end{proof}

\subsubsection{Remark}\label{rem:pslnn}
\Cor{cor:moddelta} implies that $\Delta^{\ree}(\fg)$ does not contain
$\Delta(\mathfrak{psl}(n|n))$, since $\Delta(\mathfrak{psl}(2|2))$
does not satisfy (GR3), and, for $n>2$, the set
$\Delta(\mathfrak{psl}(n|n))$ contains isotropic roots $\beta_i:=\vareps_i-\delta_i$
with  $\sum\limits_{i=1}^n \beta_i=0$. By~\Cor{cor:moddelta},
$\cl(\beta_i)=\pm\cl(\beta_j)$ for some $i\not=j$ which is
impossible since $(\beta_i\pm\beta_j, \Delta(\mathfrak{psl}(2|2))\not=0$.

\subsection{Real root system  $A(2m-1|2n-1)^{(2)}$}\label{A112}
Take $L_{m,n}$ as in~$\S$~\ref{basismn} and consider the subset
$$R_{m,n}=\{\pm\vareps_i\pm\vareps_j, \pm\delta_i\pm\delta_j,
2\mathbb{Z}\delta\pm 2\delta_i,\ (2\mathbb{Z}-1)\delta\pm 2\vareps_i\}_{1\leq i\not=j\leq n}
\bigcup \{\mathbb{Z}\delta\pm\vareps_i\pm\delta_j\}_{1\leq i,j\leq n}.$$
Note that $R_{p,q}=R_{m,n}\cap L_{p,q}$ for $m\geq p$, $n\geq q$.
The set $R_{m,n}$ satisfies (GR0)--(GR4). We will denote this set by $A(2m-1|2n-1)^{(2)}$. If $(m,n)\not=(1,1)$, then
$A(2m-1|2n-1)^{(2)}$ is isomorphic to the set of real roots
of the corresponding affine Kac-Moody superalgebra.

\subsubsection{}\label{A112notKM}
Let us show that 
$$A(1|1)^{(2)}=\{\mathbb{Z}\delta\pm\vareps_1\pm\delta_1,\ 
2\mathbb{Z}\delta\pm 2\delta_1,\ (2\mathbb{Z}-1)\delta\pm 2\vareps_1\}
$$
is not isomorphic to the set of real roots
of a Kac-Moody superalgebra.
Suppose the contrary, i.e. that $A(1|1)^{(2)}=R_{1,1}$
is isomorphic to $\Delta^{\ree}(\fg')$ for some Kac-Moody superalgebra
$\fg'$.

The set $\Sigma_{1,2}:=\{\delta-2\vareps_1,\vareps_1-\delta_1,\delta_1-\delta_2,2\delta_2\}$ is a base of $R_{1,2}$ which is of type $A(3|1)^{(2)}$.
Take $\omega:\mathbb{Z}\Sigma_{1,2}\to \mathbb{R}$
given by 
$\omega(\alpha)=1$ for all $\alpha\in \Sigma_{1,2}$.
Then $\omega(\alpha)\in \mathbb{Z}_{\not=0}$
for all $\alpha\in R_{2,1}$. 
Define a triangular decomposition
of $\fg'$ via $\omega$ (see~$\S$~\ref{Sherman}).
By Hoyt-Serganova classification~\cite{Hoyt},\cite{S}, $\fg'$
is a symmetrizable affine Kac-Moody superalgebra
(since $\Delta^{\ree}(\fg')$ is an infinite indecomposable set).
By~\cite{Shay}, Theorem 0.4.3, the triangular decomposition
of $\fg'$ admits a base $\Sigma'$ and 
$\fg'=\fg(A',\tau')$, where $A'$ is
the Gram matrix of $\Sigma'$  and $\tau'$ is the set
of odd roots in this base. The set $\Sigma'$
coincides with the set of indecomposable
elements in $\{\alpha\in R_{1,1}|\ \omega(\alpha)>0\}$.
Since 
$\{\alpha\in R_{1,1}|\ \omega(\alpha)>0\}$ is equal to
$$
\{\mathbb{Z}_{>0}\delta\pm\vareps_1\pm\delta_1,\ 
2\mathbb{Z}_{>0}\delta\pm 2\delta_1,\ (2\mathbb{Z}_{>0}-1)\delta\pm 2\vareps_1\}\cup \{\vareps_1\pm\delta_1,2\delta_1\},
$$
one has $\Sigma'=\{\delta-2\vareps_1,\vareps_1-\delta_1,\ 
 2\delta_1\}$.
The Gram matrix of $\Sigma'$ is 
$A'=\begin{pmatrix}  4 & -2 & 0\\ -2& 0 & 2\\ 0 & 2 & -4\end{pmatrix}$.
We have $\tau'=\{2\}$, since 
$a_{22}=0$ (so $2\in \tau'$) and 
$\frac{2a_{12}}{a_{11}}$,$\frac{2a_{32}}{a_{33}}$ are odd integers
(so $1,3\not\in\tau'$). Thus $(A',\tau')$ is the Cartan datum
of $A(1|1)$, so $\fg'=\fg(A',\tau')=\fgl(2|2)$. However,
the set $\Delta^{\ree}(\fgl(2|2))$ is finite, so
$R_{1,1}$ is not isomorphic to a quotient of this set.

\subsubsection{}
Assume that 
 $\Delta_{\lambda}$ has an
indecomposable component $R$ which is 
isomorphic to $A(1|1)^{(2)}$. Let us show  that the bijection 
$R\cong A(1|1)^{(2)}$ preserves the parity.
Indeed, the roots of the form $\mathbb{Z}\delta'\pm\vareps'_1\pm\delta'_1$ are isotropic, 
so they are odd. All other roots in $R$
are non-isotropic. If $R$ contains a non-isotropic odd root
$\alpha$, then $2\alpha\in R$. However, for any non-isotropic root 
$\alpha\in R$ we have $2\alpha\not\in R$. Therefore all non-isotropic roots
in $R$ are even. Hence the bijection between $R$ and $A(1|1)^{(2)}$
preserves the parity.

\subsection{}
\begin{thm}{thm:Deltalambda}
Let $\fg$ be a symmetrizable Kac-Moody superalgebra of finite rank
and let $\lambda\in\fh^*$.
 
 \begin{enumerate}
\item If $\Delta_{\lambda}$ is a bijective quotient of 
$\Delta^{\ree}(\fg')$, where $\fg'$ is a symmetrizable Kac-Moody 
superalgebra of at most countable rank, 
then the bijection $\Delta_{\lambda}\to\Delta^{\ree}(\fg')$
preserves the parity. 
 
\item If $\Delta_{\lambda}$
 does not contain isotropic roots, then $\Delta_{\lambda}$
is  a bijective quotient of   the  system of real roots 
of an anisotropic symmetrizable Kac-Moody superalgebra $\fg'$ 
of at most countable rank.

\item If $\fg$ is a finite-dimensional Kac-Moody superalgebra, then
$\fg'$ is   as well.
If  $\fg$ is an affine symmetrizable
 Kac-Moody superalgebra and each indecomposable component $R$ of
 $\Delta_{\lambda}$ is not of type $A(1|1)^{(2)}$,
then $\Delta_{\lambda}$ is isomorphic
to   $\Delta^{\ree}(\fg')$, where 
$\fg'=\prod\limits_{i\in I}\fg'_i$ 
 and each $\fg'_i$ is either finite-dimensional or  affine
symmetrizable Kac-Moody superalgebra.
 Moreover, the set $I$ is finite if  $(\lambda,\delta)\not=0$.

\item  If $\Delta_{\lambda}$ has an indecomposable component
of type $A(1|1)^{(2)}$, then 
$\fg$ is  of the type $A(2m-1|2n-1)^{(2)}$ and 
$(\lambda,\delta)=\frac{p}{q}$, where $p,q$ are coprime integers and $q$ is odd (see~$\S$~\ref{dualCoxeter} for the normalization).
\end{enumerate}
 \end{thm}

\subsubsection{Remark}\label{rem:bib}
\Thm{thm:Deltalambda}  was stated in~\cite{GKadm} without  proof
and the case 
$A(1|1)^{(2)}$ was not excluded.

\subsubsection{Proof of (i)}
Let $\psi: \mathbb{C}\Delta_{\lambda}\to \mathbb{C} \Delta^{\ree}(\fg')$
be the linear map which gives  the bijection $\Delta_{\lambda}\to\Delta^{\ree}(\fg')$.

Take an odd root  $\alpha\in \Delta_{\lambda}$. If
$\alpha$ is isotropic, then  $\psi(\alpha)$ is isotropic, so it is odd.
If $\alpha$ is anisotropic, then, by (d) in~$\S$~\ref{abcd},
$2\alpha\in  \Delta_{\lambda}$. Then
 $2\psi(\alpha)\in \Delta^{\ree}(\fg')$, so 
$\psi(\alpha)$ is odd.

Take an even root $\alpha\in \Delta_{\lambda}$.
By~\Lem{lem:chainiso}, $\Delta_{\lambda}$ contains an isotropic $\beta$
such that $(\beta,\alpha)\not=0$. By (c) in~$\S$~\ref{abcd}, $\Delta_{\lambda}$ contains 
$\beta-\alpha$ or $\beta+\alpha$. Without loss of generality we can assume
$\beta-\alpha\in \Delta_{\lambda}$. Since $\beta$ is odd and $\alpha$ is even,
$\beta-\alpha$ is odd. By above, $\psi(\beta)$ and $\psi(\beta-\alpha)$
are odd, so $\psi(\alpha)$ is even. \qed

\subsubsection{}\label{observB6}
Assume that  $R\subset \Delta^{\ree}(\fg)$
satisfies (GR0)--(GR2) and each indecomposable component $R_i$ of
$R$ is  a bijective quotient of   the  system of real roots 
of a symmetrizable Kac-Moody superalgebra $\fg'_i$ of at most countable rank. Then $R$ is a  bijective quotient of 
$\Delta^{\ree}(\fg')$, where $\fg'=\prod\limits_{i}\fg'_i$
is  a symmetrizable Kac-Moody superalgebra $\fg'$ of at most countable rank.

\subsubsection{Proof of (ii)}
 This statement follows from~$\S$~\ref{observB6} and~\Prop{prop:RR'an} (i) below.

\subsubsection{}
\begin{lem}{lem:decompositionR}
Let $\fg$ be a symmetrizable affine Kac-Moody superalgebra. 
Then for each $\lambda\in\fh^*$ the set $\Delta_{\lambda}$
has finitely many indecomposable components  not isomorphic
to $\Delta(\fsl(1|1))$. If $(\lambda,\delta)\not=0$, then 
the set $\Delta_{\lambda}$
has finitely many indecomposable components.
\end{lem}
\begin{proof}
Set $R_{\lambda,i}:=\Delta_{\lambda,i}\cap R_{\lambda}$.
Note that $(\Delta_{\lambda,i}, R_{\lambda,j})=0$ for $i\not=j\in I$.
If $R_{\lambda,i}=\emptyset$, then for any $\gamma\in \Delta_{\lambda_i}$
one has $\gamma\not\in R_{\lambda}$ and $(\gamma,R_{\lambda})=0$ which contradicts to~\Lem{lem:Xi} (ii). Hence $R=\coprod_{i\in I} R_{\lambda,i}$ where $R_{\lambda,i}$ are pairwise orthogonal non-empty sets.

Let $J$ be the set of indices $j$ such that  $\Delta_{\lambda,j}\not\subset\Delta^{\is}$.

For $j\in J$ the sets 
 $\Delta_{\lambda_j}$  contains a non-isotropic root  $\gamma_j$, and
 for $j_1,j_2\in J$ one has 
 $\cl(\gamma_{j_1})\not=\cl(\gamma_{j_2})$. Since
 $\cl(\Delta^{\ree})$ is finite, the set $J$ is finite.

 Take $i\in I\setminus J$. Then $\Delta_{\lambda,j}\subset\Delta^{\is}$.
  For each $i\in I\setminus J$
fix $\beta_i\in\ R_{\lambda,i}$; note that $\beta_i\in\Delta^{\is}$
(since $\Delta_{\lambda,i}\subset\Delta^{\is}$), so
$(\lambda,\beta_i)=0$ (by definition of $R_{\lambda}$).

Let us show that $\Delta_{\lambda,i}$ is the root system of $\fsl(1|1)$
for each  $i\in I\setminus J$.
Indeed, if $\Delta_{\lambda,i}\not=\{\pm\beta_i\}$, then
$\Delta_{\lambda,i}$ contains $\beta$ which is not orthogonal
to $\beta_i$ (since $\Delta_{\lambda,i}$ is indecomposable).
Then, by (GR2),
 $\Delta_{\lambda,j}$ contains $\gamma\in \{\beta_i+\beta,\beta_i-\beta\}$.
 Since $\beta,\beta_i$ are isotropic and non-orthogonal, $\gamma$
 is anisotropic, a contradiction.

Suppose that $I\setminus J$ is infinite.
Since $\cl(\Delta^{\ree})$ is finite, one has $\cl(\beta_{i_1})=\cl(\beta_{i_2})$
for some $i_1,i_2\in I\setminus J$, so $\beta_{i_1}-\beta_{i_2}\in\mathbb{Z}_{\not=0}\delta$.
Hence $(\lambda,\delta)=0$ as required.
\end{proof}

\subsubsection{}
We will prove  (iv)  in~$\S$~\ref{DeltalambdaA112} below.
Using~$\S$~\ref{observB6} and~\Lem{lem:decompositionR}
we reduce~\Thm{thm:Deltalambda} (iii)  to the following statement
which will be proved in~$\S\S$~\ref{Falpha}--
\ref{DeltalambdaA112}  below.

 \subsubsection{}
\begin{prop}{prop:RsubsetDelta}
Let $\fg$ be a finite-dimensional or affine
symmetrizable Kac-Moody superalgebra.

\begin{enumerate}
\item
If $R\subset \Delta^{\ree}(\fg)$ 
satisfying (GR0)--(GR2) is indecomposable, contains an isotropic root, 
 and $\cl(R)$ 
is not isomorphic to $\Delta(\mathfrak{psl}(2|2))$, then $R$ is isomorphic
to  $\Delta^{\ree}(\fg')$, where 
$\fg'$ is a finite-dimensional or affine symmetrizable Kac-Moody superalgebra.
\item
If $\lambda\in\fh^*$ and  $R$ is an indecomposable component of $\Delta_{\lambda}$
 and $\cl(R)$ 
is isomorphic to $\Delta(\mathfrak{psl}(2|2))$, then $R$
is isomorphic to $A(1|1)^{(2)}$ and $(\lambda,\delta)=\frac{p}{q}$,
where $q$ is odd.
\end{enumerate}
\end{prop}

\subsubsection{Remark}
In~\cite{VGRS} V.~Serganova defined ``generalized root systems'' (GRS) as finite sets
satisfying the condition that $(-,-)$ is non-degenerate on the span of $R$, and  axioms (GR0)--(GR2),
 and showed that any indecomposable  GRS 
is isomorphic to the root system of  a basic classical Lie superalgebra different from $\mathfrak{psl}(2|2)$.
Several attempts to generalize this approach to affine Lie superalgebras were made,
but, to the best of our knowledge, root systems of   affine Lie superalgebras
are not singled out by a set of axioms. 
In~\cite{You} (cf.~\cite{Yos}) M.~Yousofzadeh  introduced and classified the
extended affine supersystems (with the ``string property'' instead of the axiom (GR2)),
which include root systems of affine Lie superalgebras.
\Prop{prop:RsubsetDelta} (i)
can be deduced from the classification theorem~\cite{You}, Theorem 2.2 (which is based on the Serganova's result).
Since the latter classification is rather complicated 
we deduce~\Prop{prop:RsubsetDelta} (i)
directly from  Serganova's result.

\subsection{Gram matrices}\label{Gramm}
For a finite  subset 
$\alpha_1,\ldots,\alpha_s$ of $L$
we denote by 
$$\Gamma_{\alpha_1,\ldots,\alpha_s}=\bigl((\alpha_i,\alpha_j)\bigr)$$ the corresponding Gram matrix.
We will use the following fact.

\subsubsection{}\begin{lem}{lem:Gramm}
Let $\fg$ be a symmetrizable finite-dimensional or affine Kac-Moody superalgebra and let $\alpha_1,\ldots,\alpha_s\in\Delta^{\an}(\fg)$ be such that
$(\alpha_i,\alpha_{i+1})\not=0$ for $i=1,\ldots, s-1$.
Then the matrix $(\alpha_1,\alpha_1)^{-1}\Gamma_{\alpha_1,\ldots,\alpha_s}$ is positive semidefinite and  $\cl(\alpha_1),\ldots,\cl(\alpha_s)$ are linearly dependent
 if $\det\Gamma_{\alpha_1,\ldots,\alpha_s}=0$.
\end{lem}
\begin{proof}
Set $\gamma_i:=\cl(\alpha_i)$ and note that 
 $\Gamma_{\alpha_1,\ldots,\alpha_s}=\Gamma_{\gamma_1,\ldots,\gamma_s}$.
 Since $(\gamma_i,\gamma_{i+1})\not=0$ for $i=1,\ldots, s-1$, 
the vectors $\gamma_1,\ldots,\gamma_s$ lie in the same indecomposable component of the set $\cl(\Delta^{\an})$.
These components are the sets of real roots of
finite-dimensional anisotropic Kac-Moody superalgebras.
The restriction of the bilinear form $(\gamma_1,\gamma_1)^{-1}(-,-)$ to the span of such component is positive definite.
This implies the statement.
\end{proof}

 \subsection{}
\begin{prop}{prop:RR'an}
Let $\fg$ be a  symmetrizable Kac-Moody superalgebra, possibly of countable rank.
 Let $R\subset \Delta^{\an}(\fg)$ be an indecomposable set (in the sense of~\Defn{def:indec})
satisfying $s_{\alpha}\beta\in R$ for all $\alpha,\beta\in R$
and $2\alpha\in R$ for all odd $\alpha\in R$. Then
\begin{enumerate}
\item
There exists an indecomposable symmetrizable  
anisotropic Kac-Moody superalgebra $\fg'$, 
possibly of countable rank, such that $R$ is a 
bijective quotient of $\Delta(\fg')^{\ree}$. 

\item If $\fg$ is finite-dimensional or affine Kac-Moody superalgebra, then
  $\fg'$ in (i) is finite-dimensional or affine Kac-Moody superalgebra, and $R$ is isomorphic to $\Delta(\fg')^{\ree}$.
 \end{enumerate}
 \end{prop}
 \begin{proof}
If $\alpha\in\Delta^{\an}(\fg)$,
then $\fg_{\pm\alpha}$ act locally nilpotently,
so for each $\gamma\in\Delta(\fg)$, 
$\frac{2(\alpha,\gamma)}{(\alpha,\alpha)}$ is integral
and is even if $\alpha$ is odd. Set $R^+:=\Delta^{+}(\fg)\cap R$ and
\begin{equation}\label{eq:KTformula}
\Sigma':=\{\alpha\in R|\ s_{\alpha}\bigl(R^+\setminus \{\alpha,2\alpha\}\bigr)\subset R^+\}.
\end{equation}
Arguing as in~\cite{KT98} we obtain $\frac{2(\alpha,\beta)}{(\beta,\beta)} \in\mathbb{Z}_{\leq 0}$
 for all $\alpha,\beta\in\Sigma'$ with $\alpha\not=\beta$
 and 
 \begin{equation}\label{eq:KTR}
 R=W[\Sigma']\bigl(\Sigma'\coprod \{2\alpha| \alpha\in\Sigma', p(\alpha)=\ol{1}\}\bigr).\end{equation}
 By above, $\frac{2(\alpha,\beta)}{(\beta,\beta)}$
 is even if $\beta$ is odd. Therefore
  the matrix $A':=(\langle \alpha,\beta^{\vee}\rangle )_{\alpha,\beta\in \Sigma'}$ 
 is a Cartan matrix for  a anisotropic symmetrizable generalized Kac-Moody superalgebra $\fg'$
 with a base $\Sigma''$
  and the parity function $p':\Sigma''\to \mathbb{Z}_2$ induced
 by the parity function $p: \Delta(\fg)\to \mathbb{Z}_2$.  Moreover, the map $\Sigma''\to \Sigma'$ extends
 to a linear homomorphism $\psi: \mathbb{Z}\Sigma''\to\mathbb{Z}\Sigma'$ which preserves the parity and
 the bilinear forms,
 and maps $\ol{\Delta}^{\ree}(\fg')=W[\Sigma''](\Sigma'')$ 
 to $R\cap \ol{\Delta}^{\ree}(\fg)=W[\Sigma'](\Sigma')$.
 By~(\ref{eq:real}), this implies 
 $\psi \bigl(\Delta^{\ree}(\fg')\bigr)=
 R$. Using  formula~(\ref{eq:KTR}) and the assumption that $R$ is indecomposable, we conclude
 that the Dynkin diagram of
 $\Sigma''$ is connected.  This establishes (i).

For (ii) assume that $\fg$ is finite-dimensional or affine
Kac-Moody superalgebra.

 Let us show that $\fg'$ is a Kac-Moody superalgebra, i.e. that $\Sigma'$ is finite. Indeed, if $\Delta$ is finite, then $R$ is finite.
If $\Delta$ is affine, then the image of $\Delta$ in $\mathbb{Z}\Delta/\mathbb{Z}\delta$ is finite, so,
if $\Sigma'$ is infinite, then 
 $\alpha,\alpha+j\delta\in \Sigma'$ for some $j\not=0$, but $\frac{2(\alpha,\alpha+j\delta)}{(\alpha,\alpha)}=2$,
a contradiction (by above, this number is non-positive).
 Hence $\Sigma'$ is finite, so $\fg'$ is a Kac-Moody superalgebra.
 
Renormalize the bilinear form on $\Delta(\fg)$ in such a way 
that $(\gamma,\gamma)=2$ for some $\gamma\in\Sigma'$.
By (i), the Dynkin diagram of $\Sigma''$  is connected.
By~\Lem{lem:Gramm}, the Gram matrix $\Gamma_{\Sigma'}=\Gamma_{\Sigma''}$ 
is either positive definite or
 positive semidefinite; the algebra $\fg'$ is finite-dimensional in the first case
 and affine in the second case (see~\cite{Kbook}, Theorem 4.3 and~\cite{K78}).

 It remains to verify that the linear map 
 $\psi:\mathbb{C}\Sigma''\to\mathbb{C}\Sigma'$ is injective.  
 Since $\psi$ preserves the bilinear forms, $\Ker\psi$ lies
 in the kernel of the bilinear form on $\mathbb{C}\Sigma''$. In particular, $\Ker\psi=0$ if
 $\fg'$ is finite-dimensional.  In the remaining case $\fg'$ is affine. Suppose the contrary, that $\Ker\psi\not=0$, so
   $\Ker\psi=\mathbb{C}\delta'$
 where $\delta'$ is the minimal imaginary root in $\Delta(\fg')$.  Then 
 $R\cap \ol{\Delta}^{\ree}(\fg)=\psi(\ol{\Delta}^{\ree}(\fg'))$  
 is a classical finite root system 
 and $\Sigma'$ is a set of simple roots of this root system.
 Thus the Cartan matrix $A'$ is invertible, so $\fg'$ is not affine, a contradiction. This completes the proof of (ii).
\end{proof}

\subsection{Set $F(\alpha)$}\label{Falpha}
Let $R\subset L$ be a set satisfying (GR0), (GR1) and (WGR2). 
Let $K$ be the kernel of the restriction of $(-,-)$ to $\mathbb{C}R$.
Fix any   element $\delta_0\in K$ 
and let $\cl_0$ be the canonical map  
$\cl_0: \mathbb{C}R\to \mathbb{C}R/\mathbb{C}\delta_0$. 
For example, if $R\subset \Delta^{\ree}(\fg)$, then 
for $\delta_0=\delta$ we have $\cl=\cl_0$.
For any 
$\alpha\in R$ we set
$$F(\alpha):=\{c\in\mathbb{C}|\ \alpha+c\delta_0\in R\}.$$
  Since $\mathbb{Z}R$ is a free module of finite rank,
  we have $F(\alpha)\subset \mathbb{Z}a$ for some $a\in\mathbb{C}$.

\subsubsection{}
\begin{lem}{lem:Falpha}
\begin{enumerate}
\item For all $\alpha\in R$ we have  $0\in F(\alpha)$, $F(-\alpha)=-F(\alpha)$
and $F(w\alpha)=F(\alpha)$ for all $w\in W[R]$.

\item  Let $\alpha\in R$ be such that $(\alpha,\alpha)\not=0$.
Then $F(\alpha)=\mathbb{Z}a$ for some $a\in\mathbb{C}$. Moreover, 
for $\gamma\in R$ the set $F(\gamma)$
has a structure of $\mathbb{Z}$-module with respect to the action
$j*c:=c-j\langle \gamma,\alpha^{\vee}\rangle a$
(for $j\in \mathbb{Z}$ and $c\in F(\gamma)$). In particular, 
$$\langle \gamma,\alpha^{\vee}\rangle F(\alpha)\subset F(\gamma)\ \ 
\text{ for all }\alpha,\gamma\in R\ \text{ with } (\alpha,\alpha)\not=0.$$
\item  If $\beta,\gamma\in R$ are such that 
$$(\beta,\beta)=0,\ \ \ (\gamma,\beta)\not=0,\ 
\ \ \bigl(\cl_0(\gamma)-\cl_0(\beta)\bigr)\not\in\cl_0(R),
$$
then $\beta+\gamma\in R$ and $F(\beta), F(\gamma)\subset F(\beta+\gamma)$.
\end{enumerate}
\end{lem}
\begin{proof}
(i) follows from (GR0) and (GR1) and (iii) follows from (GR2).
For
(ii) take $\alpha\in R$ such that $(\alpha,\alpha)\not=0$. For $x\in F(\alpha)$ we have
$x\delta_0-\alpha\in R$. Thus for 
$\gamma\in R$ we have 
$$s_{x\delta_0-\alpha}s_{\alpha} (\gamma)=\gamma-x\langle \gamma,\alpha^{\vee}\rangle \delta_0\in R,$$ 
so $c\in F(\gamma)$ forces $(c-x \langle \gamma,\alpha^{\vee}\rangle)\in F(\gamma)$.
In particular, fo all $x,c\in F(\alpha)$ we have $c-2x\in F(\alpha)$. Hence
 $F(\alpha)$ is a subgroup of $\mathbb{C}$.
  Since $R\subset \Delta$
  we have $F(\alpha)\subset \mathbb{Z}a$ for some $a\in\mathbb{C}$. 
  By above,
  $c-x\langle \gamma,\alpha^{\vee}\rangle\in R$. Since $F(\gamma)$
contains $0$ this gives $\langle \gamma,\alpha^{\vee}\rangle 
F(\alpha)\subset F(\gamma)$.
\end{proof}

\subsubsection{}
By (ii) we have $F(\alpha)=F(-\alpha)$ for a non-isotropic 
$\alpha\in R$.
In~$\S\S$~\ref{FA01}--\ref{FBC11} we will see that
the same formula holds for isotropic $\alpha$ if 
$\cl(R)\not=C(1|1)$, using that in all other cases  any isotropic 
root $\alpha$ can be included in $A(1|0)$, $B(1|1)$ or $BC(1|1)$.

\subsubsection{}\label{FA01}
Consider the case when 
$\cl_0 (R)$ is of type $A(1|0)$. Take $\alpha,\beta\in R$
such that $\cl(\alpha)=\vareps_1-\vareps_2$,
$\cl(\beta)=\vareps_2-\delta_1$. Then $\alpha\in R^{\an}$, 
$\beta\in R^{\is}$
and $s_{\alpha}\beta=\beta+\alpha$. Let $E$ be the 
span of $\alpha,\beta$. By~\Lem{lem:Falpha} (i), 
$R^{\an}=\{\pm \alpha+\mathbb{Z}a\delta_0\}$
for some $a\in\mathbb{C}$. Then
$W[R]\beta=
\{\beta+\mathbb{Z}a\delta_0,\beta+\alpha+\mathbb{Z}a\delta_0\}$.
The Gram matrix $\Gamma_{\alpha,\beta}$
is equal to the Gram matrix $\Gamma_{\cl(\alpha),\cl(\beta)}$
which is non-degenerate, so 
$E\cap\mathbb{C}\delta_0=0$ which forces $-\beta\not\in W[R]\beta$.
Using~\Lem{lem:chainiso} (ii) we obtain
$R^{\is}=-W[R]\beta\coprod W[R]\beta$, that is
$$R=\{\pm\alpha+\mathbb{Z}a\delta_0, \pm\beta+\mathbb{Z}a\delta_0,
\pm(\beta+\alpha)+\mathbb{Z}a\delta_0\}.$$
 Hence $E\cap R$
 can be identified with $A(1|0)$ and 
$R=(E\cap R)+\mathbb{Z}a\delta_0$ 
is isomorphic to  $A(1|0)$ if $a\delta_0=0$ and to 
$A(1|0)^{(1)}$  otherwise.

\subsubsection{}\label{FB11}
Consider the case when  $\cl_0 (R)$ is of type $B(1|1)$, i.e.
$$\cl_0 (R)=\{\pm \vareps_1, \pm\vareps_1\pm \delta_1, 
\pm\delta_1, \pm 2\delta_1\}.$$

Take $\alpha,\beta\in R$ such that  $\cl_0(\alpha)=\delta_1$ 
and  $\cl_0(\beta)=\vareps_1-\delta_1$. 
By~\Lem{lem:Falpha}  (iii), $R$ contains $\alpha+\beta$
and $F(\alpha), F(\beta)\subset F(\alpha+\beta)$. 
 Applying~\Lem{lem:Falpha} (iii) to the pair $-\beta,\alpha+\beta$
we obtain $F(\alpha+\beta)\subset F(\alpha)$, that is
$F(\alpha)=F(\alpha+\beta)$. By (GR1), $s_{\alpha}\beta=\beta+2\alpha\in R$ and
 $F(\beta+2\alpha)=F(\beta)$ (by~\Lem{lem:Falpha} (i)). 
Applying~\Lem{lem:Falpha} (iii) to the pair
$-\beta,\beta+2\alpha$ we obtain $2\alpha\in R$
and $F(\beta)\subset F(2\alpha)$.
 Using~\Lem{lem:Falpha} (ii)
we get $F(\beta)=F(2\alpha)=\mathbb{Z}c$ for some $c\in\mathbb{C}$.
By above, $R$ contains $\pm \beta$, $\pm \alpha$
$\pm 2\alpha$ and $\pm (\beta+2\alpha)$; these elements form $B(1|1)$
and we identify this subset of $R$ with $\cl_0(R)$.
By~\Lem{lem:Falpha} (ii),  
$F(\delta_1)=\mathbb{Z}a$ for some $a\in\mathbb{C}$.
Summarizing we have
$$F(\vareps_1)=F(\delta_1)=\mathbb{Z}a,\ \ 
F(\pm\vareps_1\pm\delta_1)=F(\vareps_1-\delta_1)=F(2\delta_1)=\mathbb{Z}c\subset \mathbb{Z}a.$$
By~\Lem{lem:Falpha} (ii), for any $b\in F(\vareps_1-\delta_1)$ 
we have $(b+2a)\in F(\vareps_1-\delta_1)$. 
Thus either $c=\pm a$ or $c=\pm 2a$ for $a\not=0$. In the first case  $R=B(1|1)+\mathbb{Z}a$, 
that is $R\cong B(1|1)$ or
$R\cong B(1|1)^{(1)}$; in the second case 
$R\cong D(2|1)^{(1)}$. 

\subsubsection{}\label{FBC11}
Consider the  case when  $R$ satisfies (GR2) and
$\cl_0 (R)$ is of type $BC(1|1)$ that is
$$\cl_0 (R)=\{\pm \vareps_1, \pm\vareps_1\pm \delta_1, 
\pm\delta_1, \pm 2\delta_1, \pm 2\vareps_1\}$$
where $(\vareps_1,\delta_1)=0$, $(\vareps_1,\vareps_1)=-(\delta_1,\delta_1)=1$.

Take $\beta,\alpha$ in $R$ such that $\cl_0(\beta)=\vareps_1-\delta_1$
and $\cl_0(\alpha)=\delta_1$. Notice that $R$ lies in the span
of $\alpha,\beta$ and $\delta_0$. We identify
$\alpha,\beta$
with $\delta_1$ and $\vareps_1-\delta_1$.
Then $R$ contains $\pm\delta_1$ and $\pm\vareps_1\pm\delta_1$.
By~\Lem{lem:Falpha} (i), $F(\delta_1)=\mathbb{Z}a$ for some $a\in\mathbb{C}$.
\Lem{lem:Falpha} (ii) gives $\vareps_1\in R$
and $F(\delta_1), F(\vareps_1-\delta_1)\subset F(\vareps_1)$.
Similarly, $F(\vareps_1)\subset F(\delta_1)$. Hence
$F(\vareps_1)=\mathbb{Z}a$ and  $F(\vareps_1-\delta_1)\subset \mathbb{Z}a$. By~\Lem{lem:Falpha} (i), 
$$F(\vareps_1-\delta_1)+2\mathbb{Z}a=F(\vareps_1-\delta_1).$$
Therefore 
\begin{equation}\label{eq:BC112a}
F(\pm\vareps_1\pm \delta_1)=F(\vareps_1-\delta_1)
\text{  is $\mathbb{Z}a$ or $2\mathbb{Z}a$},\ \ \ \ F(\vareps_1)=F(\delta_1)=\mathbb{Z}a.
\end{equation}
for some $a\in\mathbb{C}$. 
Since $R$ contains
$\vareps_1\pm\delta_1$ and satisfies (GR2), $R$ contains either
$2\delta_1$ or $2\vareps_1$. The case  $2\vareps_1\in R$ reduces to the case
$2\delta_1\in R$ by
 permuting the roles of $\vareps_1$ and $\delta_1$.
 Thus, without loss of generality, we can (and will) 
assume that $2\delta_1\in R$.  Applying~\Lem{lem:Falpha} (i)
to the pair $2\delta_1$ and $\delta_1$ we get
$4\mathbb{Z}a\subset F(2\delta_1)$. On the other hand, \Lem{lem:Falpha} (ii) gives  $F(2\delta_1)\subset F(\vareps_1+\delta_1)$.
Thus
\begin{equation}\label{eq:BC114a}
4\mathbb{Z}a\subset F(2\delta_1) \subset F(\vareps_1-\delta_1).
\end{equation}

Set $Y:=\{y\in\mathbb{C}|\ 2\vareps_1+y\delta_0\in R\}$.
Since $2\vareps_1\in \cl_0(R)$,  $Y$
is not empty. For $y\in Y$ one has
$$s_{2\vareps_1+y\delta}(\delta_1-\vareps_1)=\delta_1+\vareps_1+y\delta\in R$$
so $y\in  F(\vareps_1-\delta_1)$.
For each $x\in F(\vareps_1-\delta_1)$,
$R$  contains $\vareps_1-\delta_1$
and $\vareps_1+\delta_1+x\delta_0$, so, by (GR2),  the set
$\{2\vareps_1+x\delta_0,2\delta_1+x\delta_0\}\cap R$
has the cardinality one. Therefore 
$F(\vareps_1-\delta_1)=F(2\delta_1)\coprod Y$.  
By~\Lem{lem:Falpha} (i)  we have
$F(2\delta_1)=\mathbb{Z}c$
and $Y=y_0+\mathbb{Z}d$ for some $c,d,y\in\mathbb{C}$.
Applying~\Lem{lem:Falpha} (i)
to the pair $2\vareps_1+y\delta_0$ and $\vareps_1$ we get
$4\mathbb{Z}a\subset \mathbb{Z}d$.
Using~(\ref{eq:BC112a}) and~(\ref{eq:BC112a})
we obtain
$$4\mathbb{Z}a\subset \mathbb{Z}c,\ \ \  c,d,y_0\in\mathbb{Z}a,\ \ \ 
2a\in F(\pm\vareps_1\pm \delta_1)=\mathbb{Z}c\coprod (y_0+\mathbb{Z}d).$$ 
This gives the following possibilities:
$$\begin{array}{lllll}
(1) & F(2\delta_1)=2\mathbb{Z}a, & F(\vareps_1-\delta_1)=\mathbb{Z}a, &
 2\vareps_1+\delta_0\in R, & F(2\vareps_1+\delta_0)=2\mathbb{Z}a\\
(2) & F(2\delta_1)=4\mathbb{Z}a, & F(\vareps_1-\delta_1)=2\mathbb{Z}a, &
 2\vareps_1+2\delta_0\in R, & F(2\vareps_1+\delta_0)=4\mathbb{Z}a.\\
\end{array}$$
We have $R\cong\Delta^{\ree}(\fg')$, where 
$\fg'$ is of type $A(2|1)^{(2)}$ for (1),
and of type $A(2|2)^{(4)}$ for (2).

We will use the following observation: if $\gamma\in R$ is such that
$\cl(\gamma)=\pm 2\delta_1$,
then $\gamma/2\in R$ (this holds,
because $s/2\in F(\delta_1)$ for any $s\in F(2\delta_1)$).

\subsubsection{}\label{FAmn}
Let $\cl_0(R)$ be of the type $A(m-1|n-1)$ for $m,n\geq 1$
with $mn>1$, or $\Delta(\mathfrak{psl}(n|n))$ for $n\geq 3$.

We consider the distinguished base 
$\{\alpha'_i\}_{i=1}^{m+n-1}$
in $\cl_0(R)=A(m-1|n-1)$:
 $\alpha'_i:=\vareps_i-\vareps_{i+1}$
for $i=1,\ldots, m$,  $\alpha'_{m+j})=\delta_j-\delta_{j+1}$
for $j=1,\ldots,n$ and $\alpha'_m:=\vareps_m-\delta_1$.
Let $\alpha_i\in R$
be such that $\cl_0(\alpha_i)=\alpha'_i$.
Let $E$ be the span of $\{\alpha_i\}_{i=1}^{m+n-1}$.

Set $\dot{R}:=R\cap E$. Then $\dot{R}$
satisfies (GR0)--(GR2) and contains $\{\alpha_i\}_{i=1}^{m+n-1}$.
The Gram matrix of $\{\alpha_i\}_{i=1}^{m+n-1}$ is the
Gram matrix of $\{\alpha'_i\}_{i=1}^{m+n-1}$. Therefore
 $s_{\alpha_i}$ for $i=1,\ldots,m+n-1$, $i\not=m$,
generate the group isomorphic to $W[\cl_0(R)]$.
Since any element in $\cl_0(R)=A(m-1|n-1)$
is conjugated to one of the elements in
$\{\alpha'_i\}_{i=1}^{m+n-1}\cup\{-\alpha_m\}$,
the restriction of $\cl_0$ to $\dot{R}$
gives a surjective map $\dot{R}\to \cl_0(R)$.
If the restriction of
$\cl_0$ to $\mathbb{C}\dot{R}$ is injective, then
$\dot{R}\cong \cl_0(R)$.

For $m\not=n$, 
the Gram matrix of $\{\alpha_i\}_{i=1}^{m+n-1}$ is
non-degenerate, so  $\dot{R}\cong A(m-1|n-1)$.
For $m=n$ the Gram matrix of $\{\alpha_i\}_{i=1}^{m+n-1}$ 
has a one-dimensional kernel,
so $\dot{R}$ is isomorphic to 
$\Delta(\mathfrak{psl}(n|n))$ 
or to $A(n-1|n-1)=\Delta(\mathfrak{gl}(n|n))$.
Notice that if $R\subset \Delta^{\ree}(\fg)$, then
$\dot{R}\cong A(m-1|n-1)$, by~\Rem{rem:pslnn}.

\subsection{Set $\cl_K(R)$}\label{phi(R)}
From now on $\fg$ is  an indecomposable  symmetrizable  
finite-dimensional or affine
Kac-Moody superalgebra with the set of real roots
 $\Delta^{\ree}$.  If $\fg$ is finite-dimensional, we set $\delta:=0$.
In all cases 
$\cl: \mathbb{C}\Delta^{\ree}\to
\mathbb{C}\Delta^{\ree}/\mathbb{C}\delta$ is the canonical map.

From now on  $R\subset \Delta^{\ree}$ is  an 
 indecomposable set 
 satisfying (GR0)--(GR3) and  containing an isotropic root.
As before, $K=\{v\in \mathbb{C}R|\ (v,R)=0\}$ and
we let $\cl_K:\mathbb{C}R\to \mathbb{C}R/K$ be the canonical map.
The bilinear form $(-,-)$ 
induces a non-degenerate bilinear form on $\mathbb{C}R/K$. 

\subsubsection{}
\begin{lem}{lem:phi(R)}
Let $R\subset \Delta^{\ree}$ be  an 
 indecomposable set 
 satisfying (GR0)--(GR3). 
Then 
$\cl_K(R)\subset \mathbb{C}R/K$ is a finite indecomposable set
satisfying (GR0), (GR1), (WGR2) and (GR3).
\end{lem}
\begin{proof}
It is clear that $\cl_K(R)\subset \mathbb{C}R/K$ 
is an indecomposable set
satisfying (GR0), (GR1), (WGR2) and (GR3). It remains to verify
that $\cl_K(R)$ is finite. Suppose that $R$ is infinite.
Since $\cl(\Delta^{\ree})$ is a finite set, 
one has $\cl(\alpha)=\cl(\gamma)$
for some $\alpha,\gamma\in R$, so $\delta\in\mathbb{C}R$.
Then $\delta\in K$, so $\cl_K$ can be decomposed
as $\cl_K=\cl_K'\circ \cl$, where 
$\cl_K': \mathbb{C}R/\mathbb{C}\delta\to \mathbb{C}R/K$
is the canonical map. Since $\cl(R)\subset \cl(\Delta^{\ree})$ is a finite set, $\cl_K(R)$ is finite.
\end{proof}

 \subsubsection{}\label{phiRlist}
By~\Lem{lem:phi(R)}, 
$\cl_K(R)$ is one of the sets 
described in~$\S$~\ref{finR}, i.e.
 $\cl_K(R)$ lies in the list:
$$A(m-1|n-1) \text{ for }m\not=n,\ 
B(m|n), D(m|n), C(m|n), BC(m|n), D(2|1,a), G(3), F(4),$$
where $m,n\geq 1$,
or $\cl_K(R)\cong \Delta(\mathfrak{psl}(n|n))$ for $n\geq 2$.
(Recall that  $A(n-1|n-1)$
stands for $\Delta(\fgl(n|n))$).
For $D(2|1,a), G(3), F(4)$ we set $n=1$ and $m=1,2,3$ respectively.
If $\cl_K(R)\not=D(2|1,a), F(4)$, we view $\cl_K(R)$
as a subset of  $L_{m,n}$ (see~$\S$~\ref{basismn} for notation).

\subsection{Case $\cl_K(R)\not=C(1|1)$}\label{pfRsubsetDeltai}
Our goal is to show that in this case
$R\cong \Delta^{\ree}(\fg')$.
The cases when $\cl_K(R)$ is $A(1|0)$, $B(1|1)$, $BC(1|1)$
were treated in~$\S$~\ref{Falpha}. Now we consider the remaining cases,
i.e.  $mn\geq 2$ (and $mn\geq 3$ for $A(m|n)$, since $A(1|0)$
corresponds to $m=2$, $n=1$).

\subsubsection{}\label{ER'}
If $\cl_K(R)\not=D(2|1,a), F(4)$ we view $\cl_K(R)$
as a subset of  $L_{m,n}$ and set
$$E':=\{\sum_{i=1}^m x_i \vareps_i+\sum_{j=1}^n y_j \delta_j|\ \sum_{i=1}^m x_i+\sum_{j=1}^n y_j=0, x_i,y_j\in\mathbb{C}\}.$$
For $\cl_K(R)=D(2|1,a)$ take $E'$ spanned by two isotropic roots
in $\cl_K(R)$, and 
for $\cl_K(R)=F(4)$ take $E'$ spanned by  
 $(\delta_1-\vareps_1+\vareps_2-\vareps_3)/2$,
$\vareps_1-\vareps_2$  and $\vareps_2-\vareps_3$.

One readily sees that if $\cl_K(R)\not\cong \Delta(\mathfrak{psl}(n|n))$,
then 
$\cl_K(R)\cap E'\cong A(m-1|n-1)$
and this is a maximal root subsystem of type
$A(i|j)$ contained in $\cl_K(R)$.  For $\cl_K(R)\cong \Delta(\mathfrak{psl}(n|n))$
we have $\cl_K(R)\cap E'=\cl_K(R)$.

Let $\{\alpha'_i\}_{i=1}^{m+n-1}$
be as in~$\S$~\ref{FAmn}. Fix $\alpha_i\in R$
such that $\cl_K(\alpha_i)=\alpha'_i$
and let $E$ be the span of $\alpha_1,\ldots,\alpha_{m+n}$.
By~$\S$~\ref{FAmn}, $R\cap E$
contains $R'$ such that
$$R'\cong A(m-1|n-1).$$

\subsubsection{}
\begin{lem}{lem:mngeq2}
We have
\begin{enumerate}
\item $K\cap \mathbb{C}R'=0$
 if 
$\cl_K(R)\not\cong
 \Delta(\mathfrak{psl}(n|n))$ for $n\geq 3$;
% \item $\mathbb{C}\delta\cap \mathbb{C}R'=0$ 
%if $\cl_K(R)\cong  \Delta(\mathfrak{psl}(n|n))$;
 \item $\{\gamma\in R|\ 
 \cl_K(\gamma)\in E'\}=R'+\mathbb{Z}a\delta\ \ \text{ for some }a\in\mathbb{Z}$.
\end{enumerate}
\end{lem}
\begin{proof}
For $m\not=n$, $R'$ isomorphic to
$A(m-1|n-1)$ and the restriction of $(-,-)$ to
$R'$ is non-degenerate, so $K\cap \mathbb{C}R'=0$.
Consider the case $m=n$. Let $\beta_i\in R'$ be such that
$\cl_K(\beta_i)=\vareps_i-\delta_i$. 
The  kernel of the restriction of $(-,-)$ to $\mathbb{C}R'$
is spanned by 
 $\mu:=\sum\limits_{i=1}^n \beta_i$. 
 Therefore $K\cap \mathbb{C}R'=K\cap \mathbb{C}\mu$.
Hence the restriction of $\cl_K$ to $R'$ is injective
(since $n>1$, so $\alpha-\beta\not\in\mathbb{C}\mu$
if $\alpha,\beta\in R'$ and $\alpha\not=\beta$). 

 If 
 $\cl_K(R)\not\cong \Delta(\mathfrak{psl}(n|n))$, then
 $\cl_K(R)$ is $D(n|n)$ or $B(n|n)$, so
 $2\delta_1\in \cl_K(R)$. Taking
 $\gamma\in R$ such that $\cl_K(\gamma)=2\delta_1$
 we obtain $(\gamma,\beta_i)=(2\delta_1,\vareps_1-\delta_1)\not=0$
and $(\gamma,\beta_i)=0$ for $i>1$, that is 
$(\mu, \gamma)\not=0$, so  $\mu\not\in K$.
Hence $\cl_K(R)\not\cong \Delta(\mathfrak{psl}(n|n))$ forces
 $K\cap \mathbb{C}R'=0$. This completes the proof of (i).

 For (ii) recall that $\cl_K(R)\cap E'$
is isomorphic to  $A(m-1|n-1)$ 
if $\cl_K(R)\not\cong \Delta(\mathfrak{psl}(n|n))$
and $\cl_K(R)\cap E'=\cl_K(R)$ if  
$\cl_K(R)\cong \Delta(\mathfrak{psl}(n|n))$.
Since $R'\cong A(m-1|n-1)$, we have $\cl_K(R')=\cl_K(R)$.
 Assume that 
 $\{\gamma\in R|\ \cl_K(\gamma)\in \cl_K(R')\}\not=R'$. 
 Take $\gamma\in R$
 such that $\cl_K(\gamma)\in E'$ and $\gamma\not\in R'$.
Since $\cl_K(R')=\cl_K(R)$,
 $\cl_K(\gamma)=\cl_K(\alpha)$ for some $\alpha\in R'$.
 Set  $\delta_0:=\gamma-\alpha$ and define 
 $F(\alpha)$ as in~$\S$~\ref{Falpha}.  Since $\gamma\in R$
 we have $1\ F(\alpha)$, so
 $\mathbb{Z}\subset F(\alpha)$ by $\S$~\ref{FA01}.
 Therefore $\cl(\alpha+\mathbb{Z}\delta_0)=
 \cl(\alpha)+\mathbb{Z}\cl(\delta_0)$
lies in $\cl(R)\subset\cl(\Delta^{\ree})$ which is finite.
 Hence $\cl(\delta_0)=0$, so $\delta_0\in\mathbb{C}\delta$.
This establishes (ii).
\end{proof}

\subsubsection{} 
Consider the case $\cl_K(R)=A(m-1|n-1)$. 
Since $R'\cong A(m-1|n-1)$ and $mn\geq 2$,
the restriction of $\cl_K$ 
to $R'$ is injective.
Since $R'$ and $\cl_K(R)$ contain 
the same number of elements, we have
$\cl_K(R)=\cl_K(R')$.  By~\Lem{lem:mngeq2} (ii),
 we obtain
$R=R'+\mathbb{Z}j\delta$, so 
$R$ is isomorphic to $\Delta^{\ree}(\fg')$
for $\fg'$ of type $A(m-1|n-1)$.

\subsubsection{}\label{dotR}
Now consider the cases when $\cl_K(R)\not\cong A(m-1|n-1)$.

Retain notation of~\ref{ER'}. Recall that $R'\cong A(m-1|n-1)$.
By~\Lem{lem:mngeq2} (i), $K\cap \mathbb{C}R'=0$, so the span of
$\cl_K(R')$ is of dimension $m+n-1$. Since 
$\cl_K(R)\not\cong A(m-1|n-1)$, 
 the span of $\cl_K(R)$ is of  dimension $m+n$.
In each case we  choose $\alpha'_{m+n}\in \cl_K(R)$ 
which does not lie in $\cl_K(R')$:
if $\cl_K(R)$ is $B(m|n)$ or $BC(m|n)$ 
we take $\alpha'_{m+n}:=\delta_n$;
in other cases we take $\alpha'_{m+n}$ arbitrary.

Then we choose $\alpha_{m+n}$
such that $\cl_K(\alpha_{m+n})=\alpha'_{m+n}$.
We denote by $V$ the span of $\{\alpha_i\}_{i=1}^{m+n}$.
We set 
$$\dot{R}:=V\cap R.$$

Since $\{\alpha'_i\}_{i=1}^{m+n}$ is a basis of $\cl_K(R)$
we identify $\mathbb{C}R/K$ with $V$.  Then
$\dot{R}\subset \cl_K(R)$.

We have $\dim V=m+n$ and 
$\dot{R}\subset V$ satisfies (GR0)--(GR2) 
(since $R$ satisfies these properties). 
The Gram matrices of $\{\alpha_i\}_{i=1}^{m+n}$
and of $\{\alpha'_i\}_{i=1}^{m+n}$ coincide
and the latter is the Gram matrix of a distinguished base
of a root system which is not of type $A(m-1|n-1)$.
Thus this matrix is invertible, so the restriction
of $(-,-)$ to $V$ is non-degenerate. By Serganova's classification ($\S$~\ref{finR})
$\dot{R}$ is the root system of a finite-dimensional Kac-Moody 
superalgebra. Recall that 
\begin{equation}\label{eq:dotRphiR}
\dot{R}\subset R,\ \ A(m-1|n-1)\cong R'\subsetneq \dot{R}\subset \cl_K(R),\ \ \ 
\mathbb{C}\dot{R}=\mathbb{C}\cl_K(R)
\end{equation}
and  that $\cl_K(R)$ is not  the types 
$A(i|j)$ or $\Delta(\mathfrak{psl}(i|i))$ for all $i,$.
If $\dot{R}\not=\cl_K(R)$, then $\dot{R}\subset\cl_K(R)$
is one of the following inclusions:
$$D(m|n)\subset B(m|n),\ \ D(m|n)\subset C(m|n),\ \ D(m|n)\subset BC(m|n), \ \ 
B(m|n)\subset BC(m|n).$$

If $\cl_K(R)$ is $B(m|n)$ or $BC(m|n)$, $\dot{R}$
contains $\alpha_{m+n}=\delta_n$, so $\dot{R}$
is not of the type $D(m|n)$.
Hence $\cl_K(R)=\dot{R}$
if $\cl_K(R)\not=C(m|n)$, $BC(m|n)$,
$\dot{R}=D(m|n)$  if $\cl_K(R)\not=C(m|n)$, 
and  $\dot{R}=B(m|n)$  if $BC(m|n)$.

\subsubsection{}\label{FalphadotR}
Take a non-zero $\delta_0\in K$.
By~\Lem{lem:chainiso}, any isotropic root in $\beta\in\cl_K(R)$
is $W[\dot{R}]$-conjugate to a root
in $R'$. By~\Lem{lem:Falpha}, $F(\beta)=0$ if $\delta_0\not\in\mathbb{C}\delta$ 
and
$F(\beta)=\mathbb{Z}a$  for $\delta_0=\delta$, where $a$
is as in~\Lem{lem:mngeq2} (ii).

If $\gamma\in R$ is anisotropic, then, by~\Lem{lem:Falpha},
$F(\gamma)=\mathbb{Z}c\delta_0$ for some $c\in\mathbb{C}$.
Since $\cl(R)\subset \cl(\Delta^{\ree})$
is finite, $\cl(c\delta_0)=0$. 
We conclude that $F(\alpha)=0$ for each $\alpha\in R$ if
$\delta_0\in\mathbb{C}\delta$. Since $\cl_K(R)\subset V=\mathbb{C}\dot{R}$
we obtain
$$R\subset (V+\mathbb{C}\delta).$$

\subsubsection{}
\begin{cor}{}
Let $\dot{\fg}'$
be the finite-dimensional Kac-Moody superalgebra
with the root system $\dot{R}$.
If $\dot{R}=\cl_K(R)$, then
$R\cong \Delta^{\ree}({\fg'})$, where $\fg'$ is either
$\dot{\fg}'$, or 
 the affinization of $\dot{\fg}'$,
or $\fg'\cong D(m+1|n)^{(2)}$.
\end{cor}
\begin{proof}
By~\Lem{lem:chainiso}, $\dot{R}$ contains an isotropic
$\beta$ such that $(\beta,\gamma)\not=0$.
By~$\S\S$~\ref{FA01}, \ref{FB11} we have
$c=\pm  a$ if  $\cl_K(R)\cap (\mathbb{C}\beta+\mathbb{C}\gamma)$
is of type $A(1|0)$, and $c\in \{\pm a, \pm a/2\}$
if this is of type $B(1|1)$.

If $c=\pm a$, then
$R=\dot{R}+\mathbb{Z}a\delta$, and 
$R=\dot{R}\cong \Delta^{\ree}(\dot{\fg'})$
if $a=0$ and $R\cong \Delta^{\ree}({\fg'})$
where $\fg'$ is the affinization of $\dot{\fg}'$
if $a\not=0$.  

If $a\not=0$ and
$c\not=\pm a$, then $\dot{R}=B(m|n)$ or $G(3)$. 
In this case~$\S$~\ref{FB11} gives 
$F(\alpha_s)=\mathbb{Z}\frac{a}{2}$ if $\alpha_s\in \dot{R}$
is a short root and $F(\alpha)=\mathbb{Z}a$ 
otherwise. If  $\dot{R}=\cl_K(R)$
is of  type $B(m|n)$, this gives  $R\cong \Delta^{\ree}({\fg'})$
where $\fg'=D(m+1|n)^{(2)}$.
If $\dot{R}=G(3)$, then, for a long root $\alpha_l$
in $G(2)$ we have $\langle \alpha_l,\alpha_s^{\vee}\rangle=\pm 3$, 
so~\Lem{lem:Falpha}
(ii) gives  $3 F(\alpha_s)\subset F(\alpha_l)$
which contradicts to $F(\alpha_s)=\mathbb{Z}\frac{a}{2}$
and $F(\alpha_l)=\mathbb{Z}a$. 
\end{proof}

\subsubsection{Case $C(m|n)$ with $mn\geq 2$}\label{Cmn}
Without loss  of generality we can (and will) assume $n\geq 2$.
Recall that $R\subset V\oplus\mathbb{C}\delta$.
Set $\delta_0:=\delta$.
By~\Lem{lem:Falpha},  $F(\alpha)=\mathbb{Z}a$
for any $\alpha\in W[\dot{R}]R'$ that is 
for all $\alpha\in\dot{R}$ with $\alpha\not=\pm 2\delta_j$
for $i=1,\ldots,m$.

Set 
$$\beta=\delta_1-\vareps_1.$$

Applying~\Lem{lem:Falpha} (ii) to the pair $\beta$
and $2\delta_1$,
we obtain $F(2\delta_1)\subset \mathbb{Z}a$.
Since $\langle 2\delta_1, (\delta_1-\delta_2)^{\vee}\rangle=2$
we get $2\mathbb{Z}a\subset F(2\delta_1)$.
Hence $F(2\delta_1)$ is $\mathbb{Z}a$ or $2\mathbb{Z}a$.

Set 
$$Y:=\{y\in \mathbb{C}|\ 2\vareps_1+y\delta\in R\}.$$
Since $\cl_K(R)=C(m|n)$ and $R\subset V+\mathbb{C}\delta$, $Y$
is not empty. For $y\in Y$ one has
$$s_{2\vareps_1+y\delta}(\beta_1)=\beta_1+(2\vareps_1+y\delta)=\delta_1+\vareps_1+y\delta\in R$$
so $y\in\mathbb{Z}a$ (since $F(\delta_1+\vareps_1)=\mathbb{Z}a$).
Thus $Y$ is a non-empty subset of $\mathbb{Z}a$. 
For each $i\in\mathbb{Z}$,
$R$  contains $\vareps_1-\delta_1$
and $\vareps_1+\delta_1+ai\delta$, so  the set
$\{2\vareps_1+ai\delta,2\delta_1+ai\delta\}\cap R$
has the cardinality one. Therefore 
$$\mathbb{Z}a=F(2\delta_1)\coprod Y.$$ 
Since $F(2\delta_1)$ is $\mathbb{Z}a$ or $2\mathbb{Z}a$,
and $Y$ is non-empty, we obtain 
 $F(2\delta_1)=2\mathbb{Z}a$, 
 $Y=a+2\mathbb{Z}a$ and $a\not=0$.

 Acting by $W[\dot{R}]$
 we obtain 
 $$F(2\delta_j)=2\mathbb{Z}a,\ \ \  \
 \{y\in \mathbb{C}|\ 2\vareps_i+y\delta\in R\}a+2\mathbb{Z}a$$
 for  $j=1,\ldots,n$
 and $i=1,\ldots,m$. This gives 
 $$R=\{\alpha+\mathbb{Z}a\delta|\ \alpha\in D(m|n),\ \alpha\not=\pm 2\delta_j\}\cup \{\pm 2\delta_j+\mathbb{Z}a\delta,
 \pm 2\vareps_i+a\delta+2\mathbb{Z}a\delta\}$$
 (where $i=1,\ldots,m$ and 
 $j=1,\ldots,n$). Hence
 $R\cong \Delta^{\ree}(\fg')$
for $\fg'=A(2m-1|2n-1)^{(2)}$.

\subsubsection{Case $BC(m|n)$ with $mn\geq 2$}
Without loss  of generality we can (and will) assume $n\geq 2$.
By~$\S\S$~\ref{dotR},\ref{FalphadotR}, $R$ contains $\dot{R}\cong B(m|n)$
and $R\subset (\mathbb{C}\dot{R}\oplus\mathbb{C}\delta)$.
Set $\delta_0:=\delta$. Using~\Lem{lem:mngeq2} (ii)
and the fact
that $F(w\alpha)=F(\alpha)$ for 
$w\in W[R]$ and $\alpha\in R$ 
we get
$$F(\pm\vareps_i\pm \delta_j)=F(\pm \vareps_i\pm \vareps_{i_1})=F(\pm\delta_j\pm\delta_{j_1})=\mathbb{Z}a$$
for $1\leq i\not=i_1\leq m$ and $1\leq j\not= j_1\leq n$.
Now~$\S$~\ref{FBC11} gives the following possibilities 
$$\begin{array}{lllll}
(1) & F(2\delta_j)=2\mathbb{Z}a,  & F(\vareps_i)=F(\delta_i)=\mathbb{Z}a
& 
 2\vareps_i+\delta\in R, & F(2\vareps_i+\delta)=2\mathbb{Z}a\\
(2) & F(2\delta_j)=2\mathbb{Z}a,  &  F(\vareps_i)=F(\delta_j)=\mathbb{Z}a/2 &
 2\vareps_i+2\delta\in R, & F(2\vareps_i+\delta)=2\mathbb{Z}a.\\
\end{array}$$
Hence $R\cong\Delta^{\ree}(\fg')$, where 
$\fg'$ is of type $A(2m+1|2n)^{(2)}$ for (1),
and of type $A(2m|2n)^{(4)}$ for (2).

\subsection{Case $\cl_K(R)=C(1|1)$}\label{C11}
Take $\alpha_1$, $\alpha_2$, and $\beta$ in $R$ such that $\cl_K(\alpha_1)=-2\vareps_1$, 
$\cl_K(\alpha_2)=2\delta_1$, and 
$\cl_K(\beta)=\vareps_1-\delta_1$. 
We normalize the form $(-,-)$ by the condition
$(\alpha_1,\alpha_1)=2$.
Let $V$ be the span of $\alpha_1$, $\alpha_2$ and $\beta$. Set
$$\xi:=2\beta+\alpha_1+\alpha_2.$$ 

\subsubsection{}
\begin{lem}{lem:C11}
The vectors $\alpha_1$, $\alpha_2$, $\beta$ are linearly independent.
If $\xi$ is not proportional to $\delta$, then 
$\alpha_1$, $\alpha_2$,
$\beta$ and $\delta$ are linearly independent. 
\end{lem}
\begin{proof}
Since $R$ contains $\beta$ and $s_{\alpha_2}\beta=\beta+\alpha_2$,
one has $2\beta+\alpha_2\not\in R$  (by (GR2)), so
$\xi\not=0$.

Notice that the kernel of the Gram matrix of $\alpha_1,\alpha_2,\beta,\delta$
is spanned by $\xi$ and $\delta$. Since $\xi\not=0$, the vectors
$\alpha_1,\alpha_2,\beta$ are linearly independent.
If $\xi$ is not proportional to $\delta$, then 
$\alpha_1$, $\alpha_2$,
$\beta$ and $\delta$ are linearly independent. 
\end{proof}

\subsubsection{}\label{RC11}

The non-isotropic elements in $\cl_K(R)$ are $\pm \cl_K(\alpha_i)$
for $i=1,2$. Using~\Lem{lem:Falpha} we obtain
\begin{equation}\label{eq:C11an}
R^{\an}=\{\pm \alpha_1+\mathbb{Z}a_1\delta\}\coprod
\{\pm \alpha_1+\mathbb{Z}a_1\delta\}
\end{equation}
for some $a_1,a_2\in\mathbb{Z}_{\geq 0}$.Let $W'$ be the image of $W[R]$ in $GL(\mathbb{C}R)$.
By~\Lem{lem:chainiso} (ii), 
$$R^{\is}=-W'\beta\cup W'\beta.$$

Since $s_{\alpha_i}s_{\alpha_i+s\delta}\beta=\beta+s\delta$
for any $s\in\mathbb{Z}$ 
we have
\begin{equation}\label{eq:W'beta}
W'\beta=\{\beta+\mathbb{Z} a\delta, \beta+\alpha_1+\mathbb{Z} a\delta,\beta+\alpha_2+\mathbb{Z}a\delta, \beta+\alpha_1+\alpha_2+\mathbb{Z}a\delta\},\end{equation}
where $a\in\mathbb{Z}_{>0}$ is 
such that $\mathbb{Z}a_1+\mathbb{Z}a_2=\mathbb{Z}a$
(i.e., $a=0$ if $a_1=a_2=0$, and 
$a$ is the greatest common divisor of $a_1,a_2$ otherwise).
% In particular, $\mathbb{C}R$ is spanned by $\alpha_1$, $\alpha_2$,
%$\delta'$ and $a\delta$.

Since $R$ contains $\beta+\alpha_1$, $\beta+\alpha_2$ and
$\beta+aj\delta\in R$
for any $j\in\mathbb{Z}$, (GR2) gives 
 \begin{equation}\label{eq:C11eq2}
|\{\alpha_1-aj\delta,
\xi-\alpha_2+aj\delta\}\cap R=
 |\{\alpha_2-aj\delta,\xi-\alpha_1+aj\delta\}\cap R |=1
 \end{equation}
 where $|X|$ stands for the cardinality of $X$.

\subsubsection{Case $\xi\not\in\mathbb{C}\delta$}\label{xinot}
By~\Lem{lem:C11}, $\beta$, $\alpha_1$, $\alpha_2$, $\delta$
are linearly independent. This implies
 $W'\beta\cap (-W'\beta)=\emptyset$.
Moreover, $\xi-\alpha_i+aj\delta\not\in R$ by~(\ref{eq:C11an}) 
for $i=1,2$. Combining~(\ref{eq:C11an}) and~(\ref{eq:C11eq2}) we conclude
 $a\in\mathbb{Z}a_i$ for $i=1,2$, that is $a=a_1=a_2$.
 Hence
 $$R=\{\pm \alpha_i+\mathbb{Z}a\delta,\ \pm\beta+a\delta, \ \pm(\beta+\alpha_i+\mathbb{Z} a\delta),\ \pm(\beta+\alpha_1+\alpha_2)+\mathbb{Z}a\delta\}_{i=1,2}.$$
 Since $\beta$, $\alpha_1$, $\alpha_2$, $\delta$ are linearly independent, the map given by $\alpha_1\mapsto \vareps_1-\vareps_2$,
 $\beta\mapsto \vareps_2-\delta_1$, 
 $\alpha_2\mapsto \delta_1-\delta_2$,  $a\delta\mapsto \delta$
induces the isomorphism 
 $R\cong \Delta(\fgl(2|2))=A(1|1)$
 if $a=0$ and $R\cong A(1|1)^{(1)}$ if $a\not=0$.

 Note that $\cl(R)\cong A(1|1)$ ($\cong \Delta(\fgl(2|2))$).

  \subsubsection{Case $a=0$}
  In this case $a_1=a_2=0$.
  By~\Lem{lem:C11}, $\beta$, $\alpha_1$, $\alpha_2$ 
are linearly independent. This implies
 $W'\beta\cap (-W'\beta)=\emptyset$ and
 $$R=\{\pm \alpha_i, \pm\beta,   \pm(\beta+\alpha_i),  \pm(\beta+\alpha_1+\alpha_2)\}_{i=1,2}\cong A(1|1),$$
where the isomorphism is the same as in~$\S$~\ref{xinot}.
Note that $\cl(R)\cong A(1|1)$ ($\cong \Delta(\fgl(2|2))$).

\subsubsection{Case $\xi\in\mathbb{C}\delta$, $a\not=0$}
Then $\xi=d\delta$ for some  $d\in\mathbb{Z}_{\not=0}$. 

Recall that $a_1,a_2\in\mathbb{Z}_{\geq 0}$, $a_1+a_2\not=0$,
and $a$
is the greatest common divisor of $a_1,a_2$.
 Combining~(\ref{eq:C11an}) and~(\ref{eq:C11eq2}) we conclude
 that for each $j\in\mathbb{Z}$ either $aj\in\mathbb{Z}a_2$
or $aj+d\in\mathbb{Z}a_1$, and, similarly, 
 either $aj\in\mathbb{Z}a_1$
or $aj+d\in\mathbb{Z}a_2$. Taking $j=0$ we obtain
$d\not\in \mathbb{Z}a_i$ for $i=1,2$.
Then $aa_1+d\not\in \mathbb{Z}a_1$, so $aa_1\in\mathbb{Z}a_2$.
Similarly, $aa_2\in\mathbb{Z}a_1$. Therefore $a_1,a_2\not=0$,
so $k_1:=\frac{a_1}{a}$, $k_2:=\frac{a_2}{a}$ are coprime
positive integers.
By above, $ak_1|k_2$ and $ak_2|k_1$, which implies $k_1=k_2=1$.
Hence $a_1=a_2=a$ and $a\not|d$.

Let us show that
$-W'\beta\cap W'\beta=\emptyset$.
We will use formula~(\ref{eq:W'beta}).
Take $\gamma\in -W'\beta\cap W'\beta$.
Then $\cl(\gamma)=\cl(w\beta)$ for some $w\in W'$
and $w^{-1}\gamma\in -W'\beta\cap W'\beta$.
Thus, without loss of generality, we can assume that 
 $\cl(\gamma)=\cl(\beta)$. In this case
$$\gamma=\beta+j_1a\delta=-\beta+\alpha_1+\alpha_2+j_2a\delta$$
so 
$2\beta+\alpha_1+\alpha_2=d\delta=(j_2-j_1)a\delta$ which contradicts to $a\not|d$. 

In the light of~$\S$~\ref{RC11} we conclude that $R=R(a,d)$
where
\begin{equation}\label{eq:Rad}
R(a,d)=\{\pm \alpha_i+\mathbb{Z}a\delta,\ \pm\beta+a\delta, \ \pm(\beta+\alpha_i+\mathbb{Z} a\delta),\ \pm(\beta+\alpha_1+\alpha_2)+\mathbb{Z}a\delta\}_{i=1,2}
\end{equation}
where $\beta=\frac{d\delta-\alpha_1-\alpha_2}{2}$
and $a\not|d$. One has $R(a,d)\cong R(a,d\pm 2a)$
(the isomorphism is given by $\beta\mapsto \beta+a\delta$),
 $R(a,d)\cong R(a,-d)$ (the isomorphism is given by 
$\beta\mapsto -\beta$, $\alpha_i\mapsto -\alpha_i$),
so we can assume that $a\geq 2d$.

One has 
$R(2d,d)\cong A(1|1)^{(2)}$ (with the isomorphism given
by  $\alpha_1\mapsto \delta-2\vareps_1$, $\beta\mapsto \vareps_1-\delta_1$ and $\alpha_2\mapsto 2\delta_1$).

For $a\not=2d$, one has $R(a,d)\not\cong A(1|1)^{(2)}$
(and  $R$ is not a bijective quotient of $\Delta^{\ree}(\fg')$
for any $\fg'$), since 
for any isotropic root $\gamma\in R$ and any non-zero $\delta_0\in K$
we have 
$$F(\gamma)=\mathbb{Z}xa\cup (\mathbb{Z}a-d)x,$$
where $\delta_0=x\delta$, and $F(\gamma)$
forms a subgroup of $\mathbb{C}$
if $a=2d$ (or if $R$ is  a bijective quotient of $\Delta^{\ree}(\fg')$),
and does not form a subgroup if $a\not=2d$.

 \subsubsection{}
 \begin{cor}{}
Let $R\subset \Delta^{\ree}(\fg)$ be such that $\cl_K(R)\cong C(1|1)$
($\cong \Delta(\mathfrak{psl}(2|2))$. Then we have the following possibilities:
$$\begin{array}{lll}
R\cong A(1|1) (=\Delta(\fgl(2|2)) & \cl(R)=R,\\
 R\cong A(1|1)^{(1)} & \cl(R)\cong A(1|1)\\
 R\cong A(1|1)^{(2)} & \cl(R)=\cl_K(R)\cong C(1|1)\\
  R\cong R(a,d), a,d\in \mathbb{Z}_{>0}, a>2d & \cl(R)=\cl_K(R)\cong C(1|1).
  \end{array}$$
  \end{cor}

This completes the proof of~\Prop{prop:RsubsetDelta} (i).

\subsection{Proof of~\Prop{prop:RsubsetDelta} (ii)}\label{DeltalambdaA112}
Now assume that $\Delta_{\lambda}$ has an indecomposable component
$R$ such that $\cl(R)=C(1|1)$.

\subsubsection{}
Since $\cl(R)$ does not satisfy (GR2), 
$\cl(\Delta^{\ree})$ does not satisfy (GR2), so
$\cl(\Delta^{\ree})$ is isomorphic either to $BC(m|n)$
for $mn\geq 1$ or to $C(m|n)$ for $mn>1$, and 
$$\cl(R)=\{\pm 2\vareps_i,\pm \vareps_i\pm \delta_j,\pm 2\delta_j\}$$
for some $i, j$ ($1\leq i\leq m$, $1\leq j\leq n$).

Take $\gamma\in R$ such that  $\cl(\gamma)=2\delta_j$.
Note that $\gamma/2\not\in R$ (since $\cl(\gamma/2)=\delta_i\not\in\cl(R)$), so
 $\gamma/2\not\in \Delta_{\lambda}$.
By~\Cor{cor:Delta2beta}, $\gamma/2\not\in\Delta^{\ree}$.

If $\cl(\Delta^{\ree})\cong BC(m|n)$, then
$\cl(\mathbb{C}R\cap \Delta^{\ree})\cong BC(1|1)$, so
$\mathbb{C}R\cap \Delta^{\ree}$ is  one of the real root systems 
(i.e., $A(2|1)^{(2)}$, $A(2|2)^{(4)}$)
described in~$\S$~\ref{FBC11}. For these root 
system for any root $\gamma$ satisfying $\cl(\gamma)=2\delta_j$, $\gamma/2$ is a root 
(see the observation in the end of~$\S$~\ref{FBC11}), 
so $\gamma/2\in \Delta^{\ree}$, a contradiction.

Hence $\cl(\Delta^{\ree})=C(m|n)$,  so, by~$\S$~\ref{Cmn},
$\fg$ is of the type $A(2m-1|2n-1)^{(2)}$.

\subsubsection{}
Set  $k:=(\lambda,\delta)$. 
Since $(\alpha,\alpha)\in \{0,\pm 2\}$ for any $\alpha\in R$, 
we have $(\lambda,\alpha)\in\mathbb{Z}$ for all $\alpha\in R_{\lambda}$,
so $(\lambda,\alpha)\in\mathbb{Z}$ for all $\alpha\in R$.

Take a non-isotropic $\alpha\in R$.
Then $(\alpha,\alpha)=\pm 2$, so 
$\alpha+s\delta\in \Delta_{\lambda}$ 
if and only if $\alpha+s\delta\in \Delta^{\ree}$ 
and $(\lambda,\alpha+s\delta)\in\mathbb{Z}$.
The first condition gives $2|s$ and the second condition
gives $sk\in\mathbb{Z}$.

Recall that $R=R(a,d)$ for some $a,d\in\mathbb{Z}_{>0}$,
$a\geq 2d$. By~(\ref{eq:Rad}), $\alpha+s\delta\in R$
if and only if $a|s$. 
We conclude that $a|s$ is equivalent to $2|s$
and $sk\in\mathbb{Z}$. Therefore $a$ is even and
$k=\frac{p}{q}$ where $p,q$ are coprime and either $q$ is odd and
$a=2q$ or $q$ is even and $a=q$.
Since $d\delta\in\mathbb{Z}R$ we have 
$(\lambda,d\delta)=dk\in\mathbb{Z}$. Thus $d$ is divisible by $q$.
Hence $a=2q$ and $d=q$. Thus $R(a,d)\cong A(1|1)^{(2)}$
and $q$ is odd. This completes the proof of~\Prop{prop:RsubsetDelta} (ii).

\end{document}